\documentclass[letterpaper,11pt,reqno]{amsart}
\usepackage{mlmodern}
\usepackage[utf8]{inputenc}
\usepackage[T1]{fontenc}
\usepackage{amsmath, bm}
\usepackage{amssymb}
\usepackage{graphicx}
\usepackage[counterclockwise]{rotating}
\usepackage{mathtools}

\usepackage{enumerate}
\usepackage{amsfonts}
\usepackage{dirtytalk}
\usepackage{chngcntr}
\counterwithin{figure}{section}
\usepackage{xy}
\usepackage{amsthm}
\usepackage{xfrac}
\usepackage[dvipsnames]{xcolor}

\usepackage[obeyspaces,hyphens,spaces]{url}
\makeatletter

\usepackage[
backend=bibtex,
giveninits=true,
style=alphabetic,
url=false, isbn=false, maxbibnames=99,mincitenames=1, maxcitenames=4]{biblatex}
\renewbibmacro{in:}{}

\DeclareFieldFormat[article,periodical]{volume}{\mkbibbold{#1}} 

\usepackage[urlcolor=Magenta, 
linkcolor=red, 
citecolor=Green,
colorlinks=true,
linktocpage=true]{hyperref}
\usepackage{xr-hyper}

\usepackage{tikz}
\usepackage{tikz-cd}
\tikzstyle{lien}=[->,>=stealth,rounded corners=5pt,thick]
\usepackage{geometry}
\usepackage{cleveref}

\crefformat{section}{\S#2#1#3} 
\crefformat{subsection}{\S#2#1#3}
\crefformat{subsubsection}{\S#2#1#3}

\newtheorem{thm}{Theorem}[section]
\newtheorem{prop}[thm]{Proposition}
\newtheorem{lem}[thm]{Lemma}
\newtheorem{cor}[thm]{Corollary}

\theoremstyle{definition}
\newtheorem{defi}[thm]{Definition}
\newtheorem{rem}{Remark}[section]
\newtheorem{ex}{Example}[section]

\newtheorem{assum}{Assumption}

\numberwithin{equation}{section}

\usepackage{import}
\usepackage{graphicx}
\graphicspath{{images/}}

\usepackage{blindtext}

\author{Julien Moy}
\address{Laboratoire de Mathématiques d'Orsay, Universit{\'e} Paris-Saclay, 91405 Orsay cedex, France}
\email{julien.moy@universite-paris-saclay.fr}
\title[Trace formula for perturbations of Schr\"odinger operators]{A long time Gutzwiller trace formula for perturbations of Schr\"odinger operators near Anosov energy levels}

\def\lesssimA#1#2{\mathrel{ \vcenter{\offinterlineskip%
			\ialign{\hfil##\hfil\cr \noalign{\vskip 0.15ex} $#1<$\cr \noalign{\vskip 0.25ex} $#1\sim$\cr}%
}}}

\def\lesssim{\mathpalette\lesssimA{}}

\def\gtrsimA#1#2{\mathrel{ \vcenter{\offinterlineskip%
			\ialign{\hfil##\hfil\cr \noalign{\vskip 0.15ex} $#1>$\cr \noalign{\vskip 0.25ex} $#1\sim$\cr}%
}}}

\def\gtrsim{\mathpalette\gtrsimA{}}
\begin{document}
	\maketitle
	
	\begin{abstract} Let $P_h$ be a semiclassical pseudodifferential operator on a closed Riemannian manifold $(M,g)$, with principal symbol $p$. We consider perturbations $P_h+h^\alpha Q$, with $0<\alpha\le 1$, where $Q$ is a semiclassical pseudodifferential operator whose symbol oscillates at scale $h^\beta$, with $2\beta <\alpha$. 
		
		Assuming that the classical dynamics of the unperturbed system on a regular energy surface $p^{-1}(E)$ is hyperbolic, we prove a Gutzwiller trace formula for the perturbed operator that holds up to times of order $|\log h|$. This formula relates the distribution of energy levels of the quantum system near the energy $E$ to the periodic orbits of its classical counterpart. The perturbation $h^\alpha Q$ affects only the phases appearing in the trace formula. 
	\end{abstract}
	
	\tableofcontents
	
	\section{Introduction}
	
	\subsection{Spectral statistics in quantum chaos} This work is motivated by some classical conjectures in quantum chaos about the distribution of energy levels of quantum systems whose classical counterparts exhibit chaotic dynamics. The Bohigas--Giannoni--Schmit conjecture \cite{PhysRevLett.52.1,BGS2} asserts that generically---in a sense to be made precise---, the local spectral statistics of such systems coincide with those of large Gaussian random matrices. To fix ideas, consider a semiclassical Schr\"odinger operator $P_h=h^2\Delta_g+V$ on a closed Riemannian manifold $(M,g)$, where $\Delta_g$ is the nonnegative Laplace--Beltrami operator and $V$ is a smooth potential. Then $P_h$ is elliptic and its spectrum on $L^2(M)$ consists of a nondecreasing sequence of eigenvalues $\{E_{j,h}\}_{j\ge 0}$ that diverges to $+\infty$. To measure statistical properties of the spectrum near an energy $E$, one introduces a smoothed counting function
	\[N_f(E,L):=\sum_j f\Big(L\frac{E_{j,h}-E}{h}\Big),\]
	where $f$ is a test function with smooth compactly supported Fourier transform. One should think of $f$ as a regularized version of $\mathbf 1_{[-1,1]}$. The larger $L$ is, the finer properties of the spectrum are encoded in the fluctuations of $N_f(E,L)$ as $E$ varies. A simple quantity to measure the fluctuations of the number of eigenvalues in a microscopic neighborhood of $E$ is the smoothed number variance
	\begin{equation} \label{eq: smoothed number variance}\Sigma(E,E+\Delta E,L):=\frac{1}{\Delta E}\int_{E}^{E+\Delta E} |N_f(E',L)-\overline N_f|^2 \mathrm{d}E',\end{equation}
	where $\overline N_f$ is the smoothed average number of eigenvalues\[\quad \overline N_f:=\frac{1}{\Delta E}\int_{E}^{E+\Delta E} N_f(E',L)\mathrm{d}E',\]
	and we assume $\frac hL\ll \Delta E\ll 1$. The assumption $\Delta E\gg \frac hL$ ensures that we consider many disjoint windows of width $\frac hL$ when computing the variance while the condition $\Delta E\ll 1$ means that we stay near the energy level $E$ at the classical level. An avatar of the BGS conjecture is that for a generic $P_h$, when the classical dynamics on the energy level $E$ is chaotic, the smoothed number variance \eqref{eq: smoothed number variance} should match that of the spectrum of a large random matrix of the GOE ensemble.
	
	Arguments in favor of the conjecture typically rely on the use of the Gutzwiller trace formula, which expresses the regularized density of energy levels near the energy $E$ as the sum of a Weyl term and an oscillating sum running over the set of periodic orbits of the classical system on the corresponding energy level \cite{JHHannay1984,Berry1,BerryRigidity,Sieber_2001}. Given a test function $f$ with smooth compactly supported Fourier transform, for $L=1$, the trace formula takes the form
	\[\sum_j f\Big(\frac{E_{j,h}-E}{h}\Big)=N_{\rm Weyl}+N_{\rm osc} +\text{remainder}.\]
	Here $N_{\rm Weyl}$ is an integral over the energy shell $p^{-1}(E)$ involving local geometric invariants which is responsible for the Weyl's law asymptotics, and is given at leading order by
	\[N_{\rm Weyl}\underset{h\to 0}\sim \frac{\mu_{\rm L}(p^{-1}(E))}{(2\pi h)^{d-1}} \widehat f(0),\]
	where $\mu_{\rm L}$ is the Liouville measure on $p^{-1}(E)$. The term $N_{\rm osc}$ is an oscillating sum running over the periodic orbits of the classical flow on the energy shell $p^{-1}(E)$:
	\[N_{\rm osc} =\sum_{\substack{\gamma\ \text{periodic} \\ \text{orbit}}} A_\gamma \exp(\mathrm{i}h^{-1}S(\gamma)) \widehat f(T_\gamma).\]
	Here $T_\gamma$ and $S(\gamma)$ denote the period and the classical action of the periodic orbit $\gamma$, and $A_\gamma$ is a dynamical invariant measuring the instability of the orbit, see Theorem \ref{prop: long time gutz}. Berry \cite{Berry1,BerryRigidity} proposed to study the fluctuations of the oscillating sum when $E$ varies by performing a \emph{diagonal approximation}. It consists in discarding pairs of distinct periodic orbits when computing energy averages of $|N_{\rm osc}|^2$, by arguing that the nondiagonal terms $\mathrm{e}^{\frac{\mathrm{i}}{h}(S(\gamma)-S(\gamma'))}$ with $\gamma'\notin \{\gamma,\gamma^{-1}\}$ (here we denote by $\gamma^{-1}$ the orbit $\gamma$ traveled backwards) interfere destructively. This approximation lacks however a rigorous justification, and is even known to predict incorrect results for certain systems that present some \say{accidental} symmetries \cite{Luo1994}. Nevertheless, the BGS conjecture is still expected in a generic sense.
	
	To pursue this line of work, we study perturbations of quantum systems that are too weak to affect the classical evolution, but strong enough to modify the phases $\exp(\mathrm{i}h^{-1}S(\gamma))$ appearing in the trace formula. In a forthcoming work \cite{Moyspecstat}, we will use this trace formula to study  the spectral statistics in the model of random Schr\"odinger operators introduced by Ingremeau and Vogel \cite{ingremeau2023emergence}.

	\subsection{The semiclassical trace formula}
	
	The semiclassical trace formula has a long history that we briefly cover here, although we do not aim at being exhaustive. Early instances of the trace formula appeared in the pioneering works of Balian--Bloch on wave propagation in cavities \cite{BalianBlochI,BalianBlochII,BalianBlochIII}, and in Gutzwiller's work on Schr\"odinger operators \cite{Gutzwiller71}. Mathematically rigorous results were obtained shortly after by Colin de Verdière \cite{ColindeVerdière1973} for the heat operator and soon after by Chazarain \cite{Chazarain1974} and Duistermaat--Guillemin \cite{Duistermaat1975} for first order elliptic operators on a closed manifold, using the theory of homogeneous Fourier integral operators (FIOs). Semiclassical versions of the trace formula were later obtained by Brummelhuis--Uribe \cite{Brummelhuis1991} and Paul--Uribe \cite{PAUL1995192} who employed a reduction to the classical homogeneous theory of FIOs via a trick of Colin de Verdière. A direct semiclassical treatment using $h$-FIOs was given by Meinrenken \cite{MEINRENKEN1992279} for Schr\"odinger operators on $\mathbf{R}^n$; Combescure--Ralston--Robert then gave a streamlined proof based on coherent state decomposition \cite{Combescure1999-dq}. Another proof was obtained by Sj\"ostrand--Zworski, who reduced the proof to the study of a \emph{quantum monodromy operator} \cite{SJOSTRAND20021}. Trace formulae for longer time scales --- valid up to small logarithmic times in $h^{-1}$ --- were obtained by Jakobson--Polterovitch--Toth \cite{JakobsonWeyl2007} for the half-wave propagator on closed negatively curved surfaces using the Hadamard parametrix, and in a different setting by Schenck \cite{Schenk} for compact perturbations of the Laplacian on $\mathbf{R}^n$. Faure was able to reach arbitrarily large multiples of $|\log h|$ in the context of quantized Anosov diffeomorphisms of the torus \cite{faure2007Trace}. Strohmaier and Zelditch \cite{strohmaier} proved a relativistic generalization of the Duistermaat--Guillemin trace formula for stationary spacetimes. The recent work \cite{wunsch2025gutzwillertraceformulasingular} investigates the effect of discontinuities of the potential.

	\subsection{Main results} Let $(M,g)$ be a smooth, closed, orientable, Riemannian manifold. Denote by $T^*M$ its cotangent bundle. We refer to Appendix \ref{app: semiclassical on manifolds} for the notation of semiclassical calculus used here. We fix a quantization procedure $\operatorname{Op}_h^{M}(\cdot)$, that associates to any reasonable symbol $a\in C^\infty(T^*M)$ an operator $\operatorname{Op}_h^{M}(a):C^\infty(M)\to \mathcal D'(M)$, see Definition \ref{defi: weyl quanti manifold}.\medskip
	
	\begin{assum} \label{assum: classical flow} Let $m>0$ and let $p_0\in S^m(T^*M)$ be an elliptic, $h$-independent symbol. Denote by $\Phi^t_0$ the associated Hamiltonian flow on $T^*M$. Let $E\in \mathbf{R}$. We assume that $\mathrm{d}p_0$ is nonvanishing on $\Sigma_E:=p_0^{-1}(E)$, so that $\Sigma_E$ is a smooth compact hypersurface in $T^*M$. We further assume that the restriction $\Phi^t_0:\Sigma_E\to \Sigma_E$ is an Anosov flow. This means that the space $T\Sigma_E$ admits a continuous flow invariant decomposition
		\[T\Sigma_E=\mathbf{R}H_{p_0}\oplus E^s\oplus E^u,\]
		such that there exist constant $C,\lambda>0$ satisfying
		\[\forall v\in E^s, \forall t\ge 0, \quad \|\mathrm{d}\Phi_0^t\cdot v\|\le C \mathrm{e}^{-\lambda t}\|v\|,\]
		and
		\[\forall v\in E^u, \forall t\ge 0, \quad \|\mathrm{d}\Phi_0^{-t}\cdot v\|\le C \mathrm{e}^{-\lambda t}\|v\|.\]
		Here $\|\cdot\|$ is any smooth norm on $T(T^*M)$.
	\end{assum}
	
	\begin{assum}\label{assum: derivatives symbol} Let $0<\alpha\le 1$ and let $0\le 2\beta<\alpha$. Let $m>0$, and let $h\in (0,1]$. We consider a continuous family of real symbols $\{p_\delta\}_{\delta\in [0,h^\alpha]}$, such that $p_0$ is as above, and for all $n\ge 1$, the symbol $\partial_\delta^n p_\delta$ belongs to $S^m_\beta(T^*M)$ uniformly in $\delta \in [0,h^\alpha]$. Also, let $\{r_\delta\}_{\delta\in [0,h^\alpha]}$ be a family of real symbols in $h^{2-2\beta}S^{m}_\beta(T^*M)$. We define
		\begin{equation} \label{eq: defi P h delta}P_{h,\delta}=\operatorname{Op}_h^{M}(p_\delta+r_\delta).\end{equation}	\end{assum}
	We point out that the lower order term $r_\delta$ will not affect the leading order in the trace formula, on the timescales that we consider. One reason for adding a remainder is that we want the semiclassical Laplacian $h^2\Delta_g$ to fit under our assumptions. This operator is not simply given by the quantization of the symbol $|\xi|_g^2$, but one has instead $h^2\Delta_g=\operatorname{Op}_h^{M}(|\xi|_g^2+r)$ for some symbol $r\in h^2S^0(T^*M)$, up to a $\mathcal O(h^\infty)$ remainder.
	\begin{ex} For instance one may consider a family $p_\delta=p_0+\delta q$, with $q\in S_\beta^m(T^*M)$.
	\end{ex}
	We obtain the following trace formula for the operator $P_{h,\delta}$.
	\begin{thm}[Long time Gutzwiller trace formula] \label{prop: long time gutz} For any $(\alpha,\beta)$ satisfying Assumption \ref{assum: derivatives symbol}, there are constants $\varepsilon=\varepsilon(\alpha,\beta)>0$ and $\mathfrak c=\mathfrak c(\alpha,\beta)>0$ such that for any $h$ small enough, for any $\delta\in [0,h^\alpha]$ and $1\le L\le \mathfrak c|\log h|$ the following holds. For any test function $f$, with Fourier transform $\widehat f$ smooth and compactly supported in $[-1,1]$, we have
		\[\operatorname{Tr} f\Big(\frac Lh(P_{h,\delta}-E)\Big)=N_{\rm Weyl}+N_{\rm osc},\]
		where the Weyl term is given by
		\[N_{\rm Weyl}=\frac{1}{(2\pi h)^{d-1}L}\sum_{k\le N} \frac{\mu_{k,\delta}}{L^k}\widehat f^{(k)}(0)+\mathcal O(h^{N(1-2\beta)-C_d}\|f\|_{C^{2N+d}}),\]
		and the oscillating part admits an expansion in powers of $h$ which is given at first order by
		\begin{equation} \label{eq: N osc in mainthm}N_{\rm osc}=\frac 1L\sum_{\gamma\in \mathcal P_E} T_\gamma^{\#} \frac{\mathrm{e}^{\mathrm{i} ( \frac 1h S(\gamma_\delta)-\frac\pi 2\mathfrak{m}_\gamma)}}{|\det(I-P_\gamma)|^{\frac 12}} \widehat f\Big(\frac{T_\gamma}{L}\Big)+\mathcal O(h^\varepsilon \|f\|_{C^{C_d}}).\end{equation}
		Here,
		\begin{itemize}\item $\mu_{k,\delta}$ is a function depending on $h,\delta$ satisfying $|\mu_{k,\delta}|\le C_k h^{\lfloor \frac{k+1}{2}\rfloor}$ uniformly in $\delta\in [0,h^\alpha]$, where $\lfloor \cdot\rfloor$ is the floor function.
			\item $\mathcal P_E$ denotes the set of periodic orbits of the flow $\Phi^t_0:\Sigma_E\to \Sigma_E$, and for $\gamma\in \mathcal P_E$, we denote by $T_\gamma$ and $T_\gamma^\#$ the period and primitive period of $\gamma$, respectively.
			\item For each $\gamma\in \mathcal P_E$, there is a continuous deformation of periodic orbits $\{\gamma_\delta\}_{\delta\in [0,h^\alpha]}$ of the perturbed flow $\Phi^t_\delta$ with $\gamma_\delta\subset p_\delta^{-1}(E)$. The action of $\gamma_\delta$ is  
			\[S(\gamma_\delta)=\int_{\gamma_\delta} \alpha,\]
			where $\alpha$ is the canonical $1$-form on $T^*M$.
			\item $P_\gamma$ is the linearized Poincar\'e map of the orbit, given by $P_\gamma=\mathrm{d}\Phi^{T_\gamma}_0:E^s\oplus E^u\to E^s\oplus E^u$.
			\item $\mathfrak{m}_\gamma$ is the Maslov index of $\gamma$, interpreted as the winding number of the weak unstable manifold around $\gamma$, see \S \ref{subsec: elucidating Maslov}.\end{itemize}
	\end{thm}
	
	\begin{rem}We make a few comments.
		\begin{itemize} 
			\item The sum \eqref{eq: N osc in mainthm} is finite because $\widehat f$ has compact support and there are finitely many periodic orbits with period smaller than $L$. If the phases $\mathrm{e}^{\frac{\mathrm{i}}{h} S(\gamma_\delta)}$ were mutually independent, $N_{\rm osc}$ would be of order $1$. However, if the phases align --- and we assume for simplicity that all Maslov indices vanish --- then $N_{\rm osc}$ is of order $\mathrm{e}^{cL}$ for some positive constant $c$.
			\item The variations of the action with respect to $\delta$ are given by the formula
			\[\frac{\partial S(\gamma_\delta)}{\partial\delta}=-\int_{\gamma_\delta} \frac{\partial p_\delta}{\partial \delta}\mathrm{d}t.\]
	\end{itemize} \end{rem}
	
	Even if the flow $\Phi^t_0:p_0^{-1}(E)\to p_0^{-1}(E)$ is not Anosov, provided periodic orbits are hyperbolic, we can prove a Gutzwiller trace formula for the operator $P_{h,\delta}$ that holds for finite times. That is, we are not able to let $L$ go to $+\infty$ as $h\to 0$. In fact, the Anosov property provides a quantitative estimate on the separation of periodic orbits in phase space which allows working up to $L\le \mathfrak c|\log h|$.
	
	Recall that a periodic orbit of a Hamiltonian flow is \emph{hyperbolic} if the linearized Poincar\'e map of the orbit, which acts on the symplectic orthogonal to the flow direction, admits no eigenvalue of modulus $1$. Given a hyperbolic periodic orbit, one can define stable and unstable distributions along the orbit, which are still denoted $E^{s}$ and $E^u$.
	\begin{thm}\label{thm: only assume hyperbolic per orbit} We drop the assumption that the classical flow is Anosov. Let $\gamma$ be a hyperbolic periodic orbit of the flow on $\Sigma_E$, and assume that the period of $\gamma$ is isolated in the set of all periods. Let $f$ be a test function with smooth compactly supported Fourier transform, and such that $T_\gamma$ is the only period in the support of $\widehat f$. Then, with the notation of Theorem \ref{prop: long time gutz}, one has
		\[\operatorname{Tr} f\Big(\frac{P_{h,\delta}-E}{h}\Big)=T_\gamma^\#\frac{\mathrm{e}^{\mathrm{i} ( \frac 1h S(\gamma_\delta)-\frac\pi 2\mathfrak{m}_\gamma)}}{|\det(I-P_\gamma)|^{\frac 12}} \widehat f(T_\gamma)+\mathcal O(h^\varepsilon).\]
	\end{thm}
	
	To our knowledge, the interpretation of the Maslov index as the winding number of the unstable manifold around the orbit goes back to Creagh, Robbins and Littlejohn \cite{creagh1990,robbins1991}. In the main application of Theorem \ref{prop: long time gutz} that we have in mind, there are actually no Maslov indices. Indeed, by a result of Paternain--Paternain \cite{Paternain1994}, we know that if a symbol $p(x,\xi)\in C^\infty(T^*M)$ satisfies
	\begin{itemize}
		\item $\nabla_{\xi}^2p(x,\xi)\ge 0$;
		\item The classical flow on the energy shell $p^{-1}(E)$ is Anosov;
	\end{itemize}
	then the weak-unstable distribution $E^{wu}=\mathbf{R}H_{p_0}\oplus E^u$ is globally transverse to the vertical distribution $\mathbf V$, thus all Maslov indices vanish. This applies in particular to electromagnetic Hamiltonians $p(x,\xi)=|\xi-\mathbf A(x)|_g^2 +V(x)$, where $\mathbf A\in C^\infty(M,T^*M)$ is a $1$-form called a \emph{magnetic potential}, and $V\in C^\infty(M)$. For the free Hamiltonian $p(x,\xi)=|\xi|_g^2$, this result was known since Klingenberg \cite{Klingenberg1974}.
	
	\subsection{Organization of the paper} In \S \ref{sec: energy loc}, we perform an energy localization, based on the ellipticity of the operator $P_{h,\delta}$, that will allow us to work in a compact subset of the phase space $T^*M$ in the rest of the paper. This relies on the functional calculus for exotic pseudodifferential operators presented in the appendix. In \S \ref{sec: classical dynamics}, we study the dynamics generated by the perturbed Hamiltonian $p_\delta$. The main body of the work consists in controlling the behavior of the flow, and therefore the generating functions, when $\delta$ varies. The estimates that we obtain are similar to those in \cite[\S 4]{ingremeau2023emergence}.

	In \S \ref{sec: FIO}, we review some elements of the local theory of Fourier integral operators (FIOs), then we construct a WKB approximation for the short time propagator $\mathrm{e}^{-\frac{\mathrm{i}}{h} tP_{h,\delta}}$. While this procedure is fairly classical, some care is needed because we work with an exotic symbol $p_\delta$ whose derivatives of order $\ge 3$ blow up in the limit $h\to 0$. A notable feature is that the perturbations are included in the phase and not only in the amplitude of the FIOs. 
	
	In \S \ref{sec: the short time trace}, we then compute the short time trace, which is responsible for the Weyl term in the trace formula. 
	
	In \S\ref{sec: long time propagation}, we study the long time quantum dynamics. A WKB approximation for the long time propagator is obtained by composing short time approximations. As one may encounter \emph{caustics} which prevent expressing the propagator nicely, as in \cite{Schenk} we work in unstable/stable coordinates that are adapted to the dynamics.  At the quantum level, this classical change of coordinates amounts to conjugating the quantum propagator by some metaplectic operators. Similar results were obtained for perturbations $P_{h,\delta}=h^2\Delta_g+\delta Q$ on negatively curved manifolds by Ingremeau and Vogel \cite{ingremeau2023emergence}, by lifting the computations to the universal cover. Finally, we compute the long time trace using a stationary phase expansion. In the unperturbed case $\delta=0$, this section is very similar to \cite{Schenk}, aside from Maslov indices considerations.
	
	The final section \S \ref{sec: metaplectic} provides the material on metaplectic operators and Maslov indices which is needed in the paper. 
	
	The appendix, which is very much inspired by \cite{DJN21,bonthonneauWeyl}, contains a presentation of the semiclassical pseudodifferential calculus on Riemannian manifolds, with an emphasis on exotic symbols.
	
	\subsection*{Acknowledgments} I wish to thank Emmanuel Schenck, Maxime Ingremeau, and Martin Vogel for our discussions during the preparation of this work. Thanks also to St\'ephane Nonnenmacher for his guidance during the writing of this manuscript and his feedback on a previous version of the paper.
	\subsection{Notations and conventions}
	If $k$ is an integer, we let $k_+:=\max(0,k)$.
	
	$\mathcal S(\mathbf{R}^d)$ is the Schwartz space of rapidly decaying functions and $\mathcal S'(\mathbf{R}^d)$ is its dual.
	A function $f\in \mathcal S(\mathbf{R})$ and its Fourier transform $\widehat f\in \mathcal S(\mathbf{R})$ are given by
	\[\widehat f(\xi)=\frac{1}{2\pi}\int_\mathbf{R} \mathrm{e}^{-\mathrm{i} t\xi}f(t)\mathrm{d}t,\qquad f(t)=\int_{\mathbf{R}}\mathrm{e}^{\mathrm{i} t\xi}\widehat f(\xi)\mathrm{d}\xi.\]

	Given a function $f:\mathbf{R}^d\to \mathbf{R}$, the notation $\partial^\ell f$ represents any linear combination of partial derivatives of $f$ of order $\ell$. If $G=(G^1,\ldots,G^d):\mathbf{R}^d\to \mathbf{R}^d$ is a smooth map, the notation $\partial^\ell G$ represents $\partial^\ell G^i$ for some $i\in \{1,\ldots,d\}$, where $\partial^\ell G^i$ is defined as above.

	One can then write Fa\`a di Bruno's formula in a compact way. If $f:\mathbf{R}^d\to \mathbf{R}$ and $G:\mathbf{R}^d\to \mathbf{R}^d$ are as above, one has
	\[\partial^\ell(f\circ G)=\sum_{r=0}^{\ell}\sum_{\ell_1+\ldots+\ell_r=\ell} (\partial^r f)(G)\cdot \Big(\prod_{i=1}^r \partial^{\ell_i}G\Big).\]
	By convention, an empty sum equals $0$ and an empty product equals $1$. Thus the formula still makes sense for $\ell=0$. When $\ell\ge 1$ we can make the sum start at $r=1$.
	
	We write $f\lesssim g$ or $f=\mathcal O(g)$ if there is a constant $C$ such that $f\le Cg$. If $C$ depends on a set of relevant parameters, denoted by $\bullet$, we write $f\lesssim_{\!\! \bullet} g$.
	
	If $f$ depends on a semiclassical parameter $h\in (0,1]$, we write $f=\mathcal O(h^\infty)$ if $f=\mathcal O_N(h^N)$ for all $N\ge 0$.
	
	Given two Banach spaces $E,F$ and a bounded linear operator $A:E\to F$, we denote by $\|A\|_{E\to F}$ the operator norm of $A$. If $E=F$ we often write $\|A\|_E$ instead. The notation $\mathcal O_\bullet(g)_{E\to F}$ represents an operator $A:E\to F$ such that $\|A\|_{E\to F}=\mathcal O_\bullet(g)$.
	
	If $H$ is a separable Hilbert space and $A:H\to H$ is a linear operator, we denote by $\|A\|_{\rm tr}$ the trace-class norm of $A$.

	\section{Energy localization} \label{sec: energy loc} Let $f$ be a real function with smooth compactly supported Fourier transform. Recall that we are studying
	\[N_f(E,L):=\operatorname{Tr}f\Big(L\frac{P_{h,\delta}-E}{h}\Big)=\sum_j f\Big(L\frac{E_{j,h}-E}{h}\Big),\]
	where $\{E_{j,h}\}_{j\ge 0}$ denotes the sequence of eigenvalues of $P_{h,\delta}$. 
	\begin{lem}\label{lem: first energy loc} Let $\chi\equiv 1$ on $[-1,1]$ and $\chi\equiv 0$ outside $[E-2,E+2]$. Also, let $\tilde \chi(x)\equiv x$ on $[E-1,E+1]$ and $\tilde \chi(x)\equiv 0$ outside $[-2,2]$. Then, if $\nu<1$ and $1\le L\le C|\log h|$ we have
		\begin{equation} \label{eq: energy localized} N_f(E,L)=\operatorname{Tr}f\Big(L\frac{\tilde \chi(P_{h,\delta})-E}{h}\Big) \chi\Big(\frac{P_{h,\delta}-E}{h^\nu}\Big)+\mathcal O_C(h^\infty). \end{equation}
	\end{lem}
	\begin{proof}Since $\widehat f$ is compactly supported we have $|f(t)|\le C_N |t|^{-N}$ for $t\ge 1$. Then the absolute value of the difference between the two sides of \eqref{eq: energy localized} can be bounded by
		\begin{equation} \label{eq: diff trace can be bdd}\sum_{h^\nu\le |E_{j,h}-E|\le 2 h^\nu} \Big|f\Big(L\frac{E_{j,h}-E}{h}\Big)\Big|\lesssim \sum_{h^\nu\le |E_{j,h}-E|\le 2 h^\nu} \Big|\frac{L(E_{j,h}-E)}{h}\Big|^{-N}.\end{equation}
		Since $p_\delta$ is elliptic of order $m$, Proposition \ref{prop: rough bounds counting func} gives an upper bound
		\[\#\{E_{j,h}\le x\}\lesssim h^{-d}x^{d/m},\]
		which is uniform in $\delta\in [0,h^\alpha]$. Because $L=\mathcal O(|\log h|)$, we deduce that  \eqref{eq: diff trace can be bdd} is $\mathcal O(h^\infty)$.
	\end{proof}
	In practice we will only apply this result for $\nu<\frac 12$, this condition ensures that $\chi\Big(\frac{P_{h,\delta}-E}{h^\nu}\Big)$ is a pseudodifferential operator lying in an exotic class, by Theorem \ref{thm: func calc singular func}. Instead of working with this operator, we prefer localizing near the energy shell $\Sigma_E=p_0^{-1}(E)$, which is more or less equivalent.
	\begin{prop}\label{prop: energy cutoff} Let $\chi_E\in S_\nu(T^*M)$ be equal to $1$ in a $3h^\nu$-neighborhood of $\Sigma_E$. Then, if $\nu<\min(\frac 12,\alpha)$, in the notation of Lemma \ref{lem: first energy loc}, we have
		\[N_f(E,L)=\operatorname{Tr}\Big[f\Big(L\frac{\tilde \chi(P_{h,\delta})-E}{h}\Big)\operatorname{Op}_h^{M}(\chi_E)\Big]+\mathcal O(h^\infty).\]
	\end{prop}
	\begin{proof}Let $\chi\equiv 1$ on $[-1,1]$ and $\chi\equiv 0$ outside $[-2,2]$. Then, write
		\begin{align*}
			\operatorname{Tr} f\Big(L\frac{\tilde \chi(P_{h,\delta})-E}{h}\Big)\operatorname{Op}_h^{M}(\chi_E)& = 	\operatorname{Tr}  f\Big(L\frac{\tilde \chi(P_{h,\delta})-E}{h}\Big)\operatorname{Op}_h^{M}(\chi_E) \chi\Big(\frac{P_{h,\delta}-E}{h^\nu}\Big) \\ & +	\operatorname{Tr}  f\Big(L\frac{\tilde \chi(P_{h,\delta})-E}{h}\Big)\operatorname{Op}_h^{M}(\chi_E) \Big[1-\chi\Big(\frac{P_{h,\delta}-E}{h^\nu}\Big)\Big].
		\end{align*}
		By cyclicity of the trace, and by the same argument as in Lemma \ref{lem: first energy loc}, the second term on the right-hand side is smaller than
		\[\Big\|\Big[1-\chi\Big(\frac{P_{h,\delta}-E}{h^\nu}\Big)\Big]f\Big(L\frac{\tilde \chi(P_{h,\delta})-E}{h}\Big) \Big\|_{\rm tr}\cdot \|\operatorname{Op}_h^{M}(\chi_E)\|_{L^2}=\mathcal O(h^\infty).\]
		To estimate the first term, we use Theorem \ref{thm: func calc singular func} to write
		\[ \chi\Big(\frac{P_{h,\delta}-E}{h^\nu}\Big)=\operatorname{Op}_h^{M}(q)+\mathcal O(h^\infty)_{\rm tr},\]
		for some symbol $q\in S^m_{\max(\beta,\nu)}(T^*M)$ equal to $1$ on $|p_\delta-E|\le h^\nu$ and vanishing on $|p_\delta-E|\ge 2 h^\nu$. Now, $|p_\delta-E|\le 2h^\nu$ implies $|p_0-E|\le 2h^\nu+C\delta=2h^\nu+Ch^\alpha\le 3h^\nu$ provided $h$ is small enough. It follows that $\chi_E\equiv 1$ on the support of $q$, hence by the composition formula from Proposition \ref{prop: composition manifold}, we have
		\[\operatorname{Op}_h^{M}(\chi_E)\operatorname{Op}_h^{M}(q)=\operatorname{Op}_h^{M}(q)+\mathcal O(h^\infty)_{L^2/\rm{tr}}.\]
		This finishes the proof.
	\end{proof}
	Note that by Theorem \ref{thm: functional calculus}, if $a\in C^\infty(T^*M)$ is supported on a small neighborhood of the energy shell $p_0^{-1}(E)$, we have
	\[\tilde \chi(P_{h,\delta})\operatorname{Op}_h^{M}(a)=\operatorname{Op}_h^{M}(p_\delta+r_\delta)\operatorname{Op}_h^{M}(a),\]
	so we will always work with a compactly supported symbol equal to $p_\delta$ in the region of interest.
	
\section{Classical dynamics} \label{sec: classical dynamics} In this section, we study the dynamics of the perturbed classical system. We either work on $M=\mathbf{R}^d$ or on a closed manifold $M$. The setting is the following: $\{p_\delta\}_{\delta\in [0,h^\alpha]}$ is a smoothly varying family of smooth symbols in $S^m(T^*M)$ such that $p_0$ is elliptic of order $m$ and such that $\partial_\delta^n p_\delta$ is uniformly bounded in $S^m_\beta(T^*M)$ for all $n\ge 1$.
We assume $|\delta|\le h^\alpha$ with $\beta<2\alpha$. We denote by $H_{p_\delta}$ the Hamiltonian vector field associated with $p_\delta$, and by $\Phi^t_\delta:T^*M\to T^*M$ the associated flow. \medskip

\begin{assum}[Assumptions on the classical flow]~ 
	
	\begin{enumerate}
		\item 
		Let $\Sigma_E:=p_0^{-1}(E)$. We assume that $\mathrm{d}p_0$ does not vanish on $\Sigma_E$. It implies that $\Sigma_E$ is a smooth submanifold of $T^*M$. We assume moreover that $\Sigma_E$ is closed and connected.
		\item 
		The flow $\Phi_0^t:\Sigma_E\to \Sigma_E$ is \emph{Anosov}. More precisely, there is a continuous flow invariant splitting $T_\rho\Sigma_E=\mathbf R H_{p_0}(\rho) \oplus E^u(\rho)\oplus E^s(\rho)$, such that the following holds. There are constants $C,c>0$ such that
		\begin{itemize} 
			
			\item for all $v$ in $E^u(\rho)$, one has $|\mathrm{d}\Phi^{-t}_0(\rho)v|\le C\mathrm{e}^{-ct}$ for all $t\ge 0$,
			\item for all $v$ in $E^s(\rho)$, one has $|\mathrm{d}\Phi^{t}_0(\rho)v|\le C\mathrm{e}^{-ct}$ for all $t\ge 0$.\end{itemize}
		Here $|\cdot|$ denotes any continuous norm on $T(T^*M)$. The spaces $E^u$ and $E^s$ are respectively called \emph{unstable} and \emph{stable}, and both have dimension $d-1$. Note that the \emph{weak unstable} and \emph{weak stable} spaces $E^{wu}:=\mathbf RH_{p_0}\oplus E^u$ and $E^{ws}:=\mathbf RH_{p_0}\oplus E^s$ are Lagrangian subspaces of $T(T^*M)$ for the natural symplectic form $\omega$.
	\end{enumerate}
\end{assum}

\subsection{Estimates on the perturbed flow} The derivatives of the flow are estimated as follows. There is an analogous result in \cite[Lemma 4.2]{ingremeau2023emergence}.
\begin{prop}\label{prop: short time derivative flow}
There is a constant $C>0$ such that for all $\ell\ge 0$,
\[\forall t\ge 0, \ \|\Phi^t_\delta\|_{C^\ell_\rho}\lesssim h^{-\beta(\ell-1)_+}\mathrm{e}^{C\ell t}.\]
Moreover, $|\partial_t^n\partial_\rho^\ell\Phi^t_\delta(\rho)|\lesssim h^{-\beta(\ell+n-1)}\mathrm{e}^{C\ell t}$ for all $n\ge 1$.\end{prop}
\begin{proof}The estimate on spatial derivatives follows from the general estimates of \cite[Theorem 4.5]{prouff2024egorovstheoremweylhormandercalculus}, see also \cite[Lemma C.1]{DyGu14}, or \cite[Lemma 2.2]{bouzouinarobert}. The estimate on $|\partial_t^n\partial_\rho^\ell \Phi^t_\delta(\rho)|$ is proved by induction on $n$, using $\partial_t\Phi_\delta^t=H_{p_\delta}\Phi^t_\delta$. \end{proof}

\subsubsection{A Duhamel formula for the perturbed flow} The variations of $\Phi_\delta^t$ with respect to $\delta$ are expressed as follows.
\begin{lem}\label{lem: duhamel perturbed flow}Let $\{p_\delta\}$ be a smoothly varying family of smooth Hamiltonian on $T^*M$. Denote by $\Phi_\delta^t$ the Hamiltonian flow generated by $p_\delta$. Then, for any $\rho\in T^*M$, 
\begin{equation}\label{eq: duhamel perturbed flow bis}\frac{\partial}{\partial \delta}\Phi^t_\delta(\rho)=\int_0^t \mathrm d\Phi_\delta^s\cdot H_{\partial_\delta p_\delta}(\Phi_\delta^{t-s}(\rho)) \mathrm ds.\end{equation}
Here $H_{\partial_\delta p_\delta}$ is the Hamiltonian vector field associated with $\partial_\delta p_\delta$.\end{lem}
\begin{proof} Fix $a\in C^\infty(T^*M)$ and let $u(t)=\partial_\delta (\mathrm{e}^{tH_{p_\delta}}a)=\partial_\delta ( a(\Phi_\delta^t(\cdot)))$. Then, by the chain rule,
\[\partial_t u=H_{p_\delta}u+H_{\partial_\delta p_\delta} \mathrm{e}^{tH_{p_\delta}}a.\]
Since $u(0)=0$, we infer from Duhamel's formula that
\[u(t)=\int_0^t \mathrm{e}^{(t-s)H_{p_\delta}} \big[H_{\partial_\delta p_\delta} \mathrm{e}^{sH_{p_\delta}}a\big]\mathrm{d}s.\]

This holds for all $a$ thus \eqref{eq: duhamel perturbed flow bis} follows.
\end{proof}

\subsection{Adapted coordinates} In this subsection we only consider the unperturbed dynamics $\Phi^t_0$. Let $(U_i, f_i)_{i\in I}$ be a finite atlas of $M$. It gives rise to an atlas $(T^*U_i,\widetilde f_i)$ of $T^*M$.

Let $\rho\in \Sigma_E$. Fix two charts $(T^*U_i,\widetilde{f_i})$ and $(T^*U_j,\widetilde{f_j})$ near $\rho$ and $\Phi^t_0(\rho)$ respectively. The flow in coordinates is defined locally by
\[\Phi^t_{ij}:= \widetilde{f_j}\circ \Phi^t_0\circ \widetilde{f_i}^{-1}.\] 
At the quantum level, this amounts to considering the operator in charts $(f_j)_*\mathrm{e}^{-\frac{\mathrm{i}}{h} tP_h}f_i^*$. In order to write this operator as a local Fourier integral operator with a generating function having a minimal number of variables, it is necessary that the graph of the flow $(y,\eta)\mapsto (x,\xi)$ projects (locally) diffeomorphically on the $(x,\eta)$ coordinates. An equivalent condition is that the upper left block of the differential $\mathrm{d}\Phi^t_0$ is invertible. Unfortunately, this is not necessarily the case, but this always holds after performing some change of coordinates in phase space.  A convenient choice, which is adapted to the dynamics, is given in \cite[Lemma 4.3]{NZ09}, that we quote here.
\begin{lem}\label{lem: adapted bases} Let $\rho_0\in \Sigma_E$. Then, one can find a symplectic basis $(u_1,\ldots,u_d;s_1,\ldots,s_d)$ of $T_{\rho_0} M$, such that
\[u_1=H_{p_0}(\rho_0), \qquad \mathrm{Span}(u_2,\ldots,u_d)=E^u(\rho_0), \qquad \mathrm{Span}(s_2,\ldots,s_d)=E^s(\rho_0),\]
and $\omega(u_i,s_j)=\delta_{ij}$. Moreover, this choice can be made continuous in a neighborhood of $\rho$. Note that $s_1$ is automatically transverse to $T_\rho\Sigma_E$.\end{lem}
\begin{proof}Fix a continuous Riemannian metric $g$ on the subbundle $T\Sigma_E=\mathbf{R}H_{p_0}\oplus E^u\oplus E^s\subset T(T^*M)$, such that $\mathbf{R}H_{p_0}$ and $E^u\oplus E^s$ are orthogonal at every point. Then extend $g$ to a continuous metric on the whole tangent space $T_\rho(T^*M)$, for $\rho\in \Sigma_E$. Let $J:T_\rho (T^*M)\to T_\rho (T^*M)$ be defined by $\omega(\cdot,Ju)=g(u,\cdot)$. Now, fix an orthonormal basis $(u_1,\ldots,u_d)$ for $g$ of $\mathbf{R}H_{p_0}(\rho_0)\oplus E^u(\rho_0)$, such that $u_1=H_{p_0}(\rho_0)$ and $E^u(\rho_0)=\mathrm{Span}(u_2,\ldots,u_d)$. Then let $s_i=Ju_i$. Since $(u_i)_{1\le i\le d}$ is orthonormal, one can readily check that 
\[\omega(u_i,s_j)=g(u_i,u_j)=\delta_{ij},\]
for all $i,j\in \{1,\ldots,d\}$, where $\delta_{ij}$ is the Kronecker symbol. It remains to verify that $\mathrm{Span}(s_2,\ldots,s_d)=E^s$. To see this, notice that if $j\ge 2$ then $s_j= \lambda s_1+v$ for some $\lambda\in \mathbf{R}$ and $v\in E^s$. Then,
\[0=\omega(u_1,s_j)=\lambda\omega(u_1,s_1)+\omega(u_1,v)=\lambda+0.\]
Here one has $\omega(u_1,v)=0$ for all $v\in E^s$ because $\mathbf{R}H_{p_0}\oplus E^s$ is Lagrangian. Thus $s_j=v\in E^s$ as we wished. \end{proof}
\begin{rem}  Beware that the choice of $(u_2,\ldots,u_n;s_1,s_2,\ldots,s_n)$ cannot be made globally continuous with respect to $\rho$. If it were, this would imply that $E^s,E^u$ are trivial bundles. However, the construction can always be made continuous locally.
\end{rem}
We quote another result of \cite{NZ09} that show the existence of local symplectic charts adapted to the dynamics that trivialize the flow.
\begin{lem}[{\cite[Lemma 4.3]{NZ09}}]\label{lem: straightening coordinates} Fix $\rho_0\in p_0^{-1}(E)$. There are local symplectic coordinates $(y_1,\ldots,y_d;\eta_1,\ldots,\eta_d)$ on $T^*M$ near $\rho_0\in T^*M$ such that
\[\eta_1(\rho)= p_0(\rho), \quad \frac{\partial}{\partial y_i}(\rho_0)=u_i, \quad \frac{\partial}{\partial \eta_i}(\rho_0)=s_i, \qquad i=1,\ldots,d.\]
In these coordinates, the unperturbed flow $\Phi_0^t$ becomes the $t$-translation along the $y_1$ coordinate.
\end{lem}

Then, one has the following result, essentially contained in \cite[\S 4]{NZ09}.
\begin{lem}\label{lem: matrice differentielle flot} Consider bases of $T_\rho M$ and $T_{\Phi^t_0(\rho)}M$ as in Lemma \ref{lem: adapted bases}. In these bases, the differential of the flow can be expressed
\[\mathrm{d}\Phi^t_0(\rho)=\begin{pmatrix}

1 & 0 & \beta_t & 0 \\ 
0 & A & 0 & 0 \\
0 & 0 & 1 & 0 \\
0 & 0 & 0 & A^{-\top} \end{pmatrix}.\]
Here we have $\|A^{-1}\|\le C \mathrm{e}^{-c t}$, and $|\beta_t|\le C t+1$, uniformly with respect to the choices of bases.
\end{lem}

\begin{proof} \textbf{1.} Since $\mathrm{d}\Phi^t_0(\rho)\cdot H_{p_0}(\rho)=H_{p_0}(\Phi^t_0(\rho))$, one gets the first column. Also, $\mathrm{d}\Phi^{t}_0$ preserves the spaces $E^u$ and $E^s$, giving that $\mathrm{d}\Phi^t_0$ has the form
\[\mathrm{d}\Phi^t_0=\begin{pmatrix}

1 & 0 & \beta_t & 0 \\
0 & A & * & 0 \\
0 & 0 & * & 0 \\ 
0 & 0 & * & D \end{pmatrix}=:\begin{pmatrix}\mathbf A & \mathbf B \\ 0 & \mathbf D\end{pmatrix}.\] 

Since $\mathrm{d}\Phi^t_0$ is represented by a symplectic matrix in these coordinates one finds $\mathbf A^\top\mathbf D=\mathrm{Id}$ and $\mathbf A \mathbf B^\top=\mathbf B \mathbf A^\top$, implying
\[\mathbf D=\begin{pmatrix}
1 & 0 \\ 0 & A^{-\top}
\end{pmatrix}, \qquad \mathbf B=\begin{pmatrix}\beta_t & 0 \\ 0 & 0\end{pmatrix},\]
for some $\beta_t\in \mathbf R$, which gives the claimed expression for $\mathrm{d}\Phi^t(\rho)$.

\textbf{2.} The estimate on $\|A^{-1}\|$ is a consequence of the Anosov property of the flow. The estimate on $\beta_t$ relies on the following observation. First, $\beta_t(\rho)$ is uniformly bounded for $t\le 1$ and $\rho\in \Sigma_E$. Then, for all $t,s$, choosing bases of $T_\rho M$, $T_{\Phi_0^t(\rho)}M$ and $T_{\Phi_0^{t+s}}(\rho)$ as above, the trivial identity $\mathrm{d}\Phi_0^{t+s}(\rho)=\mathrm{d}\Phi_0^s(\Phi_0^t(\rho))\circ \mathrm{d}\Phi_0^t(\rho)$ implies, from the specific form of $\mathrm{d}\Phi_0^t$ and $\mathrm{d}\Phi_0^s$ in these bases, that one has
\[\beta_{t+s}(\rho)=\beta_s(\Phi_0^t(\rho))+\beta_t(\rho).\]
It is implicit here, but $\beta_t$ depend on the choices of bases of $T_\rho M$ and $T_{\Phi_0^t(\rho)}M$. 
Letting $\|\beta_t\|_\infty$ denote the maximum value of $|\beta_t|$ among all possible choices of $\rho\in \Sigma_E$ and orthonormal bases of $T_\rho M$ and $T_{\Phi_0^t(\rho)}M$ as described in Lemma \ref{lem: adapted bases}, one has $\|\beta_t\|_\infty\lesssim 1$ for $t\le 1$, and the following subadditivity property holds:
\[\|\beta_{t+s}\|_{\infty}\le  \|\beta_t\|_\infty+\|\beta_s\|_\infty.\]
This implies $\|\beta_t\|_\infty\lesssim C t+1$.\end{proof}

\subsection{Generating functions} One can describe the action of symplectomorphisms of $T^*\mathbf{R}^d$ by the use of generating functions, which are easier to handle.\begin{defi}\label{defi: gen func} Consider a symplectomorphism $\kappa:(y,\eta)\in T^*\mathbf{R}^d\to (x,\xi)\in T^*\mathbf{R}^d$. A function $\varphi(x,\eta)$ is a \emph{generating function} for $\kappa$ if for any $(x,\xi);(y,\eta)\in T^*\mathbf{R}^d$, one has
\[\kappa(y,\eta)=(x,\xi)\quad \Longleftrightarrow \quad (y,\xi)=(\nabla_\eta \varphi,\nabla_x\varphi).\]\end{defi}
With this definition, generating functions are only defined up to the choice of a constant.	We give another definition for flows generated by time-dependent Hamiltonians.
\begin{defi}
Consider a flow $\Psi^t$ of diffeomorphisms generated by a time-dependent Hamiltonian $\{q_t\}$. We say that $\varphi(t,x,\eta)$ is a generating function for $\{\Psi^t\}$ if for all $t$, $\varphi(t,\cdot)$ is a generating function for $\Psi^t$ in the sense of Definition \ref{defi: gen func}, and if additionally the Hamilton-Jacobi equation
\[\partial_t \varphi(t,x,\eta)=-q_t(x,\nabla_x \varphi(t,x,\eta))\]
is satisfied.\end{defi}

One can also define \emph{local generating functions} associated with local symplectomorphisms, the definition being the same up to some domain considerations. In general such a generating function need not exist but it always does locally provided the upper-left block of $\mathrm{d}\kappa$ is invertible. Precise statements will be given later.

\subsubsection{Hamiltonian flows in charts} Fix $\rho_0 \in \Sigma_E$ and let $\rho_t:=\Phi^t(\rho_0)$. For $t\ge 0$, we consider a basis $(u_1(t),\ldots,u_d(t);s_1(t),\ldots,s_d(t))$ of $T_{\rho_t}M$ as in Lemma \ref{lem: adapted bases}, that varies continuously with $t$. Note that $\rho_t=\rho_{t'}$ does not imply $u_k(t)=u_k(t')$ nor $s_k(t)=s_k(t')$. Let $t_0>0$, and fix open charts $U_i,U_j$ of $M$ such that $\rho_0\in T^*U_i$ and $\rho_{t_0}\in T^*U_j$. We consider the flow in coordinates
\[\Phi^t_{\delta,ij}:=\widetilde{f_j}\circ \Phi^t_\delta\circ \widetilde{f_i}^{-1}.\]
We are not done since we want the graph of the flow to have good projectability properties. We introduce symplectic matrices $\mathcal A_i(t)$ and $ \mathcal A_j(t)$ using the bases introduced in the previous subsection. Letting $(e_1,\ldots,e_d;f_1,\ldots,f_d)$ denote the canonical basis of $T^*\mathbf{R}^d$, we define $\mathcal A_i(t)$ (for $t$ close to $0$) by 
\begin{equation} \label{eq: def A_i(t)} \mathcal A_i(t) e_k=(\widetilde f_i)_* u_k(t), \qquad \mathcal A_i(t) f_k=(\widetilde f_i)_*s_k(t)
\end{equation}
and define $\mathcal A_j(t)$ similarly but using the chart $\widetilde f_j$, for $t$ close to $t_0$. Then, for $t$ near $t_0$, we look at the flow in new coordinates
\begin{equation} \label{eq: def twisted flow}\widetilde {\Phi}^t_{\delta,ij}:= \mathcal A_j(t_0)^{-1}\circ \widetilde{f_j}\circ \Phi^t_\delta\circ \widetilde{f_i}^{-1}\circ \mathcal A_i(0).\end{equation}
To lighten notation, we will often just write $\Phi^t_\delta$ instead of $\widetilde \Phi^t_{\delta,ij}$. One should keep in mind that we work in charts and that we have applied linear symplectic transformations. Similarly, we denote abusively $\rho_0,\rho_{t_0}$ the same points in the charts. The next proposition is a variation of \cite[Proposition 10]{Schenk}. The only difference here is that we consider singular Hamiltonians.
\begin{prop}[Projectability of the graph] \label{prop: projectability} Fix $t_0>0$, $\rho_0\in \Sigma_E$, and let $\rho_{t_0}=\Phi^{t_0}_0(\rho_0)$. We fix charts $(T^*U_i,\widetilde f_i)$, $(T^*U_j, \widetilde f_j)$ near $\rho_0$ and $\rho_{t_0}$, respectively. We still write $\rho_0$ and $\rho_{t_0}$ for the points in coordinates $\widetilde f_i(\rho_0), \widetilde f_j(\rho_{t_0})$. Points in $\mathbf{R}^{2d}$ are denoted by $(y,\eta)$ or $(x,\xi)$. There is a small but fixed constant $\varepsilon>0$, such that if 
\[r+\delta h^{-2\beta}\le \varepsilon \mathrm{e}^{-2Ct},\] then the graph of $ \Phi_{\delta}^t$ restricted to the ball $B_y(x_0,r)\times B_\eta(\xi_0,r)$, that is the set 
\[\big\{(y,\eta;x,\xi)~:~ (x,\xi)= \Phi^t_{\delta}(y,\eta), \ (y,\eta)\in B(x_0,r)\times B(\xi_0,r)\big\}, \]
projects diffeomorphically on the $(x,\eta)$ coordinates for all $|t-t_0|\le r$. Let
\[\Omega_r:=\pi_{x,\eta}\big(\mathrm{Graph}( \Phi^t_{\delta})_{|B(x_0,r)\times B(\xi_0,r)}\big)\]
be the image of the graph by this projection. Then, there is a small constant $c_1>0$, such that the following inclusion holds:
\begin{equation} \label{eq: Ball times ball in omega}B_x\big(x_1,\tfrac{c_1r}{t_0+1}\big)\times B_\eta(\xi_0,\tfrac{c_1r}{t_0+1})\subset \Omega_r.\end{equation} 
\end{prop}
\begin{rem}In particular, the last statement of the proposition says that for any $(x,\eta)\in B(x_1,\tfrac{c_1 r}{1+t})\times B(\xi_0,\tfrac{c_1 r}{1+t})$, there exists a unique $y\in B(x_0,r)$ such that $\Phi_\delta^t(y,\eta)=(x,\bullet)$.
\end{rem}

\begin{proof}Let us compute, for $\rho,\rho'\in B(\rho_0,r)$ and $|t-t_0|\le r$, using Taylor's formula:
\begin{equation}\label{eq: express diff flots}	\begin{aligned}
\Phi_\delta^t(\rho)-\Phi_\delta^t(\rho')& =\int_0^1 \mathrm{d}\Phi_\delta^t(s\rho+(1-s)\rho')(\rho-\rho')\mathrm{d}s \\
& =\int_0^1 \mathrm{d}\Phi_0^t(s\rho+(1-s)\rho')\mathrm{d}s\cdot (\rho-\rho')+\mathcal O\big(\delta\|\partial_\delta \mathrm{d}\Phi_\delta^t\|_{\infty}\cdot |\rho-\rho'|\big) \\
& =   \mathrm{d}\Phi_0^{t_0}(\rho_0)(\rho-\rho')+ \mathcal O(r \|\Phi^t_0\|_{C^2}\cdot |\rho-\rho'|)+\mathcal O\big(\delta\|\partial_\delta \mathrm{d}\Phi_\delta^t\|_{\infty}\cdot |\rho-\rho'|\big).
\end{aligned}\end{equation}
Now let us write $\rho=(y,\eta)$ and $\rho'=(y',\eta')$. Let $F_\delta^t(y,\eta)=(\pi_x(\Phi^t(y,\eta)),\eta)$. If $\mathrm{d}\Phi_\delta^t=\begin{pmatrix}
A & B \\ C & D
\end{pmatrix}$ then the differential of $F_\delta^t$ is given by
\[\mathrm{d}F_\delta^t=\begin{pmatrix}
A & B \\ 0 & \operatorname{Id}
\end{pmatrix},\]
which is invertible if and only if $A$ is invertible. By our choice of adapted coordinates in Lemma \ref{lem: adapted bases}, we deduce that $F_\delta^t$ is a local diffeomorphism provided $|t-t_0|+|\rho-\rho_0|+\delta h^{-2\beta}\le \mathrm{e}^{-2CT}$. Since $\|\Phi^t_0\|_{C^2}\lesssim \mathrm{e}^{2Ct}$ and $\|\partial_\delta \mathrm{d}\Phi_\delta^t\|\lesssim h^{-2\beta} \mathrm{e}^{2Ct}$, we infer from \eqref{eq: express diff flots} that
\[|F_\delta^t(\rho)-F_\delta^t(\rho')|\ge \big|\mathrm{d}F_0^{t_0}(\rho_0)\cdot (\rho-\rho')\big|-C(r+\delta h^{-2\beta})\mathrm{e}^{2Ct}|\rho-\rho'|.\]
Recalling the expression of $\mathrm{d}\Phi_0^{t_0}(\rho_0)$ from Lemma \ref{lem: matrice differentielle flot}, we have
\[\big|\mathrm{d}F_0^{t_0}(\rho_0)\cdot (\rho-\rho')\big|\ge \|(\mathrm{d}F_0^{t_0})^{-1}\|^{-1}\cdot |\rho-\rho'|\gtrsim \frac{1}{1+t}|\rho-\rho'|.\]
We deduce that provided 
\[(r+\delta h^{-2\beta})\le \frac{c_0}{2}(1+t)\mathrm{e}^{2C t},\]
for some small enough constant $c_0>0$, we have
\begin{equation}\label{eq: F delta inverse lip}|F_\delta^t(\rho)-F_\delta^t(\rho')|\ge \frac{c_0}{(1+t)}|\rho-\rho'|. \end{equation}
It implies that $F_\delta^t$ is injective on $B(\rho_0,r)$, thus for all $t\in B(t_0,r)$, the map $F_\delta^t:B(\rho_0,r)\to F_\delta^t(B(\rho_0,r))$ is a diffeomorphism. Moreover, if we let $(x_1,\xi_1)=\Phi_0^{t_0}(x_0,\xi_0)$, it follows from \eqref{eq: F delta inverse lip} that there is a small constant $c_1>0$ depending on $c_0$ such that 
\[B(x_1,\tfrac{c_1r}{(t+1)})\times B(\xi_0,\tfrac{c_1r}{(t+1)})\subset F_\delta^t\big(B(\rho_0,r)\big).\]

\end{proof}
This result implies the existence of a local generating function for the flow.
\begin{prop}[Existence of a generating function]\label{prop: existence gen fun} In the notation and assumptions of Proposition \ref{prop: projectability}, one can find a function $\varphi_{\delta}(t,x,\eta)$ defined for $(x,\eta)\in \Omega_{r}$ and $|t-t_0|<r$ such that for any $(y,\eta)\in B(\rho_0,\frac{c_0 r}{1+t})$, one has
\[\Phi^t_{\delta}(y,\eta)=(x,\xi)\Leftrightarrow (y,\xi)=(\nabla_\eta \varphi_{\delta}, \nabla_x\varphi_{\delta}).\]
Moreover, $\varphi_\delta$ solves the Hamilton-Jacobi equation
\begin{equation} \label{eq: hamilton jacobi eq}\partial_t \varphi_\delta(t,x,\eta)=-p_\delta(x,\nabla_x \varphi_\delta(t,x,\eta)).\end{equation}\end{prop}

\begin{proof}Introduce the Lagrangian
\[\Lambda=\big\{(t,p_\delta(\Phi^t_\delta(y,\eta));\Phi^t_\delta(y,\eta);(y,\eta))\big\}\subset T^*\mathbf{R}\times T^*\mathbf{R}^d\times T^*\mathbf{R}^d.\]
Then the Hamilton equations imply that the $1$-form $\omega:=-\tau\mathrm{d}t+\xi\mathrm{d}x+y\mathrm{d}\eta$ satisfies $\mathrm{d}\omega_{|\Lambda}=0$. Thus $\omega$ is locally exact and there exists a function $\varphi_\delta(t,x,\eta)$ such that $\omega=\mathrm{d}\varphi_\delta$ on $\Lambda$, in other words
\[\tau=-\partial_t \varphi_\delta, \quad \xi=\nabla_x \varphi_\delta, \quad y=\nabla_\eta \varphi_\delta,\]
which exactly means that $\varphi_\delta$ is a generating function for the flow.

\end{proof}

\subsubsection{Composition of generating function} We record a lemma about the composition of generating functions under an assumption of projectability of the graphs.

\begin{lem} \label{lem: composition generating functions} Consider two symplectomorphisms of $T^*\mathbf{R}^d$,   $\kappa_1:(x_0,\xi_0)\mapsto (x_1,\xi_1)$ and $\kappa_2:(x_1,\xi_1)\mapsto (x_2,\xi_2)$, and let $\kappa_3=\kappa_2\circ \kappa_1$. Assume that the graphs of $\kappa_i$ projects diffeomorphically on the $(x,\eta)$ coordinates for all $i=1,2,3$. Let $\varphi_i(x_{i+1},\xi_i)$ be a generating function for $\kappa_i$, for $i=1,2$. Then, a generating function for $\kappa_3$ is given by
\begin{equation} \label{eq: formule phi 3} \varphi_3(x_2,\xi_0):=\varphi_2(x_2,\xi_1)-\langle x_1,\xi_1\rangle+\varphi_1(x_1,\xi_0),\end{equation}
where $(x_1,\xi_1)$ is fixed by the condition $(\bullet,\xi_0)\overset{\kappa_1}{\longrightarrow} (x_1,\xi_1)\overset{\kappa_2}{\longrightarrow} (x_2,\bullet)$.\end{lem}
\begin{proof}Consider a curve $(x_0(s),\xi_0(s))$ and let $(x_1(s),\xi_1(s)):=\kappa_1(x_0(s),\xi_0(s))$ and $(x_2(s),\xi_2(s)):=\kappa_2(x_1(s),\xi_1(s))$. Then, with $\varphi_3$ defined by the formula \eqref{eq: formule phi 3} above, one has
\[\frac{\mathrm{d}}{\mathrm{d}s} \varphi_3(x_2(s),\xi_0(s))=\langle \nabla_{x_2} \varphi_2,\dot x_2\rangle +\langle \nabla_{\xi_1} \varphi_2,\dot \xi_1\rangle-\langle \dot x_1,\xi_1\rangle-\langle x_1,\dot \xi_1\rangle+\langle \nabla_{x_1} \varphi_1,\dot x_1\rangle+\langle \nabla_{\xi_0} \varphi_1,\dot \xi_0\rangle.\]
But we know already that $\nabla_{x_2}\varphi_2=\xi_2$, $\nabla_{\xi_1} \varphi_2=x_1$, $\nabla_{x_1}\varphi_1=\xi_1$, $\nabla_{\xi_0}\varphi_1=x_0$, giving
\[\langle \nabla_{x_2} \varphi_3,\dot x_2\rangle+\langle \nabla_{\xi_0} \varphi_3,\dot \xi_0\rangle=\frac{\mathrm{d}}{\mathrm{d}s} \varphi_3(x_2(s),\xi_0(s))=\langle \xi_2,\dot x_2\rangle+\langle x_0,\dot \xi_0\rangle.\]
Since the graph of $\kappa_3$ projects diffeomorphically on the $(x_2,\xi_0)$ coordinates, this implies that on the graph of $\kappa_3$ we have
\[\xi_2=\nabla_{x_2}\varphi_3, \qquad x_0=\nabla_{\xi_0}\varphi_3,\]
in other words $\varphi_3$ is a generating function for $\kappa_3$.\end{proof}
\subsection{Interpretation of the generating function} \label{subsec: gen func twisted flow} As we work in charts, we have to study the classical dynamics in coordinates, e.g. consider $\widetilde f_j\circ \Phi_\delta^t\circ \widetilde f_i^{-1}$ (and accordingly we shall study the quantum propagator in coordinates $(f_j)_*\mathrm{e}^{-\frac{\mathrm{i}}{h} tP_{h,\delta}}f_i^*$). The generating functions for the flow in coordinates are identified with actions of trajectories on the base manifold $T^*M$, modulo some chart-dependent terms that are not invariantly defined. These chart-dependent terms will eventually cancel out when computing the trace, but we have to keep track of them when constructing approximations for $\mathrm{e}^{-\frac{\mathrm{i}}{h} tP_{h,\delta}}$ in charts.

\begin{defi}Let $\alpha=\xi\mathrm{d}x$ denote the canonical $1$-form on $T^*M$. If $\gamma:[0,T]\to T^*M$ is a smooth path, one defines the \emph{action} of $\gamma$ as
\[S(\gamma):=\int_{\gamma}\alpha.\]
Let $\{q_t\}_{t\ge 0}$ be a time-dependent Hamiltonian on $T^*M$, which generates a flow $\{\Psi^t\}_{t\ge 0}$.
Let $\gamma=\{\Psi^t(\rho)\}_{t\in [0,T]}$ be an orbit of the flow $\Psi^t$. We define
\begin{equation} \label{eq: def action bis} \widehat S(\gamma):=S(\gamma)-\int_0^T q_t(\Psi^t(\rho))\mathrm{d}t.\end{equation}
\end{defi}

\begin{prop}\label{prop: generating function is action} Let $\{q_t\}$ be a time-dependent Hamiltonian on $T^*\mathbf{R}^d$, which generates a flow $\Psi^t$. Let $\psi(t,x,\eta)$ be a generating function for the flow $\Psi^t$, for $t\in [0,T]$, with initial condition $\psi(0,x,\eta)=\langle x,\eta\rangle$. Then one has
\[\Psi^t(y,\eta)=(x,\xi)\Longrightarrow \psi(t,x,\eta)-\langle y,\eta\rangle=\widehat S(\gamma),\]
where $\gamma$ denotes the piece of orbit $\{\Psi^s(y,\eta)\}_{s\in [0,t]}$.
\end{prop}

\begin{proof}We reuse the notation from the proof of Proposition \ref{prop: existence gen fun}, but now allow the Hamiltonian to depend on time. Let $\Psi^t(y,\eta)=(x,\xi)$. Consider the curve $\gamma(s)=\{(s,q_s(\Psi^s(y,\eta));\Psi^s(y,\eta);y,\eta)\}_{s\in [0,t]}$ in $\Lambda$. Then, on the one hand
\[\int_{\gamma} \mathrm{d}\psi=\psi(t,x,\eta)-\psi(0,y,\eta),\]
and on the other hand
\[\int_{\gamma}\mathrm{d}\psi=\int_{\gamma}-\tau\mathrm{d}s+\xi\mathrm{d}x+y\mathrm{d}\eta.\]
Since $\eta$ is constant along the trajectory this simplifies to
\[\int_{\gamma} \psi=\int_{\gamma} \xi\mathrm{d}x-q_s\mathrm{d}s,\]
which is the sought formula.
\end{proof}
It is important for us to understand the action of the image of a path by a symplectomorphism. While there is no basic answer in general to my knowledge, some simple formulae are available in some cases.
\begin{lem}\label{lem: symplectic lift action} Let $f:\mathbf{R}^d\to \mathbf{R}^d$ be a smooth diffeomorphism and let $\widetilde f:T^*\mathbf{R}^d\to T^*\mathbf{R}^d$ denote its symplectic lift. For any path $\gamma$, we have
\[S(\widetilde f(\gamma))=S(\gamma).\]
\end{lem}
\begin{proof}Just check that the one form $\xi \mathrm{d}x$ is preserved by $\widetilde f$.
\end{proof}
We now look at the action of the image of $\gamma$ by a linear symplectomorphism.

\begin{defi}\label{defi: cocycle action}Let $\mathcal A=\begin{pmatrix}
A & B \\ C & D
\end{pmatrix}\in \mathrm{Sp}(\mathbf{R},2d)$. Define, for $\rho=(x,\xi)\in T^*\mathbf{R}^d$:
\[F(\mathcal A,\rho):=\langle B\xi,Cx\rangle +\frac{1}{2}\langle Cx,Ax\rangle +\frac{1}{2}\langle D\xi,B\xi\rangle.\]
\end{defi}
\begin{prop}\label{prop: action of image}Let $\mathcal A\in \operatorname{Sp}(\mathbf{R},2d)$. Let $\gamma$ be a path in $T^*\mathbf{R}^d$ with end points $\rho_0,\rho_1$. Then,
\[S(\mathcal A\gamma)-S(\gamma)=F(\mathcal A,\rho_1)-F(\mathcal A,\rho_0).\]
\end{prop}
\begin{proof}This is a direct computation using the fact that $\mathcal A$ is symplectic.\end{proof}

\begin{prop} \label{prop: properties cocycle}  We record for the following some useful properties of $F$.

\begin{enumerate}\item $F$ satisfies the cocycle property $F(\mathcal A_2\circ \mathcal A_1,\rho)=F(\mathcal A_2,\mathcal A_1(\rho))+F(\mathcal A_1,\rho)$. In particular, if $F(\mathcal A^{-1},\mathcal A \rho)=-F(\mathcal A,\rho)$.

\item If the upper-left block $A$ of $\mathcal A$ is invertible, then the function
\begin{equation} \label{eq: gen func for symplectic map}\psi(x,\eta):=\frac{1}{2} \langle C A^{-1} x, x\rangle +\langle A^{-1} x, \eta\rangle -\frac{1}{2}\langle A^{-1}B\eta,\eta\rangle\end{equation}
is a generating function for the map $\mathcal A:(y,\eta)\mapsto ( x,\xi)$, and whenever $\mathcal A(y,\eta)=(x,\xi)$, one has
\[F(\mathcal A,(y,\eta))=\psi(x,\eta)-\langle y,\eta\rangle.\]\end{enumerate}\end{prop}
As a consequence, one gets the following:
\begin{cor}\label{cor: generating function as action plus difference} Let $p$ be a Hamiltonian on $T^*\mathbf{R}^d$, with associated Hamiltonian flow $\Phi^t$. Given a symplectic matrix $\mathcal A\in \operatorname{Sp}(\mathbf{R},2d)$, we define a new flow $\Phi^t_{\mathcal A}:=\mathcal A^{-1}\circ \Phi^t\circ \mathcal A$, and let $\varphi_{\mathcal A}(t,\tilde x_1,\tilde \xi_0)$ be a short time generating function for $\Phi^t_{\mathcal A}$, with initial condition $\varphi_{\mathcal A}(0,\tilde x_1,\tilde \xi_0)=\langle \tilde x_1,\tilde \xi_0\rangle$. Then, whenever $\Phi^t_{\mathcal A}(\tilde x_0,\tilde \xi_0)=(\tilde x_1,\tilde \xi_1)$, one has
\[\varphi_{\mathcal A}(t,\tilde x_1,\tilde \xi_0)-\langle \tilde x_1,\tilde \xi_0\rangle=\widehat S(\gamma)+F(\mathcal A,(\tilde x_0,\tilde \xi_0))-F(\mathcal A,(\tilde x_1,\tilde \xi_1)),\]
where $\widehat S(\gamma)$ is the action of the path $\gamma:=\{\Phi^s(x_0,\xi_0)\}_{s\in [0,t]}$. \end{cor}
\begin{proof}Denote by $(x_0,\xi_0);(x_1,\xi_1)$ the endpoints of $\gamma$. Let $\tilde \gamma=\mathcal A\gamma$, and accordingly denote its endpoints by $(\tilde x_0,\tilde \xi_0);(\tilde x_1,\tilde \xi_1)$. The flow $\Phi^t_{\mathcal A}$ is associated with the Hamiltonian $p_{\mathcal A}:=p\circ \mathcal A^{-1}$. By Proposition \ref{prop: generating function is action}, one has
\[\varphi_{\mathcal A}(t,\tilde x_1,\tilde \xi_0)-\langle \tilde x_0,\tilde \xi_0\rangle=\widehat S(\tilde \gamma)=\int_{\tilde \gamma} \tilde \xi \mathrm{d}\tilde x-p\circ \mathcal A^{-1}\mathrm{d}t=S(\tilde \gamma)-\int_{\gamma} p\mathrm{d}t,\]
where we have made a change of variable in the second equality. By Proposition \ref{prop: action of image}, one has
\[S(\tilde \gamma)-S(\gamma)=F(\mathcal A,(\tilde x_0,\tilde \xi_0))-F(\mathcal A,(\tilde x_1,\tilde \xi_1)).\]
The result follows.
\end{proof}

\subsection{Iteration}\label{subsec: iteration} We anticipate the next sections and explain how one can build a long time generating function by an iteration procedure based on Lemma \ref{lem: composition generating functions}, using short time generating functions. While one can always obtain a long time generating function valid for $t\in [t_0-\varepsilon,t_0+\varepsilon]$ by applying the Poincar{\'e} lemma as in Proposition \ref{prop: existence gen fun}, there is no apparent canonical choice: if $\varphi(t,x,\eta)$ is a generating function then so is $\varphi(t,x,\eta)+c$, and the choice of $c$ is not obvious. In contrast, for short times, a natural choice---because the associated Fourier integral operator must be the identity---is to impose the initial condition $\varphi(0,x,\eta)=\langle x,\eta\rangle$.

Fix $\rho=\rho_0$ then let $\rho_k=\Phi^{k}_0(\rho)$. Let $\varphi_k(x_{k+1},\xi_k)$ be a generating function for $\Phi^{1}_\delta$ in adapted coordinates  as described in Proposition \ref{prop: existence gen fun}, defined on $B_{\pi_1(\rho_{k+1})}(\varepsilon)\times B_{\pi_2(\rho_k)}(\varepsilon)$, for some small $\varepsilon=\varepsilon(T_0)$. We define a sequence of generating functions $\phi_k(x_k,\xi_0)$, for $\xi_0\in B(0,h^\nu)$, by:
\begin{equation}\label{eq: sequence gen func} \phi_{k+1}(x_{k+1},\xi_0)=\varphi_k(x_{k+1},\xi_k)-\langle x_k,\xi_k\rangle +\phi_k(x_k,\xi_0), \end{equation}
where $(x_k,\xi_k)$ is fixed by the condition $(\bullet,\xi_0)\overset{\Phi^{k}}{\longrightarrow} (x_k,\xi_k) \overset{\Phi^{1}}{\longrightarrow} (x_{k+1},\bullet)$. The next proposition quantifies Lemma \ref{lem: composition generating functions} by precising the domain of definition of the generating functions.
It is proved by using the same estimates as in the proof of Proposition \ref{prop: projectability}.
\begin{prop}\label{prop: iterated generating function} There is a constant $\mathfrak c=\mathfrak c(\alpha,\beta)>0$ and a large constant $C>0$ such that for $k\le \mathfrak c(\alpha,\beta)|\log h|$, letting $R_k=\mathrm{e}^{-Ck}h^\nu$, the function $\phi_k$ is well-defined on $B_{\pi_1(\rho_k)}(R_k)\times B_{\pi_2(\rho_0)}(h^\nu)$ and generates the flow at time $k$. That is, for any $(x,\eta)\in B(\pi_1(\rho_k),R_k)\times B(\pi_2(\rho_0),R_0)$, there is a unique $y\in B(\pi_1(\rho_0),R_0)$ such that $\Phi^k(y,\eta)=(x,\bullet)$.
\end{prop}

\begin{rem}We can also consider time-dependent generating functions. We can compose the generating function for $\Phi^{k}_\delta$ with a generating function for $\Phi^{1+s}_\delta$, for $|s|\le \mathrm{e}^{-Ck}$. This produces a generating function for $\Phi^{k+1+s}_\delta$, for $|s|\le c \mathrm{e}^{-Ck}$.\end{rem}

\subsection{Estimates on the generating function} In the rest of the section we fix a time threshold $T=\mathfrak c|\log h|$ and work in the regime $|t|\le T$. We fix $\rho_0\in T^*M$ and a time $|t_0|\le T$. We consider the flow in adapted coordinates defined in \eqref{eq: def twisted flow}. To lighten the notation, we will still denote it by $\Phi_\delta^t$. Similarly, we still write $\rho_0$ for the point in coordinates $\mathcal A_i(0)^{-1}(\widetilde f_i(\rho_0))$. In the whole subsection, we assume $|t-t_0|\le r$ and $|\rho-\rho_0|\le r$ with
\begin{equation} \label{eq: range of param} r+\delta h^{-2\beta}\ll \mathrm{e}^{-2CT}.\end{equation}
Then, in the notation of Proposition \ref{prop: existence gen fun}, for any $|t-t_0|\le r$, the flow map $\Phi^t_\delta$ admits a local generating function $\varphi_\delta(t,x,\eta)$, defined for $(x,\eta)\in \Omega_r$.

\begin{lem}\label{lem: derivees coeff matrix} We work in adapted coordinates near $\rho_0$ and $\Phi^{t_0}(\rho_0)$. Write the differential of the flow as $\mathrm{d}\Phi_\delta^t:=\begin{pmatrix}
A & B \\ C & D
\end{pmatrix}$.  Then, $\|A^{-1}\|_{C^\ell}\lesssim h^{-\beta \ell} \mathrm{e}^{2C \ell t}$. Also, any product of the matrices $A,B,C,D,A^{-1}$ has $C^\ell$-norm controlled by $h^{-\beta \ell} \mathrm{e}^{2C \ell t}$, for $\ell\ge 1$.

\end{lem}

\begin{proof}The $C^\ell$ norms of the matrices $A,B,C,D$ are controlled by the $C^{\ell+1}$ norm of $\Phi^t$ thus are $\lesssim h^{-\beta \ell} \mathrm{e}^{C(\ell+1)t}$ by Proposition \ref{prop: short time derivative flow}. We compute $\partial^\ell(A^{-1})$ and notice that because $\|A^{-1}\|\le 2$ (this is Lemma \ref{lem: matrice differentielle flot}), we have for $\ell\ge 1$:
\[\|\partial^\ell A^{-1}\|\lesssim \sup_{\substack{\ell_1+\ldots+\ell_n=\ell \\ \ell_i\ge 1}} \prod_{i=1}^n\|\partial^{\ell_i} A\|\lesssim \sup_{\ell_1+\ldots+\ell_n=\ell} \prod_{i=1}^n h^{-\beta \ell_i}\mathrm{e}^{C(\ell_i+1)t}\lesssim h^{-\beta \ell}\mathrm{e}^{2C \ell t}. \]
The result on the products follows.\end{proof}
This implies the following estimates on the space derivatives of $\varphi_\delta$.
\begin{lem}[Space derivatives of $\varphi_\delta$]\label{lem: space derivatives generating function} For all $\ell\ge 2$, one has
\begin{equation} \label{eq: space derivative gen fun} \|\varphi_\delta\|_{C^\ell_{x,\eta}}\lesssim h^{-\beta(\ell-2)}(1+t)\mathrm{e}^{2C(\ell-2)t}.\end{equation}
\end{lem}
\begin{proof}We let $G(x,\eta):=(\xi(x,\eta),y(x,\eta))=\nabla_{x,\eta} \varphi_\delta(t,x,\eta)$. Writing $\mathrm{d}\Phi^t=\begin{pmatrix} A & B \\ C & D\end{pmatrix}$, one has
\[\mathrm{d}G(x,\eta)=\begin{pmatrix}A^{-1}& -A^{-1}B \\ C A^{-1} & D-CA^{-1}B\end{pmatrix},\]
where the matrix on the right-hand side is evaluated at $(y(x,\eta),\eta)$. We define a function that depends artificially on the variable $\xi$:
\[M(\xi,y,\eta):=\begin{pmatrix}A^{-1}& -A^{-1}B \\ C A^{-1} & D-CA^{-1}B\end{pmatrix}(y,\eta).\]
Then, one can express
\begin{equation} \label{eq: dG= M(G)} \mathrm{d}G (x,\eta)=M(G(x,\eta),\eta).\end{equation}
By Lemma \ref{lem: derivees coeff matrix}, we have $\|M\|_{C^\ell_{\xi,y,\eta}}\lesssim h^{-\beta \ell} (1+t)\mathrm{e}^{2C \ell t}$ for all $\ell\ge 0$. By induction on $\ell$ using \eqref{eq: dG= M(G)} we obtain 
\[\forall \ell\ge 1, \quad \|G\|_{C^\ell_{x,\eta}}\lesssim h^{-\beta(\ell-1)}(1+t)\mathrm{e}^{2C(\ell-1)t}.\]
This is equivalent to \eqref{eq: space derivative gen fun} because $G=\nabla \varphi_\delta$.
\end{proof}
Using repeatedly the identity $\partial_t \varphi_\delta=-p_\delta(x,\nabla_x \varphi_\delta)$, Lemma \ref{lem: space derivatives generating function} allows getting estimates on the time derivatives of $\varphi_\delta$.
\begin{lem}[Incorporating time derivatives] \label{lem: spacetime derivatives generating function} For all $\ell+n\ge 2$, one has
\begin{equation} \label{eq: time der phi t}\|\partial_t^n \varphi_\delta\|_{C^\ell_{x,\eta}}\lesssim h^{-\beta(\ell+n-2)} (1+t)\mathrm{e}^{2C(\ell+n-2)t}.\end{equation}
At $t=0$ the estimates improve slightly:
\begin{equation} \label{eq: time der phi t improves}\forall n\ge 1, \qquad \|\partial_t^n \varphi_\delta(t=0)\|_{C^\ell_{x,\eta}}\lesssim h^{-\beta(\ell+n-3)_+}.\end{equation}
\end{lem}
\begin{proof}\textbf{1.} The proof of \eqref{eq: time der phi t} goes by induction on $n$. We use the shorthand $\mathbf x:=(x,\eta)$. For $n=1$, we have $\partial_t \varphi_\delta=-p_\delta(x,\nabla_x \varphi_\delta)$, then
\[\|\partial_t \varphi_\delta\|_{C^\ell_{\mathbf x}}\lesssim h^{-\beta(\ell-1)}(1+t) \mathrm{e}^{2C(\ell-1)t}\qquad \forall \ell \ge 1,\]
which proves the result for $n=1$. Assume that \eqref{eq: time der phi t} holds up to rank $n\ge 1$. Starting from
\[\partial_t^{n+1}\partial_{\mathbf x}^\ell \varphi_\delta=-\partial_t^n\partial_{\mathbf x}^\ell  \big(p_\delta(x,\nabla_x\varphi_\delta)\big),\]
we obtain by Fa\`a di Bruno's formula:
\[\big|\partial_t^{n+1}\partial_{\mathbf x}^\ell \varphi_\delta\big|\lesssim \sum_{k=1}^{\ell+n} \|p_\delta\|_{C^k}\sum_{\substack{\ell_1+\ldots+\ell_k=\ell \\ n_1+\ldots+n_k=n \\ \substack \ell_i+n_i\ge 1}} \prod_{i=1}^k \|\partial_t^{n_i} \varphi_\delta\|_{C_{\mathbf x}^{\ell_i+1}}.\]
By the assumptions on $p_\delta$ and the induction hypothesis, we infer
\begin{equation}\label{eq: induction proof bounds derivatives phi}\big|\partial_t^{n+1}\partial_{\mathbf x}^\ell \varphi_\delta\big|\lesssim \sum_{k=1}^{\ell+n} h^{-\beta(k-2)_+}\sum_{\substack{\ell_1+\ldots+\ell_k=\ell \\ n_1+\ldots+n_k=n \\ \substack \ell_i+n_i\ge 1}} \prod_{i=1}^k h^{-\beta(\ell_i+n_i-1)_+} \mathrm{e}^{2C(\ell_i+n_i-1)t}.\end{equation}
Any product of exponentials appearing on the right-hand side of \eqref{eq: induction proof bounds derivatives phi} is smaller than $\mathrm{e}^{2C(\ell+n-1)t}$. Also, we use
\[(k-2)_++(\ell+n-k)\le \ell+n-1, \quad \forall k\ge 1,\]
to deduce \eqref{eq: time der phi t} at rank $n+1$.

\textbf{2.} The proof of the improved estimate at $t=0$ also works by induction. For $n=1$, using $\partial_t \varphi_\delta(0,x,\eta)=-p_\delta(x,\eta)$, we have
\[\|\partial_t \varphi_\delta(t=0)\|_{C^\ell}\lesssim \|p_\delta\|_{C^\ell}\lesssim h^{-\beta(\ell-2)_+}.\]	
Assume that \eqref{eq: time der phi t improves} holds up to rank $n\ge 1$. Then, \eqref{eq: induction proof bounds derivatives phi} becomes
\[\big|\partial_t^{n+1}\partial_{\mathbf x}^\ell  p_\delta(x,\nabla_x\varphi_\delta(t=0))\big|\lesssim \sum_{k=1}^{\ell+n}h^{-\beta(k-2)_+}\sum_{\substack{\ell_1+\ldots+\ell_k=\ell \\ n_1+\ldots+n_k=n \\ \substack \ell_i+n_i\ge 1}} \prod_{i=1}^k h^{-\beta(\ell_i+n_i-2)_+}.\]
This time we use
\[(k-2)_++\sum(\ell_i+n_i-2)_+\le \big( \big(k+\sum(\ell_i+n_i)\big)-2\big)_+\le (\ell+n-2)_+\]
to conclude.
\end{proof}
Lemma \ref{lem: spacetime derivatives generating function} implies the following estimates on the function $a_0:=|\partial_{x\eta}^2 \varphi_\delta|^{\frac 12}$, which will be the leading amplitude appearing in the FIO parametrix for the quantum propagator.
\begin{prop}\label{prop: derivees a0} Define $a_0(t,x,\eta):=\det \Big(\frac{\partial^2\varphi_\delta}{\partial x\partial \eta}\Big)^{\frac 12}$. Then, we have
\[\forall \ell\ge 0, \quad |\partial_{t,x,\eta}^\ell a_0|\lesssim a_0 h^{-\beta \ell}\mathrm{e}^{2C \ell t}.\]
\end{prop}

\subsubsection{Control of the intermediate points} Let $\rho_0\in T^*M$ and $0<s_0<t_0$. Set $\rho_{s_0}=\Phi^{s_0}(\rho_0)$ and $\rho_{t_0}=\Phi^{t_0}(\rho_0)$. We consider bases of $T_{\rho_0}(T^*M)$,  $T_{\rho_{s_0}}(T^*M)$ and $T_{\rho_{t_0}}(T^*M)$ adapted to the dynamics as constructed in Lemma \ref{lem: adapted bases}. In turn, we consider the flow maps $\Phi^s_\delta$ and $\Phi_\delta^t$ in these adapted coordinates, for $s$ close to $s_0$ and $t$ close to $t_0$, respectively.

In adapted coordinates, the flow at time $t$ is projectable on the $(x_t,\xi_0)$ coordinates, for $|t-t_0|\le \varepsilon \mathrm{e}^{-2C t_0}$. Thus we can define a map
\[(x_t,\xi_0)\mapsto (x_0(x_t,\xi_0),\xi_t(x_t,\xi_0)),\]
and also a map 
\[(x_t,\xi_0)\mapsto (x_s(x_t,\xi_0),\xi_s(x_t,\xi_0)):=\Phi^s_\delta(x_0(x_t,\xi_0),\xi_0),\]
that associates to the endpoints $(x_t,\xi_0)$ the intermediate point $(x_s,\xi_s)$ in adapted coordinates near $\rho_{s_0}$. Letting $\varphi_\delta(t,x,\eta)$ be a generating function for the time $t$ flow in adapted coordinates, we may write $(x_s,\xi_s)=\Phi_\delta^s(\nabla_\eta \varphi_\delta(t,x_t,\xi_0),\xi_0)=\Phi_\delta^{s-t}(x_t,\nabla_x \varphi_\delta(t,x_t,\xi_0))$, which implies, using Proposition \ref{prop: short time derivative flow} and Lemma \ref{lem: spacetime derivatives generating function} together with Fa\`a di Bruno's formula that
\begin{equation} \label{eq: C ell norm interm point ell ge 1}\forall \ell\ge 1, \quad \|x_s\|_{C^\ell_{t,x_t,\xi_0}}, \|\xi_s\|_{C^\ell_{t,x_t,\xi_0}}\lesssim h^{-\beta(\ell-1)_+}\mathrm{e}^{C\ell t}. \end{equation}
Actually, we have some slightly better results for $\ell=1$, that we describe here.
\begin{prop}\label{prop: intermediate points} 

Let $(x_s,\xi_s)$ be defined above. Write $x_s=(x_s^1,x_s')\in \mathbf{R}\times \mathbf{R}^{d-1}$ and $\xi_0=(\xi_0^1,\xi_0')\in \mathbf{R}\times \mathbf{R}^{d-1}$. 

\begin{enumerate} \item All the partial derivatives
\begin{equation} \label{eq: intermediate points}
\frac{\partial x_s'}{\partial \xi_0}, \quad \frac{\partial x_s^1}{\partial \xi_0'}, \quad \frac{\partial x_s}{\partial x_t}, \quad \frac{\partial \xi_s}{\partial \xi_0}, \quad \frac{\partial \xi_s}{\partial x_t}\end{equation}
are bounded by $\mathcal O(1)$. The partial derivative $\frac{\partial x_s^1}{\partial \xi_0^1}$ is bounded by $\mathcal O(1+|t-s|)$. Combined with \eqref{eq: C ell norm interm point ell ge 1}, there is a constant $C>0$ such that
\[\forall \ell\ge 1, \quad \|x_s\|_{C^\ell_{t,x_t,\xi_0}}, \|\xi_s\|_{C^\ell_{t,x_t,\xi_0}}\lesssim h^{-\beta(\ell-1)_+}(1+|t-s|)\mathrm{e}^{C(\ell-1)t}.\]

\item If $|s-t|=\mathcal O(1)$, the partial derivatives $\frac{\partial x_s(t,x,\eta)}{\partial t}, \frac{\partial \xi_s(t,x,\eta)}{\partial t}$ are $\mathcal O(1)$. \end{enumerate}\end{prop}
\begin{proof} 
\textbf{1.} Consider the flow maps in adapted coordinates, for $|s-s_0|\le \mathrm{e}^{-2CT}$ and $|t-t_0|\le \mathrm{e}^{-2CT}$, respectively:
\[\mathrm{d}\Phi^s=\begin{pmatrix}
A_s & B_s \\ C_s & D_s
\end{pmatrix}, \qquad \mathrm{d}\Phi^t=\begin{pmatrix}
A_t & B_t \\ C_t & D_t
\end{pmatrix}.\] 
Given endpoints $(x_t,\xi_0)$, we define associated points $(x_0,\xi_t,x_s,\xi_s)$ by the conditions
\[(x_0,\xi_0)\overset{\Phi^s}{\longrightarrow}(x_s,\xi_s)\overset{\Phi^{t-s}}{\longrightarrow}(x_t,\xi_t).\]
Then, one has
\begin{equation} \label{eq: partial s x_s par rapport t}\begin{pmatrix}
\frac{\partial x_s}{\partial x_t}	&   \frac{\partial x_s}{\partial \xi_0} \\
\frac{\partial \xi_s}{\partial x_t}	& \frac{\partial \xi_s}{\partial \xi_0}
\end{pmatrix}=\begin{pmatrix}
A_sA_t^{-1}&   -A_sA_t^{-1}B_t+B_s \\
C_sA_t^{-1}& D_s-C_sA_t^{-1}B_t
\end{pmatrix}=:M(x_0,\xi_0,s,t;\delta).\end{equation}

\textbf{2.} We claim that $\|M\|\lesssim (1+|s-t|)$ in operator norm. At the point $\rho_0$, for $s=s_0,t=t_0$ and for $\delta=0$, we have
\[M(\rho_0,s_0,t_0,\delta=0)=\begin{pmatrix}A_{s_0} A_{t_0}^{-1} & -A_{s_0} A_{t_0}^{-1}B_{t_0}+B_{s_0} \\ 0 & D_{s_0}
\end{pmatrix}.\]
The upper-left block is actually the upper left-block of $\big(\mathrm{d}\Phi_{t_0-s_0}\big)^{-1}$ in adapted coordinates, which has norm $\lesssim 1$. In the notation of Lemma \ref{lem: matrice differentielle flot}, the upper-right block is simply $\begin{pmatrix}
\beta_{s_0}-\beta_{t_0} & 0 \\ 0 & 0
\end{pmatrix}$, which is bounded by $1+|t_0-s_0|$, by the proof of Lemma \ref{lem: matrice differentielle flot}. Finally, the lower-right block is bounded by $\lesssim 1$ in norm, by Lemma \ref{lem: matrice differentielle flot}. This implies the proposition at $(\rho_0,s_0,t_0,\delta=0)$. The estimates persist for different $(x_0,\xi_0),s,t,\delta$ provided
\[|(x_0,\xi_0)-\rho_0|+|s-s_0|+|t-t_0|+\delta h^{-2\beta}\le \mathrm{e}^{-2CT},\]
by the proof of Lemma \ref{lem: matrice differentielle flot}.

\textbf{3.} For the statement about time derivatives, letting $\rho_s=(x_s,\xi_s)$ depend on endpoints $(x,\xi_0)$ and the time $t$, we have
\[\rho_s(t,x,\xi_0)=\Phi^{s-t}(x,\nabla_x \varphi(t,x,\xi_0)).\]
Then,
\[\partial_t \rho_s=-H_p(\rho_s)+\mathrm{d}\Phi^{s-t}\cdot (0,\partial_t \nabla_x\varphi(t,x,\xi_0)).\]
The quantity $\partial_t \nabla_x \varphi=\nabla_x (-p(x,\nabla_x \varphi))$ is bounded by $\|p\|_{C^1}\cdot \|\nabla_x^2 \varphi\|$ which is $\mathcal O(1)$ (indeed $\nabla_x^2 \varphi$ is the upper half-block of the matrix $\mathrm{d}G$ appearing in the proof of Lemma \ref{lem: space derivatives generating function}). Because $|s-t|=\mathcal O(1)$ and the Hamiltonian vector field is bounded independently of $h$, we deduce $|\partial_t \rho_s|=\mathcal O(1)$.

\end{proof}

\subsection{$\delta$-derivatives of the generating function} We now give a formula for the $\delta$-derivative of $\varphi_\delta$. \begin{prop}\label{prop: delta derivative gen fun} Let $\{p_\delta\}$ be a smoothly varying family of smooth Hamiltonians on $T^*\mathbf{R}^d$, and let $\{\Phi_\delta^t\}$ be the associated family of Hamiltonian flows. Let $\varphi_\delta$ be a short-time generating function for the flow $\Phi_\delta^t$, with initial condition $\varphi_\delta(0,x,\eta)=\langle x,\eta\rangle$. Then for all $t$ small enough, one has
\begin{equation}\label{eq: delta derivative generating function} \partial_\delta \varphi_\delta(t,x,\eta)=-\int_0^t \frac{\partial p_\delta}{\partial \delta}(\Phi^s_\delta(\nabla_\eta \varphi_\delta(t,x,\eta),\eta))\mathrm{d}s.\end{equation}
\end{prop}
\begin{proof} Let $(x_s,\xi_s):=\Phi_\delta^s(\nabla_\eta \varphi_\delta(t,x,\eta),\eta)$ for $s\in [0,t]$. By differentiating with respect to $\delta$ the formula
\[\varphi_\delta(t,x,\eta)-\langle x_0,\eta\rangle=\int_0^t \langle \dot x_s,\xi_s\rangle -p_\delta(x_s,\xi_s)\mathrm{d}s\]
proved in Proposition \ref{prop: generating function is action}, we obtain
\[\frac{\partial \varphi_\delta}{\partial \delta}=\langle \partial_\delta x_0,\eta\rangle+\int_0^t\big( \langle \partial_\delta \dot x_s,\xi_s\rangle+\langle \dot x_s,\partial_\delta \xi_s\rangle -(\partial_\delta p_\delta)(x_s,\xi_s)-\langle \nabla_{x,\xi} p_\delta, (\partial_\delta x_s,\partial_\delta \xi_s)\rangle \big) \mathrm{d}s.\]
By an integration by parts, 
\[\langle \partial_\delta x_0,\eta\rangle+\int_0^t\langle \partial_\delta \dot x_s,\xi_s\rangle\mathrm{d}s=\langle \partial_\delta x_0,\eta\rangle+ \big[\langle\partial_\delta x_s,\xi_s\rangle\big]_{s=0}^{s=t}-\int_0^t\langle \partial_\delta  x_s,\dot \xi_s\rangle\mathrm{d}s=-\int_0^t\langle \partial_\delta  x_s,\dot \xi_s\rangle\mathrm{d}s.\]
The first two terms in the second line have canceled out because $x_t=x$ is independent of $\delta$ and $\xi_0=\eta$. Also, using the Hamilton equations
\[\dot x_s =\frac{\partial p_\delta}{\partial \xi_s}, \quad \dot \xi_s=-\frac{\partial p_\delta}{\partial x_s}, \]
we obtain
\[-\langle \partial_\delta  x_s,\dot \xi_s\rangle+\langle \dot x_s,\partial_\delta \xi_s\rangle=\langle \nabla p_\delta,(\partial_\delta x_s,\partial_\delta \xi_s)\rangle,\]
and eventually
\[\frac{\partial \varphi_\delta}{\partial \delta}=-\int_0^t(\partial_\delta p_\delta)(x_s,\xi_s) \mathrm{d}s.\]

\end{proof}
A similar formula holds for the $\delta$-derivative of the long-time generating function from \S \ref{subsec: iteration} obtained by composing short time generating functions, due to the following elementary lemma.
\begin{lem}\label{lem: delta derivative composition gen fun} Consider some $\delta$-dependent maps $\kappa_1,\kappa_2$, admitting generating functions $\varphi_1,\varphi_2$. Let $\kappa_3=\kappa_2\circ \kappa_1$ and assume that the graph of $\kappa_3$ is projectable, so that one can define the composition $\varphi_3$ of $\varphi_2$ with $\varphi_1$ using Lemma \ref{lem: composition generating functions}. Then
\[\partial_\delta \varphi_3(x_2,\xi_0)=(\partial_\delta \varphi_2)(x_2,\xi_1)+(\partial_\delta\varphi_1)(x_1,\xi_0),\]
where $(x_1,\xi_1)=(x_1^\delta,\xi_1^\delta)$ depends on $\delta$ and is fixed by the condition
\[(\bullet,\xi_0)\overset{\kappa_1}{\longrightarrow}(x_1,\xi_1)\overset{\kappa_2}{\longrightarrow}(x_2,\bullet).\]

\end{lem} 
\begin{proof} Start from
\[\varphi_3(x_2,\xi_0)=\varphi_2(x_2,\xi_1^\delta)-\langle x_1^\delta,\xi_1^\delta\rangle+\varphi_1(x_1^\delta,\xi_0).\]
We differentiate this identity with respect to $\delta$:
\[\partial_\delta \varphi_3=\partial_\delta \varphi_2 +\langle \partial_{\xi_1}\varphi_2,\partial_\delta \xi_1^\delta\rangle -\langle \partial_\delta x_1^\delta,\xi_1^\delta\rangle-\langle x_1^\delta,\partial_{\delta}\xi_1^\delta\rangle+\partial_\delta \varphi_1+\langle \partial_{x_1}\varphi_1,\partial_\delta x_1^\delta\rangle.\]
Since $\tfrac{\partial \varphi_2}{\partial \xi_1}=x_1^\delta$ and $\tfrac{\partial \varphi_1}{\partial x_1}=\xi_1^\delta$, this simplifies to
\[\partial_\delta \varphi_3(x_2,\xi_0)=\partial_\delta \varphi_2(x_2,\xi_1)+\partial_\delta\varphi_1(x_1,\xi_0),\]
as we wished.
\end{proof}

\subsection{Estimates on $\delta$-derivatives}
					
			In the following, given vector valued functions $F^{(1)},\ldots,F^{(n)}$ and integers $\ell_1,\ldots,\ell_n$, the notation $\partial_x^{\ell_1}F^{(1)}\cdots \partial_x^{\ell_n}F^{(n)}$ represents any linear combination of terms of the form
			\begin{equation} \label{eq: symbol prod derivatives} \prod_{i=1}^n \partial_x^{\alpha_i} F^{(i)}_{j_i},\end{equation}
						where $\alpha_i$ is a multiindex of length $\ell_i$ and $F^{(i)}_{j_i}$ denotes a component of the vector $F^{(i)}$. The coefficients in the linear combination are independent of the $F^{(i)}$. These expressions may change from line to line.

						\begin{prop}[$\delta$-derivatives of the flow] \label{prop: bound delta derivatives of the flow} Let $\{p_\delta\}_{\delta\in [0,h^\alpha]}$ be a smoothly varying family of smooth symbols either on $T^*\mathbf{R}^d$ or on $T^*M$ for a closed manifold $M$. Fix $E\in \mathbf R$ and define the union of the energy shells \[\Theta:=\bigcup_{\delta\in [0,h^\alpha]} p_\delta^{-1}(E-1,E+1).\]
							We assume that $\|p_0\|_{C^\ell(\Theta)}\lesssim 1$ and $\|\partial_\delta^n p_\delta\|_{C^\ell(\Theta)}\lesssim h^{-\beta \ell}$ for all $n\ge 1$. Let $T\ge 1$. For all $n+k+\ell\ge 1$ and $|t|\le T$, we have
							\begin{equation} \label{eq: delta deriv flow bdd} \forall \rho\in \Theta, \qquad \Big|\partial_\delta^n \partial^\ell_\rho \partial_t^k \Phi^t_\delta(\rho)\Big|\lesssim T^n\mathrm{e}^{C(n+\ell) t}h^{-\beta(\ell+k+2n-1)_+}.\end{equation}
							If working on $T^*M$ the inequalities are understood in charts.
						\end{prop}
						\begin{rem}Using the Duhamel formula of \eqref{eq: duhamel perturbed flow bis} for the perturbed flow, it is in principle possible to express $\partial_\delta^n \partial^\ell_\rho \partial_t^k \Phi^t_\delta(\rho)$ as a sum of integrals of products of $(\delta,\rho)$ derivatives of $p_\delta$ and $\rho$-derivatives of $\Phi^t_\delta$ evaluated at different times. Because \eqref{eq: duhamel perturbed flow bis} involves splitting the $t$ variable as $t_0+t_1=t$, the expression of $\partial_\delta^n \partial^\ell_\rho \partial_t^k \Phi^t_\delta(\rho)$ (which is obtained by iterating this kind of procedure) involves some integrals over tuples of times that satisfy some recursive structure that can be encoded in a combinatorial tree. Because such an expression is fairly intricate and we do not need it in the present paper, we refrain from writing it here. Instead, the bound \eqref{eq: delta deriv flow bdd} is obtained directly by induction, without obtaining the exact expression of the derivative.
						\end{rem}
						\begin{rem}Up to increasing the constant $C$, for $n+\ell+k\ge 1$, we have
							\begin{equation}\label{eq: delta derivatives spacetime flow} \forall \rho\in \Theta, \qquad \Big|\partial_\delta^n \partial^\ell_\rho \partial_t^k \Phi^t_\delta(\rho)\Big|\lesssim \mathrm{e}^{C(n+\ell) t}h^{-\beta(\ell+k+2n-1)_+}, \end{equation}
							which we shall use in the following.
						\end{rem}
						\begin{proof} We prove the result on $T^*\mathbf{R}^d$, we will explain below how to deal with the case of closed manifolds. 
							
							\textbf{1.} We first prove \eqref{eq: delta derivatives spacetime flow} for $k=0$. The result is known for $n=0$. Then, the proof goes by induction on $n$. In the convention described at the beginning of the subsection for products of the form \eqref{eq: symbol prod derivatives}, one can rewrite \eqref{eq: duhamel perturbed flow bis} as
							\[\partial_\delta \Phi_\delta^t(\rho)=\int_0^t \partial_\rho\Phi_\delta^{s}(\Phi_\delta^{t-s}(\rho))\partial_{\rho}\partial_\delta p_\delta(\Phi_\delta^{t-s}(\rho))\mathrm{d}s.\]
							By differentiating in $\delta$ and $\rho$, one obtains
							\begin{equation} \label{eq: partial rho ell partial delta n+1 phi}\partial_\rho^\ell\partial_\delta^{n+1} \Phi_\delta^t(\rho)=\int_0^t \partial_\delta^n\partial_\rho^\ell\big[\partial_\rho\Phi_\delta^{s}(\Phi_\delta^{t-s})\partial_{\rho}\partial_\delta p_\delta(\Phi_\delta^{t-s}(\rho))\big]\mathrm{d}s.\end{equation}
							Let $f_\delta$ be either $\partial_\rho\Phi_\delta^s$ or $\partial_\rho\partial_\delta p$. Then,
							\[\partial_\delta^n\partial_\rho^\ell \big[f_\delta(\Phi_\delta^{t-s}(\rho))\big]=\sum_{n_1+n_2=n}\partial_\rho^\ell \partial_{\delta_1}^{n_1}\partial_{\delta_2}^{n_2} \big[f_{\delta_1}(\Phi_{\delta_2}^{t-s}(\rho))\big]\Big|_{\delta_1=\delta_2=\delta},\]
							which equals, by Fa\`a di Bruno's formula and the product rule
							\begin{equation}\label{eq: derivee n ieme f delta} \begin{aligned}
									\partial_\delta^n\partial_\rho^\ell \big[f_\delta(\Phi_\delta^{t-s}(\rho))\big]& =\sum_{n_1+n_2=n}\partial_\rho^\ell\Big[\sum_{k=1}^{n_2} (\partial_\rho^{k} \partial_\delta^{n_1} f_\delta)(\Phi_\delta^{t-s}(\rho))\sum_{r_1+\ldots+r_k=n_2} \prod_{i=1}^k \partial_\delta^{r_i}\Phi_\delta^{t-s}(\rho)\Big] \\
									&=\sum_{n_1+n_2=n}\sum_{\ell_1+\ell_2=\ell}\Big[\sum_{k=1}^{n_2} \partial_\rho^{\ell_1}\big[(\partial_\rho^{k} \partial_\delta^{n_1} f_\delta)(\Phi_\delta^{t-s}(\rho))\big] \partial_\rho^{\ell_2}\sum_{r_1+\ldots+r_k=n_2} \prod_{i=1}^k \partial_\delta^{r_i}\Phi_\delta^{t-s}(\rho)\Big].
							\end{aligned}\end{equation}
							By the induction hypothesis, one has
							\begin{equation} \label{eq: partial ell 2 flow} \Big|\partial_\rho^{\ell_2}\sum_{r_1+\ldots+r_k=n_2} \prod_{i=1}^k \partial_\delta^{r_i}\Phi_\delta^{t-s}(\rho)\Big|\lesssim T^{n_2}h^{-\beta(\ell_2+2n_2-k)}\mathrm{e}^{C(n_2+\ell_2)(t-s)},\end{equation}
							and
							\[\big|\partial_\rho^{\ell_1}\big[(\partial_\rho^{k} \partial_\delta^{n_1} f_\delta)(\Phi_\delta^{t-s}(\rho))\big]\big|\lesssim\sum_{m=1}^{\ell_1}\big|\partial_\rho^{m+k}\partial_\delta^{n_1}f_\delta\big|\cdot\prod_{r_1+\ldots+r_m=\ell_1} \big|\partial_\rho^{r_i}\Phi_\delta^{t-s}(\rho)\big|.\]
							For $f_\delta=\partial_\rho\Phi_\delta^s$, we have by the induction hypothesis
							\begin{equation} \label{eq: for f = partial Phi}\big|\partial_\rho^{\ell_1}\big[(\partial_\rho^{k} \partial_\delta^{n_1} \partial_\rho\Phi_\delta^s)(\Phi_\delta^{t-s}(\rho))\big]\big|\lesssim h^{-\beta(k+2n_1+\ell_1)}T^{n_1}\mathrm{e}^{C(\ell_1+k+n_1+1)s}\mathrm{e}^{C\ell_1(t-s)},\end{equation}
							while for $f_\delta=\partial_\rho\partial_\delta p_\delta$:
							\begin{equation} \label{eq: for f = partial p}\big|\partial_\rho^{\ell_1}\big[(\partial_\rho^{k} \partial_\delta^{n_1}\partial_\rho\partial_\delta p_\delta)(\Phi_\delta^{t-s}(\rho))\big]\big|\lesssim h^{-\beta(k+\ell_1+1)} \mathrm{e}^{C\ell_1(t-s)}.\end{equation}
							Recalling \eqref{eq: derivee n ieme f delta} and combining \eqref{eq: partial ell 2 flow} with \eqref{eq: for f = partial Phi} or \eqref{eq: for f = partial p}, we end up with
							\begin{itemize}
								\item $\big|\partial_\delta^n\partial_\rho^\ell \big[\partial_\rho\Phi_\delta^s(\Phi_\delta^{t-s}(\rho))\big]\big|\lesssim T^n h^{-\beta(\ell+2n)}\mathrm{e}^{C(n+\ell+1)t}$.
								\item $|\partial_\delta^n\partial_\rho^\ell \big[\partial_\delta\partial_\rho p(\Phi_\delta^{t-s}(\rho))\big]|\lesssim T^n h^{-\beta(\ell+2n+1)} \mathrm{e}^{C(n+\ell)(t-s)}$.
							\end{itemize}
							By expanding the derivative of the product in \eqref{eq: partial rho ell partial delta n+1 phi}, we obtain eventually
							\[\big|\partial_\rho^\ell\partial_\delta^{n+1} \Phi_\delta^t(\rho)\big|\le T^{n+1} h^{-\beta(\ell+2(n+1)-1)} \mathrm{e}^{C(n+\ell)t}.\]
							
							\textbf{2.} The case $k\ge 1$ follows from Lemma \ref{lem: formula time derivative Phi} and Lemma \ref{lem: intermediate bound delta deriv p} below, together with the case $k=0$ that we have just proved above.
							
							\textbf{3.} We explain how to adapt the proof when $\mathbf{R}^d$ is replaced by a closed manifold $M$. Fix a finite atlas $(U_\ell,f_\ell)$ of $M$, and let $(T^*U_\ell,\widetilde f_\ell)$ denote the associated atlas of $T^*M$. Fix $\rho\in \Theta$. Up to rescaling time, for each $k\ge 0$ one can find an index $\ell_k$ such that the piece of trajectory $\{\Phi^t_0(\rho)\}_{t\in [k-1,k+1]}$ is entirely contained in $T^*U_{\ell_k}$. We then define flow maps in coordinates:
							\[\Phi^s_{\ell_k,\ell_k'}=\widetilde f_{\ell_{k'}}\circ \Phi^s_\delta\circ\widetilde f_{\ell_k}^{-1},\]
							whenever the composition makes sense. Also, define the Hamiltonian in coordinates $p_{\ell_k}:=p\circ \widetilde f_{\ell_k}^{-1}$. 
							Then, the identity \eqref{eq: duhamel perturbed flow bis} becomes in charts 
							\[\frac{\partial \Phi^t_{\ell_0,\ell_n}}{\partial \delta}(\rho)=\sum_{k=0}^n\int_{k}^{k+1}\mathbf 1_{s\le t} \mathrm{d}\Phi_{\ell_k,\ell_n}^{t-s} \cdot H_{\partial_\delta p_{\ell_k}}(\Phi^s_{\ell_0,\ell_k}(\rho))\mathrm{d}s.\]
							This expression only involves the flow maps in charts along the trajectories and the symbol in charts. The above proof can thus be adapted to this new setting.
						
						\end{proof}
						
						We prove the missing estimate in the proof of \textbf{2.} above.
						\begin{lem}\label{lem: formula time derivative Phi}For all $k\ge 1$,
							\[\partial_t^k \Phi_\delta^t(\rho)=\sum_{(\mu_i)} \prod_i (\partial^{\mu_i}_xp_\delta)(\Phi_\delta^t(\rho)),\]
							where the sum runs over tuples of integers $\mu_i\ge 1$ satisfying $\sum_i \mu_i=2k-1$ and $\sum_i (\mu_i-1)=k-1$.
						\end{lem}
						\begin{proof}Perform an induction on $k\ge 1$ using $\partial_t \Phi^t_\delta=H_{p_\delta}\Phi^t_\delta$.
						\end{proof}
						\begin{lem}\label{lem: intermediate bound delta deriv p} For all $\mu\ge 1$, for all $n,\ell\ge 0$, we have
							\[|\partial_\delta^n\partial_\rho^\ell \big[(\partial_\rho^\mu p_\delta)(\Phi)\big]|\lesssim (1+h^{-\beta(\mu-2)}) T^n h^{-\beta(2n+\ell)}\mathrm{e}^{C(n+\ell)t}.\]
						\end{lem}
						\begin{proof}By Fa\`a di Bruno's formula, one has
							\[\partial_\delta^n\partial_\rho^\ell \big[(\partial_\rho^\mu p_\delta)(\Phi^t_\delta)\big]
							=\sum_{n_1+n_2=n}\sum_{\ell_1+\ell_2=\ell}\sum_{r=0}^{n_2}\sum_{\substack{m_1+\ldots+m_r=n_2, \\ m_i\ge 1}} \partial_\rho^{\ell_1}\big[(\partial_\rho^{\mu+r}\partial_\delta^{n_1}p_\delta)(\Phi^t_\delta)\big]\cdot \partial_\rho^{\ell_2}\prod_{m_1+\ldots+m_r=n_2}\partial_\delta^{m_i}\Phi^t_\delta. \]
							By the assumptions on $p_\delta$ and the estimates on the flow, by Fa\`a di Bruno's formula, we have an upper bound (this relies on the fact that $\mu\ge 1$):
							\begin{equation}\label{eq: deriv first term proof time der} \begin{aligned}\big|\partial_\rho^{\ell_1}\big[(\partial_\rho^{\mu+r}\partial_\delta^{n_1}p_\delta)(\Phi_\delta^t)\big]\big|& \lesssim T^{n_1}\mathrm{e}^{C\ell_1 t} \left\{\begin{array}{ll}(1+h^{-\beta(\mu+r+\ell_1-2)}) & \text{if $n_1=0$} \\ 
										h^{-\beta(\mu+r+\ell_1)}& \text{if $n_1\ge 1$}\end{array}\right. \\
									& \lesssim T^{n_1} \mathrm{e}^{C\ell_1t} (1+h^{-\beta(\mu+r+\ell_1+2n_1-2)}).\end{aligned}\end{equation}
							Meanwhile, by the first part of the proof of Proposition \ref{prop: bound delta derivatives of the flow}:
							\begin{equation} \label{eq: deriv second term proof time der}\Big|\partial_\rho^{\ell_2}\prod_{m_1+\ldots+m_r=n_2}\partial_\delta^{m_i}\Phi_\delta^t\Big|\lesssim  h^{-\beta(2n_2+\ell_2-r)}T^{n_2} \mathrm{e}^{C(n_2+\ell_2)t}.\end{equation}
							The result follows by combining \eqref{eq: deriv first term proof time der} with \eqref{eq: deriv second term proof time der}.
						\end{proof}

					\subsection{Controlling $\delta$-derivatives of generating functions} In the rest of the section, we allow the Hamiltonian to depend on two auxiliary parameters $\delta,\omega$. Although the current work is only concerned with a family of $\delta$-dependent Hamiltonians, the next result will prove useful to study the model of random Schr\"odinger operators of \cite{ingremeau2023emergence}. \medskip
					
					\noindent \textbf{Notation.} If $\mathbf s=(s_1,\ldots,s_r)$ is an array of real number, the notation $\Phi_\delta^{\mathbf s}$ represents the flow $\Phi_\delta^{s_i}$ for some $i\in \{1,\ldots,r\}$ that may change from line to line and even in the same formula. For example $\partial_\rho^3\Phi_{\delta}^{\mathbf s}\partial_\rho \Phi_\delta^{\mathbf s}$ might represent $\partial_\rho^3 \Phi_{\delta}^{s_1}\partial_\rho \Phi_\delta^{s_4}$. We denote by $\mathrm{d}\mathbf s$ the Lebesgue measure on $\mathbf{R}^r$.
					\begin{prop}[Higher $\delta$-derivatives of the generating function]\label{prop: higher der gn func} Let $p=p_{\delta,\omega}$ be a smooth Hamiltonian on $T^*\mathbf{R}^d$ that depends on two parameters $\delta,\omega$, and let $V=\partial_\omega p$. We assume that the flow $\Phi^t$ admits a local generating function $\varphi(t,x,\eta)$. Let $\mathbf x=(x,\eta)$ and introduce $G(t,\mathbf x)=(\nabla_\eta \varphi(t,x,\eta),\eta)$. Then, for all $n\ge 1$, one has
						\begin{equation}\label{eq: formula high deriv gen fun}\forall \ell\ge 0, \quad \partial_{\mathbf x}^\ell\partial_\delta^{n-1}\partial_{\omega} \varphi=\sum_{r=1}^n\int_{[0,t]^r} \sum_{\mathbf p} \partial_\delta^N\partial_\rho^{L}V\cdot \prod_i \partial_\delta^{n_i}\partial_\rho^{\ell_i}\Phi_\delta^{\mathbf s}\cdot \prod_j \partial_\delta^{m_j}\partial_\rho^{q_j}p\cdot \prod_k \partial_{\mathbf x}^{u_k}G ~\mathrm{d}\mathbf s.\end{equation}
						Here we evaluate $\partial_{\mathbf x}^{u_k}G$ at $(t,\mathbf x)$, $\partial_\delta^{n_i}\partial_\rho^{\ell_i}\Phi_\delta^{\mathbf s}$ at $G(t,\mathbf x)$, and $\partial_\delta^N\partial_\rho^{L}V$ and $ \partial_\delta^{m_j}\partial_\rho^{q_j}p$ at $\Phi_\delta^{\mathbf s}(G(t,\mathbf x))$.
						
						It is understood that each of the sums in \eqref{eq: formula high deriv gen fun} runs over families 
						\[\mathbf p=\big\{N,L,\{n_i,\ell_i\},\{m_j,q_j\},\{u_k\}\big\}\] with $N,L\ge 0$, $\ell_i+n_i\ge 1$, $q_j\ge 0$, $m_j\ge 1$ and $u_k\ge 1$, such that the following holds:
						\begin{itemize}
							\item $\mathbf C:=L+2N+\sum(\ell_i+2n_i-1)+\sum q_j+(m_j-1)+\sum (u_k-1)\le \ell+2n-2$.
							\item $\mathbf D:=\sum (\ell_i+n_i)+\sum (u_k-1)\le \ell+2n-2$. 
						\end{itemize}
						
					\end{prop}

					\begin{proof}\textbf{1.} We first show the result for $n=1$. Let us differentiate \eqref{eq: delta derivative generating function} $\ell$ times in the $\mathbf x:=(x,\eta)$ variables:
						\[\partial_{\mathbf x}^\ell \partial_\omega \varphi=-\int_0^t \partial_{\mathbf x}^{\ell}\big[V(\Phi_\delta^s(G))\big]\mathrm{d}s=-\int_0^t\sum_{r=0}^{\ell}\sum_{\substack{\ell_1+\ldots+\ell_r=\ell \\ \ell_i\ge 1}} (\partial_{\rho}^rV)\cdot \prod \partial_{\mathbf x}^{\ell_i} (\Phi_\delta^s(G))\mathrm{d}s,\]
						where $G$ is evaluated at $(t,\mathbf x)$. For each $\ell_i$, by Fa\`a di Bruno's formula again
						\[ \partial_{\mathbf x}^{\ell_i} (\Phi_\delta^s(G))=\sum_{k_i=0}^{\ell_i}\sum_{\substack{u_i^{(1)}+\ldots+u_i^{(k_i)}=\ell_i \\ u_i^{(j)}\ge1 }}\partial_\rho^{k_i}\Phi_\delta^s(G)\cdot \prod_{j=1}^{k_i} \partial_{\mathbf x}^{u_i^{(j)}} G,\]
						where $k_i+\sum (u_i^{(j)}-1)=\ell_i$. In total, we check that for each term of the sum
						\[\mathbf C:=r+\sum_i \Big[(k_i-1)+\sum_j (u_i^{(j)}-1)\Big]=r+\sum_i (\ell_i-1)=\ell.\]
						\[\mathbf D:=\sum_i \Big[k_i+\sum_j (u_i^{(j)}-1)\Big]=\sum_i \ell_i=\ell.\]

						\textbf{2.} Assume by induction that \eqref{eq: formula high deriv gen fun} holds up to rank $n$. Then, for $r\in \{1,\ldots,n\}$, compute
						\[\partial_\delta\Big[\int_{[0,t]^r} \sum_{\mathbf p} \partial_\delta^N\partial_\rho^{L}V\cdot \prod \partial_\delta^{n_i}\partial_\rho^{\ell_i}\Phi_\delta^{\mathbf s}\cdot \prod \partial_\delta^{m_j}\partial_\rho^{q_j}p\cdot \prod \partial_{\mathbf x}^{u_k}G~\mathrm{d}\mathbf s\Big].\]
						We look at the action of $\partial_\delta$ on the product, and see how we must update $r$ and the tuple $\mathbf p$, and accordingly how $\mathbf C$ and $\mathbf D$ are modified.
						\begin{itemize}
							\item If $\partial_\delta$ hits $(\partial_\delta^N\partial_\rho^L V)(\Phi_\delta^{\mathbf s}(G))$ then we compute
							\[\partial_\delta(\partial_\delta^N\partial_\rho^L V)(\Phi_\delta^{\mathbf s}(G))=(\partial_\delta^{N+1}\partial_\rho^L V)(\Phi_\delta^{\mathbf s}(G))+(\partial_\delta^N\partial_\rho^{L+1}V)\cdot \big(\partial_\delta \Phi_\delta^{\mathbf s} +\partial_\rho \Phi_\delta^{\mathbf s}\cdot \partial_\delta G).\]
							Since 
							\[\partial_\delta G=\partial_{\mathbf x}\int_0^t -\partial_\delta p(\Phi^s_\delta(G))\mathrm{d}s=-\int_0^t(\partial_\rho\partial_\delta p)\cdot \partial_\rho \Phi^s_\delta\cdot \partial_{\mathbf x}G\mathrm{d}s,\]
							for each of the three terms of the sum we find that
							\begin{itemize}
								\item either $N$ increases by $1$ and $\mathbf C\leftarrow\mathbf C+2$.
								\item either $L$ increases by $1$ and we add $(n_\star,\ell_\star)=(1,0)$, so $\mathbf C\leftarrow\mathbf C+1$ and $\mathbf D\leftarrow\mathbf D+1$.
								\item $L$ increases by $1$, we add $(n_\star,\ell_\star)=(0,1)$ and $(n_\star',\ell_\star')=(0,1)$, $(m_\star,q_\star)=(1,1)$ and $u_\star=1$, so $\mathbf C\leftarrow 2$ and $\mathbf D\leftarrow \mathbf D+2$. Also $r\leftarrow r+1$.
							\end{itemize}
							\item If $\partial_\delta$ hits $\partial_\delta^{n_i}\partial_\rho^{\ell_i}\Phi_\delta^{\textbf{s}}$ then we find 
							\[(\partial_\delta^{n_i+1}\partial_\rho^{\ell_i}\Phi_\delta^{\textbf{s}})+\partial_\delta^{n_i}\partial_\rho^{\ell_i+1}\Phi^{\textbf{s}}_\delta\cdot \int_0^t(\partial_\rho\partial_\delta p)\cdot \partial_\rho \Phi^s_\delta\cdot \partial_{\mathbf x}G\mathrm{d}s,\]
							where the last integral is just $\partial_\delta G$. We separate the two terms
							\begin{itemize}
								\item In the first case, we replace $(n_i,\ell_i)$ by $(n_{i}+1,\ell_i)$ so $\mathbf C\leftarrow\mathbf C+2$ and $\mathbf D\leftarrow\mathbf D+1$
								\item In the second case, we replace $(n_i,\ell_i)$ by $(n_{i},\ell_i+1)$ and add $(m_\star,q_\star)=(1,1)$ and $(n_\star,\ell_\star)=(0,1)$ and $u_\star=1$, so $\mathbf C\leftarrow\mathbf C+2$ and $\mathbf D\leftarrow\mathbf D+2$. Also we replace $r\leftarrow r+1$.
							\end{itemize}

							\item If $\partial_\delta$ hits  $(\partial_\delta^{m_j}\partial_\rho^{q_j}p)(\Phi_\delta^{\mathbf s}(G))$ then we obtain
							\[(\partial_\delta^{m_j+1}\partial_\rho^{q_j}p)+(\partial_\delta^{m_j}\partial_\rho^{q_j+1}p)\cdot(\partial_\delta \Phi_\delta^{\mathbf s}+\partial_\rho\Phi_\delta^{\mathbf s}\cdot \partial_\delta G).\]
							\begin{itemize}
								\item the first term replaces $(m_j,q_j)$ by $(m_{j}+1,q_j)$ so $\mathbf C\leftarrow\mathbf C+1$.
								\item the second term replaces $(m_j,q_j)$ by $(m_j,q_{j}+1)$ and adds $(n_\star,\ell_\star)=(1,0)$ so $\mathbf C\leftarrow\mathbf C+2$ and $\mathbf D\leftarrow\mathbf D+1$.
								\item the last term replaces $(m_j,q_j)$ by $(m_j,q_{j}+1)$, adds $(n_\star,\ell_\star)=(0,1)$ and adds $(m_\star,q_\star)=(1,1)$ and $(n_\star',\ell_\star')=(0,1)$ and $u_\star=1$, so $\mathbf C\leftarrow\mathbf C+2$ and $\mathbf D\leftarrow\mathbf D+2$.
								
							\end{itemize}
							\item If $\partial_\delta$ hits $\partial_{\mathbf x}^{u_\star} G$ then we use the result for $n=1$:
							\[\partial_{\delta}\partial_{\mathbf x}^{u_\star+1} \varphi=\int_0^t \partial_\delta^{N'}\partial_\rho^{L'}\partial_\delta p\cdot \prod \partial_\delta^{n_i'}\partial_\rho^{\ell_i'}\Phi_\delta^s\cdot \prod \partial_\delta^{m_j'}\partial_\rho^{q_j'}p\cdot \prod \partial_{\mathbf x}^{u_k'}G\mathrm{d}s,\]
							and here we have
							\[\mathbf C':=L'+N'+\sum(\ell_i'+2n_i'-1)+\sum q_j'+(m_j'-1)+\sum (u_k'-1)\le u_\star+1+2-2=u_\star+1\]
							and
							\[\mathbf D':=\sum (\ell_i'+n_i')+\sum (u_k'-1)\le u_\star+1.\]
							So we update $\mathbf p$ by removing $u_\star$, we add $(m_\star,q_\star)=(N'+1,L')$ and all the $\{\ell_i',n_i'\},\{q_j',m_j'\}$,$\{u_k'\}$ to the family $\mathbf p$. The new value of $\mathbf C$ is
							\[\mathbf C'\le \mathbf C-(u_\star-1)+(u_\star+1)\le \mathbf C+2.\]
							The new value of $\mathbf D$ is
							\[\mathbf D'-(u_\star-1)+(u_\star+1)=\mathbf D+2.\]
							Lastly, we update $r\leftarrow r+1$.
						\end{itemize}
						This finishes the proof.

					\end{proof}
					
					\begin{rem} Consider the case where $\mathbf{R}^d$ is replaced by a closed manifold $M$ and $p_\delta$ is a Hamiltonian on $T^*M$. We reuse the notation of the step \textbf{3.} in the proof of Proposition \ref{prop: bound delta derivatives of the flow}. For $t\in (n-1,n+1)$, let $\Phi^t_{\ell_0,\ell_n}:=\widetilde f_{\ell_n}\circ \Phi_\delta^t\circ \widetilde f_{\ell_0}^{-1}$ be the flow map in coordinates, and let $\varphi(t,x,\eta)$ be a local generating function for $\Phi^t_{\ell_0,\ell_n}$. Let $G(t,\mathbf x)=(\nabla_\eta \varphi(t,x,\eta),\eta)$. The starting point is the formula
						\[\partial_\omega \varphi(t,\mathbf x)=-\sum_k\int_{k}^{k+1}\mathbf 1_{s\le t} \frac{\partial p_{\ell_k}}{\partial \omega}(\Phi_{\ell_0,\ell_k}^s(G(t,\mathbf x)))\mathrm{d}s,\]
						which only involves objects defined on $\mathbf{R}^d$.
						Then, one can run again the above proof and obtain a similar expression for $\partial_{\mathbf x}^\ell\partial_\delta^{n-1}\partial_{\omega} \varphi$. We denote by $\mathcal J_r(t)$ the collection of sets of the form  \[[k_1,\max(k_1+1,t)]\times \cdots [k_r,\max(k_r+1,t)],\]
						with $(k_1,\ldots,k_r)\in \mathbf N^r$. We obtain an expression
						\[\partial_{\mathbf x}^\ell\partial_\delta^{n-1}\partial_\omega \varphi=\sum_{r=1}^n\sum_{J\in \mathcal J_r(t)}\int_{J} \sum_{\mathbf p} \partial_\delta^N\partial_\rho^{L}\partial_\omega p_{\bullet}\cdot \prod_i \partial_\delta^{n_i}\partial_\rho^{\ell_i}\Phi_{\ell_0,\bullet}^{\mathbf s}\cdot \prod_j \partial_\delta^{m_j}\partial_\rho^{q_j}p_{\bullet}\cdot \prod_k \partial_{\mathbf x}^{u_k}G ~\mathrm{d}\mathbf s.\]
						Here for each factor of the product, the index $\bullet$ represents an index of the chart $\Xi=\{(U_\ell,f_\ell)\}$, that may be different for different factors. The tuples $\mathbf p$ satisfy the same assumptions as in Proposition \ref{prop: higher der gn func}.

					\end{rem}
					
					In the rest of this section, in view of future application in \cite{Moyspecstat}, we make the following assumption on the Hamiltonian.
					
					\begin{assum}Assume that $p=p_{\delta,\omega}$ now depends on two parameters $(\delta,\omega)$, and that $p_{0,\omega}=:p_0$ is independent of $\omega$. Moreover, assume that in a compact neighborhood of the energy shell $p_0^{-1}(E)$, we have for all $\ell\ge 0$:
						\[\forall n\ge 1,\forall m\in \{0,1\}, \quad |\partial_\rho^\ell \partial_\delta^n \partial_\omega^m p_{\delta,\omega}|\lesssim h^{-\beta \ell}, \quad \text{and} \quad \quad |\partial_\rho^\ell\partial_\omega p_{\delta,\omega}|\lesssim h^{-\beta (\ell+2)}.\]
						Moreover, assume that $\partial_\omega p_{\delta,\omega}$ is supported in a ball of radius $\le h^\beta$.
					\end{assum}
					For example, we may have $p_{\delta,\omega}=p_\delta+ \delta \omega V$, with $\delta\in [0,h^\alpha]$ and $2\beta<\alpha$, where $V$ is supported in a ball of radius $h^\beta$ and satisfies $\|V\|_{C^\ell}\lesssim h^{-\beta \ell}$. Note that under these assumptions, for each fixed $\omega$, the family of symbols $\{p_{\delta,\omega}\}$ satisfies the assumptions \ref{assum: classical flow}--\ref{assum: derivatives symbol} thus all the previous results apply to $p_{\delta,\omega}$. We now have a family of flows $\Phi^t=\Phi^t_{\delta,\omega}$ and families of generating functions $\varphi=\varphi_{\delta,\omega}$. By applying Proposition \ref{prop: higher der gn func}, we obtain the following corollary.
					
					\begin{cor}\label{cor: higher der gn func} Let $\varphi(t,x,\eta)$ be the long time generating function for the flow in coordinates defined in \S \ref{subsec: iteration}. For all $n\ge 1$ and $\ell\ge 0$, one has:
						\[|\partial_\delta^n\partial_{\mathbf x}^\ell \varphi|\lesssim h^{-\beta(\ell+2n-2)}(1+t) \mathrm{e}^{C(\ell+2n-2)t}.\] 
						Moreover, for all $n\ge 1$ and $\ell\ge 0$, we have
						\[|\partial_\delta^{n-1}\partial_{\omega}\partial_{\mathbf x}^\ell \varphi|\lesssim h^{-\beta(\ell+2n-5)}(1+t) \mathrm{e}^{C(\ell+2n-2)t}.\]
					\end{cor}
					Note that the proof of the second item crucially relies on the fact that an orbit of the flow $\{\Phi^t_\delta\}_{t\in [0,T]}$ spends at most a time $\lesssim T h^{\beta}$ in the support of $V=\partial_\omega p$.

					\begin{proof} The estimates follow from the formula \eqref{eq: formula high deriv gen fun}, along with the upper bound \eqref{eq: delta derivatives spacetime flow} on derivatives of the flow and the bound from Lemma \ref{lem: space derivatives generating function} on the space derivatives of the generating function.
					\end{proof}

					We now incorporate time derivatives to the picture.
					
					\begin{prop}[Controlling all derivatives of the generating function]\label{prop: control all deriv gen fun} One has
						\begin{equation} \label{eq: controlling all der no omega}\forall n\ge 0, \quad \forall \ell,k\ge 0, \quad |\partial_\delta^n\partial_{\mathbf x}^\ell\partial_t^k \varphi|\lesssim h^{-\beta(\ell+k+2n-2)_+}\mathrm{e}^{C(\ell+2n)t}. \end{equation}
						and
						\begin{equation} \label{eq: controlling all der with omega}\forall n\ge 1, \quad \forall \ell,k\ge 0,|\partial_\delta^{n-1}\partial_{\omega}\partial_x^{\ell}\partial_t^k \varphi|\lesssim h^{-\beta(\ell+k+2n-5)}\mathrm{e}^{C(\ell+2n)t}. \end{equation}
					\end{prop}
					\begin{proof} Let us now denote $G(t,\mathbf x)=(x,\nabla_x \varphi)$.
						
						\textbf{1.} We begin with \eqref{eq: controlling all der no omega}. For $n=0$, \eqref{eq: controlling all der no omega} is already known by Lemma \ref{lem: spacetime derivatives generating function}, thus we can assume $n\ge 1$. We perform an induction on $k$. For $k=0$ this is corollary \ref{cor: higher der gn func}. For $k=1$ there is nothing to do because $\partial_t \varphi=-p(G)$ and we can invoke Lemma \ref{lem: intermediate lemma partial delta p} below. Assume that the result holds up to rank $k\ge 1$, then use
						\[\partial_{\mathbf x}^\ell\partial_\delta^n \partial_t^{k+1} \varphi=-\partial_{\mathbf x}^\ell\partial_\delta^n \partial_t^{k} p(G).\]
						By Fa\`a di Bruno's formula:
						\[\partial_t^k p(G)=\sum_{j=1}^k \partial_\rho^jp(G)\sum_{k_1+\ldots+k_j=k}\prod_{i=1}^j \partial_t^{k_i}G.\]
						Then by the product rule and the chain rule
						\[\partial_{\mathbf x}^\ell\partial_\delta^n \partial_t^{k+1} \varphi=-\sum_{j=1}^k \sum_{k_1+\ldots+k_j=k}\sum_{r+\sum \ell_i=\ell}\sum_{m+\sum n_i=n} \partial_{\mathbf x}^{r}\partial_\delta^m(\partial_\rho^jp (G))\cdot \prod_{i=1}^j \partial_{\mathbf x}^{\ell_i}\partial_\delta^{n_i}\partial_t^{k_i}G.\]
						By Lemma \ref{lem: intermediate lemma partial delta p} below, since $j\ge 1$ we have
						\begin{equation} \label{eq: truc1}|\partial_{\mathbf x}^{r}\partial_\delta^m(\partial_\rho^jp (G))|\lesssim h^{-\beta(r+2m+j-1)}\mathrm{e}^{C(r+2m)t},\end{equation}
						and by the induction hypothesis:
						\begin{equation} \label{eq: truc2} \begin{aligned}\Big|\prod_i \partial_{\mathbf x}^{\ell_i}\partial_\delta^{n_i}\partial_t^{k_i}G\Big|& \lesssim \prod_i h^{-\beta(\ell_i+k_i+2n_i-1)} \mathrm{e}^{C(\ell_i+2n_i)t}\\ &\lesssim h^{-\beta((k-j)+(\ell-r)+2(n-m))}\mathrm{e}^{C(\ell-r+2n-2m)t}. \end{aligned}\end{equation}
						The product of the two terms \eqref{eq: truc1}--\eqref{eq: truc2} is $\lesssim h^{\beta(\ell+2n+k-1)}=h^{-\beta(\ell+2n+(k+1)-2)}$, as we wished.
						
						\textbf{2.} The proof of \eqref{eq: controlling all der with omega} also goes by induction on $k$ and follows the same lines, but also uses \eqref{eq: one has for omega =omegai p}.
					\end{proof}

					It remains to prove the next lemma.
					\begin{lem}\label{lem: intermediate lemma partial delta p} Let $G(t,\mathbf x)=(x,\nabla_x \varphi(t,x,\eta))$. Let $n\ge 1, \ell\ge 0$. Then,
						\begin{equation} \label{eq: one has for omega =delta p} \forall \mu\ge 0, \quad \big|\partial_{\mathbf x}^\ell\partial_\delta^{n}(\partial_\rho^{\mu}p_{\delta,\omega}(G))\big|\lesssim h^{-\beta(\ell+\mu+2n-1)}\mathrm{e}^{C(\ell+2n)t},\end{equation}
						and
						\begin{equation} \label{eq: one has for omega =omegai p} \forall \mu\ge 1, \quad  \big|\partial_{\mathbf x}^\ell \partial_\delta^{n-1}\partial_{\omega_i} (\partial_\rho^{\mu}p_\delta(G))\big|\lesssim h^{-\beta(\ell+\mu+2n-4)}\mathrm{e}^{C(\ell+2n)t},\end{equation}
						Also, if $\ell\ge 0$ and $\mu\ge 1$ then
						\begin{equation}\label{eq: when no delta der on p} \quad \big|\partial_{\mathbf x}^\ell(\partial_\rho^{\mu}p_{\delta,\omega}(G))\big|\lesssim h^{-\beta(\ell+\mu-1)}\mathrm{e}^{C\ell t}.\end{equation}
					\end{lem}
					\begin{proof}
						
						\textbf{1.} We first prove \eqref{eq: when no delta der on p}. The result follows from Fa\`a di Bruno's formula:
						\[\partial_{\mathbf x}^{\ell}(\partial_\rho^\mu p(G))=\sum_{r=0}^{\ell}\sum_{\ell_1+\ldots+\ell_r=\ell}|\partial^{r+\mu}_\rho p|\cdot \prod_{i=1}^r \big|\partial_{\mathbf x}^{\ell_i}G\big|.\]
						This is
						\begin{itemize}
							\item $\le h^{-\beta(\mu-2)_+}$ if $\ell=0$, hence $\le h^{-\beta(\ell+\mu-1)}$ since $\mu\ge 1$.
							\item $\le h^{-\beta(\ell-1)}h^{-\beta(\mu-1)_+}+h^{-\beta(\ell+\mu-2)}\lesssim h^{-\beta(\ell+\mu-2)}$ if $\ell\ge 1$.
						\end{itemize}

						\textbf{2.} We now assume $n\ge 1$. Let $\zeta\in \{\delta,\omega\}$. One has
						\[\partial_\delta^{n-1}\partial_\zeta \big(\partial_\rho^{\mu}p_\delta)(G)=\sum_{r,(m,m^\star)}\sum_{\{n_i,n_i^\star\}_{i\in I}} (\partial_\rho^{\mu+r}\partial_\delta^m\partial_\zeta^{m^\star}p_\delta)\cdot \prod_i \partial_\delta^{n_i}\partial_\zeta^{n_i^\star}G,\]
						where:
						
						\begin{itemize}
							\item $r+m+m^\star+\sum_{i\in I} (n_i+n_i^\star-1)=n$.
							\item $m+\sum n_i=n-1$.
							\item $m^\star+\sum n_i^\star=1$.
							\item $r\ge 0$ and $n_i+n_i^\star\ge 1$ for all $i\in I$.
						\end{itemize}
						The case $r=0$ corresponds to the case where $I$ is empty. As so,
						\begin{equation} \partial_{\mathbf x}^\ell \partial_\delta^{n-1}\partial_\zeta \big(\partial_{\rho}^{\mu}p_\delta)(G)=\sum_{\ell_1+\ell_2=\ell} \sum_{r,(m,m^\star)}\sum_{\{n_i,n_i^\star\}_{i\in I}} \partial_{\mathbf x}^{\ell_1}\Big[(\partial_\rho^{\mu+r}\partial_\delta^m\partial_\zeta^{m_\star}p_\delta(G)\Big]\cdot \partial_{\mathbf x}^{\ell_2} \Big[\prod_{i\in I} \partial_\delta^{n_i}\partial_\zeta^{n_i^\star}G\Big].\end{equation}
						Let us develop the first term, using Fa\`a di Bruno's formula:
						\begin{equation} \partial_{\mathbf x}^{\ell_1}\Big[\partial_\rho^{\mu+r}\partial_\delta^m\partial_\zeta^{m_\star}p_\delta(G)\Big]=\sum_{k=0}^{\ell_1}\sum_{L_1+\ldots+L_k=\ell_1} \partial_\rho^{k+\mu+r}\partial_\delta^m\partial_\zeta^{m^\star}p_\delta\cdot \Big[ \prod_{j=1}^{\ell_1} \partial_{\mathbf x}^{L_j}G\Big]. \end{equation}
						By the previous Corollary \ref{cor: higher der gn func} on space derivatives, we have
						\begin{equation} \label{eq: bound prod d G} \Big|\prod_{j=1}^{\ell_1}\partial_{\mathbf x}^{L_j}G\Big|\lesssim h^{-\beta(\ell_1-k)}\mathrm{e}^{C(\ell_1-k)t}.\end{equation}
						Now we must distinguish whether $\zeta=\omega$ or $\zeta=\delta$.
						
						\textbf{Case 1: $\zeta=\omega$.} First, by the assumption on $p_{\delta,\omega}$:
						\begin{equation} \label{eq: bound partial der p} |\partial_x^{\mu+k+r}\partial_\delta^m\partial_{\omega}^{m_\star}p_\delta| \le \left\{
							\begin{array}{ll} 1+h^{-\beta(\mu+r+k-2)} & \text{if $m+m^\star=0$,} \\
								h^{-\beta(\mu+r+k-2)}	  &	\text{if $m=0$ and $m^\star=1$,} \\
								h^{-\beta(\mu+r+k)}     & \text{if $m\ge 1$.} \end{array}\right.\end{equation}
						By Corollary \ref{cor: higher der gn func} again, one has
						\begin{equation} \label{eq: bound prod G n i star} \begin{aligned}
								\Big|\partial_{\mathbf x}^{\ell_2}\Big[\prod_i \partial_\delta^{n_i}\partial_\omega^{n_i^\star}G\Big]\Big|
								& \le \Big|\sum_{\sum_i \ell_2^{(i)}=\ell_2}\prod_i \partial_{\mathbf x}^{\ell_2^{(i)}}\partial_\delta^{n_i}\partial_{\omega}^{n_i^\star} G\Big|\\
								&\lesssim \sum_{\sum_i \ell_2^{(i)}=\ell_2}\Big(\prod_{i~:~n_i^\star=0} h^{-\beta(\ell_2^{(i)}+2n_i-1)}\Big)\Big(\prod_{i~:~n_i^\star=1}  h^{-\beta(\ell_2^{(i)}+2n_i+2n_i^\star-4)}\Big) \\
								& \lesssim h^{-\beta (\ell_2+\sum n_i+\sum(n_i-1)-|n^\star|)}\mathrm{e}^{C(\ell_2+\sum n_i+\sum (n_i-1)+ 2|n^\star|)t},
						\end{aligned}\end{equation}
						where $|n^\star|=\sum n_i^\star\in \{0,1\}$. From the conditions on $m,m^\star,r,\{n_i,n_i^\star\}_{i\in I}$, we have
						\[\sum n_i+\sum (n_i-1)=2n-2m-r-2.\] By combining \eqref{eq: bound prod d G},\eqref{eq: bound partial der p},\eqref{eq: bound prod G n i star}, we find that for any $k\in \{0,\ldots,\ell_1\}$, the product
						\[|\partial_x^{\mu+k+r}\partial_\delta^m\partial_{\omega}^{m_\star}p_\delta|\cdot \Big|\prod_{j=1}^{\ell_1}\partial_{\mathbf x}^{L_j}G\Big|\cdot\Big|\partial_{\mathbf x}^{\ell_2}\prod_{i\in I}   \partial_\delta^{n_i}\partial_\omega^{n_i^\star}G\Big|\]
						is bounded by\[\lesssim\mathrm{e}^{C(\ell+2n)t} h^{-\beta (\ell-k)} h^{-\beta(2n-r-2m-2-|n^\star|)}\cdot \left\{
						\begin{array}{ll} 1+h^{-\beta(\mu+r+k-2)} & \text{if $m+m^\star=0$,} \\
							h^{-\beta(\mu+r+k-2)}	  &	\text{if $m=0$ and $m^\star=1$,} \\
							h^{-\beta(\mu+r+k)}     & \text{if $m\ge 1$.} \end{array}\right.\]
						We control the three contributions separately.
						\begin{itemize}
							\item If $m+m^\star=0$ then $|n^\star|=1$ and $r\ge 1$, leading to
							\[h^{-\beta (\ell-k)} h^{-\beta(2n-r-2m-2-|n^\star|)}\lesssim h^{-\beta(\ell-k+2n-r-3)}\lesssim h^{-\beta(\ell+2n-4)},\]
							and also
							\[h^{-\beta (\ell-k)} h^{-\beta(2n-r-2m-2-|n^\star|)}h^{-\beta(\mu+r+k-2)}\lesssim h^{-\beta(\ell+2n+\mu-5)}.\]
							\item For $(m,m^\star)=(0,1)$, one has $|n^\star|=0$ hence
							\[h^{-\beta (\ell-k)} h^{-\beta(2n-r-2m-2-|n^\star|)}h^{-\beta(\mu+r+k-2)} \lesssim h^{-\beta(\ell+2n+\mu-4)}.\]
							\item For $m\ge 1$, one has
							\[h^{-\beta (\ell-k)} h^{-\beta(2n-r-2m-2-|n^\star|)} h^{-\beta(\mu+r+k)}\lesssim h^{-\beta(\ell+2n+\mu-4)}.\]
						\end{itemize}
						In each case, the contribution is $\lesssim h^{-\beta(\ell+2n+\mu-4)}$. This proves \eqref{eq: one has for omega =omegai p}.
						
						\textbf{Case 2: $\zeta=\delta$.} 
						In this case we just replace $m\leftarrow m+m^\star$ and replace $n_i\leftarrow n_i+n_i^\star$ and work with
						\begin{itemize}
							\item $r+m+\sum (n_i-1)=n$.
							\item $m+\sum n_i=n$.
							\item $r\ge 0$ and $n_i\ge 1$. In particular $r=0$ if and only if $I$ is empty.
						\end{itemize}
						Then, one has
						\[|\partial_\rho^{\mu+k+r}\partial_\delta^{m}p_\delta| \le \left\{
						\begin{array}{ll} 1+h^{-\beta(\mu+r+k-2)} & \text{if $m=0$,} \\
							h^{-\beta(\mu+r+k)}	  &	\text{if $m\ge 1$.} \end{array}\right.\]
						Also,
						\[\Big|\partial_{\mathbf x}^{\ell_2}\Big[\prod_i \partial_\delta^{n_i}G\Big]\Big|\lesssim h^{-\beta (\ell_2+2n-2m-r)}\mathrm{e}^{C(\ell_2+2n-2m-r)t}.\]
						Eventually, each term of the form
						\[|\partial_x^{\mu+k+r}\partial_\delta^mp_\delta|\cdot \Big|\prod_{j=1}^{\ell_1}\partial_{\mathbf x}^{L_j}G\Big|\cdot\Big|\partial_{\mathbf x}^{\ell_2}\prod_{i\in I}   \partial_\delta^{n_i}G\Big|\]
						is bounded by
						\[\lesssim\mathrm{e}^{C(\ell+2n)t} h^{-\beta (\ell-k)} h^{-\beta(2n-r-2m)}\cdot \left\{
						\begin{array}{ll} 1+h^{-\beta(\mu+r+k-2)} & \text{if $m=0$,} \\
							h^{-\beta(\mu+r+k)}	  &	\text{if $m\ge 1$.} \end{array}\right.\]
						As before, by noticing that $m=0$ forces $r\ge 1$, we deduce that all the terms are $\lesssim h^{-\beta(\ell+2n+\mu-1)}$. This finishes the proof of \eqref{eq: one has for omega =delta p}.
					\end{proof}

	\section{Fourier Integral Operators} \label{sec: FIO}

We recall some elements of the local theory of (semiclassical) Fourier integral operators. We only consider phase functions of the form $\Phi(x,y,\eta)=\varphi(x,\eta)-\langle y,\eta\rangle$, which are associated with local symplectomorphisms $\kappa:(y,\eta)\mapsto (x,\xi)$ whose graph project diffeomorphically on the $(x,\eta)$ coordinates. A possible reference is \cite[\S 4]{NZ09}. The setting below is slightly exotic: we consider $h$-dependent phase functions that may blow up as $h\to 0$. Such a situation was already encountered by Blair \cite{Blair2022} who obtained a long-time van Vleck formula for the kernel of quantum propagators on the Euclidean space. In contrast to our setting, the Hamiltonians considered by Blair do not blow up as $h\to 0$, however the tools developed in \cite{Blair2022} allow dealing with this new situation. We refer to the appendix for the notions of pseudodifferential calculus used in this section.

\begin{defi}Let $\varphi(x,\eta)$ be a local generating function of a symplectomorphism $\kappa$ defined on a subset $U_x\times U_\eta\subset \mathbf{R}^{2d}$ (that is, one has $\kappa(\nabla_\eta \varphi,\eta)=(x,\nabla_x \varphi)$ for all $(x,\eta)\in U_x\times U_\eta$), and let $a_h(x,y,\eta)\in S^{\rm comp}_\beta(T^*\mathbf{R}^d)$ be a symbol supported on $U_x\times U_\eta$ in the $(x,\eta)$ variables. One defines a \emph{Fourier integral operator} associated with $\kappa$, with phase $\varphi$ and amplitude $a$, by
	\begin{equation} \label{eq: defi FIO nice form}\mathcal I[\varphi,a] u(x):=\frac{1}{(2\pi h)^d}\int_{\mathbf{R}^{2d}} \mathrm{e}^{\frac{\mathrm{i}}{h}(\varphi(x,\eta)-\langle y,\eta\rangle)}a_h(x,y,\eta)u(y)\mathrm{d}y\mathrm{d}\eta.\end{equation}
	The operator $\mathcal I[\varphi,a]$ is called a \emph{local semiclassical quantization} of $\kappa$. We will allow the phase and amplitude to depend on a time parameter $t$, but will still denote by $\mathcal I[\varphi,a]$ the resulting operator.\end{defi}

\begin{ex} When $\varphi(x,\eta)=\langle x,\eta\rangle$, one recovers pseudodifferential operators. In particular, if $\tilde a(x,y,\eta)=a\big(\frac{x+y}{2},\eta\big)$ then 
	\[\mathcal I[\langle x,\eta\rangle,\tilde a]=\mathrm{Op}_h^{\rm w}(a).\]\end{ex}
\subsection{A stationary phase lemma}
We quote a result of Blair \cite{Blair2022} that gives a quantitative stationary phase expansion of certain oscillatory integrals with singular phase functions and amplitudes, in a slightly simplified version.

\begin{prop}[Stationary phase estimate, {\cite[Theorem 1.1]{Blair2022}}] \label{prop: stat phase blair} Let $\phi_h(x)$ be a phase function and $u_h(x)$ be an amplitude depending on a semiclassical parameter $h$, and consider the oscillatory integral
	\[I_h:=\int_{\mathbf{R}^d} \mathrm{e}^{\frac{\mathrm{i}}{h} \phi_h(x)}u_h(x)\mathrm{d}x.\]
	Let $\sigma,\mu\in [0,1/2)$, and assume $5\mu+2\sigma<1$. Set $\rho:=\sigma+\mu$. We make the following assumptions:
	\begin{itemize}\item $\|\phi_h\|_{C^\ell}\le C_\ell h^{-\mu-\sigma(\ell-2)}$ for $\ell\ge 2$.
		\item $\phi_h$ admits a single critical point at $x=0$, and $\big\|(\nabla^2\phi_h(0))^{-1}\big\|\lesssim 1$.
		\item The amplitude $u_h$ is supported in a ball of radius $\varepsilon h^\rho$ for $\varepsilon$ small enough depending only on the constants involved in the upper bounds on $\|\phi\|_{\dot C^\ell}$ for $\ell\le 3$. Moreover, we assume $\|u_h\|_{C^\ell}\lesssim h^{-\rho \ell}$.\end{itemize} Then, we have the following expansion in powers of $h$:
	\begin{equation}\label{eq: stat phase blair}I_h=(2\pi h)^{\frac d2}\frac{\mathrm{e}^{\frac{\mathrm{i}}{h} \phi_h(0)+ \frac{\mathrm{i}\pi}{4}\operatorname{sgn}(\nabla^2\phi_h(0))}}{|\det(\nabla^2 \phi_h(0))|^{\frac 12}}\sum_{j=0}^{N-1}h^j\mathcal L_{j,h}{u_h}\Big|_{x=0}+ \mathcal O(h^{(1-2\rho)N+\frac d2}).\end{equation}
	The error depends only on the implied constants in the estimates of $\|\phi_h\|, \|u_h\|$.
	The operators $\mathcal L_{j,h}$ are given by
	\begin{equation} \label{eq: L j h u from blair}\mathcal L_{j,h}u=\sum_{\ell-m=j}\sum_{2\ell\ge 3m} \frac{\mathrm{i}^{-j}2^{-\ell}}{\ell!m!} \big\langle \big(\nabla^2\phi_h(0)\big)^{-1}D_x,D_x\big\rangle^{\ell}(g_h^m u),\end{equation}
	where $g_h$ is obtained by removing the quadratic part of $\phi_h$:
	\[g_h(x)=\phi_h(x)-\phi_h(0)-\frac{1}{2}\langle \nabla^2\phi_h(0) x,x\rangle. \]
\end{prop}
As noted by Blair, a potential drawback of this result is the condition imposed on the support of $u_h$. However, with additional assumptions on the phase function, this condition can be removed.

\begin{prop}[Stationary phase, bis] \label{prop: stat phase bis}In the setting of Proposition \ref{prop: stat phase blair}, if $|\nabla \phi_h(x)|\ge \frac 12 |x|$, then the expansion \eqref{eq: stat phase blair} holds under the weaker assumption $\mathrm{Supp}(u_h)\subset \{|x|\le 1\}$, but the error in \eqref{eq: stat phase blair} becomes $\mathcal O_N(h^{(1-2\rho-\mu)N})$. \end{prop}

\begin{proof}We just have to show that the contribution of the integral outside of $\{x:|x|\le \varepsilon h^\rho\}$ is negligible. To do so, introduce a cutoff $\chi\in C^\infty_{\rm comp}([-1,1])$ with $\chi\equiv 1$ on $[-\frac 12,\frac 12]$. Then, decompose $I_h=I_h^{(1)}+I_h^{(2)}$ with
	\[I_h^{(1)}=\int_{\mathbf{R}^d} \mathrm{e}^{\frac{\mathrm{i}}{h} \phi_h(x)}u_h(x)\chi(\frac{x}{\varepsilon h^\rho})\mathrm{d}x, \quad  \text{and} \quad I_h^{(2)}=\int_{\mathbf{R}^d} \mathrm{e}^{\frac{\mathrm{i}}{h} \phi_h(x)}u_h(x)\big(1-\chi(\frac{x}{\varepsilon h^\rho})\big)\mathrm{d}x.\]
	One obtains an asymptotic expansion for $I_h^{(1)}$ using Proposition \ref{prop: stat phase blair}. It remains to control $I_h^{(2)}$. This is an easy consequence of nonstationary phase estimates. Indeed, by a change of variables, one has
	\[I_h^{(2)}=h^\rho\int_{\mathbf{R}^d} \mathrm{e}^{\frac{\mathrm{i}}{h} \phi_h(h^\rho x)}u_h(h^\rho x)\big(1-\chi(\frac{x}{\varepsilon})\big)\mathrm{d}x=h^\rho\int \mathrm{e}^{\frac{\mathrm{i}}{\tilde h} \tilde \phi_h(x)} \tilde u_h(x)(1-\chi(\frac{x}{\varepsilon}))\mathrm{d}x,\]
	where we have set $\tilde \phi_h(x):=h^{\mu-2\rho}\phi_h(h^\rho x)$, $\tilde u_h(x):=u_h(h^\rho x)$ and $\tilde h:=h^{1+\mu-2\rho}$. The advantage is that now the derivatives of $\tilde \phi_h$ of order $\ge 2$ are uniformly bounded independently of $h$, and $\tilde u_h$ has all its derivatives bounded uniformly in $h$. Also, on the support of the integrand, one has 
	\[|\nabla \tilde \phi_h(x)|\ge h^{\mu-\rho}|\nabla\phi_h(h^\rho x)|\ge h^{\mu}\frac \varepsilon4.\]
	A nonstationary phase estimate, see e.g. \cite[Theorem 2.33]{NonnenmacherLectures}, leads to
	\[|I_h^{(2)}|\lesssim_{\!\! N,\varepsilon} \tilde h^N h^{-2\mu N}=h^{(1-2\rho-\mu)N}.\]
	This finishes the proof. \end{proof}


\subsection{Composition of FIOs} We give a composition formula for Fourier integral operators with phases and amplitudes lying in exotic symbol classes. The setting is somewhat \emph{ad hoc}, but it encompasses all the situations that we will encounter in the next sections. Motivated by the study of quantum propagators, we also consider a continuous family of Fourier integral operators parameterized by the time variable, by letting $\mathcal U_i=\mathcal U_i(t)$ depend on $t$, for $t$ in some interval $ [t_0-\varepsilon,t_0+\varepsilon]$. We have the following composition formula.

\begin{prop}[Composition of FIOs] \label{thm: composition FIO} Let $T\ge 1$. Fix $\sigma\in [0,1/2)$. Let $\kappa_1,\kappa_2(t,\cdot)$ be two symplectomorphisms, with $\kappa_2$ allowed to depend on a time variable $t$. We assume that the graphs of all the maps $\kappa_1$, $\kappa_2(t,\cdot)$, $\kappa_3(t,\cdot):=\kappa_2(t,\cdot)\circ \kappa_1$ are projectable, for $t$ close to $t_0$, in a sense specified below. Let $\varphi_2(t,x_2,\xi_1)$ and $\varphi_1(x_1,\xi_0)$ be local generating functions for $\kappa_2(t,\cdot)$ and $\kappa_1$ respectively. Following Lemma \ref{lem: composition generating functions}, we define
	\[\varphi_3(t,x_2,\xi_0):=\varphi_2(t,x_2,\xi_1)-\langle x_1,\xi_1\rangle+\varphi_1(x_1,\xi_0),\]
	where $(x_1,\xi_1)=:z_c(\mathbf x)$ depends on $\mathbf x=(t,x_2,x_0,\xi_0)$ and is fixed by the condition
	\[(\bullet,\xi_0)\overset{\kappa_1}{\longrightarrow}(x_1,\xi_1)\overset{\kappa_2}{\longrightarrow} (x_2,\bullet).\]
	We consider two Fourier integral operators
	\[\mathcal U_1=\mathcal I[\varphi_1,a_h^{(1)}], \qquad \mathcal U_2(t)=\mathcal I[\varphi_2,a_h^{(2)}],\]
	Assume $h^{-\beta}\mathrm{e}^{CT}\le h^{-\sigma}$. We make the following assumptions.
	
	\begin{enumerate}\item \label{assum 1 on diffeo} \emph{Projectability.} All the graphs of $\kappa_1,\kappa_2,\kappa_3$ are projectable when we restrict to the support of $a_h^{(2)}(x_2,x_1,\xi_1)a_h^{(1)}(x_1,x_0,\xi_0)$. Moreover we have $\big|\frac{\partial \xi_1(x_1,\xi_0)}{\partial x_1}\big|\le c_1$ and $\big| \frac{\partial x_1(x_2,\xi_1)}{\partial \xi_1}\big|\le c_2$ with $c_1$ small enough depending on $c_2$, on the supports of the symbols $a_h^{(1)}$ and $a_h^{(2)}$, respectively. 
	
		\item \label{assum phase function} \emph{Control on the phase.} $\|\varphi_i\|_{C^\ell_{t,x_i,\xi_{i-1}}}\lesssim T h^{-\sigma(\ell-2)}$ for $\ell\ge 2$ and $\|z_c\|_{C^\ell_{t,x_2,\xi_0}}\lesssim  T h^{-\sigma(\ell-1)}$ for $\ell\ge 1$. 
		
		\item \emph{Control on the amplitude}. \label{assum amplitude} the amplitudes $a_h^{(i)}$ have compact support and satisfy $\|a_h^{(i)}\|_{C^\ell}\lesssim h^{-\nu \ell}$ for $i=1,2$, for some $\nu\in [0,1/2)$.
	\end{enumerate}
Then, letting $\mathcal U_3(t)=\mathcal U_2(t)\mathcal U_1$, one has
\begin{equation} \label{eq: U_3(t) as FIO} \mathcal U_3(t)=\mathcal I[\varphi_3,a_h^{(3)}]+R_N,\end{equation}
where:
\begin{itemize}\item Letting $\mathbf x=(t,x_2,x_0,\xi_0)$ and $z=(x_1,\xi_1)$ the symbol $a_h^{(3)}$ admits an expansion $a_h^{(3)}=\sum_{j\le N} h^ja_j^{(3)}$, where
	\begin{equation}\label{eq: symbol aj composition FIO} a_j^{(3)}(\mathbf x)=\sum_{m+i+k=j} \sum_{|\alpha|\le 2m} c_{m,\alpha}(\mathbf x)\partial_z^\alpha \big(a_i^{(2)}a_k^{(1)}\big)\Big|_{z=z_c(\mathbf x)}.\end{equation}
	The coefficients $c_{m,\alpha}$ and the point $z_c$ are controlled by
	\begin{equation}\label{eq: estimee coefficient derivee j composition FIO} \|c_{m,\alpha}\|_{C^\ell_{\mathbf x}}\le C_{m,\alpha,\ell} T^{(\ell+2m-|\alpha|)}h^{-\sigma(\ell+2m-|\alpha|)}, \qquad \|z_c\|_{C^\ell_{\mathbf x}}\le C_{j,\ell} T h^{-\sigma (\ell-1)_+}.\end{equation}
	\item The leading part of the symbol is given by
	\begin{equation} \label{eq: leading part symbol composition} a_0^{(3)}(\mathbf x)=\frac{a_0^{(2)}a_0^{(1)}(\mathbf x,z_c(\mathbf x)) }{|\det^{\frac 12}(\nabla^2_z \phi_h(\mathbf x,z_c(\mathbf x)))|},\end{equation}
	where $\phi_h(\mathbf x,z)=\varphi_2-\langle x_1,\xi_1\rangle+\varphi_1$.
	\item The error satisfies
	\[\|R_N\|_{L^2/\rm tr}\le C_{N}T^{3N} h^{(1-2\max(\sigma,\nu))N-C}.\]
	The constant $C_N$ only depends on a finite number (depending on $N$) of the implicit constants appearing in the estimates on $\|\varphi_i\|_{C^\ell}$, $\|a_h^{(i)}\|_{C^\ell}$ and $\|z_c\|_{C^\ell}$.
	\end{itemize}\end{prop}
	
		\begin{proof}[Proof of Proposition \ref{thm: composition FIO}] 
We denote by $z$ the variable $(x_1,\xi_1)$, and by $\mathbf x$ the variable $(t,x_2,x_0,\xi_0)$. Let us set
\[\phi_{h}(z,\mathbf x)=\varphi_2(t,x_2,\xi_1)-\langle x_1,\xi_1\rangle +\varphi_1(x_1,\xi_0),\]
and 
\[b_h(z,\mathbf x):=a_h^{(2)}(x_2,x_1,\xi_1)a_h^{(1)}(x_1,x_0,\xi_0).\]
Then, the Schwartz kernel of $\mathcal U_3(t)$ is given by
\[K(t,x_2,x_0)=\frac{1}{(2\pi h)^{2d}}\int \mathrm{e}^{-\frac{\mathrm{i}}{h}\langle x_0,\xi_0\rangle}\left(\int \mathrm{e}^{\frac{\mathrm{i}}{h}\phi_h(z,\mathbf x)}b_h(z,\mathbf x)\mathrm{d}z\right)\mathrm{d}\xi_0.\]
\textbf{1.} We perform a stationary phase expansion in the $z$ variable. Let us check the assumptions of Proposition \ref{prop: stat phase bis}. The gradient of the phase is given by
\[\nabla_z \phi_h(x_1,\xi_1)=\begin{pmatrix}\nabla_{x_1}\varphi_1(x_1,\xi_0)-\xi_1 \\ \nabla_{\xi_1}\varphi_2(t,x_2,\xi_1)-x_1\end{pmatrix}.\]
It vanishes on the support of $b_h$ if and only if $\kappa_1(\bullet,\xi_0)=(x_1,\xi_1)$ and $\kappa_2(x_1,\xi_1)=(x_2,\bullet)$. This implies $\kappa_3(\bullet,\xi_0)=(x_2,\bullet)$; by the projectability assumption on the graph, we deduce that there is a unique critical point given by 
\[z_c(\mathbf x)=\kappa_2^{-1}(x_2,\partial_{x_2}\varphi_3)=\kappa_1(\partial_{\xi_0}\varphi_3,\xi_0).\]
The Hessian of the phase is given by
\begin{equation} \label{eq: hessian of phase in compo lemma} \nabla^2_z\phi_h=\begin{pmatrix}\frac{\partial \xi_1(x_1,\xi_0)}{\partial x_1} & -\mathrm{Id} \\ -\mathrm{Id} & \frac{\partial x_1(x_2,\xi_1)}{\partial \xi_1}\end{pmatrix}.\end{equation}
By the assumption (\ref{assum 1 on diffeo}), the Hessian $\nabla^2_z\phi_h$ is invertible on the support of $b_h$, with inverse satisfying
\begin{equation} \label{eq: bound inverse hess}\|(\nabla^2 \phi_h)^{-1}\|\lesssim 1.\end{equation}
Since $\nabla_z \phi_h(z_c)=0$, assumption (\ref{assum 1 on diffeo}) also implies
\[|\nabla_z \phi_h(z)|\ge \frac 12|z-z_c|,\]
for all $z$ in the support. Finally, since $\nabla^2_z\phi_h(z_c)$ can be continuously deformed in $\mathrm{GL}_{2d}(\mathbf{R})$ to the matrix $J=\begin{pmatrix}
	0 & -\mathrm{Id} \\ -\mathrm{Id} & 0
\end{pmatrix}$, one has $ \operatorname{sgn} \nabla^2_z \phi_h(z_c)=\operatorname{sgn}J=0.$

By assumption (\ref{assum phase function}) and \eqref{eq: hessian of phase in compo lemma}, the derivatives of the phase with respect to $z$ are bounded by a constant times $Th^{-\sigma(\ell-2)}$ for $\ell\ge 2$. 
Similarly, by assumption (\ref{assum amplitude}), the amplitude $b_h(z,\mathbf x)$ is compactly supported as a function of $z$ and satisfies $\|b_h\|_{C^\ell_z}\lesssim h^{-\nu\ell}$.

By applying Proposition \ref{prop: stat phase bis} with $T=h^{-\nu}$, one finds
\[K(t,x_2,x_0)=\frac{1}{(2\pi h)^d}\int \mathrm{e}^{\frac{\mathrm{i}}{h} (\varphi_3(t,x_2,\xi_{0})-\langle x_0,\xi_0\rangle)}\sum_{j=0}^{N-1} \mathcal L_{j,h,\mathbf x} b_h\Big|_{z=z_c(t,x_2,\xi_0)} \mathrm{d}\xi_0+R_N(t,x_2,x_0),\]
where the remainder is compactly supported and satisfies pointwise estimates
\[R_N(t,x_2,x_0)=\mathcal O(T^{3N}h^{(1-2\max(\sigma,\nu))N-C}).\]
This readily implies $\|R_N(t)\|_{L^2/\rm tr}\lesssim T^{3N}h^{(1-2\max(\sigma,\nu))N-C}$.

\textbf{2.} To proceed, we need to write down explicitly the operators $\mathcal L_{j,h,\mathbf x}$ appearing in the stationary phase expansion. Consider the function $g$ obtained by subtracting from $\phi_h$ its quadratic part at $z=z_c$, that is
\[g_h(\mathbf x,z):=\phi_h(\mathbf x,z)-\phi_h(\mathbf x,z_c(\mathbf x))-\frac{1}{2}\langle \nabla^2_z \phi_h(\mathbf x,z_c(\mathbf x))(z-z_c(\mathbf x)),z-z_c(\mathbf x)\rangle,\]
Then, the operator $\mathcal L_{j,h,\mathbf x}$ writes (\cite[Theorem 1.1]{Blair2022})
\[\mathcal L_{j,h,\mathbf x}u(z)=\frac{1}{|\det(\nabla^2_z\phi_h)|^{\frac 12}}\sum_{k-m=j}\sum_{2k\ge 3m} \frac{\mathrm{i}^{-j}2^{-k}}{k!m!}\left\langle \big(\nabla_z^2\phi_h(\mathbf x,z_c(\mathbf x))\big)^{-1} D_z,D_z\right\rangle^k (g_h^m(\mathbf x,z) u(z)).\]
In particular $\mathcal L_{0,h,\mathbf x}=\displaystyle \frac{1}{|\det(\nabla^2_z\phi_h(\mathbf x,z_c(\mathbf x)))|^{\frac 12}}$. By Fa\`a Di Bruno's formula, along with the estimates on $\phi_h$ and $z_c$, the function
\[H:\mathbf x\mapsto \nabla_z^2 \phi_h(\mathbf x,z_c(\mathbf x))\]
satisfies $\|H\|_{C^\ell}\lesssim T^{\ell+1}h^{-\sigma \ell}$ for $\ell\ge 1$. Thus, the differential operators appearing in the definition of $\mathcal L_{j,h,\mathbf x}$ can be written
\begin{equation} \label{eq: H moins un Dz Dz}\langle H(\mathbf x)^{-1} D_z,D_z\rangle^k=\sum_{|\alpha|=2k} a_\alpha(\mathbf x) \frac{\partial^{|\alpha|}}{\partial z^\alpha},\end{equation}
for some coefficients $a_\alpha$ that satisfy 
\begin{equation} \label{eq: the coeffs a alpha} \|a_\alpha(\mathbf x)\|_{C^\ell}\lesssim T^{2\ell}h^{-\sigma \ell}.\end{equation}

We turn to $g_h(\mathbf x,z)$. By Taylor's formula, one can write
\begin{equation}\label{eq: taylor g_h} g_h(\mathbf x,z)=\sum_{|\gamma|=3} (z-z_c(\mathbf x))^\gamma R_\gamma(\mathbf x,z),\end{equation}
where the sum runs over multi-indices $\gamma$ of length $3$, and
\[R_\gamma(\mathbf x,z)=\int_0^1 \frac{(1-s)^2}{2!} \frac{\partial^3}{\partial z^\gamma} \phi_h(\mathbf x,sz+(1-s)z_c(\mathbf x))\mathrm ds.\]
From the estimates $\|\phi_h\|_{C^\ell}\lesssim Th^{-\sigma(\ell-2)_+}$ and $\|z_c\|_{C^\ell}\lesssim Th^{-\sigma(\ell-1)_+}$, one finds
\begin{equation} \label{eq: Cell norm Rgamma}\forall \ell\ge 0, \quad \|R_\gamma\|_{C^\ell}\lesssim T^{\ell+1} h^{-\sigma(\ell+1)}.\end{equation}
Taking \eqref{eq: taylor g_h} to $m$-th power, we get
\[g_h^m(\mathbf x,z)=\sum_{|\bm \gamma|=3m} (z-z_c(\mathbf x))^{\bm \gamma} R_{\bm \gamma},\]
where the sum runs over tuples $\bm \gamma=(\gamma_1,\ldots,\gamma_m)\in (\mathbf N^d)^m$ satisfying $|\gamma_i|=3$ for all $i$, and we have used the shorthand $R_{\bm \gamma}=\prod_{i=1}^m R_{\gamma_i}$. One can check from \eqref{eq: Cell norm Rgamma} that if $|\bm \gamma|=3m$ then
\begin{equation} \label{eq: Cell norm R bold gamma}\|R_{\bm \gamma}\|_{C^\ell}\lesssim T^{\ell+m}h^{-\sigma(\ell+m)}.\end{equation}
By the product rule, for $|\alpha|=2k\ge 3m$, we have
\[\frac{\partial^{2k}}{\partial z^{\alpha}} (g_h^m u)\big|_{z=z_c(\mathbf x)}=\sum_{\alpha_1+\alpha_2=\alpha}\partial_z^{\alpha_1}g_h^m(\mathbf x,z)\big|_{z=z_c(\mathbf x)}\cdot \partial_z^{\alpha_2} u\big|_{z=z_c(\mathbf x)}.  \]
Now, to get nonzero terms when evaluating $\partial_z^{\alpha_1}(z-z_c(\mathbf x))^{\bm \gamma}R_{\bm \gamma}$ at $z=z_c$, at least $3m$ derivatives need to hit the products $(z-z_c)^{\bm \gamma}$. Thus, if $|\alpha_1|=j$ we have
\[\frac{\partial^j}{\partial z^{\alpha_1}}(g_h^m u)\big|_{z=z_c}=\sum_{|\bm \gamma|=3m\le j} \sum_{|\beta|=j-3m} C_{\beta,\bm \gamma}\partial_z^{\beta} R_{\bm \gamma}\big|_{z=z_c}.\]
Here the $C_{\beta,\bm \gamma}$ are combinatorial coefficients. It follows that
\[\frac{\partial^{2k}}{\partial z^{\alpha}} (g_h^m u)\big|_{z=z_c(\mathbf x)}=\sum_{k_1+k_2=2k}\sum_{|\bm \gamma|=3m\le k_1} \sum_{|\alpha_1|=k_1-3m}\sum_{|\alpha_2|=k_2} \partial_z^{\alpha_1}R_{\bm \gamma}\cdot \partial_z^{\alpha_2}u\Big|_{z=z_c(\mathbf x)}.\]
Gathering everything, one finds
\[\mathcal L_{j,h,\mathbf x}u=\sum_{|\alpha|\le 2j} c_{j,\alpha}(\mathbf x)\partial^{\alpha}_{z} u\Big|_{z=z_c(\mathbf x)}.\]
Here, for $|\alpha|=k_2$, the coefficients $c_{j,\alpha}$ are of the form
\begin{equation} \label{eq: formula for c j alpha} c_{j,\alpha}=\sum_{k-m=j}\sum_{2k\ge 3m}\sum_{k_1+k_2=2k} \sum_{|\beta|=2k}\sum_{|\bm \gamma|=3m\le k_1} a_\beta(\mathbf x)\partial^{k_1-3m}R_{\bm \gamma}.\end{equation}
By \eqref{eq: Cell norm R bold gamma} and the bound \eqref{eq: the coeffs a alpha}, we have
\[\|c_{j,\alpha}\|_{C^\ell_{\mathbf x}}\le \sum_{k-m=j}\sum_{2k\ge 3m}\sum_{k_1+k_2=2k} T^{2(\ell+k_1-2m)}h^{-\sigma(\ell+k_1-2m)},\]
it remains to notice that the condition on the indices forces $k_1-2m= 2j-k_2$ to conclude that
\[\|c_{j,\alpha}\|_{C^\ell_{\mathbf x}}\le C_{j,\alpha,\ell} T^{2(\ell+2j-|\alpha|)} h^{-\sigma(\ell+2j-|\alpha|)}.\]
\textbf{3.} We are ready to define the symbol $a^{(3)}$. By the above computation, we have
\[\sum_{j=0}^{N-1} h^j\mathcal L_{j,h,\mathbf x} b_h=\sum_{j=0}^{N-1} h^j\sum_{|\alpha|\le 2j}c_{j,\alpha}(\mathbf x)\partial^{\alpha}_{z} b_h\Big|_{z=z_c(\mathbf x)}.\]
We expand $b_h=\sum_{i\le N,k\le N} h^{i+k}a_i^{(2)}a_k^{(1)}$, then for $j<N$, set
\[a_j^{(3)}:=\sum_{j'+i+k=j} \sum_{|\alpha|\le 2j'} c_{j',\alpha}(\mathbf x)\partial_z^\alpha \big(a_i^{(2)}a_k^{(1)}\big)\Big|_{z=z_c(\mathbf x)}.\]
Letting $a^{(3)}_h:=\sum_{j< N} h^j a_j^{(3)}$, one arrives at
\begin{align*}\sum_{j<N}h^j\mathcal L_{j,h,\mathbf x}b_h& =a_h^{(3)}+\sum_{j<N}\sum_{i,k\le N, \ i+k\ge N} h^{i+k+j} \sum_{j'+i+k=j} \sum_{|\alpha|\le 2j'} c_{j',\alpha}(\mathbf x)\partial_z^\alpha \big(a_i^{(2)}a_k^{(1)}\big)\Big|_{z=z_c(\mathbf x)} \\ & = a_h^{(3)}+\mathcal O(T^{2N}h^{N(1-2\max(\sigma,\nu))}),\end{align*}
where the sum on the right-hand side has been estimated using \eqref{eq: estimee coefficient derivee j composition FIO} and the assumptions on $a^{(1)},a^{(2)}$. The error $\mathcal O(T^{2N}h^{N(1-2\max(\sigma,\nu))})$ has compact support in $(x_2,x_0)$ thus can be incorporated in the remainder $R_N$.
\end{proof}

\subsubsection{Two particular cases} As a corollary, we specify the result to the case where $\mathcal U_1$ or $\mathcal U_2$ is a semiclassical pseudodifferential operator with compactly supported symbol.

\begin{prop}Let $\chi\in C^\infty_{\rm comp}(\mathbf{R}^{2d})$ satisfy $\|\chi\|_{C^\ell}\lesssim h^{-\nu\ell}$. Then,
\[\mathrm{Op}_h^{\rm w}(\chi)\mathcal U_1=\mathcal I[\varphi_1,b_h]+\mathcal O(T^{3N}h^{N(1-2\max(\sigma+\varepsilon))-C}),\]
with $b_h=\sum_{j\le N}h^j b_j$ and $b_0=\chi(x_1,\partial_{x_1} \varphi_1(x_1,\xi_0))a_0(x_1,x_0,\xi_0)$. Also,
\[\mathcal U_1\mathrm{Op}_h^{\rm w}(\chi)=\mathcal I[\varphi_1,c_h]+\mathcal O(T^{3N}h^{N(1-2\max(\sigma+\varepsilon))-C}),\]
with $c_h=\sum_{j\le N}h^jc_j$ and $c_0=a_0(x_1,x_0,\xi_0)\chi(x_0,\xi_0)$.\end{prop}

\subsection{Short time WKB method}\label{sec: short time WKB}

In this section, we provide a WKB approximation of the quantum propagator $\mathrm{e}^{-\frac{\mathrm{i}}{h} tP_{h}}$ for short times, when $P_h$ has compactly supported symbol.
We follow the proof presented in \cite[Chapter 10]{zworski2012semiclassical}. The novelty here is that the Hamiltonian belongs to the class of exotic symbols $S_\beta$, 
hence much of the work consists of carefully tracking dependencies on the derivatives of the Hamiltonian in the course of the proof. 

\subsubsection{The setting} Let $p(x,\xi)$ be a smooth symbol on $T^*\mathbf{R}^d$ with compact support. We assume that $p$ has a particular structure, namely,
\begin{equation}p=p_\delta+r_\delta, \end{equation}
where $p_\delta,r_\delta$ both depend on $\delta$ and are compactly supported. Moreover, we assume that
\begin{equation} \label{eq: assumptions p WKB} \|p_\delta\|_{C^\ell}\lesssim h^{-\beta(\ell-2)_+}, \qquad \|r_\delta\|_{C^\ell}\lesssim h^2h^{-\beta(\ell+2)}.\end{equation}
The motivation behind this assumption is that  in coordinate charts, the operator considered in Theorem \ref{prop: long time gutz} is precisely of the form $\operatorname{Op}_h^{\rm w}(p)$ with $p$ as above, see Proposition \ref{prop: egorov exotic}.

Let $\Phi^t_\delta$ denote the Hamiltonian flow generated by $p_\delta$. Then for small $t$, we have $\Phi^t_\delta(y,\eta)=(y,\eta)$ outside a compact set. Thus, for short time $t$, say $|t|\le 1$, the flow $\Phi^t_\delta$ admits a global generating function $\varphi_\delta(t,x,\eta)$, which depends implicitly on $h$, such that $\varphi_\delta(0,x,\eta)=\langle x,\eta\rangle$ and $\varphi_\delta(t,x,\eta)=\langle x,\eta\rangle$ outside a compact set. Recall from Lemmas \ref{lem: space derivatives generating function}--\ref{lem: spacetime derivatives generating function} that we have 
\begin{equation} \label{eq: sec WKB estimate phi}\forall \ell\ge 2, \quad \|\varphi_\delta\|_{C^\ell_{t,x,\eta}}\le h^{-\beta(\ell-2)}, \end{equation}
with a slight improvement at $t=0$.

\begin{prop} \label{prop: short time wkb}	Let $P_h:=\mathrm{Op}_h^{\rm w}(p)$, and consider the unitary group $U(t):=\mathrm{e}^{-\frac{\mathrm{i}}{h}tP_h}:L^2(\mathbf R^d)\to L^2(\mathbf R^d)$ defined as the solution to the Cauchy problem
\begin{equation} \label{eq: equation U(t)} \left\{\begin{array}{ll}U(0)=\mathrm{Id} \\ (hD_t +P_h)U(t)=0\end{array}\right..\end{equation}
Then, we have for $|t|\le 1$:
\begin{equation} \label{eq: short time parametrix}U(t)=\mathcal I[\varphi_\delta,a_h]+\mathcal R_N(t).\end{equation}
The amplitude expands as $a_h=\sum_{j=0}^N h^j a_j$, and the leading term $a_0$ is given by
\[a_0(t,x,y,\eta)=|\det(\partial_{x\eta}^2 \varphi_\delta)|^{\frac 12},\]
and the $C^\ell$ norms of the symbols $a_j$ are estimated by
\[\|a_j\|_{C^\ell_{t,x,\eta}}\lesssim h^{-\beta(\ell+2j)}.\]
Finally, the error satisfies $\|\mathcal R_N(t)\|_{L^2}\lesssim h^{N(1-2\beta)-C}$, where $C$ is a dimensional constant.
\end{prop}
The proof of Proposition \ref{prop: short time wkb} involves several steps: 
\begin{enumerate}
\item Work out the WKB method: inject the oscillatory integral Ansatz $\mathcal I[\varphi_\delta,a_h]$ in the equation \eqref{eq: equation U(t)}, then find a sequence of transport equations for $(a_j)$. This step is standard, but we need to track the symbol class of the different functions that appear in the computations.
\item Solve the sequence of transport equations for $(a_j)$, and give estimates on $\|a_j\|_{C^\ell}$. This is also classical.
\item Show the remainder estimate $\|\mathcal R_N(t)\|_{L^2}=\mathcal O(h^{N(1-2\beta)-C})$.
\end{enumerate}
We follow \cite[Chapter 10]{zworski2012semiclassical}.
\begin{prop}	\label{prop: form of the differential operator} Define a $(t,\eta)$-dependent vector field
\begin{equation}
	\label{eq: define V} V(x;t,\eta):=\sum_j \partial_{\xi_j}p_\delta(x,\nabla_x \varphi_\delta(t,x,\eta))\partial_{x_j}.
\end{equation}
Then, as operators acting on functions in the $x$ variable: \begin{equation}\label{eq: equation for ansatz in WKB}\mathrm{e}^{-\frac{\mathrm{i}}{h} \varphi_\delta}(hD_t+P_h)\mathrm{e}^{\frac{\mathrm{i}}{h} \varphi_\delta}=h\left(\partial_t +V+\frac{1}{2}\mathrm{div}(V) +h R\right)
	,\end{equation}
where:
\begin{itemize}
	\item $R=R(t,\eta)$ is a pseudodifferential operator acting on functions in the $x$ variable, of the form
	\begin{equation} \label{eq: R is of the form} R:=\mathrm{Op}_h^{\rm w}(r_2)+\mathrm{Op}_h^{\rm w}(r_1)\partial_x +\mathrm{Op}_h^{\rm w}(r_0)\partial_{xx}^2,\end{equation}
	and we have estimates on the symbols:
	\begin{equation} \label{eq: estimates on ri}\|r_i\|_{C^\ell_{t,x,\xi,\eta}}\le h^{-\beta(\ell+i)}, \quad i=0,1,2.\end{equation}
\end{itemize}
\end{prop}
\begin{proof} In the course of the proof, the functions $r_0,r_1,r_2$ may vary from line to line, but they will always satisfy \eqref{eq: estimates on ri}.

\textbf{1.} Let $u=u(x)\in C^\infty_{\rm comp}(\mathbf{R}^d)$. We first treat $t,\eta$ as parameters and look for an expression of
\[\mathrm{e}^{-\mathrm{i} \varphi_\delta/h} P_h \mathrm{e}^{\mathrm{i} \varphi_\delta/h}u(x).\]
Let $F(x,y;t,\eta)$ be defined by
\[F(x,y;t,\eta)=\int_0^1 \nabla_x\varphi_\delta(t,sx+(1-s)y,\eta)\mathrm ds,\]
so that $\varphi_\delta(t,x,\eta)-\varphi_\delta(t,y,\eta)=\langle F(x,y;t,\eta),x-y\rangle$.
Then,
\begin{equation} \label{eq: bounds F WKB}\|F\|_{C^\ell_{t,x,y,\eta}}\lesssim h^{-\beta(\ell-1)_+},\end{equation}
and the conjugation of $P_h$ by the multiplication operator by $\mathrm{e}^{\frac{\mathrm{i}}{h} \varphi_\delta}$ can be expressed
\begin{equation} \label{eq: conjugated Ph}\mathrm{e}^{-\mathrm{i} \varphi_\delta/h} P_h \mathrm{e}^{\mathrm{i} \varphi_\delta/h}u(x)=\frac{1}{(2\pi h)^d}\int \mathrm{e}^{\frac{\mathrm{i}}{h} \langle x-y,\xi\rangle} p\left(\frac{x+y}{2},\xi +F(x,y;t,\eta)\right) u(y)\mathrm dy\mathrm d\xi.\end{equation}
Define
\begin{equation}\label{eq: definitjon q} q(x,\xi;t,\eta)=p(x,\xi+\nabla_x \varphi_\delta(t,x,\eta)).\end{equation}
Then, the symbol in the oscillatory integral \eqref{eq: conjugated Ph} writes as
\begin{equation} \label{eq: one can write}p\left(\frac{x+y}{2},\xi +F(x,y;t,\eta)\right)- q\left(\frac{x+y}{2},\xi;t,\eta\right)=e(x,y,\xi;t,\eta),\end{equation}
with
\begin{equation} \label{eq: defi symbole e} e(x,y,\xi;t,\eta)=\left\langle G(x,y,\xi;t,\eta);F(x,y;t,\eta)-\nabla_x \varphi_\delta\big(t,\frac{x+y}{2},\eta\big)\right\rangle,\end{equation}
and where we have set
\begin{equation} \label{eq: defi G}G(x,y,\xi;t,\eta)=\int_0^1 \nabla_\xi p\Big(\frac{x+y}{2},\xi+ sF(x,y;t,\eta)+(1-s)\nabla_x \varphi_\delta\big(t,\frac{x+y}{2},\eta\big)\Big)\mathrm{d}s.\end{equation}
By \eqref{eq: conjugated Ph}, \eqref{eq: definitjon q} and \eqref{eq: one can write}, we see that
\[\mathrm{e}^{-\mathrm{i} \varphi_\delta/h} P_h \mathrm{e}^{\mathrm{i} \varphi_\delta/h}=\mathrm{Op}_h^{\rm w}(q(\operatorname{\cdot};t,\eta))+\operatorname{Op}_h(e(\cdot;t,\eta)).\]
By \eqref{eq: bounds F WKB}, \eqref{eq: sec WKB estimate phi}, \eqref{eq: assumptions p WKB}, and using Fa\`a di Bruno's formula, the function $G$ from \eqref{eq: defi G} satisfies
\begin{equation} \label{eq: borne G_t} \|G\|_{C^\ell_{t,x,y,\xi,\eta}}	\lesssim h^{-\beta(\ell-1)_+}.\end{equation}
By a Taylor expansion of $\nabla_x \varphi_\delta$ near $\frac{x+y}{2}$, the right term appearing in the scalar product \eqref{eq: defi symbole e} may be rewritten
\[F(x,y;t,\eta)-\nabla_x \varphi_\delta\big(t,\frac{x+y}{2},\eta\big)=\sum_{i,j} (x_i-y_i)(x_j-y_j)H_{ij}(x,y;t,\eta),\]
where $(H_{ij})_{i,j}$ is a family of vector valued functions, which are estimated using \eqref{eq: sec WKB estimate phi}:
\begin{equation} \label{eq: controle H_t}\|H_{ij}\|_{C^\ell_{t,x,y,\eta}}\lesssim \|\varphi_\delta\|_{C^{\ell+3}_{t,x,y,\eta}}\lesssim h^{-\beta(\ell+1)}.\end{equation}
Going back to \eqref{eq: defi symbole e}, we have obtained
\[e(x,y,\xi;t,\eta)=\sum_{i,j}(x_i-y_i)(x_j-y_j) \langle G(x,y,\xi;t,\eta),H_{ij}(x,y;t,\eta)\rangle.\]
Thanks to the presence of the quadratic terms $(x_i-y_i)(x_j-y_j)$, we can integrate by parts twice with respect to the $\xi$ variable, giving \[\mathrm{Op}_h(e(\cdot;t,\eta))=h^2\mathrm{Op}_h\big(\tilde b(\cdot;t,\eta)\big)\] with 
\[\tilde b(x,y,\xi;t,\eta)=\sum_{i,j} \big\langle \partial_{\xi_i\xi_j}^2 G(x,y,\xi;t,\eta),H_{ij}(x,y;t,\eta)\big\rangle.\]
Using \eqref{eq: borne G_t} and \eqref{eq: controle H_t} we find
\begin{equation} \label{eq: controle b_t}\|\tilde b\|_{C^\ell_{t,x,y,\xi,\eta}}\lesssim h^{-\beta(\ell+2)}.\end{equation}
By \cite[Theorem 4.20]{zworski2012semiclassical}, we have
\[\mathrm{Op}_h(\tilde b(\cdot;t,\eta))=\mathrm{Op}_h^{\rm w }( b(\cdot;t,\eta)),\]
where $ b(x,\xi;t,\eta)$ is defined by
\[ b(x,\xi;t,\eta):=\mathrm{e}^{\mathrm{i} h \langle D_z,D_\xi\rangle } \tilde b\big(x-\frac z2,x+\frac z2,\xi;t,\eta\big)\big|_{|z=0}.\]
Moreover, \eqref{eq: controle b_t} implies a similar control on $b$:
\begin{equation} \label{eq: estimate tilde b}\|b\|_{C^\ell_{t,x,\xi,\eta}} \lesssim h^{-\beta(\ell+2)}.\end{equation}
To summarize, we have obtained
\begin{equation} \label{eq: to summarize}\mathrm{e}^{-\mathrm{i} \varphi_\delta/h} P_h \mathrm{e}^{\mathrm{i} \varphi_\delta/h}=\mathrm{Op}_h^{\rm w}(q(\cdot;t,\eta))+\mathrm{Op}_h^{\rm w}\big( r_2(\cdot;t,\eta)\big)
	,\end{equation}
with $q=p(x,\xi+\nabla_x\varphi_\delta)$ and $r_2$ estimated by \eqref{eq: estimates on ri}. \medskip

\textbf{2.} We incorporate time derivatives. By \eqref{eq: to summarize}, we have
\[(hD_t+P_h) \mathrm{e}^{\frac{\mathrm{i}}{h}\varphi_\delta}=\mathrm{e}^{\frac{\mathrm i}{h}\varphi_\delta}\left(\partial_t \varphi_\delta+\frac{h}{\mathrm{i}}\partial_t +\mathrm{Op}_h^{\rm w}(q(x,\xi;t,\eta)) +h^2\mathrm{Op}_h^{\rm w}(r_2)\right).\]
Recall the Hamilton--Jacobi equation
\[\partial_t \varphi_\delta(t,x,\eta)=-p_\delta(x,\nabla_x\varphi_\delta(t,x,\eta)).\]
Introducing 
\[f(x,\xi;t,\eta):=p(x,\xi+\nabla_x \varphi_\delta(t,x,\eta))-p_\delta(x,\nabla_x \varphi_\delta(t,x,\eta)),\]
we can then rewrite
\[\partial_t \varphi_\delta+q=f,\]
where $q$ has been defined in \eqref{eq: definitjon q}. Recall that $p=p_\delta+r_\delta$. By Taylor's formula, we can expand
\begin{equation} \label{eq: can expand} f(x,\xi;t,\eta)=\xi\cdot \nabla_\xi p_\delta(x,\nabla_x \varphi_\delta(t,x,\eta))+\langle Q(x,\xi)\xi,\xi\rangle+r_\delta(x,\xi+\nabla_x \varphi_\delta(t,x,\eta)),\end{equation}
where $Q=Q(\cdot;t,\eta)$ is given by
\[Q(x,\xi;t,\eta)=\int_0^1 (1-s)\nabla_\xi^2 p_\delta\big(x,(1-s)\nabla_x \varphi_\delta(t,x,\eta)+s\xi\big) \mathrm ds,\]
and satisfies
\begin{equation} \label{eq: bound Q} \|Q\|_{C^\ell_{t,x,\xi,\eta}}\lesssim h^{-\beta \ell}.\end{equation}
Similarly, the function $r_\delta(x,\xi+\nabla_x \varphi_\delta(t,x,\eta))$ is estimated by
\[\|r_\delta(x,\xi+\nabla_x \varphi_\delta(t,x,\eta))\|_{C^\ell}\lesssim h^2h^{-\beta(\ell+2)}.\]
We can thus incorporate the quantity $h^{-2}r_\delta(x,\xi+\nabla_x\varphi_\delta)$ in the term $r_2$.
At this point, we have
\[\mathrm{e}^{-\frac{\mathrm i}{h}\varphi_\delta}(hD_t+P_h) \mathrm{e}^{\frac{\mathrm{i}}{h}\varphi_\delta}=\frac{h}{\mathrm{i}}\partial_t+\mathrm{Op}_h^{\rm w}(L(x,\xi;t,\eta)))+ \mathrm{Op}_h^{\rm w}(\langle Q\xi,\xi\rangle)+h^2\mathrm{Op}_h^{\rm w}(r_2),\]
with $L(x,\xi;t,\eta):=\langle \xi,\nabla_\xi p_\delta(x,\nabla_x \varphi_\delta(t,x,\eta))\rangle$. This last term is linear in $\xi$, thus its Weyl quantization is given by \cite[Theorem 4.5]{zworski2012semiclassical}:
\begin{equation}
	\begin{aligned}
		\mathrm{Op}_h^{\rm w}(L)u(x)& =\frac 12\sum_{j=1}^n\left(hD_{x_j}(\partial_{\xi_j}p_\delta(x,\nabla_x \varphi_\delta) u)+(\partial_{\xi_j}p_\delta(x,\nabla_x\varphi_\delta)) hD_{x_j}u\right) \\
		& = h\Big(V+\frac 12\operatorname{div}V\Big) u,
	\end{aligned}
\end{equation}
where $V$ has been introduced in \eqref{eq: define V}. Lastly, thanks to the quadratic term $\xi^2$ by performing two integrations by parts and using the estimate \eqref{eq: bound Q}, we get that the operator
\[h^{-2}\mathrm{Op}_h^{\rm w}(\langle Q\xi,\xi\rangle)\]
is of the same form as $R$ in \eqref{eq: R is of the form}. This concludes the proof.
\end{proof}

\subsubsection{The sequence of transport equations} Recall that we are looking for a FIO representation of the propagator $\mathrm{e}^{-\frac{\mathrm{i}}{h} tP_h}\approx\mathcal I[\varphi_\delta,a_h]$. Recalling Proposition \ref{prop: form of the differential operator}, we impose
\begin{equation} \label{eq: equation on ah}\Big(\partial_t+V+\frac{1}{2}\operatorname{div}(V)+hR\Big)a_h=0.\end{equation}
Formally, we write $a_h=\sum_{j=0}^{\infty}h^j a_j$ then plug this series in \eqref{eq: equation on ah}. Identifying coefficients of equal powers of $h$, we find that the sequence $(a_j)_{j\ge 0}$ must solve
\begin{equation}\label{eq: sequence of transport equations amplitude} \left\{\begin{array}{ll} \big(\partial_t +V+\frac{1}{2}\mathrm{div}(V)\big)a_0(t,x;\eta)=0, \vspace{0.5em} \\ 
	\left(\partial_t +V+\frac{1}{2}\mathrm{div}(V)\right)a_{j+1}(t,x;\eta)=-Ra_j(t,x;\eta). \end{array}\right.\end{equation}
	Recall that $R$ is a pseudodifferential operator depending on $t$ and $\eta$ acting on functions in the $x$ variable. It is more natural to formulate the first equation geometrically in the language of half-densities. Given $\eta\in \mathbf R^d$, consider the horizontal Lagrangian $\Lambda_\eta=\{(x,\eta)~ | ~ x\in \mathbf R^d\}=\mathbf{R}^d\times \{\eta\}$. For short times, the evolved Lagrangian $\Lambda_\eta(t):=\Phi^t_\delta(\Lambda_\eta)\subset \mathbf{R}^{2d}$ projects diffeomorphically on the $x$ coordinates. The function $a_h(x;t,\eta)$ can then be seen as a function on $\Lambda_\eta(t)$, where $\Lambda_\eta(t)$ is parameterized by the $x$-coordinates. For $s<t$, we introduce the flow generated by the time-dependent vector field $V(x;t,\eta)$:
	\[\phi_{s,t}:\left\{ \begin{array}{ll}\pi_x\Lambda_\eta(s)\to \pi_x\Lambda_\eta(t) \\ 
x \mapsto \pi_x\big(\Phi^t(\nabla_\eta \varphi_\delta(s,x,\eta),\eta)\big)\end{array}\right..\]
The first equation in \eqref{eq: sequence of transport equations amplitude} then expresses the fact that the half-density $a_0(x;t,\eta) |\mathrm{d}x|^{\frac 12}$ on $\Lambda_\eta(t)$ is invariant under this flow. Indeed, one has
\[\big(\partial_t +V+\frac{1}{2}\mathrm{div}(V)\big) a_0(x;t,\eta)  |\mathrm{d}x|^{\frac 12}=(\partial_t +\mathcal L_{V})\big(a_0(x;t,\eta)|\mathrm{d}x|^{\frac 12}\big),\]
where $\mathcal L_{V}$ denotes the Lie derivative of the vector field $V(x;t,\eta)$ acting on half-densities.

We record the following estimate on the flow $\phi_{s,t}$, which follows from Proposition \ref{prop: short time derivative flow} and Lemma \ref{lem: spacetime derivatives generating function}.
\begin{lem}\label{lem ; estimate phist} For all $\ell\ge 0$, one has for $|t|<t_0$,
\begin{equation}\label{borne phist}\|\phi_{s,t}\|_{C^\ell_{s,t,x,\eta}}\lesssim h^{-\beta(\ell-1)_+}.
\end{equation}\end{lem}
\subsubsection{Solving the transport equations} We now write $a_0(t,x,\eta)$ instead of $a_0(x;t,\eta)$. At time $0$, we fix $a_0(0,x,\eta)\equiv 1$. Then, $a_0$ is given by
\[a_0(t,x,\eta)|\mathrm{d}x|^{\frac 12}=(\phi_{0,t})_* |\mathrm{d}x|^{\frac 12}= |\det \mathrm{d}_x\phi_{0,t}^{-1}|^{\frac 12} |\mathrm{d}x|^{\frac 12},\]
Since $\phi_{0,t}^{-1}(x,\eta)=(\nabla_\eta \varphi_\delta(t,x,\eta),\eta)$, we infer
\[\mathrm{d}_x\phi_{0,t}^{-1}(x,\eta)=\frac{\partial^2\varphi_\delta}{\partial x\partial \eta}(t,x,\eta),\]
and one gets eventually
\[a_0(t,x,\eta)=\big|\partial_{x,\eta}^2 \varphi_\delta\big|^{\frac 12}.\]
The higher order transport equations \eqref{eq: sequence of transport equations amplitude} are solved using Duhamel's formula. Setting $a_j(0,x,\eta)\equiv 0$ for all $j\ge 1$, the solution is given by
\begin{equation} \label{eq: defi aj+1} a_{j+1}(t,x,\eta):=-\int_0^t (R_sa_j(s,\cdot))\circ \phi_{s,t}^{-1}(x;\eta) |\det \mathrm d_x\phi_{s,t}(x;\eta)|^{-\frac 12}\mathrm ds.\end{equation}
The symbols $a_j$ are estimated in the following proposition.
\begin{prop}\label{short time estimate aj} For all $j$, we have, uniformly in $|t|\le t_0$:
\begin{equation}\label{eq: borne bj short time} \forall \ell\ge 0, \quad \|a_j\|_{C^\ell_{t,x,\eta}}\lesssim  h^{-\beta(\ell+2j)}.
\end{equation}
Moreover, we have improved estimates at $t=0$:
\begin{equation}\label{eq: bornes aj short time} \forall n\ge 1, \quad \|\partial_t^n a_j(t=0)\|_{C^\ell_{x,\eta}}\lesssim h^{-\beta(\ell+n+2j-1)}.\end{equation}\end{prop}
The proof relies on the following lemma.
\begin{lem}\label{lem: pseudo tape sur fonction lisse}

There is a constant $C$ such that for $N$ large enough (depending on the dimension $d$), for any $r\in C^\infty(T^*\mathbf{R}^d)$ with uniformly bounded derivatives, for all $u\in C^\infty(\mathbf{R}^d)$, one has 
\[\|\mathrm{Op}_h^{\rm w}(r)u\|_{C^\ell_x}\lesssim \sum_{k=0}^N h^k\Big(\sum_{i+j=2k+\ell } \|r\|_{C^{i}_{x,\xi}} \|u\|_{C^j_x}\Big)
+h^{N}\sum_{i+j\le 2N+\ell+C} \|r\|_{C^i_{x,\xi}}\|u\|_{C^j_x}.\]
\end{lem}
\begin{proof}We explain the case $\ell=0$. Start from
\[\mathrm{Op}_h^{\rm w}(r)u(x)=\frac{1}{(2\pi h)^d}\int_{\mathbf{R}^{2d}} \mathrm{e}^{\frac{\mathrm{i}}{h}\langle x-y,\xi\rangle}r\big(\frac{x+y}{2},\xi\big)u(y)\mathrm{d}y\mathrm{d}\xi.\]
We split the integral into two parts, one of which localized near $(y-x,\xi)\equiv 0$ and the other one on $|(y-x)|+|\xi|\ge 1$. The first part is estimated by a stationary phase expansion. The second part is dealt with by a non-stationary phase estimate, using the fact that $(|y-x|+|\xi|)^{-N}$ is integrable on $\{(y,\xi)~:~|(y-x)|+|\xi|\ge 1\}\subset \mathbf{R}^{2d}$ when $N>2d$. The case of higher derivatives is dealt with similarly, up to performing some additional integrations by parts in the $y$ variable.  \end{proof}
We now turn to the proof of Proposition \ref{short time estimate aj}.
\begin{proof}[Proof of Proposition \ref{short time estimate aj}]The estimates on $\|a_0\|_{C^\ell}$ have already been obtained in Proposition \ref{prop: derivees a0}.
\textbf{1.} By induction, assume that \eqref{eq: borne bj short time} holds up to index $j$. Then, by Lemma \ref{lem: pseudo tape sur fonction lisse} and \eqref{eq: estimates on ri}, one has
\begin{equation}\label{borne Raj} \|R_sa_{j}(s,\cdot)\|_{C^\ell_{s,x,\eta}}\lesssim h^{-\beta(\ell+2j+2)}
.\end{equation}
In turn, using Fa\`a di Bruno's formula and combining \eqref{borne Raj} with \eqref{borne phist}, one finds that for any $\ell\ge 0$, 
\begin{equation}\label{borne Raj fois phist}\big\|R_sa_j(s,\cdot)\circ \phi_{s,t}^{-1}\big\|_{C^\ell_{s,t,x,\eta}}\lesssim h^{-\beta(\ell+2j+2)}.\end{equation}
Combining \eqref{borne Raj} with \eqref{borne phist}, by letting
\[h_j(t,s,x,\eta):=(R_sa_j(s,\cdot))\circ \phi_{s,t}^{-1}(x;\eta) |\det \mathrm d_x\phi_{s,t}(x;\eta)|^{-\frac 12},\]
we find
\begin{equation} \label{eq: bound h_j(s,t,x,eta)} \big\|h_j\big\|_{C^\ell_{s,t,x,\eta}}\lesssim h^{-\beta(\ell+2j+2)}.\end{equation}
By the definition \eqref{eq: defi aj+1} of $a_{j+1}$, we have
\begin{equation} \label{eq: aj+1 in terms of hj} a_{j+1}(t,x,\eta)=\int_0^t h_j(t,s,x,\eta)\mathrm{d}s,\end{equation}
thus we can express derivatives of $a_{j+1}$ in terms of derivatives of $h$, and find for $|t|\le t_0$:
\[\|a_{j+1}\|_{C^\ell_{t,x,\eta}}\lesssim \|h_j\|_{C^{\ell}_{t,s,x\eta}}\lesssim h^{-\beta(\ell+2(j+1))}.\]
This completes the induction.

\textbf{2.} We show the improved estimates \eqref{eq: bornes aj short time} at $t=0$. The improved estimates for the derivatives of $\varphi_\delta$ at $t=0$ from Lemma \ref{lem: spacetime derivatives generating function} yield
\[\|\partial_t^n a_0(0,x,\eta)\|_{C^\ell_{x,\eta}}\lesssim h^{-\beta(\ell+n-1)_+}.\]
For the higher-order symbols, we use \eqref{eq: aj+1 in terms of hj} to obtain, for any $n\ge 1$,
\[\partial_t^n a_{j+1}(0,x,\eta)=\sum_{i+j\le n-1} \partial_t^i\partial_s^j h(t,s,x,\eta)\Big|_{t=s=0},\]
it then follows from \eqref{eq: bound h_j(s,t,x,eta)} that
\[\|\partial_t^n a_{j+1}(0,x,\eta)\|_{C^\ell_{x,\eta}}\le \|h\|_{C^{n-1+\ell}}\lesssim h^{-\beta(\ell+n+2j+2-1)}=h^{-\beta(\ell+n+2(j+1)-1)},\]
which finishes the proof.
\end{proof}

\subsubsection{Controlling the remainder} We let $a_h=\sum_{j\le N} h^ja_j$ denote the symbol truncated at order $N$, and let $U^N(t):=\mathcal I[\varphi_\delta,a_h]$. To control the difference between $U^N(t)$ and the true propagator $U(t)=\mathrm{e}^{-\frac{\mathrm{i}}{h} tP_{h}}$ 
we will use the following lemma, which follows from the unitarity of the group $\mathrm{e}^{-\frac{\mathrm{i}}{h} tP_{h}}$ together with Duhamel's formula.
\begin{lem}Assume $(hD_t+P_h)\mathcal R(t)=\mathcal Z(t)$. Then,
\[\|\mathcal R(t)\|\le \|\mathcal R(0)\|+\frac{1}{h}\int_0^t \|\mathcal Z(s)\|\mathrm{d}s,\]
where the norms are either the operator norm on $L^2$ or the trace-class norm.
\end{lem}

\begin{prop}Set $\mathcal R^N(t)=U(t)-U^N(t)$. Then, $\|\mathcal R^N(0)\|=0$ and $(hD_t+P_h)\mathcal R^N(t)=\mathcal Z^N(t)$ with $\|\mathcal Z^N(t)\|_{L^2}\lesssim h^{N(1-2\beta)-C}$. Thus,
\[\mathcal R^N(t)=\mathcal O(h^{N(1-2\beta)-C})_{L^2}.\]
\end{prop}
\begin{proof}We have $U^N(0)=\mathcal I[\langle x,\eta\rangle,1]=\operatorname{Id}$ thus $\mathcal R^N(0)$ vanishes. By Proposition \ref{prop: form of the differential operator} and \eqref{eq: defi aj+1}, one has
\[(hD_t+P_h)U^N(t)=h^{N+1}\mathcal I[\varphi_\delta, Ra_N].\]
We split the two contributions. Let $\psi$ be a smooth compactly supported function equal to $1$ in a neighborhood of $(t,x,\eta)=(0,0,0)$, such that $\varphi_\delta(t,x,\eta)\equiv 1$ outside the support of $\psi$. Then write
\[\mathcal I[\varphi_\delta, Ra_N]=\mathcal I[\varphi_\delta, \psi Ra_N]+\mathcal I[\langle x,\eta\rangle, (1-\psi)Ra_N]:=\mathcal Z_1(t)+\mathcal Z_2(t).\]
\textbf{1. Estimating $\|\mathcal Z_1(t)\|$.} Denoting $\mathcal F_h(v)$ the $h$-Fourier transform of a function $v$, we have
\[\mathcal Z_1(t)v(x)=\frac{h^{N+1}}{(2\pi h)^d}\int \mathrm{e}^{\frac{\mathrm i}{h} \varphi_\delta(t,x,\eta)} \psi(t,x,\eta)Ra_{N} \mathcal F_h(v)(\eta)\mathrm d\eta,\]
then by the Cauchy--Schwarz inequality and the fact that $\psi$ has compact support in $x$:
\[\|\mathcal Z_1(t)v\|_{L^2}\lesssim  h^{N+1-d}\|Ra_N\|_{C^0} \|\mathcal F_h(v)\|_{L^2}\lesssim h^{N+2-2d}h^{-2\beta N} \|v\|_{L^2}.\]
Here we have used Proposition \ref{short time estimate aj} to control $\|Ra_N\|_{C^0}$. Thus, for some dimensional constant $C_d$, we have
\[\|\mathcal Z_1(t)\|_{L^2\to L^2}\lesssim h^{N(1-2\beta)-C_d}.\]

\textbf{2. Estimating $\|\mathcal Z_2(t)\|$.} Here we use the fact that $\mathcal Z_2(t)$ is a pseudodifferential operator with amplitude satisfying $\|h^{N+1}(1-\psi)Ra_N\|_{C^\ell}\lesssim h^{N+1}h^{-\beta(2N+\ell)}$. It then follows from the Calder\'on--Vaillancourt theorem that $\|\mathcal Z_2(t)\|_{L^2}=\mathcal O(h^{N(1-2\beta)-C_d})$.

\end{proof}

\section{The short time trace} \label{sec: the short time trace} We are now ready to compute the trace near $0$. We are in the setting of Theorem \ref{prop: long time gutz} : $\{p_\delta\}_{\delta\in[0,h^\alpha]}$ is a $\delta$-dependent family of symbols such that $p_0\in S^m(T^*M)$ is elliptic and $\partial_\delta^n p_\delta\in S^m_{\beta}(T^*M)$ for all $n\ge 1$. We work in a regime $2\beta<\alpha$ and $h\ll 1$. Let $P_{h,\delta}=\operatorname{Op}_h^{M}(p_\delta)$.

\begin{prop}\label{prop: short time trace} Let $f$ be a test function with smooth compactly supported Fourier transform in a sufficiently small neighborhood of $0$. Let $\chi_E\in S^0_{\nu}(T^*M)$ be a cutoff in energy, supported in $p_0^{-1}(E-2h^\nu,E+2h^\nu)$ and equal to $1$ in $p_0^{-1}([E-h^\nu,E+h^\nu])$. Following Proposition \ref{prop: energy cutoff}, we assume $\nu<\min(\frac 12,\alpha)$. Let $\sigma=\max(\beta,\nu)$. Then
\[\operatorname{Tr} f\Big(\frac{P_{h,\delta}-E}{h}\Big) \mathrm{Op}^M_h(\chi_E)=\frac{1}{(2\pi h)^{d-1}}\sum_{k=0}^N \mu_{k,\delta} \widehat f^{(k)}(0) +\mathcal O_N(h^{(1-2\sigma)N-C_d}\|\widehat f\|_{C^{2N+C_d}}),\]
for some dimensional constant $C_d$. The coefficients $\mu_{k,\delta}$ depend on $p_\delta$ and are controlled by 
\[|\mu_{k,\delta}|\lesssim h^{\lfloor \tfrac{k+1}{2}\rfloor}.\]
\end{prop}
\begin{rem}The proof also yields an expansion in powers of $h$, but we prefer to highlight the expansion in derivatives of $\widehat f$ evaluated at $0$.
\end{rem}
\begin{rem}The cutoff function $\chi_E$ is designed to ensure that the energy shell $p_\delta^{-1}(E)$ is contained in $\{\chi_E=1\}$ for all $\delta\in [0,h^\alpha]$. Therefore, derivatives of $\chi_E$ will not contribute to the stationary phase expansion below, but only to the remainder terms.
\end{rem}

We will use in the proof a quantitative version of the implicit function theorem that can be found in a note of Liverani \cite{liverani2016implicit}. 

\begin{lem}[Quantitative implicit function theorem, \cite{liverani2016implicit}]\label{lem: implicit func}Let $F(x,y)\in C^\infty(\mathbf{R}^{m+n},\mathbf{R}^m)$ be a smooth function, such that $\partial_x F(0,0)$ is invertible. Let $r_1,r_2>0$ be such that 
\[\forall x,y\in B(0,r_1)\times B(0,r_2), \qquad \|\operatorname{Id}-(\partial_x F(0,0))^{-1}\partial_x F(x,y)\big\|\le \frac 12.\]
Set
\[C_{r_1,r_2}:= \sup_{x,y\in B(0,r_1)\times B(0,r_2)} \|\partial_y F(x,y)\|, \quad M=\|\partial_x F(0,0)^{-1}\|.\]
Then, there is a function $h(y)$ defined for $y\in B(0,\frac{r_1}{2 C_{r_1,r_2}M})$ and taking values in $B(0,r_1)$ such that
\[\forall (x,y)\in B(0,r_1)\times B\big(0,\frac{r_1}{2 C_{r_1,r_2}M}\big),\quad   F(x,y)=0\Leftrightarrow x=h(y).\] \end{lem}
\begin{rem}In the above statement all the norms are the standard operator norms on matrices. The lemma in \cite{liverani2016implicit} is stated in the case $r_1=r_2$. The case where $r_1$ may differ from $r_2$ as presented here follows by a simple rescaling argument, namely by considering $F(x,\tfrac{r_2}{r_1}y)$ instead of $F$. \end{rem}
We turn to the proof of the short time trace formula.
\begin{proof}[Proof of Proposition \ref{prop: short time trace}] \textbf{1.} Let $\{\chi_{\rho_j}\}$ be a $h$-independent smooth partition of unity of a neighborhood of the energy shell $p_0^{-1}(E)$. We assume that each $\chi_{\rho_j}$ is supported in a ball $B(\rho_j,\varepsilon)$ with $\varepsilon$ small enough to ensure that for each $j$, the ball $B(\rho_j,2\varepsilon)$ is entirely supported in a single local chart. Further, we assume that $\widehat f$ is supported in a sufficiently small neighborhood of $0$ to ensure that for all $t\in \operatorname{Supp}\widehat f$ we have $\Phi_\delta^t(B(\rho_j,\varepsilon))\subset B(\rho_j,2\varepsilon)$. Setting $\chi_{E,j}=\chi_E\chi_{\rho_j}$ we decompose $\chi_E=\sum \chi_{E,j}$. Then,
\[\operatorname{Tr}\int_{\mathbf{R}} \mathrm{e}^{-\frac{\mathrm{i}}{h} tP_{h,\delta}}\mathrm{e}^{\frac{\mathrm{i}}{h} tE} \widehat f(t) \mathrm{Op}^M_h(\chi_E)\mathrm{d}t=\sum_j \operatorname{Tr} A_j,\]
where we have set
\[A_j=\int_{\mathbf{R}} \mathrm{e}^{-\frac{\mathrm{i}}{h} tP_{h,\delta}} \mathrm{Op}^M_h(\chi_{E,j})\mathrm{e}^{\frac{\mathrm{i}}{h} tE} \widehat f(t) \mathrm{d}t.\]
We compute the trace of each $A_j$ separately. Fix an index $j$, and consider a chart $(U,f,\chi')$ such that $\chi_{\rho_j}\prec \chi'$. By applying the Egorov theorem from Proposition \ref{prop: egorov exotic}, because the flow remains in the chart for short times, if $\widehat f(t)$ is supported on a sufficiently small neighborhood of $0$, we can harmlessly insert cutoffs between the different factors of the product and obtain
\[\operatorname{Tr} A_j=\operatorname{Tr} \int_{\mathbf{R}} \chi' \mathrm{e}^{-\frac{\mathrm{i}}{h} tP_{h,\delta}}(\chi')^2 \mathrm{Op}^M_h(\chi_{E,j})\chi' \mathrm{e}^{\frac{\mathrm{i}}{h} tE} \widehat f(t) \mathrm{d}t+\mathcal O(h^\infty).\]
We can now look at the operators in charts. Let 
\[V_{j}(t):=f_*\chi'U(t)f^*\chi', \qquad \mathrm{Op}_h^{\rm w}(\overline{\chi}_{E,j}):=f_*\chi'\mathrm{Op}_h^M(\chi_{E,j})f^*\chi',\]
then the trace above equals
\begin{equation} \label{eq: trace above equals}\operatorname{Tr}A_j= \operatorname{Tr} \int_{\mathbf{R}} V_j(t) \mathrm{Op}^{\rm w}_h(\overline{\chi}_{E,j})\mathrm{e}^{\frac{\mathrm{i}}{h} tE} \widehat f(t) \mathrm{d}t+\mathcal O(h^\infty). \end{equation}
We compute each term of the sum \eqref{eq: trace above equals} separately.

\textbf{2.} Rather than studying $V_j(t)$ directly, we will work in coordinates in which the unperturbed flow acts by translation along the $x_1$ coordinate. To do so, choose a local symplectomorphism $\kappa:(x,\xi)\mapsto (\tilde x,\tilde \xi)$ such that $\kappa_*p_0=\tilde \xi_1$ locally---such a chart exists by Darboux's theorem because $\mathrm{d}p_0$ is nonvanishing on $\Sigma_E$, see \cite[Theorem 21.1.6]{HorIII}. In these new coordinates, for short times, the unperturbed dynamics are generated by $\varphi_0(t,\tilde x,\tilde \eta)=t\tilde \eta_1+\langle \tilde x,\tilde \eta\rangle$.

On the quantum side, the change of variables induced by $\kappa$ corresponds to conjugating $V_j(t)$ by a unitary Fourier integral operator $\mathcal T_\kappa:L^2(\mathbf{R}^d)\to L^2(\mathbf{R}^d)$ that local quantizes $\kappa$, see Proposition \ref{prop: egorov bis} and the discussion after it. Define
\[W_j(t):=\mathcal T_\kappa\circ V_{j}(t)\circ \mathcal T_\kappa^{*},\]
and note that
\[\mathcal T_\kappa\circ \mathrm{Op}_h^{\rm w}(\overline{\chi}_{E,j})\circ \mathcal T_\kappa^{*}=\mathrm{Op}_h^{\rm w}(\widetilde{\chi}_{E,j})+\mathcal O(h^\infty), \quad \widetilde \chi_{E,j}=\overline{\chi}_{E,j}\circ \kappa^{-1}+\mathcal O(h^{1-2\sigma}).\]
Then,
\[\operatorname{Tr}V_j(t)\mathrm{Op}_h^{\rm w}(\overline{\chi}_{E,j})=\operatorname{Tr} W_j(t)\mathrm{Op}_h^{\rm w}(\widetilde{\chi}_{E,j})+\mathcal O(h^\infty).\]
By Proposition \ref{prop: egorov bis}, we have for short times:
\[hD_t W_j(t)\mathrm{Op}_h^{\rm w}(\widetilde{\chi}_{E,j})=\mathrm{Op}_h^{\rm w}(\tilde p_j) W_j(t)\mathrm{Op}_h^{\rm w}(\widetilde{\chi}_{E,j})+\mathcal O(h^\infty),\]
for some symbol $\tilde p_j$ given by
\[\tilde p_j=\big((\chi')^2p_\delta\circ \widetilde f^{-1}\big)\circ \kappa^{-1}+r_j,\]
where $r_j$ corresponds to the standard $\mathcal O(h^2)$ remainder in the Weyl calculus and satisfies $\|r_j\|_{C^\ell}\lesssim h^2h^{-\beta(\ell+2)}$. At $t=0$ we have by the composition formula
\[W_j(0)\mathrm{Op}_h^{\rm w}(\widetilde{\chi}_{E,j})=((\chi')^2\circ \kappa) \mathrm{Op}_h^{\rm w}(\widetilde{\chi}_{E,j})+\mathcal O(h^\infty)=\mathrm{Op}_h^{\rm w}(\widetilde{\chi}_{E,j})+\mathcal O(h^\infty).\]

The assumptions of Proposition \ref{prop: short time wkb} are verified, we thus obtain a FIO parametrix
\[W_j(t)\mathrm{Op}_h^{\rm w}(\overline{\chi}_{E,j})=W_j^N(t)+\mathcal R^N_j(t)\]
where\[W_j^N(t)u(x)=\frac{1}{(2\pi h)^d}\int \mathrm{e}^{\frac{\mathrm{i}}{h} (\varphi_\delta (t,x,\eta)-\langle x,\eta\rangle)}a_h^\delta(t,x,y,\eta)u(y)\mathrm{d}y\mathrm{d}\eta,\]
and 
\begin{itemize}
\item $\varphi_\delta$ is a generating function for the perturbed flow in the new coordinates.
\item $a_h^\delta=\sum_{j\le N} h^j a_j^\delta$, with $a_0^\delta=\chi_{E,j}(y,\eta)|\partial_{x\eta}^2 \varphi_\delta(t,x,\eta)|^{\frac 12}$ (here we view $\chi_{E,j}$ as a function in charts). Moreover, we have the estimates $\|a_j^\delta\|_{C^\ell_{t,x,y,\eta}}\lesssim h^{-\beta(\ell+2j)}$, which improve at $t=0$ to
\[\forall n\ge 1, \quad \|\partial_t^n a_j^\delta(t=0)\|_{C^\ell_{x,\eta}}\lesssim h^{-\beta(\ell+n+2j-1)}.\]
\item The remainder is controlled by $\mathcal R_j^N(t)=\mathcal O(h^{N(1-2\beta)-C})_{L^2/\rm tr}.$
\end{itemize}\bigskip
\textbf{3.} Denote by $K(x,y)$ the kernel of the operator
\[\int_{\mathbf R} W_j^N(t) \mathrm{e}^{\frac{\mathrm{i}}{h}tE}\widehat{f}(t)\mathrm{d}t.\]
Then, 
\[\operatorname{Tr} A_j=\int_{\mathbf{R}^d} K(x,x)\mathrm{d}x+\mathcal O(h^{N(1-2\sigma)-C}).\]
On the diagonal, the kernel is given by
\begin{equation} \label{eq: Bj(x,x)} K(x,x)= \frac{1}{(2\pi h)^d}\iint \mathrm{e}^{\frac{\mathrm{i}}{h} (\varphi_\delta(t,x,\xi)-\langle x,\xi\rangle+tE)} a_h^\delta(t,x,x,\xi) \widehat f(t) \mathrm{d}\xi\mathrm{d}t\end{equation}
Let $\Psi(\cdot;\delta)$ and $u_h(\cdot;\delta)$ denote the phase and amplitude of the oscillatory integral \eqref{eq: Bj(x,x)}. To compute the trace, we isolate the $(t,\xi_1)$ variables. Writing $\xi=(\xi_1,\xi')\in \mathbf{R}\times \mathbf{R}^{d-1}$, one has
\begin{equation} \label{eq: tr aj equals}\int K(x,x)\mathrm{d}x=\frac{1}{(2\pi h)^d} \int_{x, \xi'} \Big(\int_{t,\xi_1}\mathrm{e}^{\frac{\mathrm{i}}{h} \Psi(t, x,\xi;\delta)} u_h(t, x, \xi;\delta)\mathrm{d}t\mathrm{d} \xi_1\Big) \mathrm{d}x\mathrm{d}\xi'.\end{equation}
We fix $(x,\xi')$ and perform a stationary phase in the $z=(t,\xi_1)$ variables. \medskip

\noindent \textbf{Critical points of the phase.} Let us set $F(z,\delta)=\nabla_{z}\Psi$. When $\delta=0$, one has $\Psi( z,\delta)=t(E-\xi_1)$ thus $F(z,0)=(E-\xi_1,-t)$ and there is a single critical point $z_0=(t^0,\xi^0_1)=(0,E)$, which does not depend on $(x,\xi')$. We now check that the assumptions of Lemma \ref{lem: implicit func} are verified. Set $r_1=ch^{\beta}$ and $r_2=ch^{2\beta}$ for some small constant $c$. Then,
\begin{itemize}
\item Setting $z_0=(0,E)$, we have $\partial_z F(z_0,0)=\begin{pmatrix}
0 & -1 \\ -1 & 0
\end{pmatrix}$.
\item For $(z,\delta)\in B(z_0,r_1)\times B(0,r_2)$, we have
\begin{equation} \begin{aligned}
	\big\|\operatorname{Id}-\big(\partial_z F(0,0)\big)^{-1} \partial_z F(z,\delta)\big\|& \lesssim \|\partial_z F(0,0)-\partial_z F(z,\delta)\|\le \|\partial_z^2 F\|r_1+ \|\partial_\delta \partial_z F\| r_2\\ & \lesssim h^{-\beta} ch^{\beta}+h^{-2\beta} ch^{2\beta}\lesssim c.
	\end{aligned}\end{equation}
	By taking $c$ small enough, we find
	\[	\big\|\operatorname{Id}-\big(\partial_z F(0,0)\big)^{-1} \partial_z F(z,\delta)\big\|\le \frac 12.\]
	\item One has $C_{r_1,r_2}\lesssim \|\partial_\delta \varphi_\delta\|_{C^1}\lesssim h^{-\beta}$ and $M=1$. 
\end{itemize}
By Lemma \ref{lem: implicit func}, because $\frac{r_1}{2 C_{r_1,r_2}M}\ge c' h^{2\beta}$ for some other constant $c'$, we deduce that for all $|\delta|\le c'h^{2\beta}$, there is a unique point $z_\delta=(t^\delta,\xi_1^\delta)\in B(z_0, ch^\beta)$ such that $F(z_\delta,\delta)=0$. In particular this is true for $\delta\in [0,h^\alpha]$ and $h$ small enough because $2\beta<\alpha$. The derivative of $z_\delta$ with respect to $\delta$ is given by
\[\partial_\delta z_\delta=-\big(\nabla^2_{t, \xi_1}\Psi(z_\delta,\delta)\big)^{-1}\partial_\delta\nabla_{t,\xi_1} \Psi( z_\delta,\delta)=\mathcal O(h^{-\beta}),\]
where we have used Proposition \ref{prop: control all deriv gen fun} to control derivatives of the generating function. In fact, one can show that $t^\delta=0$ for all $\delta$. Indeed, the phase is given by 
\[\Psi(t, x, \xi;\delta)= \varphi_\delta(t, x, \xi)-\langle  x, \xi\rangle +tE,\]
where $\varphi_\delta$ is a generating function for the perturbed flow in the new coordinates, with the initial condition $ \varphi_\delta(0, x, \xi)=\langle  x, \xi\rangle$. For fixed $(x,\xi')$, one has $\partial_{\xi_1}\Psi_{|t=0}=0$. Moreover, 
\begin{equation} \label{eq: t derivative phase at 0} \partial_t \Psi(t=0)=-p_\delta(x,\xi)+E=-\xi_1+E +\mathcal O(\delta),\end{equation}
which vanishes at a single point $ \xi_1(x,\xi';\delta)$, by the implicit function theorem. This is straightforward because the map $ \xi_1\mapsto  p_\delta(x,\xi_1,\xi')$ is strictly increasing. The critical point $z_\delta=(t^\delta,\xi^\delta)$ being unique, this forces $t^\delta=0$.

\subsection{Stationary phase estimate.}	We now fix $\delta\in [0,h^\alpha]$. In the following we will use implicitly the fact that the function $\chi_E$ is equal to $1$ near the critical point $z_\delta$. In particular, derivatives of $\chi_E$ do not appear in the stationary phase expansion below---although they appear in the error term. The Hessian of the phase at the critical point $z_\delta$ is given by
\begin{equation} \label{eq: hessian at crit point} \nabla^2_{t, \xi_1}\Psi(z_\delta,\delta)=\begin{pmatrix}\langle \nabla_\xi p_\delta,\nabla_x p_\delta\rangle & -\partial_{ \xi_1} p_\delta \\ -\partial_{ \xi_1} p_\delta &0\end{pmatrix}
. \end{equation}
In particular, $\det \nabla_z^2 \Psi(z_\delta,\delta)=-|\partial_{\xi_1} p_\delta(x,\xi_1^\delta,\xi')|^2$. 
We have the following estimates on the phase and the amplitude. Recall that $\sigma=\max(\beta,\nu)$. First, $\|\Psi\|_{ C^\ell}\lesssim h^{-\sigma(\ell-2)_+}$ and  $\|u_h\|_{C^\ell}\lesssim h^{-\sigma \ell}$ by the assumption on $\chi_E$. For any $z\in B(0,1)$, we have by Taylor's formula
\[\nabla_z\Psi(z,\delta)=\int_0^1 \nabla^2_z\Psi(sz+(1-s)z_\delta,\delta) (z-z_\delta)\mathrm{d}s.\]
By Proposition \ref{prop: control all deriv gen fun}, we have 
\[\nabla_z^2\Psi(sz+(1-s)z_\delta,\delta)=\nabla_z^2\Psi(sz+(1-s)z_\delta,0)+\mathcal O(\delta h^{-2\beta})=\begin{pmatrix}
0 & -1 \\ -1 & 0
\end{pmatrix}+\mathcal O(\delta h^{-2\beta}).\] Since $\delta h^{-2\beta}\ll 1$, this implies
\[\forall z\in B(0,1), \quad |\nabla_z \Psi(z,\delta)|\ge \frac 12 |z-z_\delta|.\]
In particular, $z_\delta$ is the unique critical point of the phase in $B(0,1)$. Note that the phase vanishes at the critical point. We can now apply Proposition \ref{prop: stat phase bis} to obtain
\begin{equation} \label{eq: we have applied prop stat phase}\int_{t, \xi_1}\mathrm{e}^{\frac{\mathrm{i}}{h} \Psi(t, x, \xi;\delta)} u_h(t, x, \xi;\delta)\mathrm{d}t\mathrm{d} \xi_1=2\pi h\frac{1}{|\det(\nabla^2 \Psi(z_\delta))|^{\frac 12}}\sum_{j=0}^{N-1} h^j\mathcal L_{j,h}u_h\Big|_{z=z_\delta}+\mathcal O(h^{N(1-2\sigma)+1}).\end{equation}
Here $\mathcal L_{0,h}=1$, and more generally $\mathcal L_{j,h}u_h\big|_{z=z_\delta}$ takes the form
\[\mathcal L_{j,h}u_h\big|_{z=z_\delta}=\sum_{|\alpha|\le 2j} c_{j,\alpha} \partial^{\alpha}u_h(z_\delta).\]
When $\delta=0$ we simply have $\mathcal L_{j,h}=(\partial_t\cdot\partial_{\xi_1})^j$. By repeating the proof of Theorem \ref{thm: composition FIO}, we see that  
\begin{equation} \label{eq: bornes cj alpha stat phase 0} \left\{\begin{array}{ll} |c_{j,\alpha}|\lesssim 1 & \text{when $|\alpha|=2j$,} \\ |c_{j,\alpha}|\lesssim h^{-\beta(2j-1-|\alpha|)} & \text{when  $|\alpha|<2j$.} \end{array}\right.\end{equation}
Note that there is an improvement compared to Theorem \ref{thm: composition FIO} due to the improved estimates on the derivatives of the phase function at $t=0$. Let us explain how the proof can be modified accordingly. Using $|\partial_t^n\partial_{x,\eta}^\ell \varphi(t,x,\eta)_{|t=0}|\lesssim h^{-\beta(\ell+n-3)_+}$ (compared to the bound $h^{-\beta(\ell+n-2)_+}$ for $t\neq 0$ in general), \eqref{eq: Cell norm Rgamma}-\eqref{eq: Cell norm R bold gamma} become
\[\|R_\gamma\|_{C^\ell_{t=0}}\lesssim h^{-\beta \ell},\qquad \|R_{\bm \gamma}\|_{C^\ell_{t=0}} \lesssim h^{-\beta \ell}.\]
Note that for $\bm \gamma=0$ we have $R_{\bm \gamma}=1$. By the formula \eqref{eq: formula for c j alpha} for $c_{j,\alpha}$, rewritten in a slightly different form, we have
\[|c_{j,\alpha}|\le \sum_{\substack{m\le 2j-|\alpha|}} \sum_{|\bm \gamma|=3m\le k_1} \|R_{\bm \gamma}\|_{\dot C^{2j-|\alpha|-m}},\]
where $\|\cdot\|_{\dot C^\ell}$ denotes the homogeneous $C^\ell$ norm. If $|\alpha|<2j$, the term $m=0$ does not contribute to the sum, because $R_0=1$ and thus $\|R_0\|_{\dot C^{2j-|\alpha|}}=0$. This implies \eqref{eq: bornes cj alpha stat phase 0}.

\medskip

\textbf{4.} Recall that the amplitude is given by $u_h=a_h \widehat f(t)$, and that at $z=z_\delta$ we have we have $\|a_h\|_{C^\ell}\lesssim h^{-\beta(\ell-1)_+}$. By Leibniz formula, one can write
\begin{equation} \label{eq: by Leibniz L j h u h}\mathcal L_{j,h}u_h(z_\delta)=(\mathcal L_{j,h}a_h)(z_\delta) \widehat f(0)+\sum_{1\le |\alpha|\le 2j} c_{j,\alpha}\sum_{\alpha_1+\alpha_2=\alpha, \ \alpha_2>0} \partial^{\alpha_1}a_h(z_\delta)\cdot (\partial^{\alpha_2}\widehat f(t))_{|t=0}.\end{equation}
From \eqref{eq: bornes cj alpha stat phase 0} and the bounds on $a_h$, we have for all $|\alpha_2|\ge 1$:
\[|c_{j,\alpha}\partial^{\alpha_1}a_h|\lesssim h^{-\beta(2j-|\alpha_2|-1)_+}.\]
Summing over all $j\le N$ we obtain
\[\sum_{2j\ge |\alpha_2|}h^j |c_{j,\alpha}\partial^{\alpha_1}a_h|\le \sum_{2j\ge |\alpha_2|} h^j h^{-\beta(2j-|\alpha_2|-1)_+}.\]
By summing all the traces \eqref{eq: trace above equals}, it follows that the coefficients $\mu_{k,\delta}$ in Proposition \ref{prop: short time trace} are bounded by
\[|\mu_{k,\delta}|\le \sum_{2j\ge k} h^j h^{-\beta(2j-k-1)_+}.\]
In the even case, we get
\[|\mu_{2k,\delta}|\le \sum_{j\ge k} h^jh^{-\beta(2j-2k-1)_+}=h^k+\sum_{j\ge k+1}h^{j(1-2\beta)}h^{\beta(2k+1)}\lesssim h^k+h^{(k+1)(1-2\beta)}h^{\beta(2k+1)}\lesssim h^k.\]
In the odd case, we have
\[|\mu_{2k+1,\delta}|\le \sum_{j\ge k+1} h^j h^{-\beta(2j-2k-2)}\le h^{(1-2\beta)(k+1)}h^{\beta(2k+2)}\le h^{k+1}.\]
So in any case $|\mu_{k,\delta}|\lesssim h^{\lfloor \frac{k+1}{2}\rfloor}$, as we wished. \end{proof}

\section{The long time propagator} \label{sec: long time propagation} In this section, we obtain a long time WKB approximation of the quantum propagator $U(t):=\mathrm{e}^{-\frac{\mathrm{i}}{h} tP_{h,\delta}}$. The first step is to perform a localization along the orbits of the classical flow. We will only work with positive times, the case of negative times is the same but one needs considering negative periods.

\subsection{Microlocalization along the orbits} We start with an application of Egorov theorem \ref{prop: egorov exotic}.

\begin{lem}\label{lem: micro along orbit lemma} Let $a\in S_\nu^{\rm comp}(T^*M)$ and assume $\sigma:=\max(\beta,\nu)<1/2$. Let $b\in S^{\nu}_{\rm comp}(T^*M)$ satisfy $b\equiv 1$ on $\Phi^1_\delta(\mathrm{Supp}(a))$ for all $|\delta|\le h^\alpha$. Then, for all $N\ge 0$, one has
\[U(1)\mathrm{Op}^M_h(a)=\mathrm{Op}^M_h(b)U(1)\mathrm{Op}^M_h(a)+\mathcal O_N(h^{N(1-2\sigma)-C})_{L^2/\rm tr}.\]
Here the error is either the $L^2$ operator norm or the trace class norm.

\end{lem}

\begin{proof} Since $U(1)$ is unitary, the result  is equivalent to
\[\mathrm{Op}^M_h(a)=U(-1)\mathrm{Op}^M_h(b)U(1)\mathrm{Op}^M_h(a)+\mathcal O(h^{N(1-2\sigma)-C})_{L^2/\rm tr},\]
but then this follows from Proposition \ref{prop: egorov exotic} and the composition formula from Proposition \ref{prop: composition manifold}. \end{proof}
By a finite induction, we obtain the following:
\begin{prop}[Microlocalization along the orbits]\label{prop: cutoffs along the orbit} Fix $\rho=(x,\xi)\in\Sigma_E$, and let $\rho_k:=\Phi^{k}_0(\rho)$. Let $\nu\in (0,1/2)$ and set $\sigma:=\max(\beta,\nu)$. Let $R_k:=\mathrm{e}^{Ck}h^\nu$, for some sufficiently large constant $C$. Let $\chi$ be a smooth cutoff such that $\chi\equiv 1$ in $[0,1]$ and $\chi\equiv 0$ outside $[0,2]$. Introduce $\chi_{\rho_k}:=\chi\big(\frac{d(\rho,\rho_k)}{R_k}\big)$. There is a constant $\mathfrak c(\alpha,\beta,\nu)$ such that for $n\le \mathfrak c|\log h|$, one has
\[U(n)\mathrm{Op}_h^M(\chi_{\rho_0})=\prod_{k=1}^n U(1)\mathrm{Op}_h^M(\chi_{\rho_k})+R_n.\]
Here the product is read from left to right and the error is controlled by
\[\big\|R_n\big\|_{L^2/\rm tr}\le (1+C_0 h^{\frac 12-\sigma})^{n-1}h^{N}.\]\end{prop}
\begin{rem}The choice $R_k:=\mathrm{e}^{Ck}h^\nu$ simply ensures that $\Phi^{1}_\delta(B(\rho_k,R_k))\subset B(\rho_{k+1},R_{k+1})$ for $k\le \mathfrak c|\log h|$ and all $\delta\in [0,h^\alpha]$.
\end{rem}
We recall before the proof that $\|\operatorname{Op}_h^{M}(\chi_{\rho_k})\|_{L^2}\le (1+C_0h^{1/2-\nu})$ and $\|\operatorname{Op}_h^{M}(\chi_{\rho_k})\|_{\rm tr}\lesssim h^{-d}$.
\begin{proof}The proof works by induction on $n$. For $n=1$ there is nothing to do. Assume the result holds up to index $n$. Then, since $U(1)$ is unitary, we have
\[U(n+1)\mathrm{Op}_h^M(\chi_{\rho_0})=U(1)\prod_{k=1}^n U(1)\mathrm{Op}_h^M(\chi_{\rho_k})+R_n,\]
with $\|R_n\|_{\rm tr}\le (1+C_0h^{\frac 12-\nu})^nh^{N}$. We rewrite the product on the right-hand side as
\[U(1) U(1)\operatorname{Op}_h^{M}(\chi_{\rho_n})\prod_{k=1}^{n-1} U(1)\operatorname{Op}_h^{M}(\chi_{\rho_k}).\]
By construction of $(R_k)_{k\ge 0}$, one has $\chi_{\rho_{n+1}}\equiv 1$ on $\Phi^{1}_\delta(\mathrm{Supp}(\chi_{\rho_n}))$. We can thus invoke Lemma \ref{lem: micro along orbit lemma} to obtain
\[U(1)\mathrm{Op}_h^M(\chi_{\rho_n})=\mathrm{Op}_h^M(\chi_{\rho_{n+1}}) U(1)\mathrm{Op}_h^M(\chi_{\rho_n})+\mathcal O_{N'}(h^{N'(1-2\sigma)-C})_{L^2/\rm tr},\]
where $N'$ will be chosen below. Eventually, one finds
\[U(n+1)\mathrm{Op}_h^M(\chi_{\rho_0})=\prod_{k=1}^{n+1} U(1)\mathrm{Op}_h(\chi_{\rho_k})+ R_{n+1},\]
with 
\[\|R_{n+1}\|\le \big(1+C_0 h^{\frac 12-\nu}\big)^{n-1}C_{N'} h^{N'(1-2\sigma)-C} +(1+C_0 h^{\frac 12-\nu})^{n-1}h^{N}.\]
Taking $N'$ large enough, one gets
\[\|R_{n+1}\|_{L^2/\rm tr} \le \big(1+C_0h^{\frac 12-\nu}\big)^{n-1}\cdot \big(C_{N'}h^{N'(1-2\sigma)-C}+h^{N}\big).\]
Now, we can take $N'$ large enough (depending on $N$) and $h<h_0(N)$ small enough to ensure $C_{N'}h^{N'(1-2\sigma)-C}+h^{N}\le (1+C_0 h^{\frac 12-\sigma})h^{N}$. This finishes the proof.
\end{proof}

\begin{rem} If $n\le \mathfrak c |\log h|$, one has
\[\|R_n\|_{L^2/\rm tr}\le (1+C_0h^{\frac 12-\nu})^{n}h^{N}\lesssim_{\!\! \mathfrak c} h^N,\]
Thus the error is controlled up to logarithmic times.\end{rem}

\subsection{Localization around periodic orbits} The following procedure is identical to \cite[\S 7.2]{Schenk} and similar to \cite[\S 3.6]{JakobsonWeyl2007}. For times $|t|\le T_{\rm max}:=\mathfrak c|\log h|$, if $\chi_\rho$ is localized on a ball of radius $\approx h^{\nu}=:\mathrm{e}^{-CT_{\rm max}}$, we will have a good representation of the propagator $U(t)\mathrm{Op}_h^M(\chi_\rho)$ as a single oscillatory integral for $|t-t_0|\le h^\nu=:\mathrm{e}^{-CT_{\rm max}}$. It is therefore natural to split the time window $[-T_{\rm max},T_{\rm max}]$ in smaller windows of width $h^\nu$. Thus, let $\omega$ be a smooth function compactly supported on $[-1,1]$, and such that
\[1\equiv\sum_{n\in \mathbf Z} \omega(t-n).\]
Let $\omega_\nu:=\omega(h^{-\nu}t)$. Then, denoting $t_n:=n h^\nu $, one has
\[1\equiv \sum_{n\in \mathbf Z} \omega_\nu(t-t_n).\]
We will compute separately, for all $n$, the trace of
\begin{equation} \label{eq: defi An}A_n:=\int \mathrm{e}^{-\frac{\mathrm{i}}{h} tP_{h,\delta}}\mathrm{e}^{\frac{\mathrm{i}}{h}tE}\omega_\nu(t-t_n)\widehat f(t/L)\mathrm{d}t.\end{equation}

For $|t_n|\ge \frac{T_0}{2}$ (where $T_0>0$ is the minimal period of a nontrivial periodic orbit), denote by $\mathcal P_n$ the set of periodic orbits $\gamma\subset \Sigma_E$ of the unperturbed flow whose periods lie in the interval $[t_n-2h^\nu,t_n+2h^\nu]$. By the separating criterion of \cite[Lemma 9]{Schenk}, such orbits are distant from each other in phase space by at least $\mathrm{e}^{-Ct_n}$. Let $\nu_1\in (0,1/2)$ and set 
\[\Omega_n:=\bigcup_{T_\gamma\in \mathcal P_n} \Theta(\gamma,h^{\nu_1}),\]
where $\Theta(\gamma,h^{\nu_1})$ denotes the $h^{\nu_1}$-neighborhood of the geodesic $\gamma$ in $T^*M$, say for the Sasaki metric. We take $\mathfrak c$ small enough to ensure that $\Theta(\gamma,h^{\nu_1})\cap \Theta(\gamma',h^{\nu_1})=\varnothing$ for $\gamma\neq \gamma'$. The complement $\Sigma_E\backslash \Omega_n$ can then be covered by balls of radius $\sim h^\nu$.

The next proposition shows that the trace is negligible away from periodic orbits.
\begin{prop}[Trace outside periodic orbits] \label{prop: trace outside periodic orb} Let $\rho\in \Sigma_E\backslash \Omega_n$. Let $\max(\beta,\nu)<1$, let $\chi_\rho\in S_\nu$ satisfy $\operatorname{Supp}\chi_\rho\subset B(\rho,h^\nu)$, and assume $h^\nu \mathrm{e}^{Ct}\ll h^{\nu_1}$. Then,
\[\forall t\in [t_{n-1},t_{n+1}], \quad \operatorname{Tr} U(t)\mathrm{Op}_h^M(\chi_\rho)=\mathcal O(h^\infty).\]
The estimate is uniform in $|t|\le T_{\rm max}$. \end{prop}
\begin{proof} \textbf{1.} Let $R_t:=\mathrm{e}^{Ct}h^\nu$, and let $b\in S_{\nu_1}(T^*M)$ satisfy $b\equiv 1$ on $B(\Phi^t(\rho),R_t)$ and $b\equiv 0$ outside $B(\Phi^t(\rho),2R_t)$. Then, by a slight variation of Proposition \ref{prop: cutoffs along the orbit}, one has
\[U(t)\mathrm{Op}_h^M(\chi_\rho)=\mathrm{Op}_h^M(b)U(t)\mathrm{Op}_h^M(\chi_\rho)+\mathcal O(h^\infty)_{\rm tr}.\]
By cyclicity of the trace, one finds
\[\operatorname{Tr} U(t)\mathrm{Op}_h^M(\chi_\rho)=\operatorname{Tr}(U(t)\mathrm{Op}_h^M(\chi_\rho)\mathrm{Op}_h^M(b))+\mathcal O(h^\infty).\]
Since $U(t)$ is unitary we infer
\[\big|\operatorname{Tr} U(t)\mathrm{Op}_h^M(\chi_\rho)\big|\le \big\|\mathrm{Op}_h^M(\chi_\rho)\mathrm{Op}_h^M(b)\big\|_{\rm tr}+\mathcal O(h^\infty).\]\medskip

\noindent \textbf{2.} We claim that $b$ and $\chi_\rho$ have disjoint support. If not, then one has $d(\rho,\Phi^t(\rho))\le 3R_t$. Then, by \cite[Lemma 8.3.13]{Lefeuvre}, there is some $|\tau|\le CR_t\le h^{\nu_1}$ and a periodic point $\rho_\star\in \Sigma_E$ of period $t+\tau$ such that
\[d_{\Sigma_E}(\rho,\rho_\star)\le 3R_t.\]
This contradicts the fact that $d(\rho,\rho_\star)\ge h^{\nu_1}$, since $R_t\ll h^{\nu_1}$ by assumption. \medskip

\noindent\textbf{3.} Now that we know that $\mathrm{Supp}(\chi_\rho)\cap \mathrm{Supp}(b)=\varnothing$, we can combine Proposition \ref{prop: composition manifold} with Lemma \ref{lem: convert sobolev estimate into trace class estimate} to obtain
\[\big\|\mathrm{Op}_h^M(\chi_\rho)\mathrm{Op}_h^M(b)\big\|_{\rm tr}=\mathcal O(h^\infty),\]
which concludes the proof.
\end{proof}
Proposition \ref{prop: trace outside periodic orb} shows that we can further refine the previous partition. For each $\gamma\in \mathcal P_n$, we let $\chi_\gamma\in S_{\nu_1}(T^*M)$ be supported on $\Omega(\gamma,2h^{\nu_1})$ satisfy $\chi_\gamma\equiv 1$ on $\Omega(\gamma,h^{\nu_1})$. Then, letting
\[A_{n,\gamma}:=\int \mathrm{e}^{-\frac{\mathrm{i}}{h}tP_h}\mathrm{e}^{\frac{\mathrm{i}}{h} tE}\omega_\nu(t-t_n)\widehat f(t/L)\mathrm{d}t \operatorname{Op}_h^{M}(\chi_\gamma),\]
we have
\[\operatorname{Tr}A_n=\sum_{\gamma\in \mathcal P_n} \operatorname{Tr} A_{n,\gamma} +\mathcal O(h^\infty).\]
We can thus study each operator $A_{n,\gamma}$ separately.
\subsubsection{Building a partition of unity along the orbit}The following lemma states that the intersection of a periodic orbit with a ball of small enough radius is either empty or consists of a single line segment. \begin{lem}\label{lem: orbit does not pass too close to itself} There are constants $C,c_0>0$ such that the following holds. Let $\gamma\subset \Sigma_E$ be a primitive periodic orbit of the flow, of period $T$. Then, for all $t,s\in \mathbf R/T\mathbf Z$ satisfying $|t-s|\ge \mathrm{e}^{-CT}$, we have $d_{T^*M}(\gamma(t),\gamma(s))\ge c_0 \mathrm{e}^{-CT}$.\end{lem}

\begin{proof}Let $T_0\gg 1$ be large enough. We can assume that $\gamma$ has period $T\ge T_0$ because there are finitely many orbits with length $\le T_0$. Suppose, for contradiction, that there exists a pair $(t,s)$ as above, and without loss of generality assume $s=0$. Then, there is some $t\in (\mathrm{e}^{-CT},T/2)$ such that $d(\Phi^t(\rho),\rho)\le \mathrm{e}^{-CT}$. Actually we have automatically $t\gtrsim 1$ because of injectivity radius considerations. Up to replacing $\mathrm{e}^{-CT}$ by $C_0\mathrm{e}^{-CT}$ we can in turn assume $t\ge T_0$. 
But then the Anosov closing Lemma \cite[Lemma 8.3.13]{Lefeuvre} tells us that there is a periodic point $\rho_\star$ with period $t+\tau$, with $|\tau|\le C_1 \mathrm{e}^{-CT}$, and such that $\mathrm{dist}(\gamma,\rho_\star)\le C_0\mathrm{e}^{-CT}$. In particular $\rho_\star$ has period $t+\tau\le T/2+\tau< T$ thus $\rho_\star$ does not belong to $\gamma$. We reach a contradiction to the orbit separation property if $C$ has been taken large enough.
\end{proof}
Lemma \ref{lem: orbit does not pass too close to itself} ensures that for each primitive periodic orbit $\gamma\subset \Sigma_E$ with period $T_\gamma\le \mathfrak c |\log h|$, we can place points $\rho_{\gamma,j}:=\Phi^{j h^{\nu_1}}(\rho_0)$ on the orbit $\gamma$, and construct associated functions $\chi_{\gamma,j}\in C^\infty_{\rm comp}(\rho_{\gamma,j},Ch^{\nu_1})$ such that $\sum \chi_{\gamma,j}\equiv 1$ on $\Omega(\gamma,h^{\nu_1})$, and such that
\[\gamma\cap B(\rho_{\gamma,j},Ch^{\nu_1})=\{\Phi^{s}(\rho_{\gamma,j})~:~ s\in [-1,1]\}\cap B(\rho_{\gamma,j},Ch^{\nu_1}).\] 
Concretely, in small balls of radius $\lesssim h^{\nu_1}$, the orbit $\gamma$ looks like a single line. Further, we define operators
\[A_{n,\gamma,j}:=\int \mathrm{e}^{-\frac{\mathrm{i}}{h}tP_h}\mathrm{e}^{\frac{\mathrm{i}}{h} tE}\omega_\nu(t-t_n)\widehat f(t/L)\mathrm{d}t \operatorname{Op}_h^{M}(\chi_{\gamma,j}),\]
so that
\[\operatorname{Tr}A_{n,\gamma}=\sum_j \operatorname{Tr}A_{n,\gamma,j}+\mathcal O(h^\infty).\]
We will study each $A_{n,\gamma,j}$ separately. In the next section we will show the following results.
\begin{prop}\label{prop: trace A n gamma j} For $\varepsilon>0$ small enough, one can find $\mathfrak c=\mathfrak c(\alpha,\beta,\varepsilon)>0$ such that for $T_\gamma\le \mathfrak c|\log h|$, in the notation of Theorem \ref{prop: long time gutz}, one has
\[\operatorname{Tr} A_{n,\gamma,j}=\frac{1}{|I-P_\gamma|^{1/2}}\mathrm{e}^{\frac{\mathrm{i}}{h} S(\gamma_\delta)}\widehat f(T_\gamma/L)\int_{\gamma} \chi_{\gamma,j}+\mathcal O(h^\varepsilon).\]
\end{prop}
\begin{rem}In fact, the proof gives an expansion of $\operatorname{Tr} A_{n,\gamma,j}$ in powers of $h$ with increasingly better error as $h\to 0$, but we shall only need the first term in our future application.
\end{rem}
By summing over $j$, we obtain, after possibly decreasing $\varepsilon$ and $\mathfrak c$:
\begin{prop} \label{prop: trace A n gamma} There is a constant $\varepsilon_0=\varepsilon_0(\alpha,\beta)$ such that for any $0<\varepsilon<\varepsilon_0$, one can find $\mathfrak c=\mathfrak c(\alpha,\beta,\varepsilon)>0$ such that for $T_\gamma\le \mathfrak c|\log h|$, one has
\[\operatorname{Tr} A_{n,\gamma}=\frac{T_\gamma^\#}{|I-P_\gamma|^{1/2}} \mathrm{e}^{\frac{\mathrm{i}}{h} S(\gamma_\delta)}\widehat f(T_\gamma/L)+\mathcal O(h^\varepsilon).\]\end{prop}

\subsection{Adapted coordinates} \label{subsec: adapted coordinate} Fix a finite atlas $\{(U_\ell,f_\ell)\}$ of $M$. Since $M$ is orientable, we can assume that $\det (\mathrm{d}(f_\ell^{}\circ f_{\ell'}^{-1}))>0$ on $f_\ell'(U_\ell\cap U_{\ell'})$ for all $\ell,\ell'$. We consider an associated quadratic partition of unity $\{\chi_\ell\}$, i.e. $\sum \chi_\ell^2=1$. Also, we take $\chi_{\ell}'\in C^\infty_{\rm comp}(U_\ell,[0,1])$ such that $\chi_\ell\prec \chi_\ell'$ and $\sum \chi_\ell'\le 2$.

Let $\rho_0\in \Sigma_E$. Up to reparameterizing the time, for all $k\ge 0$, one can find an index $\ell_k$ such that the piece of trajectory $\{\Phi^t_0(\rho_0)~:~ t\in (k-1,k+1)\}$ is entirely contained in $T^* U_{\ell_k}$. We obtain a sequence of indices $(\ell_k)_{k\ge 0}$ that depends on the starting point $\rho_0$. We let $\rho_t:=\Phi^t(\rho_0)$ and $\rho_t^k:=\widetilde f_{\ell_k}(\rho_t)$ for $t\in (k-1,k+1)$.

Assume $t\in (k-1,k+1)$. Recalling Lemmas \ref{lem: adapted bases} and \ref{lem: matrice differentielle flot}, there exists a symplectic basis $(u_1(t),\ldots,u_d(t);s_{1}(t),\ldots,s_{d}(t))$ of $T_{\rho_t} T^*M$ adapted to the decomposition $E^{wu}\oplus (E^{wu})^\bot$. We define $\mathcal A_{\ell_k}(t)\in \mathrm{Sp}(2d,\mathbf{R})$ as the matrix expressing this basis in the canonical coordinates associated with the chart $(T^* U_{\ell_k},\widetilde f_{\ell_k})$. Denoting $(e_1,\ldots,e_d;f_1,\ldots,f_d)$ the canonical basis of $T^*\mathbf{R}^d$, we have 
When $t\in (k-1,k+1)\cap (k,k+2)$, the matrices $\mathcal A_{\ell_k}(t)$ and $\mathcal A_{\ell_{k+1}}(t)$ are related by (we write $g_k:=f_{\ell_{k+1}}\circ f_{\ell_k}^{-1}$):
\begin{equation} \label{eq: transition eq A ell k}\mathcal A_{\ell_{k+1}}(t)=\mathrm{d} \widetilde{g_k}(\rho_t^k)\mathcal A_{\ell_k}(t)=\begin{pmatrix}
\mathrm{d}g_k & 0 \\ * & \mathrm{d}g_k^{-\top}
\end{pmatrix} \mathcal A_{\ell_k}(t).\end{equation}

Choose a basis $(V_1(t),\ldots,V_d(t))$ of the vertical space $\mathbf V(\rho_t)$, varying continuously in $t$, and define the associated basis of the horizontal space $\mathbf H(\rho_t)$ by
\[(H_1(t),\ldots,H_d(t)):=(JV_1(t),\ldots,JV_d(t)).\]
We define a symplectic matrix $\mathcal B_{\ell_k}(t)$ as the change-of-basis matrix from $\mathbf H(\rho_t)\oplus \mathbf V(\rho_t)$ to the canonical basis $(\partial_{x_i},\partial_{\xi_i})$ of $T_{\rho_t}T^*M$ associated with the chart $(T^*U_{\ell_k}, \widetilde f_{\ell_k})$. It has the block form
\[\mathcal B_{\ell_{k}}(t):=\begin{pmatrix}
P_k(t) & 0 \\ Q_k(t) & R_k(t)
\end{pmatrix},\]
moreover one can arrange $P_k$ to have positive determinant because $M$ is orientable.
\begin{lem}\label{lem: Bk Ak inde chart} On overlaps of charts we have \[\mathcal B_{\ell_{k+1}}(t)\mathcal A_{\ell_{k+1}}(t)=\mathcal B_{\ell_{k}}(t)\mathcal A_{\ell_k}(t).\]
This matrix represents the basis $(u_1(t),\ldots,u_d(t);s_1(t),\ldots,s_d(t))$ in the decomposition $\mathbf H\oplus \mathbf V$, and shall be denoted $\mathcal S(t)$.
\end{lem}

\subsection{Quantum conjugation}

We now quantize the family of symplectic matrices $\{\mathcal A_{\ell_k}(t)\}$ into a family of metaplectic operators $\{\mathcal M_{\ell_k}(t)\}$. The required material on metaplectic operators is provided in Section \ref{sec: metaplectic}. Recall that $g_n:=f_{\ell_{n+1}}\circ f_{\ell_n}^{-1}$.

\begin{defi}[Construction of $\{\mathcal M_k(t)\}$] \label{def: metaplectic Mk} Fix a metaplectic operator $\mathcal M_{\ell_0}(0)$ that quantizes $\mathcal A_{\ell_0}(0)$ (the only ambiguity is the choice of sign). Then, extend it continuously to obtain $\mathcal M_{\ell_0}(t)$ quantizing $\mathcal A_{\ell_0}(t)$ for $t\in [0,2)$. Assume that $\{\mathcal M_{\ell_k}(t)\}_{t\in (k-1,k+1)}$ has been defined for $k=0,\ldots,n$. Fix $s\in (n,n+1)$. Then, motivated by \eqref{eq: transition eq A ell k}, set
\[\mathcal M_{\ell_{n+1}}(s):=\mu_h^+(\mathrm{d}\widetilde g_n)\mathcal M_{\ell_n}(s),\]
where $\mu_h^+(\mathrm{d}\widetilde g_n)$ is the metaplectic operator quantizing $\mathrm{d}\widetilde g_n\big(\widetilde{f_{\ell_n}}(\rho_s)\big)$ obtained using formula \eqref{met zworski}. Then, extend continuously $\mathcal M_{\ell_{n+1}}(s)$ to $\mathcal M_{\ell_{n+1}}(t)$ for $t\in (n-1,n+2)$.\end{defi}

We provide an alternative but equivalent definition which avoids discontinuities at the chart transitions. Because
\[\mathrm{d}\widetilde g_n=\begin{pmatrix}
\mathrm{d}g_n & 0 \\ * & \mathrm{d}g_n^{-\top}
\end{pmatrix} \quad \text{and} \quad \det \mathrm{d}g_n>0,\] we can find a deformation $\{\mathcal C_{n,s}\}_{s\in [0,1]}$ from the identity to $\mathrm{d}\widetilde g_n$ in the space of symplectic matrices. Moreover, we can take $\mathcal C_{n,s}$ in a block form $\begin{pmatrix}
* & 0 \\ * & *
\end{pmatrix}$. The advantage is that the path $\{\mathcal C_{n,s}\mathcal A_{\ell_n}(n+1)\}_{s\in [0,1]}$ provides a continuous interpolation between $\mathcal A_{\ell_n}(n+1)$ and $\mathcal A_{\ell_{n+1}}(n+1)$.
\begin{defi}[Alternative definition of $\{\mathcal M_{\ell_k}(t)\}$]\label{defi: tilde A (t)} The concatenation (denoted by $\star$) of paths
\[\cdots \star \{\mathcal A_{\ell_n}(t)\}_{t\in [n,n+1]}\star \{\mathcal C_{n-1,s}\mathcal A_{\ell_{n-1}}(n)\}_{s\in [0,1]}\star \cdots 
\{\mathcal C_{0,s}\mathcal A_{\ell_0}(1)\}_{s\in [0,1]}\star \{\mathcal A_{\ell_0}(t)\}_{t\in [0,1]}\]
is a continuous path in $\mathrm{Sp}(2d,\mathbf{R})$ which can be quantized (uniquely up to the choice of starting point) into a continuous path in $\operatorname{Mp}_h(2d,\mathbf{R})$. The operators $\mathcal M_{\ell_n}(t)$ obtained in this way coincide with those of Definition \ref{def: metaplectic Mk}, provided the starting points $\mathcal M_{\ell_0}(0)$ are the same. \end{defi} 

Let  $\rho_0\in \gamma\subset \Sigma_E$ be a periodic point of the flow, of period $T_\gamma$. Up to a reparameterization of time we can assume that $T_\gamma$ is an integer, which simplifies notation in the indexing of charts. Since $\rho_0$ is periodic, we can take $\ell_{T_\gamma+1}=\ell_0$ in the sequence of charts above. For now, we assume that the weak-unstable distribution is orientable along the periodic orbit, which ensures that we can take $\mathcal A_{\ell_0}(T_\gamma)=\mathcal A_{\ell_0}(0)$ in \S \ref{subsec: adapted coordinate}. We will explain the necessary changes when the orientability assumption is dropped in \S \ref{subsec: elucidating Maslov}.

\begin{lem}\label{lem: the winding number is maslov} Assume that the weak-unstable distribution $E^{wu}$ is orientable along the orbit $\gamma$. Then, we have
\[\mathcal M_{\ell_0}(T_\gamma)=\mathrm{e}^{-\mathrm{i}\frac{\pi}{2}\mathfrak m_\gamma}\mathcal M_{\ell_0}(0),\]
where $\mathfrak m_\gamma$ is the Maslov index of the periodic orbit $\gamma$, given in Definition \ref{defi: maslov index periodic orbit}.
\end{lem}
\begin{proof}By Lemma \ref{lem: expression winding det Pt} and Remark \ref{rem: maslov is invariant under}, the Maslov index of the loop of matrices from Definition \ref{defi: tilde A (t)} coincides with that of the path $\mathcal S(t)=\mathcal B_{\ell_k(t)}\mathcal A_{\ell_k(t)}$ from Lemma \ref{lem: Bk Ak inde chart}. Writing the complexified matrix $\mathcal S(t)_c$ as
\[\mathcal S(t)_c=\begin{pmatrix}
P_t & Q_t \\ \overline{Q_t} & \overline{P_t}
\end{pmatrix},\]
Lemma \ref{lem: lifting loops} yields 
\[\mathcal M_{\ell_0}(T_\gamma)=\frac{(\det \overline{P_0})^{1/2}}{(\det \overline{P_{T_\gamma}})^{1/2}}\mathcal M_{\ell_0}(0).\]
Here one has chosen a continuous determination of $t\mapsto \det P_t$ along the path. The winding of $\det P_t$ along the orbit gives the Maslov index $\mathfrak m_\gamma$, see \S\ref{subsec: elucidating Maslov}.
\end{proof}

Let $\chi_{\rho_0}\in S_\nu(T^*M)$ be supported in $B(\rho_0,h^\nu)$. Our goal is to compute the trace of the operator
\[U(t)\mathrm{Op}_h^M(\chi_{\rho_0}),\]
for $t$ close to $T_\gamma$. In the notation of Proposition \ref{prop: cutoffs along the orbit}, define short-time propagators in coordinates:
\begin{equation} \label{eq: defi Vk(t)} V_k(t)=(f_{\ell_{k+1}})_*\chi_{\ell_{k+1}}\operatorname{Op}_h^{M}(\chi_{\rho_{k+1}})U(t)\operatorname{Op}_h^{M}(\chi_{\rho_k})(f_{\ell_k})^*\chi_{\ell_k}, \quad \text{for } |t|\le 1\end{equation}
Then, by Proposition \ref{prop: cutoffs along the orbit}, we have for $|s|\le 1$:
\[U(T_\gamma+s)\operatorname{Op}_h^{M}(\chi_{\rho_0})=(f_{\ell_{0}})^*V_{T_\gamma}(s)V_{T_\gamma-1}(1)\cdots V_0(1)(f_{\ell_0})_*+\mathcal O(h^\infty)_{\rm tr}.\]
The short time propagators $V_k(t)$ can be expressed as $h$-FIOs with phase functions of the form $\varphi(x,\eta)-\langle y,\eta\rangle$, but may not compose well due to the possible formation of caustics. To overcome this issue, we conjugate the propagators by the metaplectic operators introduced above. Note that the idea of expressing the quantum propagator in stable-unstable coordinates to reveal a simple geometrical interpretation of the Maslov indices already appeared in the physics literature \cite{creagh1990,robbins1991}. We set
\begin{equation} \label{eq: defi Wk(t)}W_k(t)=\mathcal M_{\ell_{k+1}}(k+1)^*V_k(t)\mathcal M_{\ell_k}(k),\qquad \text{for $k\le T_\gamma$}.\end{equation}
Then,
\[U(T_\gamma+s)\mathrm{Op}_h^M(\chi_{\rho_0})=(f_{\ell_0})^* \mathcal M_{\ell_{0}}(T_\gamma) W_{T_\gamma}(s)W_{T_\gamma-1}(1)\cdots W_0(1)\mathcal M_{\ell_0}(0)^*(f_{\ell_0})_*+\mathcal O_{\rm tr}(h^\infty).\]
By Lemma \ref{lem: the winding number is maslov}, we have $\mathcal M_{\ell_0}(T_\gamma)=\mathrm{e}^{-\mathrm{i}\frac{\pi}{2}\mathfrak m_\gamma}\mathcal M_{\ell_0}(0)$, thus by taking the traces in the expression above, we obtain
\begin{equation} \label{eq: by taking the traces}\operatorname{Tr} U(T_\gamma+s)\mathrm{Op}_h^M(\chi_{\rho_0})=\mathrm{e}^{-\mathrm{i}\frac{\pi}{2}\mathfrak m_\gamma} \operatorname{Tr}W_{T_\gamma}(s)W_{T_\gamma-1}(1)\cdots W_0(1)+\mathcal O(h^\infty).\end{equation}
\subsection{FIO representation of the short time twisted propagator} The next proposition, which gives a short-time FIO representation for the propagator $W_0(t)$, is obtained by combining Proposition \ref{prop: short time wkb}, the composition formula for FIOs from Proposition \ref{thm: composition FIO} and the results of \S \ref{subsec: gen func twisted flow} to give the interpretation of the phase.
\begin{prop}[Short time propagator in adapted coordinates] \label{thm: short time propagator twisted} Let $\chi_{\rho_0}\in S_\nu(T^*M)$ be as above, and let $\sigma:=\max(\beta,\nu)$. Then, for $|t|\le 1$:
\[W_0(t)=\mathcal I[\varphi_{0,1},a_h]+\mathcal O_N(h^{N(1-2\sigma)-C})_{L^2/\rm tr},\]
where:
\begin{itemize}
\item $\varphi_{0,1}$ is a generating function for the map 
\[\Phi^t_{0,1}:=\mathcal A_{\ell_1}(1)^{-1}\circ \widetilde{f_{\ell_1}}\circ\Phi^t_\delta\circ \widetilde{f_{\ell_0}}^{-1}\circ \mathcal A_{\ell_0}(0).\]
\item If $\Phi^t_{0,1}(x_0,\xi_0)=(x_1,\xi_1)$, then the phase $\varphi(t,x_1,\xi_0)-\langle x_0,\xi_0\rangle$ is given by
\[\varphi(t,x_1,\xi_0)-\langle x_0,\xi_0\rangle=\widehat S(\gamma) - F(\mathcal A_{\ell_1}(1),(x_1,\xi_1))+F(\mathcal A_{\ell_0}(0),(x_0,\xi_0)),\]
where $\widehat S(\gamma)$ is the action (defined in \eqref{eq: def action bis}) of the perturbed orbit $\{\Phi^s_\delta(\rho)\}_{s\in[0,t]}\subset T^*M$, with $\rho=\widetilde f_{\ell_0}^{-1}(\mathcal A_{\ell_0}(0)(x_0,\xi_0))$, and $F$ is the cocycle from Definition \ref{defi: cocycle action}.
\item The symbol $a_h$ is compactly supported, and admits a finite expansion $a_h=\sum_{j<N} h^j a_j$, with 
\[a_0=\chi_{\rho_1}\big(x_1,\nabla_{x_1} \varphi_{0,1}(t,x_1,\xi_0)\big)|\partial_{x_1\xi_0}^2\varphi_{0,1}(t,x_1,\xi_0)|^{\frac 12}\chi_{\rho_0}(x_0,\xi_0).\]
Here, to simplify notation, we have written $\chi_{\rho_i}$ instead of $\chi_{\rho_i}\circ \widetilde f_{\ell_i}^{-1}\circ \mathcal A_{\ell_i}(i)$ for $i\in \{0,1\}$.
\item The symbols are estimated by
\[\|a_j\|_{C^\ell_{t,x_1,x_0,\xi_0}}\lesssim h^{-\sigma(\ell+2j)}.\] \end{itemize}\end{prop}

\subsection*{Strategy of the proof} The proof proceeds in three steps.
\begin{itemize}
\item Compute $\mathbf W_0(t)=\mathcal M_{\ell_0}(0)^* \chi_{\ell_0}(f_{\ell_0})_*U(t)\mathrm{Op}_h^M(\chi_\rho)f_{\ell_0}^* \chi_{\ell_0} \mathcal M_{\ell_0}(0)$.
\item Compute $\mathbf W_1(t)=\mathcal M_{\ell_0}(1)^* (f_{\ell_0})_*U(t)f_{\ell_0}^* \mathcal M_{\ell_0}(0)$.
\item Compute $\mathbf W_2(t)=\mathcal M_{\ell_1}(1)^* (f_{\ell_1})_*U(t)f_{\ell_0}^* \mathcal M_{\ell_0}(0)$.\end{itemize}

\subsubsection{Computation of $\mathbf W_0$} By the Egorov Theorem \ref{prop: egorov bis}, $\mathbf W_0(t)$ satisfies the evolution equation
\[hD_t\mathbf W_0(t)=-\mathrm{Op}_h^{\rm w}\big((\chi_{\ell_0}^2 p)\circ \widetilde f_{\ell_0}^{-1}\circ \mathcal A_{\ell_0}(0)+ r\big)\mathbf W_0(t),\]
where $r$ satisfies $\|r\|_{C^\ell}\lesssim h^2h^{-\beta(\ell+2)}$. By Proposition \ref{prop: short time wkb}, we obtain
\[\mathbf W_0(t)=\mathcal I[\varphi_0,a_h^{(N)}]+\mathcal O(h^{N(1-2\beta)-C})_{L^2/\rm tr},\]
where:
\begin{itemize}
\item $\varphi_{0}$ is a local generating function for the Hamiltonian flow generated by $p_\delta\circ \widetilde f_{\ell_0}^{-1}\circ \mathcal A_{\ell_0}(0)$ .
\item $a_h^{(N)}=\sum_{j<N} h^j a_j$, with $a_0(t,x,y,\eta)=|\partial_{x\eta}^2 \varphi(t,x,\eta)|^{\frac 12}{\chi_\rho}(y,\eta)$, where $\chi_\rho$ is viewed as a function in the chart. The symbols are estimated by $\|a_j\|_{C^{\ell}_{t,x,\eta}}\lesssim h^{-\beta(\ell+2j)}$.
\item Whenever $(x_0,\xi_0)$ is mapped to $(x_1,\xi_1)$ by the flow $\mathcal A_{\ell_0}(0)^{-1}\circ\widetilde f_{\ell_0}\circ \Phi^t_\delta\circ \widetilde f_{\ell_0}^{-1}\circ \mathcal A_{\ell_0}(0)$, one has
\[\varphi_0(t,x_1,\xi_0)-\langle x_0,\xi_0\rangle=\widehat S(\gamma)-F(\mathcal A_{\ell_0}(0),(x_1,\xi_1))+
F(\mathcal A_{\ell_0}(0),(x_0,\xi_0)).\]
Here $\gamma$ denotes the corresponding piece of orbit on $T^*M$.
\end{itemize}

\subsubsection{Computation of $\mathbf W_1$}
The operator $\mathbf W_1(t)$ is obtained by composing $\mathcal M_{\ell_0}(1)^*\mathcal M_{\ell_0}(0)$ with $\mathbf W_0(t)$. By definition, we have
\[\mathcal M_{\ell_0}(1)^*\mathcal M_{\ell_0}(0)=\mu_h^+(\mathcal A_{\ell_0}(1)^{-1}\mathcal A_{\ell_0}(0))=\mathcal I[\psi,|\partial_{x,\eta}^2 \psi|^{\frac 12}],\] where $\mu_h^+(\bullet)$ is given by formula \eqref{met zworski}, and $\psi$ is the generating function from \eqref{eq: gen func for symplectic map}. 
We now compute the composition $\mathbf W_1(t)=\mathcal M_{\ell_0}(1)^*\mathcal M_{\ell_0}(0) \mathbf W_0(t)$. By Theorem \ref{thm: composition FIO} for the composition of FIOs, we find:
\begin{equation} \label{eq: parametrix W1(t)} \mathbf W_1(t)=\mathcal I[\psi,|\partial_{x,\eta}^2\psi|^{\frac 12}]\circ \mathcal I[\varphi_0,a_h^{(N)}]=\mathcal I[\varphi_1,b_h^{(N)}]+\mathcal O(h^{N(1-2\beta)-C})_{L^2/\rm tr},\end{equation}
where:
\begin{itemize}
\item $\varphi_1$ is the generating function for $\mathcal A_{\ell_0}(1)^{-1}\circ (f_{\ell_0})_*\circ \Phi^t_\delta \circ (f_{\ell_0})^*\circ \mathcal A_{\ell_0}(0)$ obtained by composing $\psi$ with $\varphi_0$ as in Lemma \ref{lem: composition generating functions}, that is 
\[\varphi_1(t,x_2,\xi_0)=\psi(x_2,\xi_1)-\langle x_1,\xi_1\rangle+\varphi_0(x_1,\xi_0)\big|_{(x_1,\xi_1)=(x_1(x_2,\xi_0),\xi_1(x_2,\xi_0))}.\]
\item One has $b_h^{(N)}=\sum_{j<N}h^jb_j$ with
\[b_0(t,x_2,\xi_0,x_0)=|\partial_{x_2\xi_0} \varphi_1|^{\frac 12} \chi_\rho(x_0,\xi_0).\] 
Moreover, 
\[\|b_j\|_{C^\ell}\lesssim h^{-\beta(\ell+2j)}.\]
\item By Proposition \ref{prop: properties cocycle}, we see that whenever $(\bullet,\xi_0)$ is mapped to $(x_2,\bullet)$ by $\mathcal A_{\ell_0}(1)^{-1}\circ\widetilde f_{\ell_0}\circ \Phi^t\circ \widetilde f_{\ell_0}^{-1}\circ \mathcal A_{\ell_0}(0)$, one has \[\varphi_1(t,x_2,\xi_0)-\langle x_0,\xi_0\rangle =\widehat S(\gamma)-F(\mathcal A_{\ell_0}(1),(x_2,\xi_2))+F(\mathcal A_{\ell_0}(0),(x_0,\xi_0)),\]
where $\gamma=\{\Phi^s_\delta(\rho)\}_{s\in[0,t]}\subset T^*M$, with $\rho=\widetilde f_{\ell_0}^{-1}(\mathcal A_{\ell_0}(0)(x_0,\xi_0))$.
\end{itemize} 

\subsubsection{Computation of $\mathbf W_2(t)$} 

We need to compose $\mathcal M_{\ell_1}(1)^* (f_{\ell_1})_*(f_{\ell_0})^* \mathcal M_{\ell_0}(1)$ with $\mathbf W_1(t)$. The first step is to find an oscillatory integral representation of $\mathcal M_{\ell_1}(1)^* (f_{\ell_1})_*(f_{\ell_0})^* \mathcal M_{\ell_0}(1)$. This operator quantizes the symplectic transformation $\mathcal A_{\ell_1}(1)^{-1}\widetilde {f_{\ell_1}}\widetilde{f_{\ell_0}}^{-1}\mathcal A_{\ell_0}(1)$. Let $g:=f_{\ell_1}\circ f_{\ell_0}^{-1}$. By definition, we have $\mathcal A_{\ell_1}(1)=\mathrm{d}\tilde g(\rho_1)\circ \mathcal A_{\ell_0}(1)$ and $\mathcal M_{\ell_1}(1)=\mu_h^+\left( \mathrm{d}\tilde g(\rho_1)\right)\mathcal M_{\ell_0}(1)$, where $\mu_h^+$ is defined in \eqref{met zworski}. By Lemma \ref{lem: adjoint of mu h (A)}, we find
\[\mathcal M_{\ell_1}(1)^*=\mathcal M_{\ell_0}^*(1)\mu_h^+(\mathrm{d}\tilde g(\rho_1)^{-1}),\]
and one eventually finds
\[\mathcal M_{\ell_1}(1)^* (f_{\ell_1})_*(f_{\ell_0})^* \mathcal M_{\ell_0}(1)=\mathcal M_{\ell_0}(1)^* \mu_h^+(\mathrm{d}\tilde g(\rho_1)^{-1}) (g^{-1})^* \mathcal M_{\ell_0}(1).\]
\begin{lem}\label{lem: FIO rep mu h + circ tilde g} One has
\[\mu_h^+(\mathrm{d}\tilde g(\rho_1)^{-1}) (g^{-1})^*=\mathcal I[\varphi_2,\alpha_h],\]
where $\varphi_2$ is the generating function of the map $\mathrm{d}\tilde g(\rho_1)^{-1}\circ \mathrm{d}\tilde g$ described in the proof below, and $\alpha_h=|\partial_{x,\eta}^2\varphi_2|^{1/2}+\mathcal O(h)$. Moreover, whenever $(y,\eta)$ is mapped to $(x,\xi)$ by $\mathrm{d}\tilde g(\rho_1)^{-1}\circ \mathrm{d}\tilde g$, we have
\begin{equation} \label{eq: express gen fun as action for prelemma} \varphi_2(x,\eta)-\langle y,\eta\rangle=F(\mathrm{d}\tilde g(\rho_1)^{-1}, \tilde g(y,\eta)).\end{equation}
\end{lem}
\begin{proof} Write $(g^{-1})^*=\mathcal I[\langle g^{-1}(x),\eta\rangle,a_h]$ and $\mu_h^+(\mathrm{d}\tilde g(\rho_1)^{-1})=\mathcal I[\psi,b_h]$, where $\psi$ is the generating function from Proposition \ref{prop: properties cocycle} associated with the symplectic matrix $\mathrm{d}\tilde g(\rho_1)^{-1}$ and $b_h$ is the amplitude from \eqref{met zworski}. Let $\varphi_2$ be the composition of these two generating functions described in Lemma \ref{lem: composition generating functions}. We then compute the composition $\mathcal I[\langle g^{-1}(x),\eta\rangle,a_h]\circ \mathcal I[\psi,b_h]$ using Theorem \ref{thm: composition FIO}. One obtains \eqref{eq: express gen fun as action for prelemma} by combining Proposition \ref{prop: generating function is action}, Lemma \ref{lem: symplectic lift action} and Proposition \ref{prop: properties cocycle} again.
\end{proof}
\begin{prop}\label{prop: parametrix W2(t)} We have a FIO representation
\[\mathcal M_{\ell_1}(1)^* (g^{-1})^*\mathcal M_{\ell_0}(1) u(x)=\mathcal I[\varphi_3,\alpha_h],\]
with $\alpha_h=|\partial_{x,\eta}^2 \varphi_3|^{\frac 12}+\mathcal O(h)$. Whenever $(y,\eta)$ is mapped to $(x,\xi)$ by the map $\mathcal A_{\ell_1}(1)^{-1}\circ \tilde g\circ \mathcal A_{\ell_0}(1)$, one has
\begin{equation} \label{eq: phi 4 as diff cocycle} \varphi_3(x,\eta)-\langle y,\eta\rangle=-F(\mathcal A_{\ell_1}(1),(x,\xi))+F(\mathcal A_{\ell_0}(1),(y,\eta)).\end{equation}\end{prop}

\begin{proof}	\textbf{1.} By considering a smooth deformation $\{g_s\}_{s\in [0,1]}$ from $\operatorname{Id}$ to $g$, we can realize the unitary operator $\mu_h^+(\mathrm{d}\tilde g(\rho_1)^{-1}) (g^{-1})^*$ as the solution at time $1$ of 
\begin{equation}\label{eq: derivee U(s) pb aux} \left\{\begin{array}{ll} U(0)=\mathrm{Id} \\  hD_sU(s)=\mathrm{Op}_h^{\rm w}(p_s+h^2r_s)U(s)\end{array}\right.,\end{equation}
where, denoting $\kappa_s=\mathrm{d}\tilde g_s^{-1}(\rho_1)\circ \widetilde g_s$, the symbol $p_s$ is a time-dependent Hamiltonian satisfying $H_{p_s}=\partial_s \kappa_s$, and the error $r_s$ has uniformly bounded derivatives. By using the exact Egorov theorem for metaplectic operators, this shows that the operator $\mathcal M_{\ell_0}(1)^* \mu_h^+(\mathrm{d}\tilde g^{-1}) (g^{-1})^* \mathcal M_{\ell_0}(1)$ is the solution at time $1$ of 	\[\left\{\begin{array}{ll} V(0)=\mathrm{Id} \\ hD_s V(s)=\mathrm{Op}_h^{\rm w}(q_s+h^2r_s')V(s)\end{array}\right.,\]
where $q_s=p_s\circ \mathcal A_{\ell_0}(1)$ and $r_s'$ has uniformly bounded derivatives again. The operator $V(s)$ quantizes the symplectic transformation $\tilde \kappa_s:=\mathcal A_{\ell_0}(1)^{-1}\circ \kappa_s\circ \mathcal A_{\ell_0}(1)$. Note that $\tilde \kappa_s$ is close to a translation in $T^*\mathbf{R}^d$ provided we stay in a small neighborhood of $\rho_1$. In particular $\tilde \kappa_s$ has a projectable graph. Therefore, we have a FIO representation
\[V(s)=\mathcal I[\varphi_3(s,\cdot),b],\]
where $\varphi_3(s,x,\eta)$ is a generating function of the conjugated flow $\tilde \kappa_s$ and $b\sim \sum h^j b_j$ with $b_0=|\partial_{x,\eta}\varphi_3|^{\frac 12}$. We set $\varphi_3(x,\eta):=\varphi_3(1,x,\eta)$.

\textbf{2.} Let us finally prove \eqref{eq: phi 4 as diff cocycle}. By Proposition \ref{prop: generating function is action},  whenever $(y,\eta)$ is mapped to $(x,\xi)$ by the flow $\{\tilde \kappa_s\}_{s\in [0,1]}$ generated by $(q_s)$, we have
\begin{equation} \label{eq: phi 4 as action in proof}\varphi_3(1,x,\eta)-\langle y,\eta\rangle=\int_\gamma \xi\mathrm{d}x-q_s\mathrm{d}s,\end{equation}
where $\gamma(s)=\tilde \kappa_s(y,\eta)$. Since $q_s=p_s\circ \mathcal A_{\ell_0}(1)$, letting $\tilde \gamma=\mathcal A_{\ell_0}(1)\gamma$, we have
\begin{equation} \label{eq: int gamma qs int gamma ps}\int_\gamma q_s\mathrm{d}s=\int_{\tilde \gamma} p_s\mathrm{d}s.\end{equation}
Letting $(\tilde y,\tilde \eta)=\mathcal A_{\ell_0}(1)(y,\eta)$ and $(\tilde x,\tilde \xi)=\mathcal A_{\ell_0}(1)(x,\xi)$ denote the end points of $\tilde \gamma$, we also have
\begin{equation} \label{eq: we also have}F(\mathrm{d}\tilde g(\rho_1)^{-1},\tilde g(\tilde y,\tilde \eta))=\varphi_2(\tilde x,\tilde \eta)-\langle \tilde x,\tilde \eta\rangle=\int_{\tilde \gamma} \xi\mathrm{d}x-p_s\mathrm{d}s,\end{equation}
where $\varphi_2$ is the generating function from Lemma \ref{lem: FIO rep mu h + circ tilde g}. By combining \eqref{eq: phi 4 as action in proof}, \eqref{eq: int gamma qs int gamma ps}, \eqref{eq: we also have}, we obtain
\[\varphi_3(1,x,\eta)-\langle y,\eta\rangle=\int_\gamma \xi\mathrm{d}x-\int_{\tilde \gamma} \tilde \xi\mathrm{d}\tilde x +F(\mathrm{d}\tilde g(\rho_1)^{-1},\tilde g(\tilde y,\tilde \eta)).\]
Since $\tilde \gamma=\mathcal A_{\ell_0}(1)\gamma$, the difference of actions can be expressed using Proposition \ref{prop: action of image} as
\[S(\tilde \gamma)-S(\gamma)=-F(\mathcal A_{\ell_0}(1)^{-1},(\tilde x,\tilde \xi))+F(\mathcal A_{\ell_0}(1)^{-1},(\tilde y,\tilde \eta)).\]
Recall that $(\tilde x,\tilde \xi)=\mathrm{d}\tilde g(\rho_1)^{-1}(\tilde g(\tilde y,\tilde \eta))$. Playing with the cocycle property of $F$ and using the fact that $\mathcal A_{\ell_1}(1)=\mathrm{d}\tilde g(\rho_1)\circ \mathcal A_{\ell_0}(1)$, we obtain
\[\varphi_3(x,\eta)-\langle y,\eta\rangle=-F(\mathcal A_{\ell_1}(1),(x,\xi))+F(\mathcal A_{\ell_0}(1),(y,\eta)).\]

\end{proof}

\subsubsection{Proof of Proposition \ref{thm: short time propagator twisted}}

To finish the proof of Proposition \ref{thm: short time propagator twisted}, we combine \eqref{eq: parametrix W1(t)} with Proposition \ref{prop: parametrix W2(t)} and use the composition Theorem \ref{thm: composition FIO}.

\subsection{Propagation up to logarithmic times} We now compose the short-time approximations of the propagators from Proposition \ref{thm: short time propagator twisted} to obtain a long time parametrix for $\mathrm{e}^{-\frac{\mathrm{i}}{h}tP_{h,\delta}}$.
\begin{prop}\label{prop: propag until Ehrenfest} Fix $\varepsilon>0$ and assume $T_\gamma\le \mathfrak c|\log h|$ so as to ensure $\mathrm{e}^{CT}\le h^{-\varepsilon}$. Set $\sigma=\max(\beta,\nu)+\varepsilon<1/2$. 
	Recall the definition of the operators $W_k$ from \eqref{eq: defi Wk(t)}. For all $1\le n\le T_\gamma$, for $|t|\le 2h^{\nu}$, the propagator at time $n+t$ is given by
	\[W_{n-1}(1+t)\cdots W_0(1)=\mathcal I[\phi_n,a_h^{(n-1)}]+\mathcal O(h^{N(1-2\sigma)-C})_{L^2/\rm tr}.\]
	Here, $\phi_n=\phi_n(t,x_n,\xi_0)$ is a generating function for the flow in adapted coordinates
	\[\mathcal A_{\ell_n}(n)^{-1}\widetilde f_{\ell_n} \Phi^{n+t}_\delta \widetilde f_{\ell_0}^{-1}\mathcal A_{\ell_0}(0),\qquad \forall |t|\le h^{\nu },\]
	and has been estimated in Lemma \ref{lem: spacetime derivatives generating function}. The amplitude expands as $a_h^{(n-1)}=\sum_{j=0}^{N-1} h^j a_j^{(n-1)}$, and the first term is given by
	\begin{equation} \label{eq: the first term is given by Ehrenfest}a_0^{(n-1)}(t,x_n,x_0,\xi_0)=|\partial_{x_n,\xi_0}^2 \phi_n|^{\frac 12} \chi_{\rho_n}(x_n,\xi_n(t,x_n,\xi_0))\chi_{\rho_0}(x_0,\xi_0).\end{equation}
	The symbols $a_j^{(n-1)}$ are estimated in the following way: 
	\begin{equation} \label{eq: borne ajk avec Cjl} \|a_j^{(n-1)}\|_{C^\ell_{t,x_n,x_0,\xi_0}}\le C_{j,\ell} h^{-\beta(\ell+2j)}\mathrm{e}^{C(\ell +2j)T}. \end{equation}
	In particular, if $n\le T\le \mathfrak c(\alpha,\beta,\varepsilon)|\log h|$, we have
	\[\|a_h^{(n-1)}\|_{C^\ell}\le C_{n,\ell} h^{-\sigma \ell}.\]
			\end{prop}
			\begin{rem}The fact that we take the same parameter $\nu$ in the assumption $|t|\le 2h^\nu$ and $\chi_\rho\in S_\nu(T^*M)$ is pretty much irrelevant, but it has the advantage to reduce the number of letters used.
			\end{rem}

			\begin{proof}The result for $n=1$ is the content of Proposition \ref{thm: short time propagator twisted}. Let us now assume that the result holds up to index $n<T$. Then, by the induction hypothesis and the result for short times, we have, for small $|t|$:
				\[W_{n}(1+t)W_{n-1}(1)\cdots W_0(1)=\mathcal I[\varphi,b_h^{(n)}]\circ \mathcal I[\phi_n,a_h^{(n-1)}]+\mathcal O(h^{N(1-2\sigma)-C})_{L^2/\rm tr}.\]
				Here $\varphi(t,x_{n+1},\xi_n)$ is a generating function for the flow at time $1+t$ and $\phi_n(x_n,\xi_0)$ is a generating function for the flow at time $n$, while $b_h^{(n)}$ and $a_h^{(n-1)}$ are the appropriate amplitudes.
				The generating functions $\varphi$ and $\phi_n$ compose correctly provided $T\le \mathfrak c|\log h|$ with $\mathfrak c$ small enough, thanks to Proposition \ref{prop: iterated generating function}. We can then use the composition Theorem \ref{thm: composition FIO} to compute $\mathcal I[\varphi,b_h^{(n)}]\circ \mathcal I[\phi_n,a_h^{(n-1)}]$ and obtain
				\[W_{n}(1+t)W_{n-1}(1)\cdots W_0(1)=\mathcal I[\phi_{n+1},a_h^{(n)}]+\mathcal O(h^{N(1-2\sigma)-C})_{L^2/\rm tr},\]
				where in the notation of Theorem \ref{thm: composition FIO}:
				\begin{itemize}
					\item $\phi_{n+1}(t,x_{n+1},\xi_0)$ is obtained by the formula of Lemma \ref{lem: composition generating functions} and generates the flow at time $n+1+t$.
					\item We have a finite expansion $a_h^{(n)}=\sum_{j<N} h^j a_j^{(n)}$, with
					\begin{equation}\label{eq: recursion aj3} a_j^{(n)}:=\sum_{m+i+k=j} \sum_{|\alpha|\le 2m} c_{m,\alpha}(\mathbf x)\partial_z^\alpha \big(b_i^{(n)} a_k^{(n-1)}\big)\Big|_{z=z_c(\mathbf x)}.\end{equation}
					The functions $c_{m,\alpha}$ and $z_c$ are estimated in Theorem \ref{thm: composition FIO} and Proposition \ref{prop: intermediate points} respectively:
					\begin{equation} \label{eq: bounds cj zc in proof long time} \|c_{m,\alpha}\|_{C^\ell_{\mathbf x}}\le C_{m,\ell} h^{-\beta(\ell+2m-|\alpha|)}\mathrm{e}^{CT(\ell+2m-|\alpha|)}, \qquad \|z_c\|_{C^\ell}\le C_{\ell} T h^{-\beta(\ell-1)_+}\mathrm{e}^{C(\ell-1)T}.\end{equation}
				\end{itemize}
				Recall the formula \eqref{eq: leading part symbol composition} for the leading-order symbol $a_0^{(n)}$. Let $\psi_{h}(t,x_{n+1},\xi_n,x_n,\xi_0):=\varphi(t,x_{n+1},\xi_n)-\langle x_n,\xi_n\rangle+ \phi_n(x_n,\xi_0)$ denote the phase function appearing in the stationary phase computation. To obtain \eqref{eq: the first term is given by Ehrenfest} at rank $n+1$, we need to show that
				\begin{equation}\label{eq: relations hessiennes recurrence}\frac{|\partial_{x_{n+1}, \xi_n}^2 \varphi|^{\frac 12}\cdot |\partial_{x_n,\xi_0}^2 \phi_n|^{\frac 12}}{|\nabla^2_{z_n} \psi_h|^{\frac 12}}=|\partial_{x_{n+1},\xi_0}\phi_{n+1}|^{\frac 12}.\end{equation}
				Here the Hessian of $\psi_h$ on the left-hand side is taken with respect to the $z_n=(\xi_n,x_n)$ variables, and evaluated at the critical point $z_n^c(x_{n+1},\xi_0):=(\xi_n^c(x_{n+1},\xi_0),x_n^c(x_{n+1},\xi_0))$. To see \eqref{eq: relations hessiennes recurrence}, let
				\[\mathrm{d}\Phi^{1+t}=\begin{pmatrix}
					A & B \\ C & D
				\end{pmatrix}, \qquad \mathrm{d}\Phi^n=\begin{pmatrix}
					P & Q \\ R & S
				\end{pmatrix}.\]
				Then, one has
				\[\partial_{x_{n+1}, \xi_n}^2 \varphi=\frac{\partial x_n(x_{n+1},\xi_n)}{\partial {x_{n+1}}}=A^{-1}, \qquad \partial_{x_n,\xi_0}^2 \phi_n=\frac{\partial x_0(x_n,\xi_0)}{\partial x_n}=P^{-1}.\]
				The hessian of $\psi_h$ in the $z_n=(x_n,\xi_n)$ variable can be expressed
				\[\nabla^2_{z_n}\psi_h=\begin{pmatrix} \partial_{\xi_n\xi_n}^2 \varphi & -\operatorname{Id} \\ -\operatorname{Id} & \partial_{x_nx_n}^2 \phi_n
				\end{pmatrix}=\begin{pmatrix}
					-A^{-1}B & -\mathrm{Id} \\ -\mathrm{Id} & RP^{-1}
				\end{pmatrix}.\]
				We deduce
				\[\det(\nabla^2_{z_n}\psi_h)=\det(-\mathrm{Id} -A^{-1}BRP^{-1})=\det(A^{-1})\det(P^{-1})\det(AP+BR).\]
				Since the matrix $AP+BR$ is the upper left block of $\mathrm{d}\Phi^{n+1+t}$, it is nothing but $\partial_{x_{n+1},\xi_0}^2 \phi_{n+1}$. This implies altogether the identity \eqref{eq: relations hessiennes recurrence}. The estimates on the symbol $a_j^{(n)}$ are proved in the next subsection.

			\end{proof}
			
			\subsubsection{Estimating the $C^\ell$ norms of the symbol $a_j^{(n)}$.} We prove the estimates on the symbol. Assume $|t|\le T$ with $T\ge 1$.
			\begin{prop}\label{prop: estimees aj n} While $n\le T$, the $j$-th coefficient $a_j^{(n)}$ satisfies
				\begin{equation} \label{eq: bound ajn long time}\|a_j^{(n)}\|_{C^\ell_{t,x_{n+1},\xi_0}}\lesssim |a_0^{(n)}|h^{-\beta(\ell+2j)}\mathrm{e}^{C(\ell+2j)T}.\end{equation}
			\end{prop}
			This result, and its proof, are similar to \cite[Proposition 4.1]{NZ09} or \cite[Lemma 13]{Schenk}. We denote by $z_p=(x_p,\xi_p)$ the points in the $p$-th representation of the short-time propagator. Given end points $x_{n},\xi_0$, for each $p\in \{0,\ldots,n\}$, there is a point $z_p=z_p^c(x_n,\xi_0)$ fixed by the condition
			\[(\bullet,\xi_0)\overset{\Phi^p}{\longrightarrow} z_p\overset{\Phi^{n-p}}{\longrightarrow}(x_n,\bullet).\]
			We have shown in Proposition \ref{prop: intermediate points} that for all $p\in \{0,\ldots,n\}$, one has
			\[\|z_p\|_{C^\ell_{x_n,\xi_0}}\lesssim Th^{-\beta(\ell-1)_+}\mathrm{e}^{C(\ell-1)_+T}. \]
			Recall that for $t$ close to $0$, we have
			\begin{equation} \label{eq: rappel formule aj n} a_j^{(n)}=\sum_{\substack{m+i+k=j \\ k\le j}} \sum_{|\alpha|+|\beta|\le 2m} c_{m,\alpha,\beta}^{(n)} \big(\partial_{z_n}^\alpha b_i^{(n)}\cdot \partial_{z_n}^{\beta} a_k^{(n-1)}\big)\Big|_{z_n=z_n^c(t,x_{n+1},\xi_0)}. \end{equation}
			In this expression:
			\begin{itemize}
				\item $a_j^{(n)}$ is evaluated at $(n+1+t,x_{n+1},\xi_0)$
				\item $c_{m,\alpha,\beta}^{(n)}$ is evaluated at $(1+t,x_{n+1},x_n,\xi_n,\xi_0)$, and satisfies by  \eqref{eq: estimee coefficient derivee j composition FIO}:
				\begin{equation} \label{eq: cm alpha beta}\|c_{m,\alpha,\beta}^{(n)}\|_{C^\ell}\lesssim h^{-\beta(\ell+2m-|\alpha|-|\beta|)} \mathrm{e}^{C(\ell+2m-|\alpha|-|\beta|)T}.\end{equation}
				Here we have slightly increased $C$ to obtain \eqref{eq: cm alpha beta} from \eqref{eq: estimee coefficient derivee j composition FIO}.
				\item $b_i^{(n)}$ is evaluated at $(1+t,x_{n+1},\xi_n)$,
				\item $a_k^{(n-1)}$ is evaluated at $(n,x_n,\xi_0)$.
			\end{itemize} A direct induction on $n$ using \eqref{eq: rappel formule aj n} leads to poor estimates on $\|a_j^{(n)}\|_{C^\ell}$. For this reason, we first provide a nicer expression of $a_j^{(n)}$ in the spirit of a discrete Duhamel formula. To simplify notation, we forget about the cutoffs that have been inserted along the orbit in the formulas. Taking them into account does not affect the estimates.
			For $n\ge 1$, given a tuple $\gamma_n:=(m_n,\alpha_n,\beta_n,i_n)$, we let
			\begin{equation} \label{eq: defi d gamma n } d_{\gamma_n}^{(n)}(x_{n+1},x_n,\xi_n,\xi_0):=c_{m_n,\alpha_n,\beta_n}^{(n)} \big(\partial_{z_n}^{\alpha_n} b_{i_n}^{(n)}\big).\end{equation}
			For $n=0$, we also set 
			\begin{equation} \label{eq: defi d0} d^{(0)}(x_1,x_0,\xi_0)=a_{k_1}^{(0)}.\end{equation} By iterating down \eqref{eq: rappel formule aj n}, one has
			\begin{equation} \label{eq: by iterating down}
				\begin{aligned}
					a_j^{(n)}=\sum_{(\vec m,\vec \alpha,\vec\beta,\vec i),k_1} d_{\gamma_n}^{(n)}\partial_{z_n}^{\beta_n}\big[
					d_{\gamma_{n-1}}^{(n-1)} \partial_{z_{n-1}}^{\beta_{n-1}}\big[\cdots \partial_{z_2}^{\beta_2}\big[ d_{\gamma_1}^{(1)}\partial_{z_1}^{\beta_1}a_{k_1}^{(0)}\big]_{z_1=z_1^c(x_2,\xi_0)}\big]_{z_2=z_2^c(x_3,\xi_0)}\cdots\big]_{z_n=z_n^c(x_{n+1},\xi_0)}
			\end{aligned}\end{equation}
			where the sum runs over tuples $\vec m=(m_1,\ldots,m_n)$, $\vec \alpha=(\alpha_1,\ldots,\alpha_n)$, $\vec \beta=(\beta_1,\ldots,\beta_n),k_1$ that satisfy $k_1+\sum m_p+i_p=j$, $\sum |\alpha_p|+|\beta_p|\le 2\sum m_p$. Due to these conditions, there are at most $j$ indices $p\in \{1,\ldots,n\}$ such that $(m_p,i_p)\neq (0,0)$, therefore the sum \eqref{eq: by iterating down} runs over $\mathcal O_j(n^{Cj})$ tuples $(\vec m,\vec \alpha,\vec i,k_1)$ for some positive $C$. Fix such a tuple and let
			\[n\ge n_j>n_{j-1}>\cdots >n_1\ge 1\]
			be the indices such that $(m_{n_p},i_{n_p})\neq 0$. Also, let $n_{j+1}=n$ and $n_0=0$. For $n_j<p<n_{j+1}$, we can gather the terms $d_{\gamma_p}^{(p)}=d_0^{(p)}$. For $n<n'$ define
			\begin{equation} \label{eq: defi Dn'n} \mathcal D^{n',n}(x_{n'},\xi_0):=\prod_{n<p<n'} d_{0}^{(p)}(x_{p+1},\xi_p)_{x_{p+1}=x_{p+1}^c(x_{n'},\xi_0)}=\frac{a_0^{(n'-1)}(x_{n'},\xi_0)}{a_0^{(n)}(x_{n+1},\xi_0)}\Big|_{x_{n+1}=x_{n+1}^c(x_n',\xi_0)}. \end{equation}
			Also, let 
			\begin{equation} \label{eq: def f n q +1}f_{n_{q+1}}=f_{n_{q+1}}(x_{n_{q+1}},\xi_0):=\mathcal D^{n_{q+1},n_q}(x_{n_{q+1}},\xi_0)d_{\gamma_{n_q}}^{(n_q)}(x_{n_{q}+1},x_{n_q},\xi_{n_q},\xi_0),\end{equation}
			where the right-hand side is evaluated at $x_{n_q+1}^c,x_{n_q}^c,\xi_{n_q}^c$ fixed by the end points $x_{n_{q+1}}$ and $\xi_0$. Beware of the difference between $n_{q+1}$ and $n_q+1$. Then, the summands on the right-hand side of \eqref{eq: by iterating down} can be written
			\[f_{n_{j+1}}\big[\partial_{z_{n_j}}^{\beta_{n_j}}\big[ f_{n_j}\big[\partial_{z_{n_{j-1}}}^{(n_{j-1})}\big[\cdots \partial_{z_{n_2}}^{\beta_{n_2}}\big[f_{n_2}\big[\partial_{z_{n_1}}^{\beta_{n_1}}f_{n_1}\big]_{z_{n_1}=z_{n_1}^c(z_{n_2})}\big]_{z_{n_2}=z_{n_2}^c(z_{n_3})}\cdots\big]_{z_{n_j}=z_{n_j}^c(z_{n_{j+1}})}.\]
			This quantity is equal to $F_{n_{j+1}}(x_{n_{j+1}},\xi_0)$, where $(F_{n_q})_{q\ge 1}$ is defined by induction:
			\begin{equation}\left\{\begin{array}{ll}F_{n_1}(x_{n_{1}},\xi_0)=f_{n_1}(x_{n_1},\xi_0),&  \\
					F_{n_{q+1}}(x_{n_{q+1}},\xi_0)=f_{n_{q+1}}\big[\partial_{z_{n_q}}^{\beta_{n_q}} F_{n_q}\big]_{z_{n_q}=z_{n_q}^c(z_{n_{q+1}})}& \text{for $q\ge 1$}\end{array}\right..
			\end{equation}
			The functions $F_{n_q}$ will be estimated in Proposition \ref{prop: bound Fq} below. Before, we show
			\begin{lem}\label{lem: bound the fj} Set $n_0=0$, $i_{n_0}=k_1$ and $m_{n_0}=\alpha_{n_0}=\beta_{n_0}=0$. For all $q\ge 0$, the function $f_{n_{q+1}}:=\mathcal D^{n_{q+1},n_q}d_{\gamma_{n_{q}}}^{(n_{q})}$, viewed as a function of $(x_{n_{q+1}},\xi_0)$, satisfies
				\[\|f_{n_{q+1}}\|_{C^\ell_{x_{n_{q+1}},\xi_0}}\lesssim |\mathcal D^{n_{q+1},n_q}| h^{-\beta (\ell+2m_{n_q}+2i_{n_q}-|\beta_{n_q}|)} T^\ell \mathrm{e}^{CT(\ell+2m_{n_q}-|\beta_{n_q}|)}.\]
			\end{lem}
			Up to increasing $C$, we can absorb the polynomial factor $T^\ell$ into the exponential, which we shall do in the following.
			\begin{proof}
				
				\textbf{1.} For $q=0$, we have $f_{n_1}=\mathcal D^{n_1,0}d^{(0)}=\mathcal D^{n_1,n_0}a_{k_1}^{(0)}$ and our convention for $\gamma_{{n_0}}$ gives the result with the same computations as below.
				
				\textbf{2.} We consider the case $q\ge 1$. First recall from \eqref{eq: defi Dn'n} that
				\[\mathcal D^{n_{q+1},n_q}(x_{n_{q+1}},\xi_0):=\frac{a_0^{(n_{q+1}-1)}(x_{n_{q+1}},\xi_0)}{a_0^{(n_q)}(x_{n_q+1},\xi_0)}\Bigg|_{x_{n_q+1}=x_{n_q+1}^c(x_{n_{q+1}},\xi_0)}.\]
				Because $\|a_0^{(n)}\|_{C^\ell}\lesssim a_0 h^{-\beta \ell}\mathrm{e}^{C\ell T}$ (this is Proposition \ref{prop: derivees a0}), the right-hand side, viewed as a function of the variables $x_{n_{q+1}},x_{n_q+1},\xi_0$ (where we have not evaluated $x_{n_{q}+1}=x_{n_q+1}^c(x_{n_{q+1}},\xi_0)$) has derivatives of order $\ell$ bounded by a constant times
				\[|\mathcal D^{n_{q+1},n_q}(x_{n_{q+1}},\xi_0)|\cdot h^{-\beta\ell}\mathrm{e}^{C\ell T}.\]
				Similarly, viewing $d_{\gamma_{n_q}}^{(n_q)}$ as a function of $x_{n_q+1},x_{n_q},\xi_{n_q},\xi_0$, we have
				\[\|d_{\gamma_{n_q}}^{(n_q)}\|_{C^\ell_{x_{n_q+1},x_{n_q},\xi_{n_q},\xi_0}}\lesssim h^{-\beta(\ell+2m_{n_q}+2i_{n_q}-|\beta_{n_q}|)}\mathrm{e}^{CT (\ell+2m_{n_q}-|\alpha_{n_q}|-|\beta_{n_q}|)}.\]
				Note that it still holds for $q=0$. To obtain the bounds on $\|f_{n_{q+1}}\|_{C^\ell_{x_{n_{q+1}},\xi_0}}$, we use Fa\`a di Bruno's formula together with the bound $\|x_n\|_{C^\ell_{x_{n'},\xi_0}}\lesssim T h^{-\beta(\ell-1)_+}\mathrm{e}^{C(\ell-1)T}$ for $n\le n'$ to obtain
				\begin{align*}
					\|f_{n_{q+1}}\|_{C^\ell_{x_{n_{q+1}},\xi_0}}& \lesssim |\mathcal D^{n_{q+1},n_q}(x_{n_{q+1}},\xi_0)|\times \\ & \sum_{r=1}^\ell h^{-\beta (r+m_{n_q}+2i_{n_q}-|\beta_{n_q}|)} \mathrm{e}^{CT(r+2m_{n_q}-|\alpha_{n_q}|-|\beta_{n_q}|)}\prod_{\ell_1+\ldots+\ell_r=\ell} \big(T h^{-\beta(\ell_i-1)}\mathrm{e}^{C(\ell_i-1)T}\big) \\
					& \lesssim |\mathcal D^{n_{q+1},n_q}| h^{-\beta (\ell+2m_{n_q}+2i_{n_q}-|\beta_{n_q}|)} \sum_{r=1}^\ell T^r \mathrm{e}^{CT(\ell+2m_{n_q}-|\beta_{n_q}|)} \\
					& \lesssim |\mathcal D^{n_{q+1},n_q}| h^{-\beta (\ell+2m_{n_q}+2i_{n_q}-|\beta_{n_q}|)} T^\ell \mathrm{e}^{CT(\ell+2m_{n_q}-|\beta_{n_q}|)}.
				\end{align*}
			
			\end{proof}
		
			\begin{lem} \label{prop: bound Fq} Let $C'>C$ where $C$ is as in the statement of Lemma \ref{lem: bound the fj}. Define $M_q:=\sum_{q'\le q} m_{n_{q'}}$ and $I_q=\sum_{q'\le q} i_{n_{q'}}$. Then for all $q\ge 0$:
				\begin{equation} \label{eq: C ell norm F_q} \|F_{n_{q+1}}\|_{C^\ell_{x_{n_{q+1}},\xi_0}}\lesssim_{\!\! q} |a_0^{(n_{q+1})}|h^{-\beta(\ell+2M_{q}+2I_{q})} \mathrm{e}^{C'T(\ell+2M_q+2I_q)}.\end{equation}
			\end{lem}
			\begin{proof}For $q=0$, this is Lemma \ref{lem: bound the fj}. Assume that the results holds up to index $q-1$. Recall
				\[F_{n_{q+1}}=f_{n_{q+1}}\big[\partial_{z_{n_q}}^{\beta_{n_q}} F_{n_q}\big]_{z_{n_q}=z_{n_q}^c(z_{n_{q+1}})}.\]
				By the induction hypothesis, one has
				\[\|\partial_{z_{n_q}}^{\beta_{n_q}} F_{n_q}\|_{C^\ell}\lesssim a_0^{(n_q)}T^{q+1} h^{-\beta(\ell+2M_{q-1}+2I_{q-1}+|\beta_{n_q}|)}\mathrm{e}^{C'T(\ell+2M_{q-1}+2I_{q-1}+|\beta_{n_q}|)},\]
				and by Lemma \ref{lem: bound the fj}, 
				\[\|f_{n_{q+1}}\|_{C^\ell}\lesssim |\mathcal D^{n_{q+1},n_q}|h^{-\beta(\ell+2m_{n_q}+2i_{n_q}-|\beta_{n_q}|)} \mathrm{e}^{C'T(\ell+2m_{n_q}+2i_{n_q}-|\beta_{n_q}|)}.\]
				Since $m_{n_q}+M_{q-1}=M_q$ and $i_{n_q}+I_{q-1}=I_q$, by taking the product and using Fa\`a di Bruno's formula, we deduce (note that the terms $-|\beta_{n_q}|$ and $+|\beta_{n_q}|$ cancel out):
				\[\|F_{n_{q+1}}\|_{C^\ell}\lesssim |\mathcal D^{n_{q+1},n_0}| h^{-\beta(\ell+2M_q+2I_q)}\mathrm{e}^{C'T(2M_q+2I_q)}\sum_{r=1}^\ell \mathrm{e}^{C'rT}\sum_{\ell_{1}+\ldots+\ell_r=\ell}\prod_{1\le i\le r} T \mathrm{e}^{C(\ell_i-1)T}.\]
				The sum on the right-hand side is smaller than
				\[\sum_{r=1}^\ell \mathrm{e}^{C'rT}T^r \mathrm{e}^{C(\ell-r)T}\lesssim \mathrm{e}^{C'\ell T},\]
				since we have taken $C'>C$.
			\end{proof}
			
			We finally turn to the proof of Proposition \ref{prop: estimees aj n}.
			\begin{proof}[Proof of Proposition \ref{prop: estimees aj n}] For $j=0$, this is Proposition \ref{prop: derivees a0} so there is nothing to do. Thus, we can assume $j\ge 1$.
				
				\textbf{1.} We first consider the case where no derivative falls on the time variable. Then we can invoke Proposition \ref{prop: bound Fq} with $q=j$ (recall that $n_{j+1}=n$ and $M_j+I_j=j$) to obtain
				\begin{equation} \label{eq: bound F n j +1}\|F_{n_{j+1}}\|_{C^\ell}\le C_{j,\ell} h^{-\beta(\ell+2j)} \mathrm{e}^{C'(\ell+2j) T}.\end{equation}
				We sum the $\mathcal O(T^{Cj})$ terms from the sum \eqref{eq: by iterating down}, each of which is bounded as in \eqref{eq: bound F n j +1}. The polynomial factor $T^{Cj}$ can be absorbed in the exponent and we obtain \eqref{eq: bound ajn long time}. 
				
				\textbf{2.} To incorporate time derivatives, one can simply use, for $t$ close to $1$: 
				\[a_j^{(n)}(n+t,x_{n+1},\xi_0)=\sum_{\substack{m+i+k=j \\ k\le j}} \sum_{|\alpha|+|\beta|\le 2m} c_{m,\alpha,\beta}^{(n)} \big(\partial_{z_n}^\alpha b_i^{(n)}(t)\cdot \partial_{z_n}^{\beta} a_k^{(n-1)}(n)\big)\Big|_{z_n=z_n^c(t,x_{n+1},\xi_0)}.\]
				Then the result follows from the estimates on $\|b_i^{(n)}\|_{C^\ell_{t,x_{n+1},x_n}}$, on $\|a_k^{(n-1)}\|_{C^\ell_{x_n,\xi_0}}$ obtained in \textbf{1.}, and on $\|z_n\|_{C^\ell_{t,x_{n+1},\xi_0}}$ from Proposition \ref{prop: intermediate points}.
			\end{proof}

			\subsection{The long time trace formula} \label{sec: long time trace} In this section, we prove the long time trace formula of Proposition \ref{prop: trace A n gamma j}. Recall that $\gamma$ is a periodic orbit of the flow, with period $T_\gamma$, that $\rho_0$ is a point of the orbit and $\chi_{\rho_0}\in S_\nu(T^*M)\cap C^\infty_{\rm comp}(B(\rho,h^\nu))$ is a cutoff localizing in a small ball.

			By \eqref{eq: by taking the traces} and Proposition \ref{prop: propag until Ehrenfest}, we know that for times $|s|\le 2h^\nu$, one has
			\[\operatorname{Tr} U(T_\gamma+s) \mathrm{Op}_h^{M}(\chi_{\rho_0})=\mathrm{e}^{-\mathrm{i}\frac{\pi}{2}\mathfrak m_\gamma}\operatorname{Tr} \mathcal I[\varphi_\delta,a_h^\delta]+\mathcal O(h^{N(1-2\sigma)-C}),\]
			where $\varphi_\delta$ is a generating function for the flow $\mathcal A_{\ell_0}(0)^{-1}\circ \widetilde f_{\ell_0}\circ \Phi^{T_\gamma+s}_\delta\circ \widetilde f_{\ell_0}^{-1}\circ \mathcal A_{\ell_0}(0)$. \medskip
			
			\textbf{1.} To compute the trace, it is convenient to first straighten the piece of unperturbed orbit starting from $\rho_0$. Recalling Lemma \ref{lem: straightening coordinates}, let $\kappa$ be a local symplectomorphism of $T^*\mathbf{R}^d$, which is tangent to the identity at $\tilde \rho_0:=\mathcal A_{\ell_0}(0)^{-1}\widetilde f_{\ell_0}(\rho_0)$, and such that in the new coordinates $(\tilde x,\tilde \xi)=\kappa(x,\xi)$, the short-time unperturbed flow 
			\[\kappa\circ \mathcal A_{\ell_0}(0)^{-1}\circ \widetilde f_{\ell_0}\circ \Phi^s_0\circ \widetilde f_{\ell_0}^{-1}\circ \mathcal A_{\ell_0}(0)\circ \kappa^{-1}\]
			becomes the $s$-translation along the $\tilde x_1$ coordinate, for times $|s|\le 1$. Since $\mathrm{d}\kappa$ is tangent to the identity at $\tilde \rho_0$, the graph of $\kappa$ is locally projectable hence by using the results of \S \ref{subsec: egorov time dep}, we can build a local semiclassical quantization $\mathcal T_\kappa$ of $\kappa$ such that $\mathcal T_{\kappa}^*\mathcal T_\kappa=\mathrm{Id}$. Moreover, $\mathcal T_\kappa$ and $\mathcal T_\kappa^*$ both admit local FIO representations 
			\[\mathcal T_\kappa=\mathcal I[\psi_1,\alpha_h^{(1)}]+\mathcal O(h^\infty)\qquad \text{and}\qquad \mathcal T_\kappa^*=\mathcal I[\psi_{2},\alpha_h^{(2)}]+\mathcal O(h^\infty),\] where $\psi_1,\psi_2$ are generating functions for $\kappa$ and $\kappa^{-1}$ respectively. The composition of these generating functions is simply $\langle x,\eta\rangle$. Since $\mathcal T_\kappa$ is unitary, one has
			\[\operatorname{Tr} \mathcal I[\varphi_\delta,a_h^\delta]=\operatorname{Tr} \mathcal T_{\kappa}^* \mathcal I[\varphi_\delta,a_h^\delta]\mathcal T_\kappa.\]
			The composition formula for FIOs shows that the conjugated operator $\mathcal T_{\kappa}^* \mathcal I[\varphi_\delta,a_h^\delta]\mathcal T_\kappa$ has a further FIO representation
			\[\mathcal T_{\kappa}^* \mathcal I[\varphi_\delta,a_h^\delta]\mathcal T_\kappa=\mathcal I[\tilde \varphi_\delta,\tilde a_h^\delta]+\mathcal O(h^N),\]
			where $\tilde \varphi_\delta$ is a generating function of the perturbed flow
			\[\kappa\circ \mathcal A_{\ell_0}(0)^{-1}\circ \widetilde f_{\ell_0}\circ \Phi^{t}_\delta\circ \widetilde f_{\ell_0}^{-1}\circ \mathcal A_{\ell_0}(0)\circ \kappa^{-1},\]
			obtained by composing $\psi_2$, $\varphi_\delta$ and $\psi_1$. The leading symbol is given by $\tilde a_0=|\partial_{x\eta}^2\tilde \varphi_\delta|^{\frac 12}\chi_{\rho_0}(x,\eta)$. Here again we write $\chi_{\rho_0}$ instead of the function in coordinates by an abuse of notation.
			
			\textbf{2.} We now compute the trace. We drop the tilde on $x$ and $\eta$ to alleviate notation. Let $\Psi_\delta(t,x,\eta):=\tilde \varphi_\delta(t,x, \eta)-\langle x,\eta\rangle+ tE$ and $u_h(t,x,\eta):=\tilde a_h^\delta(t,x,x,\eta) \widehat f(t/L)\omega_\nu(t-t_n)$ denote the phase and amplitude of the propagator, then
			\[\operatorname{Tr}\mathcal I[\tilde \varphi,\tilde a_h^\delta]=\frac{1}{(2\pi h)^d}\int_{\mathbf R^{2d+1}} \mathrm{e}^{\frac{\mathrm{i}}{h} \Psi_\delta(t,x,\eta)}u_h(t,x,\eta)\mathrm{d}x\mathrm{d}\eta\mathrm{d}t.\]
			We perform a stationary phase expansion. The critical manifold is given by 
			\begin{equation} \label{eq: the crit manifold is}\big\{(x,\eta)\in \mathrm{Supp}(\chi_{\rho_0}), \ t\in \mathrm{Supp}(\omega_\nu(\cdot-t_n))~:~ \partial_t \tilde \varphi_\delta=E, \ x=\nabla_\eta \tilde \varphi_\delta, \ \eta=\nabla_x \tilde \varphi_\delta\big\}.\end{equation}
			Because $\tilde \varphi_\delta$ is a generating function for the flow, the condition \eqref{eq: the crit manifold is} means that $(x,\eta)$ is a periodic point of period $t$ of the flow $\Phi^t_\delta$. When $\delta=0$, by the separation property of periodic orbits---and because in the new coordinates the unperturbed flow is just the $s$-translation along the coordinate $x_1$---the critical set is given by $\{t=T_\gamma, \xi_1=E, \ (x_i, \xi_i)_{i\ge 2}=0 \}$. With this fact in mind, we rewrite the trace as
			\[\operatorname{Tr}\mathcal I[\tilde \varphi_\delta,\tilde a_h^\delta]=\frac{1}{(2\pi h)^d}\int_{x_1} \left( \int_{t,x', \eta}  \mathrm{e}^{\frac{\mathrm{i}}{h} \Psi_\delta(t,x,\eta)}u_h(t,x,\eta) \mathrm{d} x'\mathrm{d} \eta \mathrm{d}t\right)\mathrm{d} x_1+\mathcal O(h^\infty).\]
			We view $x_1$ as a fixed parameter and consider the integral in the variable $z=(t,x',\eta)$. In the unperturbed case, there is a unique critical point $z_0=(T_\gamma,E,0)$ of the phase. We set $F(z,\delta)=\nabla_z\Psi$. Let us check the assumptions of the quantitative implicit function theorem \ref{lem: implicit func}. We fix $r_1,r_2>0$, to be chosen below.
			\begin{itemize}
				\item Let $(z,\delta)\in B(z_0,r_1)\times B(0,r_2)$. We have
				\[\|\operatorname{Id}-(\partial_zF(z_0,0))^{-1}\partial_zF(z,\delta)\|\le \|\partial_z F(z_0,0)^{-1}\|\cdot \big(\sup_z\|\partial_z^2 F(z,0)\|r_1+\sup_{z,\delta}\|\partial_\delta\partial_z F\|r_2\big).\]
				By Proposition \ref{prop: intermediate points}, we have $\|\partial_z F(z_0,0)^{-1}\|\lesssim 1$. Moreover, by the estimates on the generating function from Proposition \ref{prop: control all deriv gen fun}, we have
				\[\|\partial_z^2 F\|r_1+\|\partial_\delta\partial_z F\|r_2\lesssim r_1 \mathrm{e}^{Ct}+r_2 \mathrm{e}^{Ct}h^{-2\beta}.\]
				Let $\varepsilon$ satisfy $\mathrm{e}^{CT}=h^{-\varepsilon}$ then define $r_1=ch^{\varepsilon}$ and $r_2=ch^{2\beta+\varepsilon}$ for some small enough constant $c$. This ensures
				\[\|\operatorname{Id}-(\partial_z F(z_0,0))^{-1}\partial_zF(z,\delta)\|\le \frac 12,\]
				for $(z,\delta)\in B(z_0,r_1)\times B(0,r_2)$
				\item We have
				\[C_{r_1,r_2}=\sup_{|z-z_0|\le r_1, \ |\delta|\le r_2} \|\partial_\delta F\|\lesssim h^{-\beta} h^{-\varepsilon},\quad \|M\|=\|\partial_z F(z_0,0)^{-1}\|\lesssim 1,\]
				so that
				\[\frac{r_1}{2C_{r_1,r_2}M}\gtrsim \frac{ch^{\varepsilon}}{2 h^{-\beta-\varepsilon}}\gtrsim h^{\beta+2\varepsilon}.\]
				
			\end{itemize}
			By Lemma \ref{lem: implicit func}, we deduce that for all $|\delta|\le h^{\beta+2\varepsilon}$, there is a unique critical point $z_\delta=(T_\delta,x'_\delta,\eta_\delta)$ of the phase (which depends additionally on $x_1$), moreover
			\begin{equation} \label{eq: derivee zdelta}\partial_\delta z_\delta=-\big(\nabla^2_z \Psi_\delta(z_\delta)\big)^{-1} \partial_\delta \nabla_z \Psi_\delta(z_\delta).\end{equation}
			After lifting back to the manifold $T^*M$, the point $(x_1,x_\delta',\eta_\delta)$ corresponds to a periodic point of the perturbed flow $\Phi^t_\delta$, lying on some periodic orbit $\gamma_\delta$ with period $T_\delta$, contained in the energy level $p_\delta^{-1}(E)$. Therefore, the phase $\Psi_\delta$ at the critical point coincides with the action of $\gamma_\delta$, which we denote $S(\gamma_\delta)$. A byproduct of the proof is that any periodic orbit $\gamma_0$ of the unperturbed flow gives rise to a unique continuous deformation of periodic orbits $\{\gamma_\delta\}$ of the perturbed flow.

			We now perform a stationary phase expansion in the $z$ variable using Proposition \ref{prop: stat phase bis}, and find
			\[
			\operatorname{Tr}\mathcal I[\tilde \varphi,\tilde a_h^\delta]=\mathrm{e}^{\frac{\mathrm{i}}{h}S(\gamma_\delta)}\int_{x_1} \frac{a_0^\delta(z_\delta(x_1),x_1)}{|\nabla^2_z\Psi_\delta(z_\delta)|}\widehat f(T_\delta/L) \omega_\nu(T_\delta-t_n)\mathrm{d}x_1+O(h^{1-2\beta -\varepsilon}).
			\]
			Here we have stressed the dependence of the amplitude in $\delta$. We replace the amplitude by the unperturbed amplitude. Proposition \ref{prop: control all deriv gen fun} and the estimate 
			$|\partial_\delta z_\delta|\lesssim h^{-\beta}\mathrm{e}^{2C t}$ (which follows from \eqref{eq: derivee zdelta}) imply
			\[ \frac{a_0^\delta(z_\delta(x_1),x_1)}{|\nabla^2_z\Psi_\delta(z_\delta)|}\widehat f(T_\delta/L) \omega_\nu(T_\delta-t_n)= \frac{a_0^0(z_0(x_1),x_1)}{|\nabla^2_z\Psi_0(z_0)|}\widehat f(T_\gamma/L) \omega_\nu(T_\gamma-t_n)+\mathcal O(\delta h^{-2\beta-\varepsilon}),\]
			again for $T_\gamma\le \mathfrak c |\log h|$ where $\mathfrak c$ may have decreased.
			At this point, we have obtained
			\[\operatorname{Tr}\mathcal I[\tilde \varphi,\tilde a_h^\delta]=\mathrm{e}^{\frac{\mathrm{i}}{h}S(\gamma_\delta)}\widehat f(T_\gamma/L) \omega_\nu(T_\gamma-t_n)\int_{x_1} \frac{a_0^\delta(z_0(x_1),x_1)}{|\nabla^2_z\Psi_0(z_0)|}\mathrm{d}x_1+O(h^{\alpha-2\beta -\varepsilon}).\]
			Since in the new coordinates the short-time unperturbed flow is just the translation along the $x_1$ coordinate, the Poincar\'e map $\mathrm{d}\Phi^{T_\gamma}_0(\Phi^s(\rho))$ in the new coordinates is independent of $|s|\le 1$; thus it is given by the matrix at $s=0$ in the basis $(\partial_{x_1},\partial_{\xi_1}; (\partial_{x_i})_{i\ge 2},(\partial_{\xi_j})_{j\ge 2})$, which is, by Lemma \ref{lem: matrice differentielle flot}, of the form
			\begin{equation} \label{eq: is given by the matrix at 0}\mathrm{d}\Phi_0^{T_\gamma}=\begin{pmatrix}
					1 & \beta_{T_\gamma} & 0 & 0 \\ 
					0 & 1  & 0 & 0 \\
					0 & 0 & A & 0 \\
					0 & 0 & 0 & A^{-\top} 
			\end{pmatrix}\end{equation}
			at the point $\rho_0$ in coordinates. Note that we have $\operatorname{Span}\big(\{\partial_{x_i},\partial_{\xi_i}\}_{i\ge 2}\big)=E_s\oplus E_u$ at $\rho_0$, in coordinates.
			In particular the lower right block of \eqref{eq: is given by the matrix at 0} is conjugated to the linearized Poincar\'e map of the orbit, which is denoted by $P_\gamma$.
			The Hessian of $\Psi_0$ in the $(t,\xi_1; x',\xi')$ variables is given by
			\[\nabla^2\Psi_0=\begin{pmatrix}
				0 & 1 & 0 & 0 \\
				1 & 0 & 0 & 0 \\
				0 & 0 &0 & A^{-\top}-\mathrm{Id}\\
				0 & 0 & A^{-1} -\mathrm{Id}& 0
			\end{pmatrix}\]
			Observe that
			\[|\mathrm{Id}- P_\gamma|=\begin{vmatrix}A - \mathrm{Id} & 0 \\ 0 & D-\mathrm{Id}\end{vmatrix}=|A|\begin{vmatrix}\mathrm{Id}-A^{-1} & 0 \\ 0& D-\mathrm{Id}\end{vmatrix}=|A|\cdot |\nabla^2_z \Psi_0(z_0)|.\]
			We infer
			\[\frac{|A^{-1}|}{|\nabla^2\Psi_0(z_0)|}=\frac{1}{|\mathrm{Id}- P_\gamma|}.\]
			Since $|\partial_{x\eta}\tilde \varphi_0|^{1/2}=|A^{-1}|^{\frac 12}$, the leading term in the stationary phase expansion is given by
			\[\operatorname{Tr}\mathcal I[\tilde \varphi,\tilde a_h^\delta]=\frac{\mathrm{e}^{\frac{\mathrm{i}}{h} S(\gamma_\delta)}}{|I-P_\gamma|^{\frac 12}}\widehat f(T_\gamma/L)\omega_\nu(T_\gamma-t_n)\int_{x_1} \chi_{\rho_0} \mathrm{d}x_1+\mathcal O(\delta h^{-2\beta -\varepsilon}),\]
			where again we write abusively $\chi_{\rho_0}$ instead of the function in coordinates. The integration along $x_1$ in coordinates is simply the integration along the flow line, thus we end up with
			\[\operatorname{Tr}\mathcal I[\tilde \varphi,\tilde a_h^\delta]=\frac{\mathrm{e}^{\frac{\mathrm{i}}{h} S(\gamma_\delta)}}{|I-P_\gamma|^{\frac 12}}\widehat f(T_\gamma/L)\omega_\nu(T_\gamma-t_n)\int_{\gamma} \chi_{\rho_0}+\mathcal O(\delta h^{-2\beta -\varepsilon}).\]
			This finishes the proof of Proposition \ref{prop: trace A n gamma j}. Proposition \ref{prop: trace A n gamma} follows by summing over the contributions of each small ball along the orbit. This may multiply the error by a factor $\mathrm{e}^{CT}$, which is $\le h^{-\varepsilon}$ provided $T\le \mathfrak c(\alpha,\beta,\varepsilon)|\log h|$.
		
			\subsection{Maslov factors} \label{subsec: elucidating Maslov} Finally, we give an interpretation of the Maslov factor appearing in the trace formula
			
			\subsubsection{Maslov index of a periodic orbit}
			Let $(M,g)$ be a smooth Riemannian manifold. The cotangent bundle $\pi:T^*M\to M$ is equipped with a natural metric coming from the Riemannian structure on $M$, called the \emph{Sasaki metric} and denoted $g_S$, and a canonical \emph{symplectic form} $\omega$. The \emph{vertical bundle} $\mathbf V\subset T(T^*M)$ is defined by $\mathbf V=\ker \mathrm{d}\pi$. The Levi--Civita connection then provides a \emph{horizontal bundle} $\mathbf H\subset T(T^*M)$, moreover there is a natural identification $T(T^*M)=\mathbf H\oplus \mathbf V\simeq TM\oplus TM$.
			\begin{defi}We define an \emph{almost complex structure} $J$ on $T(T^*M)$ by setting 
				\[\omega(J u,\cdot)=g_S(u,\cdot), \qquad \forall u\in T(T^*M).\]
				It satisfies $J^2=-1$ and is given in the decomposition $\mathbf H\oplus \mathbf V\simeq TM\oplus TM$ of $T(T^*M)$ by the matrix $\begin{pmatrix} 0 & 1 \\ -1 & 0\end{pmatrix}$. \end{defi}

			Recall that a subspace $\Lambda\subset T_{(x,\xi)}(T^*M)$ is \emph{Lagrangian} if $\omega_{|\Lambda}=0$. Note that a Lagrangian subspace automatically has dimension $d$. If $\Lambda$ is Lagrangian, then $J\Lambda$ coincides with the orthogonal of $\Lambda$ for the metric $g_S$. Indeed for any $u,v\in \Lambda$, one has
			\[g_S(u,Jv)=\omega(u,J^2v)=-\omega(u,v)=0.\]
			Following the work of Maslov \cite{bookMaslov} on the WKB approximation, Arnold \cite{Arnold1967} introduced an index for closed loops of Lagrangian subspaces:
			\begin{defi}Let $\gamma:[0,1]\to T^*M$ be a closed curve. Fix a smoothly varying orthonormal basis $\mathcal E(t)=(V_1(t),\ldots,V_d(t))$ of $\mathbf V$ along $\gamma$. 
				Also, let $J\mathcal E(t):=(JV_1(t),\ldots,JV_d(t))$ denote a corresponding orthonormal basis of $\mathbf H$, so that $(\mathcal E(t),J\mathcal E(t))$ is a symplectic basis of $T_{\gamma(t)}(T^*M)$.
				
				Let $t\mapsto \Lambda(t)\subset T_{\gamma(t)}(T^*M)$ be a closed curve of Lagrangian subspaces. Fix a smoothly varying orthonormal basis $\mathcal F(t):=(\lambda_1(t),\ldots,\lambda_d(t))$ of $\Lambda(t)$, which is expressed in the basis of $\mathbf H\oplus \mathbf V$ above as a matrix
				\[\begin{pmatrix}
					A_t \\
					C_t
				\end{pmatrix}.\]
				Note that $A_t+\mathrm{i}C_t$ is automatically unitary. The \emph{Arnold--Maslov index} of the loop $\{\Lambda(t)\}$ is defined as the degree of the map
				\[t\mapsto \det\big((A_t-\mathrm{i} C_t)(A_t+\mathrm{i} C_t)^{-1}\big)\in \mathbf C^\times.\]\end{defi}
			We make a few comments.
			\begin{rem}\label{rem: maslov is invariant under} The quantity $\det\big((A_t-\mathrm{i} C_t)(A_t+\mathrm{i} C_t)^{-1}\big)$ is invariant under the following transformations:
				\begin{itemize}
					\item applying an orthogonal change of variable to the frame $(\lambda_1(t),\ldots,\lambda_d(t))$ (which need not preserve the orientation).
					\item applying an orthogonal change of variable to the frame $(V_1(t),\ldots,V_d(t))$.\end{itemize}
				Notably, the index is defined without any orientability assumption on $\mathbf V$ or $\Lambda(t)$.\end{rem}
			\begin{rem}Assume $\mathbf V$ is orientable. If $\Lambda(t)$ is orientable then we can choose $\mathcal F(0)=\mathcal F(1)$, in which case the index is even. Otherwise we can choose 
				\[\mathcal F(1)=(-\lambda_1(0),\lambda_2(0),\ldots,\lambda_d(0))\] and the index is odd.\end{rem}
			\begin{rem} The matrix $A_t$ is invertible along the trajectory if and only if $\Lambda(t)$ remains transverse to $\mathbf V(t)$. In this case, the index of the loop vanishes. Indeed, if $\det A_t$ is nonvanishing, we have
				\[\operatorname{deg} \det(A_t-\mathrm{i}C_t)=\operatorname{deg} \det(\operatorname{Id}-\mathrm{i}C_tA_t^{-1}).\]
				Since $\begin{pmatrix}
					A_t & -C_t \\ C_t & A_t
				\end{pmatrix}$ is symplectic, the matrix $C_tA_t^{-1}$ is real symmetric thus has real spectrum. In particular, the matrices $\operatorname{Id}-\mathrm{i} \varepsilon C_tA_t^{-1}$ are invertible for all $t$ and all $\varepsilon\in [0,1]$. By deformation, we obtain
				\[\operatorname{deg}\det (\operatorname{Id}-\mathrm{i}C_tA_t^{-1})=0.\]
			\end{rem}
			\begin{defi}[Maslov index of a periodic orbit]\label{defi: maslov index periodic orbit} Let $\gamma\subset \Sigma_E$ be a hyperbolic periodic orbit of the Hamiltonian flow generated by a symbol $p$. The \emph{Maslov index} of $\gamma$, denoted $\mathfrak{m}_\gamma$, is the index of the loop of Lagrangian subspaces $E^{wu}(t)=\mathbf{R}H_p\oplus E^u$ along $\gamma(t)$. \end{defi}
			\begin{ex}Consider the geodesic flow on an oriented surface $M$. Let $X$ denote the geodesic vector field on $S^*M$, and $V$ be the vertical vector field (which generates the rotation in the fibers). Lastly, let $X_\bot:=[V,X]$ denote the horizontal vector field. Let $\gamma$ be a hyperbolic periodic orbit of the flow, with period $T_\gamma$. Then $E^u(\gamma(t))\subset \mathbf RX_\bot\oplus \mathbf RV$, and one can write $E^u(\gamma(t))=\mathbf R U(t)$ with $U(t)=a(t)X_\bot+c(t)V$. We can ensure $U(T_\gamma)=\pm U(0)$ with a sign $+$ if $E^u$ is orientable along the orbit, and sign $-$ otherwise. This is immaterial as what really counts is the class $[U(t)]\in \mathbf R^2/\{\pm 1\}$. Identifying $\mathbf RX_\bot\oplus \mathbf RV\simeq \mathbf C$, we write $U(t)=r(t)\mathrm{e}^{\mathrm{i}\theta(t)}$ with $\theta(t)$ continuous. The fact that $[U(T_\gamma)]=[U(0)]$ implies $\theta(T_\gamma)=\theta(0)\pmod \pi$. The Maslov index of the orbit is
				\[\mathfrak m_\gamma=-\frac{\theta(T_\gamma)-\theta(0)}{\pi}\in \mathbf Z.\]
			\end{ex}
			\begin{rem}In the above example, the index $\mathfrak m_\gamma$ counts the number of crossings of $E^{wu}(t)$ with the vertical distribution. A crossing in the clockwise direction counts positively in $\mathfrak m_\gamma$ and negatively otherwise---we keep this convention in order to have consistent notation with the physics literature \cite{creagh1990,robbins1991}. This picture persists in higher dimension. In this case a \emph{crossing} is a time for which $E^{wu}(t)\cap \mathbf V(t)\neq \{0\}$. One associates to any crossing point $t_0$ a \emph{crossing form}, which is a quadratic form $\Gamma(\Lambda,t_0)$ on $\Lambda(t_0)\cap \mathbf V$, whose signature determines the contribution of the crossing to the Maslov index. In the case where the crossing points are all regular (i.e. give rise to a nondegenerate crossing form), one can show that the Maslov index of the path is given by minus the sum of the signatures of $\Gamma(\Lambda,t_i)$ running over all the crossing points $t_i$ along the orbit. For more details on this, see Robbin--Salamon \cite{RobbinSalamon}.	\end{rem}
			\subsubsection{Non-orientable weak-unstable manifold} We explain the necessary changes in the proof of the trace formula when the weak unstable manifold is non-orientable. In this case, one can always choose $(u_1(T_\gamma),\ldots,u_d(T_\gamma))=(-u_1(0),u_2(0),\ldots,u_d(0))$, so that
			\[\mathcal A_{\ell_0}(T_\gamma)=\begin{pmatrix}
				-1 & 0 & 0 & 0\\ 0 & \operatorname{Id} & 0&  0 \\ 0 & 0& -1 & 0 \\ 0 & 0 &0 & \operatorname{Id}
			\end{pmatrix}\mathcal A_{\ell_0}(0).\]
			Thus we need to further evolve the path of matrices to land back on $\mathcal A_{\ell_0}(0)$. To do so, we introduce symplectic rotation matrices
			\[\mathcal B(s)=\begin{pmatrix}
				\cos(s) & 0 & \sin(s) & 0 \\
				0 & I_{d-1} & 0 & 0 \\
				-\sin(s) & 0 & \cos(s) & 0 \\
				0 & 0 & 0 & I_{d-1}
			\end{pmatrix}=:\begin{pmatrix}
				C(s) & S(s) \\ -S(s) & C(s)
			\end{pmatrix}.\] 
			Then, the concatenation $\{\mathcal A(T_\gamma)\mathcal B(s)\}_{s\in [0,\pi]}\star\{\mathcal A(t)\}_{t\in [0,T_\gamma]}$ becomes a closed path of symplectic matrices. We set $\mathcal A(T_\gamma+s):=\mathcal A(T_\gamma)\mathcal B(s)$ for $s\in [0,\pi]$. Then,
			\[\mu_h(\mathcal A(T_\gamma+\pi))=\varepsilon(\gamma) \mu_h(\mathcal A(0)),\]
			for some $\varepsilon(\gamma)\in \{\pm 1\}$. Letting $\mathcal A(t)_c=\begin{pmatrix}
				P(t) & Q(t) \\ \overline{Q}(t) & \overline{P}(t)
			\end{pmatrix}$ denote the complexified matrix, we see that for $s\in [0,\pi]$, one has 
				\[\det(P(T_\gamma+s))=\det(P(T_\gamma))\cdot \mathrm{e}^{-\mathrm{i} s}.\]
				Thus, the degree of the map $t\in [0,T_\gamma+\pi]\mapsto \frac{1}{(\det \overline P_t)^{1/2}}$ is given by
				\[-\frac 14+\deg \{t\in [0,T_\gamma]\mapsto (\det P_t)^{1/2}\}=-\frac 14-\frac 14\deg \{t\in [0,T_\gamma]\mapsto (\det(A_t-\mathrm{i}C_t)(A_t+\mathrm{i}C_t)^{-1})\}.\]
				\[=-\frac 14-\frac 14\mathfrak m_\gamma,\]
				where $\mathfrak m_\gamma$ is the Maslov index of the orbit.
				It follows that $\mu_h(\mathcal A(T_\gamma+\pi))^*=\mathrm{e}^{\mathrm{i}\frac{\pi}{2}} \mathrm{e}^{\mathrm{i} \frac{\pi}{2}\mathfrak m_\gamma}\mu_h(\mathcal A(0))^*$.
				Now, by \cite[eq. (4.26), eq. (4.47)]{Folland} applied with $n=1$, letting $s\mapsto \mu_h(\mathcal B(s))\in \operatorname{Mp}_h(\mathbf{R},2d)$ denote the lift of $s\mapsto \mathcal B(s)\in \operatorname{Sp}(\mathbf{R},2d)$ starting at $\mu_h(\mathcal B(0))=\operatorname{Id}$, we have 
				\[\mu_h(\mathcal B(\pi)) u(x_1,\ldots,x_d)=\mathrm{e}^{-\mathrm{i}\frac{\pi}{2}}u(-x_1,x_2,\ldots,x_d).\]
				It follows that 
				\[\mu_h(\mathcal B(\pi))^* u(x_1,\ldots,x_d)=\mathrm{e}^{\mathrm{i}\pi/2}u(-x_1,\ldots,x_d),\]
				in particular,
				\[\mu_h(\mathcal B(\pi))^*\mathcal I[\varphi_\delta,a_h](x,y)=\mathrm{e}^{\mathrm{i}\frac{\pi}{2}}\int \mathrm{e}^{\frac{\mathrm{i}}{h}(\varphi_\delta(-x_1,x',\eta)-\langle y,\eta\rangle)}a_h(t,-x_1,x',y,\eta)\mathrm{d}\eta=\mathrm{e}^{\mathrm{i} \frac{\pi}{2}}\mathcal I[\tilde \varphi_\delta,\tilde a_h],\]
				where $\tilde \varphi_\delta$ is a generating function for $\mathcal A_{\ell_0}(0)^{-1} \widetilde f_{\ell_0}\Phi_\delta^t \widetilde f_{\ell_0}^{-1}\mathcal A_{\ell_0}(0)$. At the quantum level:
				\[\mathrm{e}^{\mathrm{i}\frac{\pi}{2}} \mathrm{e}^{\mathrm{i}\frac{\pi}{2} \mathfrak m_\gamma}\mu_h(\mathcal A(0))^* \mathrm{e}^{-\frac{\mathrm{i}}{h} tP_{h,\delta}}\operatorname{Op}_h^{M}(\chi_\rho) \mu_h(\mathcal A(0))=\mathrm{e}^{\mathrm{i} \frac{\pi}{2}}\mathcal I[\tilde \varphi_\delta,\tilde a_h].\]
				Thus the $\mathrm{e}^{\mathrm{i}\frac{\pi}{2}}$ factors cancel out and we are left with
				\[\operatorname{Tr} \mathrm{e}^{-\frac{\mathrm{i}}{h} tP_{h,\delta}}\operatorname{Op}_h^{M}(\chi_\rho)=\mathrm{e}^{-\mathrm{i}\frac{\pi}{2}\mathfrak m_\gamma} \operatorname{Tr} \mathcal I[\tilde \varphi_\delta,a_h]+\mathcal O(h^\infty).\]
				It is worth mentioning that when $M$ is orientable, the Maslov index $\mathfrak m_\gamma$ is even if and only if $E^{wu}$ is orientable along the orbit $\gamma$.
				
				\section{Metaplectic operators} \label{sec: metaplectic}
				
				For long times, the quantum propagator may fail to be expressed a FIO of the form \eqref{eq: defi FIO nice form}, due to the formation of caustics, which are locations where the graph of the flow no longer projects diffeomorphically onto the coordinates $(x,\eta)$. One way to address this issue is to perform a symplectic change of coordinates $\kappa:(x,\xi)\mapsto (\tilde x,\tilde \xi)$ in phase space, so that the graph becomes projectable in the new coordinates $(\tilde x,\eta)$. At the level of operators this amounts to composing $U(t)$ with a unitary Fourier integral operator $\widehat \kappa$ quantizing the transformation $\kappa$, meaning studying $\widehat \kappa \circ U(t)$. The action $\widehat \kappa$ on a pseudodifferential operator $\mathrm{Op}_h^{\rm w}(a)$ is given by the Egorov theorem
				\[\widehat \kappa^*\circ \mathrm{Op}_h^{\rm w}(a)\circ \widehat \kappa=\mathrm{Op}_h^{\rm w}(b), \qquad b=a\circ \kappa+\mathcal O(h^2).\] 
				When $\kappa$ is a linear symplectomorphism of $\mathbf R^{2d}$, it can be quantized into a \emph{metaplectic operator}, which has the advantage of yielding an exact Egorov theorem.
				\begin{equation}\label{eq: egorov metaplectique} \widehat \kappa^*\circ \mathrm{Op}_h^{\rm w}(a)\circ \widehat \kappa=\mathrm{Op}_h^{\rm w}(a\circ \kappa).\end{equation}
				Such an operator $\widehat \kappa$ is unique up to multiplication by a complex number of modulus 1. Let us now give a more precise definition.
				
				\begin{defi}[Metaplectic operators] Let $\mathcal A\in \mathrm{Sp}(\mathbf R,2d)$. A \emph{metaplectic operator} associated with $\mathcal A$ is a unitary operator $\mu_h(\mathcal A):L^2(\mathbf R^d)\to L^2(\mathbf R^d)$, satisfying
					\[\mu_h( \mathcal A)^*\mathrm{Op}_h^{\rm w}(a)\mu_h(\mathcal A)=\mathrm{Op}_h^{\rm w}(a\circ \mathcal A), \qquad \text{for all $a\in \mathcal S(\mathbf R^{2d})$}.\]
					Such an operator always exists, and is uniquely defined, up to multiplication by a complex number of modulus $1$. We denote by $\mathrm{Met}_h(\mathbf R^d)$ the set of metaplectic operators on $L^2(\mathbf R^d)$. \end{defi}
				
				\begin{ex}Here are some examples of metaplectic operators:
					\begin{itemize}\item The semiclassical Fourier transform
						\[\mathcal F_h:u\mapsto \frac{1}{(2\pi h)^{d/2}}\int_{\mathbf{R}^d} \mathrm{e}^{ -\frac{\mathrm i}{h} \langle x,\xi\rangle }u(x)\mathrm dx\]
						is a metaplectic operator quantizing the quarter turn $J=\begin{pmatrix}
							0 & \mathrm{Id} \\ -\mathrm{Id} & 0
						\end{pmatrix}$. 
						
						\item Assume $\mathcal A=\begin{pmatrix}A & B \\ C & D\end{pmatrix}\in \operatorname{Sp}(\mathbf{R},2d)$ satisfies $\det(A)\neq 0$. Let $\psi(x,\eta):=\frac 12\langle CA^{-1}x,x\rangle +\langle A^{-1}x,\eta\rangle -\frac 12\langle A^{-1}B\eta,\eta\rangle$. Then, the operator
						\begin{equation} \label{eq: def mu h A upper left block invert}\mu_h^+(\mathcal A)u(x):= \frac{|\det(A)|^{-\frac 12}}{(2\pi h)^d}\int_{\mathbf{R}^{2d}} \mathrm{e}^{\frac{\mathrm i}{h} (\psi(x,\eta)-\langle y,\eta\rangle)} u(y)\mathrm dy\mathrm d\eta\end{equation}
						is a metaplectic operator associated with $\mathcal A$.
				\end{itemize}\end{ex}
				If $\mathcal A$ and $\mathcal B$ are two symplectic linear maps then
				\[\mu_h(\mathcal A)\mu_h(\mathcal B)=c_{\mathcal A,\mathcal B}~\mu_h(\mathcal A\mathcal B), \qquad  \text{with}~ |c_{\mathcal A,\mathcal B}|=1.\]
				There is unfortunately no hope of finding a homomorphism $\mu_h:\mathrm{Sp}(\mathbf R,2d)\to \mathrm{Met}_h(\mathbf R^d)$, that is, satisfying $c_{\mathcal A,\mathcal B}=1$ for all $\mathcal A,\mathcal B$. However, one can build a double valued map $\mu_h:\mathrm{Sp}(\mathbf R,2d)\to \mathrm{Met}_h(\mathbf R^d)$, with image denoted $\mathrm{Mp}_h(\mathbf R,2d)$, such that $\mu_h(I_{2d})=\pm \mathrm{Id}_{L^2(\mathbf R^d)}$ and
				\[\mu_h(\mathcal A\mathcal B)=\pm \mu_h(\mathcal A)\mu_h(\mathcal B), \qquad \forall \mathcal A,\mathcal B\in \mathrm{Sp}(\mathbf R,2d).\]
				Alternatively, one can think of $\mu_h$ as a single-valued map $\mathrm{Sp}(\mathbf R,2d)\to \mathrm{Met}_h(\mathbf R^d)/\{\pm \mathrm{Id}\}$. It turns out that $\mathrm{Mp}_h(\mathbf R,2d)$ is a realization of the universal double cover of $\mathrm{Sp}(\mathbf R,2d)$, known as the \emph{metaplectic group}, in the space of unitary operators on $L^2(\mathbf{R}^d)$. We will give an explicit description of $\mathrm{Mp}_h(\mathbf R,2d)$ in the following. 

In the rest of the section, we follow the presentation of Folland \cite{Folland}. There are a few differences:
				\begin{itemize}
					\item Folland uses a different convention for the Fourier transform than ours. Therefore, factors $\sqrt{2 \pi h}$ appear here and there compared to his work.
					\item In Folland's approach, the conjugation by a metaplectic operator associated with a linear symplectomorphism $\mathcal A$ transforms a point $(x^*,\xi^*)\in \mathbf R^{2d}$ into $\mathcal A(x^*,\xi^*)$, while here it shall transform the linear form $l:=\langle x^*,\cdot\rangle +\langle \xi^*,\cdot\rangle$ into $l\circ \mathcal A$, or in other words transform $(x^*,\xi^*)$ into $\mathcal A^\top (x^*,\xi^*)$.
					\item Folland lets symplectic matrices act on $(\xi,x)$, while we let them act on $(x,\xi)$. Therefore, at some points we have to conjugate symplectic matrices by the swap $(x,\xi)\mapsto (\xi,x)$ given by the matrix $K=\begin{pmatrix}
						0 & \mathrm{Id} \\ \mathrm{Id} & 0
					\end{pmatrix}$.
				\end{itemize}
				
				\subsection{Detour by the Fock representation} While there are no available general explicit formulas (to my knowledge) for $\mu_h:\mathrm{Sp}(\mathbf R,2d)\to U(L^2(\mathbf R^d))$, there is an explicit formula in \emph{Fock representation}. Introduce the \emph{Fock space}
				\[\mathcal F(\mathbf C^d):=L^2(\mathbf C^d,\mathrm{e}^{-\frac{|z|^2}{h}}\mathrm dz)=\Big\{F:\mathbf C^d\to \mathbf C~:~(h\pi)^{-d}\int_{\mathbf C^d} |F(z)|^2 \mathrm{e}^{-\frac{|z|^2}{h}}\mathrm dz<+\infty\Big\}.\]
				Also, let $\mathcal H(\mathbf C^d)$ be the subspace of $\mathcal F$ consisting of holomorphic functions, that is
				\[\mathcal H(\mathbf C^d)=\big\{f\in \mathcal F~:~ \partial_{\bar z_j}f=0, ~ j=1,\ldots,d\big\}.\]
				The space $\mathcal F$ comes with a Hermitian product
				\[\langle F,G\rangle_{\mathcal F}=(h\pi)^{-d}\int_{\mathbf C^d} F(z)\overline{G(z)}\mathrm{e}^{-\frac{|z|^2}{h}}\mathrm dz.\]
				The \emph{Bargmann transform} is defined by 
				\[\mathtt{B}_hf(z):=\frac{1}{(\pi h)^{\frac{3d}{4}}}\int_{\mathbf R^d} \mathrm{e}^{\frac 1h (-\frac 12 z^2 +\sqrt 2 zx -\frac 12 x^2)}f(x)\mathrm dx.\]
				The operator $\mathtt B_h$ is an isometry $L^2(\mathbf R^{d})\to \mathcal F(\mathbf C^d)$ with image $\mathcal H(\mathbf C^d)$. Folland \cite{Folland} uses a different convention and considers
				\[\mathtt B_{\rm F} f(z):=2^{\frac d4}\int_{\mathbf R^d} \mathrm{e}^{-\frac{\pi}{2} z^2-2\pi xz -\pi x^2} f(x)\mathrm dx,\]
				which is an isometry from $L^2(\mathbf R^d)$ to $L^2(\mathbf C^d, \mathrm{e}^{-\pi |z|^2}\mathrm dz)$. These two transforms are related by the following formula:
				\[\mathtt B_h f(z)=\Big(\frac{2}{\pi h}\Big)^{\frac d4} \Big[\mathtt B_{\rm F} \big(f(\sqrt{2\pi h}\bullet)\big)\Big](z/\sqrt{\pi h}).\]
				In other words, letting $T_\alpha$ denote the operator $u\mapsto (T_\alpha u: x\mapsto u(\alpha x))$ (acting either on functions defined on $\mathbf R^d$ or $\mathbf C^d$), one has
				\[\mathtt B_h=\Big(\frac{2}{\pi h}\Big)^{\frac d4} T_{\sqrt{\pi h}}^{-1}\circ \mathtt B_{\rm F}\circ T_{\sqrt{2\pi h}}.\]
				Note that $(\frac{2}{\pi h})^{\frac d4}$ is the correct factor as it ensures that $\mathtt B_h$ and $\mathtt B_{\rm F}$ are both isometries from $L^2(\mathbf R^d)$ to the respective spaces $L^2(\mathbf C^d,\mathrm{e}^{-\frac{|z|^2}{h}}\mathrm dz)$ and $L^2(\mathbf C^d,\mathrm{e}^{-\pi|z|^2}\mathrm dz)$.
				
				\subsection{Action of pseudodifferential operators in Fock space}Given a point $(x^*,\xi^*)\in T^*\mathbf R^d$, we introduce the linear form $l(x,\xi)=\langle x^*,x\rangle+\langle \xi^*,\xi\rangle $ and from now on liberally identify the linear form $l$ with the point $(x^*,\xi^*)$. Recall that for any Schwartz symbol $a\in \mathcal S(\mathbf R^{2d})$, the Weyl quantization of $a$ is given by \cite[Lemma 4.10]{zworski2012semiclassical}:
				\[\mathrm{Op}_h^{\rm w}(a)=\frac{1}{(2\pi h)^{2d}}\int_{\mathbf R^{2d}} \hat a(l)\mathrm{e}^{\frac{\mathrm i}{h} l(x,hD)}\mathrm dl,\]
				where $\hat a(l)=\int_{\mathbf R^{2d}} \mathrm{e}^{-\frac{\mathrm i}{h} l(x,\xi)}a(x,\xi)\mathrm dx\mathrm d\xi$. Letting $\rho_h(\xi^*,x^*):=\mathrm{e}^{\frac{\mathrm i}{h} l(x,hD)}$, it is therefore enough to understand the action of the operators $\rho_h(\xi^*,x^*)$ in Fock space. Given $w=\xi^* +\mathrm{i} x^*$, we introduce
				\[\beta_h(w):=\mathtt B_h\circ \rho_h(\xi^*,x^*)\circ \mathtt B_h^*. \]
				Then, we have the following formula, obtained by rescaling \cite[(1.71)]{Folland}
				\[\beta_h(w) F(z)=\mathrm{e}^{\frac 1h(-\frac{|w|^2}{4} -\frac{z\bar w}{\sqrt 2})} F\big(z+\tfrac{w}{\sqrt 2}\big).\]
				
				\begin{defi}[Complexified matrix] Given a symplectic matrix $\mathcal A$, one defines the \emph{complexified} matrix
					\[\mathcal A_c:=\mathcal W \mathcal A \mathcal W^{-1},\]
					where $\mathcal W$ is the unitary matrix
					\[\mathcal W:=\frac{1}{\sqrt 2}\begin{pmatrix}I & \mathrm{i} I \\ I & -\mathrm{i} I\end{pmatrix}, \qquad \mathcal W^{-1}=\frac{1}{\sqrt 2}\begin{pmatrix}I &  I \\ -\mathrm{i}I & \mathrm{i} I\end{pmatrix}.\]
					The image of $\mathrm{Sp}(2d,\mathbf{R})$ under the mapping $\mathcal A\mapsto \mathcal A_c$ is denoted by $\mathrm{Sp}_c(2d,\mathbf{R})$.\end{defi}
				The complexified matrix expresses the action of $\mathcal A$ in the complex coordinates $(z,\bar z)=(\frac{\xi-\mathrm{i} x}{\sqrt 2},\frac{\xi+\mathrm{i} x}{\sqrt 2})$. Indeed,
				\[\mathcal A\begin{pmatrix}
					\xi \\ x
				\end{pmatrix}=\begin{pmatrix}
					\xi' \\ x'
				\end{pmatrix} \quad \Longleftrightarrow\quad \mathcal A_c\begin{pmatrix}
					z \\ \overline z
				\end{pmatrix}=\begin{pmatrix}
					z' \\ \overline{z'}
				\end{pmatrix}.\]
				Explicitly, if $\mathcal A=\begin{pmatrix} A & B \\ C & D\end{pmatrix}$ then $\mathcal A_c=\begin{pmatrix}
					P & Q \\ \overline Q& \overline{P}
				\end{pmatrix}$, with $2P=A+\mathrm{i}C-\mathrm{i} B+D$ and $2Q=A+\mathrm{i} C+ \mathrm{i} B-D$. We recall the following facts. \begin{lem}{{\cite[(4.18)]{Folland}}} \label{lem: matrices P Q} Let $\mathcal A_c=\begin{pmatrix}P & Q \\ \overline Q & \overline P\end{pmatrix}$ with $P,Q$ complex matrices. Then, 
					\begin{enumerate}[(i)]
						\item The matrix $P$ is invertible, and $\|P\|\ge 1$;
						\item  $\|P^{-1}Q\|^2=\|\overline Q P^{-1}\|^2=1-\|P\|^{-2}<1$;
						\item  $P^{-1}Q$ and $\overline QP^{-1}$ are symmetric.
					\end{enumerate}
				\end{lem}
				
				\begin{defi}[Metaplectic operators in Fock representation] Let $\mathcal A_c=\begin{pmatrix} P & Q \\ \overline Q & \overline{P} \end{pmatrix}\in \mathrm{Sp}_c(2d,\mathbf{R})$. We define
					\begin{equation}\label{eq:metaplectic fock space}\nu_h(\mathcal A_c)F(z):=(\pi h)^{-d}\int_{\mathbf C^d} K_{\mathcal A_c}(z,w) F(w)\mathrm{e}^{-\frac{|w|^2}{h}}\mathrm dw,\end{equation}
					where \[K_{\mathcal A_c}(z,w)=(\det P)^{-\frac 12}\exp\Big( \frac 1h \big( \frac 12 z \overline Q P^{-1}z+\bar w P^{-1}z -\frac 12 \bar w P^{-1}Q \bar w\big)\Big).\]
				\end{defi}
				\begin{rem} In the definition of $\nu_h(\mathcal A_c)$ one has to choose a square root of $\det(P)$ hence $\nu_h(\mathcal A_c)$ is defined up to a change of sign. In the following, when writing $\nu_h(\mathcal A_c)$ we implicitly choose a value of $\det(P)^{\frac 12}$.
				\end{rem}
				The operator $\nu_h(\mathcal A_c)$ satisfies the following properties.\begin{prop}[{\cite[(4.37)]{Folland}}] \label{prop: proprs metaplectic fock} For all $\mathcal A_c,\mathcal B_c\in \mathrm{Sp}_c(2d,\mathbf{R})$, the following holds
					\begin{enumerate}[(i)]
						\item $\nu_h(\mathcal A_c)$ is unitary on $\mathcal F(\mathbf C^d)$;
						\item 
						$\nu_h(\mathcal A_c)\beta_h(w)\nu_h(\mathcal A_c)^*=\beta_h(\mathcal A_cw)$ for all $w\in \mathbf C^d$;
						\item $\nu_h(\mathcal A_c)\nu_h(\mathcal B_c)=\pm \nu_h(\mathcal A_c \mathcal B_c)$.
				\end{enumerate}\end{prop}
				It is an enjoyable feature of working in Fock space to have an explicit formula for $\nu_h(\mathcal A_c)$, for any matrix $\mathcal A_c\in \operatorname{Sp}(2d,\mathbf{R})$. In the next section, we explain how to build the metaplectic representation $\mu_h$ by using the results on $\nu_h$.
				\subsection{Building the metaplectic representation} Recall that we are looking for an operator $\mu_h(\mathcal A)$ such that
				\begin{equation} \label{eq: requirement mu h}\mu_h(\mathcal A)^*\circ \mathrm{e}^{\frac{\mathrm i}{h} l(x,hD)}\circ \mu_h(\mathcal A)=\mathrm{e}^{\frac{\mathrm i}{h} (l\circ \mathcal A)(x,hD)},\qquad \text{for any $l\in \mathbf R^{2d}$}.\end{equation}
				If $l=\langle x^*,\cdot\rangle+\langle \xi^*,\cdot\rangle$ then $l\circ \mathcal A(x,\xi)=\langle \tilde x^*,x\rangle+\langle \tilde \xi^*,\xi\rangle$ with $\binom{\tilde x^*}{\tilde \xi^*}=\mathcal A^\top \binom{x^*}{\xi^*}$, or $\binom{\tilde \xi^*}{\tilde x^*}=K\mathcal A^\top K\binom{\xi}{x}$ where $K=\begin{pmatrix}
					0 & 1 \\ 1 & 0
				\end{pmatrix}$. Recalling the notation $\rho_h(\xi^*,x^*)=\mathrm{e}^{\frac{\mathrm i}{h} l(x,hD)}$, \eqref{eq: requirement mu h} is equivalent to
				\begin{equation} \label{eq: requirement bis}\mu_h(\mathcal A)^* \rho_h(\xi^*,x^*)\mu_h(\mathcal A)=\rho_h(K \mathcal A^\top K(\xi^*,x^*)) \quad \forall (x^*,\xi^*)\in \mathbf{R}^{2d}.\end{equation}
				We now look for a symplectic matrix $\mathcal B$ such that 
				\[\mu_h(\mathcal A)=\mathtt B_h^* \nu_h(\mathcal B_c)\mathtt B_h\]
				satisfies the above requirement. From what's above, we must have
				\[\mathtt B_h^* \nu_h(\mathcal B_c)^*\mathtt B_h\circ  \rho_h(\xi^*,x^*)\circ \mathtt B_h^* \nu_h(\mathcal B_c)\mathtt B_h=\rho_h(K\mathcal A^\top K (\xi^*,x^*)).\]
				Letting $w=\xi^*+\mathrm{i} x^*$, \eqref{eq: requirement bis} is in turn equivalent to
				\[\nu_h(\mathcal B_c)^*\beta_h(w) \nu_h(\mathcal B_c)=\mathtt B_h\rho_h((K\mathcal A^\top K)_cw)\mathtt B_h^*.\]
				Using Proposition \ref{prop: proprs metaplectic fock}, this holds for $\mathcal B=K\mathcal A^\top K$. Note that for $\mathcal A=\begin{pmatrix}
					A & B \\ C & D
				\end{pmatrix}$ we have
				\[K\mathcal A^\top K=\begin{pmatrix}A & -B \\ -C & D\end{pmatrix},\]
				and $(K\mathcal A^\top K)_c=\overline{\mathcal A_c}$. This leads us to the following definition.
				\begin{defi}[Metaplectic representation] For $\mathcal A\in \mathrm{Sp}(2d,\mathbf{R})$ define
					\[\mu_h(\mathcal A):=\mathtt B_h^* \nu_h(\overline{\mathcal A_c})\mathtt B_h.\]
					Again, $\mu_h(\mathcal A)$ is defined up to a change of sign, or can be viewed as a set with two elements. The image of $\operatorname{Sp}(2d,\mathbf{R})$ under $\mu_h$ is denoted by $\operatorname{Mp}_h(\mathbf{R},2d)$.\end{defi}
				Proposition \ref{prop: proprs metaplectic fock} for $\nu_h$ translates to the following result on $\mu_h$.
				\begin{prop}For all $\mathcal A,\mathcal B$ in $\mathrm{Sp}(2d,\mathbf{R})$, we have
					\begin{enumerate}[(i)]
						\item $\mu_h(\mathcal A)$ is unitary on $L^2(\mathbf R^d)$;
						\item $\mu_h(\mathcal A B)=\pm \mu_h(\mathcal A)\mu_h(\mathcal B)$;
						\item For any $a\in \mathcal S(\mathbf{R}^n)$,
						\[\mu_h(\mathcal A)^* \mathrm{Op}_h^{\rm w}(a)\mu_h(\mathcal A)=\mathrm{Op}_h^{\rm w}(a\circ \mathcal A).\]
					\end{enumerate}
				\end{prop}
				In certain cases one can give an explicit formula for the operator(s) $\mu_h(\mathcal A)$.
				\begin{prop}\label{prop: met zworski is correct} Assume that the upper-left block of $\mathcal A$ satisfies $\det(A)\neq 0$. Then, the two elements of $\mu_h(\mathcal A)$ are given by
					\begin{equation}\label{met zworski}\mu_h^{\pm}(\mathcal A) u(x)=\pm \frac{|\det(A)|^{-\frac 12}}{(2\pi h)^d}\int_{\mathbf{R}^{2d}} \mathrm{e}^{\frac{\mathrm{i}}{h}(\psi(x,\eta)-\langle x,\eta\rangle)} u(y)\mathrm dy\mathrm d\eta,\end{equation}
					with
					\[\psi(x,\eta)=\frac 12\langle CA^{-1}x,x\rangle +\langle A^{-1}x,\eta\rangle -\frac 12\langle A^{-1}B \eta,\eta\rangle.\]
					In other words, the formula \eqref{eq: def mu h A upper left block invert} gives the correct phase.
				\end{prop}
				We mention the following consequence.
				\begin{lem}\label{lem: adjoint of mu h (A)} Let $\mathcal A=\begin{pmatrix}
						A & 0 \\ C & D
					\end{pmatrix}\in \operatorname{Sp}(\mathbf{R},2d)$ with $\det(A)>0$. Then, the adjoint of $\mu_h^+(\mathcal A)$ is given by
					\[\mu_h^+(\mathcal A)^*=\mu_h^+(\mathcal A^{-1}).\]
				\end{lem}
				\begin{proof}Since $\det(A)>0$, we can find a continuous deformation $\mathcal A(t)=\begin{pmatrix}
						A(t) & 0 \\ C(t) & D(t)
					\end{pmatrix}$ from $\mathcal A(0)=\operatorname{Id}$ to $\mathcal A(1)=\mathcal A$ in $\operatorname{Sp}(\mathbf{R},2d)$, and such that $\det A(t)>0$ along the trajectory. Since $\mu_h^+(\mathcal A(t))^*$ and $\mu_h^+(\mathcal A(t)^{-1})$ both quantize $\mathcal A(t)^{-1}$, we know that
					\[\mu_h^+(\mathcal A(t))^*=\pm \mu_h^+(\mathcal A(t)^{-1}),\]
					where the sign $\pm$ is constant along the trajectory. By looking at $t=0$, we deduce that the sign is $+$.
				\end{proof}
			
				Next, we explain how continuous paths on $\mathrm{Sp}(\mathbf R,2d)$ lift to $\mathrm{Mp}_h(\mathbf R,2d)$.
				
				\begin{lem}[Lifting paths of symplectic matrices] Let $\{\mathcal A(t)\}_{t\in [0,1]}$ be a continuous path of symplectic matrices. Let \[\mathcal A(t)_c:=\begin{pmatrix}
						P_t & Q_t \\ \overline Q_t & \overline P_t
					\end{pmatrix}.\]
					Then, there is a unique continuous lift $t\mapsto \mathcal M(t)\in \operatorname{Mp}_h(\mathbf{R},2d)$ of $t\mapsto \mathcal A(t)$, up to the choice of sign of $\mathcal M(0)$. Moreover, if $s\le t$ are close enough, then
					\[\mathcal M(t)\mathcal M(s)^*=\mu_h^+(\mathcal A(t)\mathcal A(s)^{-1}),\]
					where $\mu_h^+(\cdot)$ has been defined in \eqref{met zworski}.
				\end{lem}
				
				\begin{proof}\textbf{1.} Fix a continuous determination of $t\mapsto \det(\overline P_t)^{\frac 12}$. This gives a continuous determination of $t\mapsto \nu_h(\overline{\mathcal A(t)_c})$, then one can just set
					\[\mathcal M(t):=\mathtt B_h^* \nu_h(\overline{\mathcal A(t)_c}) \mathtt B_h.\]
					If $\mathcal M'(t)$ is another lift of the path, then for all $t$ one has $\mathcal M(t)=\pm \mathcal M'(t)$. Since both paths are continuous (and the operators are unitary) it implies that either $\mathcal M(t)= \mathcal M'(t)$ for all $t$ or $\mathcal M(t)= -\mathcal M'(t)$ for all $t$.
					
					\textbf{2.} For the second claim, notice that for $|t-s|$ small enough, the upper left-block of $\mathcal A(t)\mathcal A(s)^{-1}$ is invertible. Thus,
					\[\mathcal M(t)\mathcal M(s)^*=\pm \mu_h^+(\mathcal A(t)\mathcal A(s)^{-1}),\]
					where $\mu_h^+(\cdot)$ is computed using \eqref{met zworski}. By continuity and because $\mathcal M(s)\mathcal M(s)^*=\operatorname{Id}$ we deduce 
					\[\mathcal M(t)\mathcal M(s)^*=\mu_h^+(\mathcal A(t)\mathcal A(s)^{-1}).\]

				\end{proof}

				The previous result can be specialized to the case of closed paths of symplectic matrices.
				\begin{lem}[Lifting loops of symplectic matrices] \label{lem: lifting loops} Let $\{\mathcal A(t)\}_{t\in [0,1]}$ be a closed path of symplectic matrices, i.e. satisfying $\mathcal A(1)=\mathcal A(0)$. If $\mu_h(\mathcal A(t))$ is one of the two lifts of $t\mapsto \mathcal A(t)$ to $\mathrm{Mp}_h(\mathbf R,2d)$ one has
					\[\mu_h(\mathcal A(1))=(-1)^d\mu_h(\mathcal A(0)).\]
					Here $d$ is the degree of the map $\det({P_t}): [0,1]\to \mathbf C^\times$, where $P_t$ is the upper-left block of $\mathcal A(t)_c$.
				\end{lem}
				
				\begin{proof}Recalling the formula \eqref{eq:metaplectic fock space} for $\nu_h(\overline{\mathcal A_c})$, one finds 
					\[\nu_h(\overline{\mathcal A(1)_c})=(-1)^d\nu_h(\overline{\mathcal A(0)_c}),\]
					it remains to conjugate the operators by $\mathtt B_h$ to obtain the result.
				\end{proof}
				The degree of the map $t\mapsto \det (P_t)$ can be expressed solely in terms of the initial matrix $\mathcal A$.
				\begin{lem}\label{lem: expression winding det Pt} Let $\mathcal A(t)=\begin{pmatrix}
						A_t & B_t \\ C_t & D_t
					\end{pmatrix}$ be a loop of symplectic matrices, and let $\mathcal A_c(t)=\begin{pmatrix}
						P_t & Q_t \\ \overline Q_t & \overline P_t
					\end{pmatrix}$ denote the complexification of $\mathcal A(t)$. Then, we have the two formulas
					\[\mathrm{deg} \det(P_t)=\operatorname{deg} \det(A_t+\mathrm{i} C_t)=\operatorname{deg} \det(A_t-\mathrm{i}B_t).\]
					
				\end{lem}
				
				\begin{proof}We prove the first equality. One has $2P_t=A_t+\mathrm{i}C_t-\mathrm{i} B_t+D_t$ and $2Q_t=A_t+\mathrm{i} C_t+ \mathrm{i} B_t-D_t$, so that $P_t +Q_t=A_t+\mathrm{i} C_t$. It follows that
					\[\det(A_t+\mathrm{i} C_t)=\det(I+P_t^{-1}Q_t)\det(P_t),\]
					and thus, letting $R_t=P_t^{-1}Q_t$, we have
					\[\operatorname{deg} \det(A_t+\mathrm{i}C_t)=\operatorname{deg}\det(P_t)+\operatorname{deg} \det(I+R_t).\]
					Since $\|R_t\|<1$ (by Lemma \ref{lem: matrices P Q}), the function $\varepsilon\mapsto \operatorname{deg}\det(I+\varepsilon R_t)$ is continuous on $[0,1]$. Since it vanishes at $\varepsilon=0$, it also vanishes at $\varepsilon=1$. 
					
					For the equality involving $A_t-\mathrm{i}B_t$, just note that $P_t+\overline Q_t=A_t-\mathrm{i}B_t$ and run the same argument but use $\|\overline Q_tP_t^{-1}\|<1$ this time. \end{proof}

	\appendix
		
\section{Semiclassical analysis on manifolds} This appendix provides a self-contained treatment of semiclassical pseudodifferential calculus on closed manifolds. We assume familiarity with the semiclassical pseudodifferential calculus on $\mathbf{R}^d$ and refer to the books of Dimassi--Sj\"ostrand \cite{dimassi} or Zworski \cite{zworski2012semiclassical} for this matter. Great sources of inspiration for these notes are the appendix of \cite{DJN21} and the notes of Guedes Bonthonneau \cite{bonthonneauWeyl,bonthonneauEDP}, which dealt with the $C^\infty_{\rm comp}$ calculus. Here we put emphasis on operators associated with symbols lying in some exotic Kohn--Nirenberg classes. In the whole appendix, $h$ is allowed to vary in $(0,1]$. The only exception is Corollary \ref{cor: exact inverse} where $h$ needs to be taken small enough. Apart from that, all the results are valid in particular for $h=1$, i.e. in the usual pseudodifferential calculus.

This section can be read independently of the rest of the paper.

\label{app: semiclassical on manifolds}
\subsection{Symbol classes}

We introduce the symbol classes that are used in this paper. We use the notation $\langle \xi\rangle:=(1+|\xi|^2)^{1/2}$ for $\xi\in \mathbf{R}^d$. We denote by $\mathrm{i}$ a square root of $-1$. Let $M$ be a smooth, compact, connected manifold, equipped with a smooth volume form $\operatorname{vol}$. The cotangent space of $M$ is denoted by $T^*M$. Functions on $T^*M$ are called \emph{symbols}.

\begin{defi}Let $m\in \mathbf R$. Consider a symbol $a\in C^\infty(T^*M)$. We define the seminorms:
\[\|a\|_{m,k}:=\sup_{|\alpha|+|\beta|\le k} \sup_{(x,\xi)\in T^*M}\langle \xi\rangle^{|\alpha|-m}|\partial_\xi^\alpha\partial_x^\beta a|.\]
Here the derivatives should be computed using a fixed, finite atlas of $M$. We say that $a\in S^m$ if $\|a\|_{m,k}<+\infty$ for all $k\ge 0$.  \end{defi}
In the following, we will allow the symbol $a\in S^m$ to depend on a semiclassical parameter $h\in (0,1]$. Some interesting cases are:
\begin{itemize}

\item The standard Kohn--Nirenberg class $S_h^m$ where $\|a\|_{m,k}\lesssim_{\!\! k} 1$ independently of $h$.
\item The exotic classes of symbols $S_{h,\delta}^m$ where $\|a\|_{m,k}\lesssim_{\!\! k} h^{-\delta k}$ for some $\delta\in [0,1/2)$.
\item The class $S^{\rm comp}_{h,\delta}(T^*M)$ where $a$ is uniformly supported in a compact set and $\|a\|_{C^k}\lesssim h^{-\delta k}$.
\item We will also consider the class $S_{h,\delta}^{m,2}$ of symbols satisfying $\|a\|_{m,k}\lesssim h^{-\delta(k-2)_+}$. This class is of interest when studying quantum dynamics because if $A$ is a pseudodifferential operator with symbol $a$, then the quantum propagator $\exp(-\mathrm{i} h^{-1}t A)$ can be expressed, at least for short times, as a Fourier integral operator (FIO) with a phase function $\varphi$ and amplitude $b$ satisfying $\|\varphi\|_{C^\ell}\lesssim h^{-\delta(\ell-2)_+}$ and $\|b\|_{C^\ell}\lesssim h^{-\delta \ell}$. In fact, the estimate $\|a\|_{m,k}\lesssim h^{-\delta(k-2)_+}$ appears to be a critical regularity assumption on the symbol which ensures that the propagator $\exp(-\mathrm{i} h^{-1}t A)$ fits into the theory of $h$-FIOs with exotic phase and amplitude, see \cite{Blair2022}.

\end{itemize}  
In practice, in most of the appendix we do not think of $a$ as a function of $h$ but rather let $(a,h)$ vary in $S^m\times (0,1]$.

\begin{rem}More generally, we could introduce an auxiliary parameter $\tilde h\in (0,1]$ and work with symbols satisfying $\|a\|_{m,k}\lesssim (h\tilde h^{-1})^{-k/2}$. The classes above correspond to the choice $\tilde h=h^{1-2\delta}$ with $\delta\in [0,1/2)$. All the results stated in the following work in this more general setting: one simply has to replace $h^{\delta}$ by $(h\tilde h^{-1})^{1/2}$ everywhere. 

\end{rem}\medskip

\noindent \textbf{Notation.} Let $\mathbf x$ vary in a space $\mathbf X$ and $\{\bullet\}$ be an auxiliary set of parameters. let $f,g$ be defined in terms of $\mathbf x$ and $\{\bullet\}$. If $\mathcal C=(c_i)_{i\in I}$ is a finite or infinite collection of nonnegative numbers depending on $\mathbf x$ but not $\bullet$, the notation $f\lesssim^{\!\! \operatorname{Pol}_k(\mathcal C)}_{\!\! \bullet} g$ or $f=\mathcal O_\bullet^{\operatorname{Pol}_k(\mathcal C)}(g)$ means that there is a polynomial expression $P_\bullet=P_\bullet(c_{i_1},\ldots,c_{i_n})$ of degree $\le k$ depending on $\bullet$ and defined on $\mathbf{R}_+^I$ such that for all $\mathbf x\in \mathbf X$, we have $f\le P_\bullet(\mathcal C) g$. We simply write $f\lesssim^{\!\! \operatorname{Pol}(\mathcal C)}_{\!\! \bullet} g$ if the degree of the polynomial is not specified.

The space $\mathbf X$ will often be implicit in the following. Typically, we may take $\mathbf X=S^m$, $\mathbf X=S^m\times (0,1]_h$ or $\mathbf X=S^{m_1}\times S^{m_2}\times (0,1]_h$ (data of one or two symbols and a semiclassical parameter $h\in (0,1]$). The auxiliary parameters $\{\bullet\}$ will always be precised.

\subsection{Formal power series and truncation operators} We introduce formal power series in $h$, which will be useful to state results and carry out the proofs. 

In various presentations of the semiclassical calculus, one works with \emph{classical symbols}, which are $h$-dependent functions $a(x,\xi;h)$ on $T^*M$ that admit a formal expansion 
\begin{equation} \label{eq: formal expansion classical symbol}a(x,\xi;h)\sim \sum h^k a_k(x,\xi),\end{equation}
in the sense that $a(x,\xi;h)-\sum_{k<N}h^k a_k(x,\xi)=\mathcal O(h^N)$ for all $N$. In turn, a family of pseudodifferential operators $A=A(h)$ is \emph{classical} if there exists a classical symbol $a$ such that $A=\mathrm{Op}_h^{\rm w}(a)$, up to some $\mathcal O(h^\infty)$ remainder in a suitable class of smoothing operators. This classicality property is invariant by change of coordinates and composition of operators. Also, it allows defining the \emph{principal symbol} of $A=\mathrm{Op}_h^{\rm w}(a)$ as $a_0$. Moreover, $a_0+ha_1$ is defined canonically and invariant under measure-preserving change of coordinates. Another advantage of working with classical symbols which simplifies the exposition is that one can identify coefficients of equal expansions in powers of $h$. By that we mean that if $a\sim \sum h^k a_k$ and $b\sim \sum h^k b_k$ are classical symbols, then $a=b$ implies $a_k=b_k$ for all $k\ge 0$. In particular, one can define truncation operators $T_N(a):=\sum_{k<N} h^k a_k$. \medskip

For \emph{nonclassical symbols} that do not admit an expansion \eqref{eq: formal expansion classical symbol}, and associated \emph{nonclassical operators}, one cannot define a principal symbol associating a genuine function to the operator, nor define truncation operators for symbols. In fact, identifying coefficients in $h$-expansions becomes meaningless if the coefficients depend themselves on $h$. \medskip

To remedy this issue, we introduce formal power series in $h$. To avoid confusion with the actual semiclassical parameter $h$, we denote the formal variable by $\underline{h}$. In practice, the expansions that appear in the calculus are rather in powers of $\mathrm{i} h$ rather than $h$ itself hence we will work with the formal variable $\mathrm{i}\underline h$. Let $\mathcal A:=C^\infty(T^*M)$. Consider a formal power series 
\[\mathbf a=\sum_{n\ge 0} (\mathrm{i}\underline{h})^n a_n\in \mathcal A[\![\mathrm{i}\underline h]\!].\]
We might also write $\mathbf a_n$ to denote the $n$-th coefficient of $\mathbf a$.
\begin{defi} Let $h\in (0,1]$. We define a truncation operator $T_N$ (that depends implicitly on $h$) that maps a formal symbol $\mathbf a\in \mathcal A[\![\mathrm{i}\underline h]\!]$ to the function $T_N(\mathbf a)\in C^\infty(T^*M)$ obtained by truncating the series at order $N-1$ and evaluating $\underline h=h$:
\[T_N(\mathbf a)=\sum_{n<N} (\mathrm{i}h)^n a_n.\]\end{defi}
Any smooth function $a\in C^\infty(T^*M)$ may be identified with a formal power series 
\[\iota(a):=a\cdot (\mathrm{i}\underline{h})^0+0\cdot (\mathrm{i}\underline{h})^1+\ldots,\]
which will usually still be denoted by $a$. Also, any linear operator $L:\mathcal A\to \mathcal A[\![\mathrm{i}\underline h]\!]$ naturally extends to an operator $\mathcal A[\![\mathrm{i}\underline h]\!]\to \mathcal A[\![\mathrm{i}\underline h]\!]$ by linearity, by setting
\[L\Big(\sum_{n\ge 0} (\mathrm{i}\underline h)^n a_n\Big):=\sum_{n\ge 0} (\mathrm{i}\underline h)^n La_n.\]

\begin{rem}Guedes Bonthonneau \cite{bonthonneauEDP} suggests an almost similar approach (see the footnote on page 2 of \cite{bonthonneauEDP}) of introducing an auxiliary parameter $0<\tilde h<h^\delta$ which is a priori independent of $h$ and considering functions $a(x,\xi,\tilde h,h)$ depending on two variables $h,\tilde h$ and satisfying bounds $\|a\|_{m,k}\lesssim \tilde h^{-\delta k}$. This also solves most of the previously raised issues because we can then again make sense of $h$-expansions.  \end{rem}

\subsection{Pseudodifferential calculus on $\mathbf{R}^d$} We recall some basic elements on the exotic semiclassical pseudodifferential calculus on $\mathbf{R}^d$, see the books of Dimassi--Sj\"ostrand \cite{dimassi} or Zworski \cite{zworski2012semiclassical}.

\subsubsection{Composition and change of variables}
\begin{defi}Let $a\in S^m(T^*\mathbf{R}^d)$. The \emph{Weyl quantization} of $a$ is the operator $\mathrm{Op}_h^{\rm w}(a):\mathcal S(\mathbf{R}^d)\to \mathcal S'(\mathbf{R}^d)$ defined by
\[\mathrm{Op}_h^{\rm w}(a)u(x):=\frac{1}{(2\pi h)^d}\int_{\mathbf{R}^d\times \mathbf{R}^d} \mathrm{e}^{\frac{\mathrm{i}}{h} \langle x-y,\xi\rangle} a\big(\frac{x+y}{2},\xi\big)u(y)\mathrm{d}y\mathrm{d}\xi.\]\end{defi}
\begin{rem}	More generally, one can quantize symbols $a(x,y,\xi)\in S^m(\mathbf{R}^{2d}\times \mathbf{R}^d)$ into operators
\[\mathrm{Op}_h(a)u(x):=\frac{1}{(2\pi h)^d}\int_{\mathbf{R}^d\times \mathbf{R}^d} \mathrm{e}^{\frac{\mathrm{i}}{h} \langle x-y,\xi\rangle} a(x,y,\xi)u(y)\mathrm{d}y\mathrm{d}\xi.\]\end{rem}

We recall the composition formula for the Weyl quantization on $\mathbf{R}^d$. If $f\in C^\infty(T^*\mathbf{R}^d\times T^*\mathbf{R}^d)$ we let $f_{|\rm Diag}$ denote the pullback of $f$ by the diagonal embedding $T^*\mathbf{R}^d\subset T^*\mathbf{R}^d\times T^*\mathbf{R}^d$.

\begin{prop}[Composition formula in $\mathbf{R}^d$]\label{prop: composition Rd} Let $a\in S^{m_1}(T^*\mathbf{R}^d)$, $b\in S^{m_2}(T^*\mathbf{R}^d)$. Then, there is a symbol $a\#_h b\in S^{m_1+m_2}(T^*\mathbf{R}^d)$ such that
\[\mathrm{Op}_h^{\rm w}(a)\mathrm{Op}_h^{\rm w}(b)=\mathrm{Op}_h^{\rm w}(a\#_h b).\]
The symbol $a\#_h b$ admits the following asymptotic expansion:
\[a\#_h b(x,\xi)=\sum_{n<N}\frac{(\mathrm{i} h)^n}{2^nn!}\big[A_{2n}(D)(a\otimes b)\big]_{|\rm Diag}+ h^Nr_N.\]
Here $A_{2n}(D)$ is a differential operator of order $2n$ acting on $C^\infty({T^*\mathbf{R}^d}_{(x,\xi)}\times {T^*\mathbf{R}^d}_{(y,\eta)})$ given by
\[A_{2n}(D)=\sigma(D)^n,\qquad \sigma(D_x,D_\xi;D_y,D_\eta)=D_y\cdot D_\xi-D_x\cdot D_\eta.\]
The remainder $r_N$ belongs to $S^{m_1+m_2-N}$. Let $\mathcal C=\{h^{\delta(j+k)}\|a\|_{m_1,j}\|b\|_{m_2,k}~:~ j,k\ge 0\}$. Then one has
\[\|r_N\|_{m_1+m_2-N,k}\lesssim_{\!\! N,k}^{\!\! \operatorname{Pol}_1(\mathcal C)} h^{-\delta(2N+k)}.\]
\end{prop}

Next, we recall how pseudodifferential operators transform under change of coordinates.

\begin{prop}\label{prop: change of variable pseudo euclidean} Let $a\in C^\infty(T^*\mathbf{R}^d)$, let $\chi$ be a cutoff function supported in an open set $U$, and let $\varphi:U\to V$ be a smooth diffeomorphism preserving the Lebesgue measure. Define
\[B:=\chi \varphi^* \mathrm{Op}_h^{\rm w}(a)\varphi_*\chi.\]
Then, $B=\mathrm{Op}_h^{\rm w}(\varphi_*b)$, for some explicit symbol $b\in S^m(T^*\mathbf{R}^d)$ which admits an expansion
\[b=\sum_{n<N, \ n \ \rm{even}} (\mathrm{i}h)^nb_n+h^{N}r_N,\]
with $b_0=\chi^2 a$. The higher order terms have the form
\[b_n=\mathbf D^{2n}_{\varphi,\chi} a.\]
Here $\mathbf D^{2n}_{\varphi,\chi}$ is a differential operator of order $2n$ with real coefficients that are independent of $\xi$, and always $\ge n$ derivatives hitting the $\xi$ variable. The remainder satisfies $r_N\in S^{m-N}$. Let $\mathcal C=\{h^{\delta k}\|a\|_{m,k}~:~k\ge 0\}$. Then,
\[\|r_N\|_{m-N,k}\lesssim_{\!\! N,k}^{\!\!\operatorname{Pol}_1(\mathcal C)} h^{-\delta(2N+k)}.\]
\end{prop}
\begin{rem}The fact that only even powers of $h$ appear in the expansion is due to the fact that $\varphi$ is volume-preserving. Indeed, if $a$ is a real symbol then $\mathrm{Op}_h^{\rm w}(a)$ is formally self-adjoint and since $\varphi$ preserves the volume, we find that $\chi \varphi^* \mathrm{Op}_h^{\rm w}(a)\varphi_*\chi$ is also formally self-adjoint. This forces all the coefficients of odd powers of $\mathrm{i}h$ to vanish in the expansion.
\end{rem}

The next lemma shows that for a differential operator $P$ on $C^\infty(T^*\mathbf{R}^d)$, the property that $P$ has coefficients independent of $\xi$ and that at least $k$ derivatives fall on the $\xi$ variable is invariant under symplectic change of coordinates induced by diffeomorphisms on $\mathbf{R}^d$. Hence this notion is well defined for differential operators on $C^\infty(T^*M)$ when $M$ is an arbitrary manifold.

\begin{lem}Let $P$ be a differential operator on $C^\infty(T^*\mathbf R^d)$ of the form
\begin{equation} \label{eq: diff op with k deriv} P=\sum_{|\alpha|= k}\sum_{|\beta|\le N} a_{\alpha \beta}(x)\partial_\xi^\alpha\partial_x^\beta.\end{equation}
Consider a diffeomorphism $\varphi:U\subset \mathbf{R}^d\to V\subset \mathbf{R}^d$ and a cutoff function $\chi\in C^\infty_{\rm comp}(V)$. Let $\widetilde \varphi:T^*U\to T^*V$ be the symplectic lift of $\varphi$. Then, the operator $Q:=\widetilde \varphi_* \chi P\widetilde \varphi^*\chi$ is still of the form \eqref{eq: diff op with k deriv}, i.e. one has
\[Q=\sum_{|\alpha|=k}\sum_{|\beta|\le N} b_{\alpha\beta}(x)\partial_\xi^\alpha\partial_x^\beta,\]
for some new coefficients $b_{\alpha\beta}$.
\end{lem}

\subsubsection{Continuity in Sobolev spaces} We first recall the Calder\'on--Vaillancourt Theorem, see e.g. \cite[Theorem 4.23]{zworski2012semiclassical}.
\begin{prop}[$L^2$-continuity for bounded symbols] \label{prop: calderon} There is a dimensional constant $C_d>0$ such that for any $a\in C^\infty(T^*\mathbf{R}^d)$, one has
\[\|\mathrm{Op}_h^{\rm w}(a)\|_{L^2\to L^2}\le C_d\sum_{j\le C_d} h^{j/2}\|\partial^j a\|_{L^\infty(T^*\mathbf{R}^d)}.\]
In particular, if $\|a\|_{C^\ell}\lesssim h^{-\ell/2}$ then $\|\mathrm{Op}_h^{\rm w}(a)\|_{L^2\to L^2}\lesssim 1$.
\end{prop}
The semiclassical Sobolev space $H_h^s(\mathbf{R}^d)$ is defined as $\mathrm{Op}_h^{\rm w}(\langle \xi\rangle^s)^{-1}(L^2(\mathbf{R}^d))$. Note that $\mathrm{Op}_h^{\rm w}(\langle \xi\rangle^s)=(1+h^2\Delta)^{s/2}$, and that the operator $\mathrm{Op}_h^{\rm w}(\langle \xi\rangle^s)$ is invertible, with inverse $\mathrm{Op}_h^{\rm w}(\langle \xi\rangle^{-s})$. The $H_h^r\to H_h^s$ norm of an operator $A$ can be expressed as
\[\|A\|_{H_h^{r}\to H_h^{s}}= \|\mathrm{Op}_h^{\rm w}(\langle \xi\rangle^{s}) A\mathrm{Op}_h^{\rm w}(\langle \xi\rangle^{-r})\|_{L^2\to L^2}.\]
\begin{prop}[Continuity on Sobolev spaces]\label{prop: continuity Sobolev} Assume $a\in S^m(T^*\mathbf{R}^d)$ and $s\le r-m$. Let $\mathcal C=\{h^{\delta k}\|a\|_{m,k}\}$. Then,
\[\|\mathrm{Op}_h^{\rm w}(a)\|_{H_h^{r}\to H_h^{s}}\lesssim_{\!\! r,s}^{\!\!\operatorname{Pol}_1(\mathcal C)} 1.\] \end{prop}

\subsection{Pseudodifferential operators on manifolds} Let $(M,\rm{vol})$ be a closed manifold, equipped with a smooth volume form $\operatorname{vol}$. To avoid working with half-densities, following Guedes Bonthonneau \cite{bonthonneauWeyl}, we endow $M$ with a finite atlas $\{(U_\ell,\varphi_\ell)\}$ of \emph{isochore charts}. It means that all the chart maps push forward the volume form onto the Lebesgue measure, i.e. $(\varphi_\ell)_*\mathrm{vol}_M=\mathrm{Leb}_{\mathbf{R}^d}$. To the atlas $\{(U_\ell,\varphi_\ell)\}$ of $M$ we add a quadratic partition of unity $\{\chi_\ell\}$, that is a family of functions $\chi_\ell\in C^\infty(M,[0,1])$ such that $\operatorname{Supp}(\chi_\ell)\subset U_\ell$ and $\sum \chi_\ell^2=1$. Also, we consider functions $\chi_\ell'\in C^\infty(M)$ such that $\operatorname{Supp}(\chi_\ell')\subset U_\ell$ and $\chi_\ell\prec \chi_\ell'$. We denote $\Xi=\{(U_\ell,\varphi_\ell,\chi_\ell',\chi_\ell)\}$. All our constructions will depend on $\Xi$. We will index elements of $\Xi$ by the letter $\ell$.

\begin{defi}[Sobolev spaces] Let $r\in \mathbf R$. The space $H_h^r(M)$ is defined as the completion of $C^\infty(M)$ with respect to the norm
\[\|u\|_{H_h^r(M)}^2=\sum_\ell \|(\varphi_\ell)_*\chi_\ell u\|_{H_h^r(\mathbf{R}^d)}^2.\]
Note that $\|\bullet\|_{H^0_h(M)}=\|\bullet\|_{L^2(M)}$.
\end{defi}

\subsection{Quantization} One can assign an operator to any symbol $a\in S^{m}$ through a quantization procedure, which will be fixed once and for all.

In the following, to avoid heavy notation, we just write $\chi_\ell$ instead of $\chi_\ell\circ \widetilde \varphi_\ell^{-1}$ for the pushforward of $\chi_\ell$ by the chart $\varphi_\ell$, and similarly for $\chi_\ell'$.
\begin{defi}[Weyl quantization on manifolds] \label{defi: weyl quanti manifold} For $a\in S^m(T^*M)$, define
\[\mathrm{Op}_h^M(a):=\sum_{\ell\in \Xi} (\varphi_\ell)^* \chi_\ell \mathrm{Op}^{\rm w}_h(\chi_\ell'a\circ \widetilde \varphi_\ell^{-1}) (\varphi_\ell)_* \chi_\ell.\]
\end{defi}

\begin{thm} The main properties of the calculus are as follows.
\begin{enumerate}[\normalfont(i)]
\item \emph{\textbf{$\xi$-independent symbols.}} If $a\in C^\infty(M)$ is viewed as a function on $T^*M$ that is constant over the fibers, then \[\operatorname{Op}_h^{M}(a)u(x)=a(x)u(x).\]
In particular, $\operatorname{Op}_h^{M}(1)=\operatorname{id}$.
\item \emph{\textbf{Adjoint.}} Let $a\in S^m(T^*M)$. For any $u,v\in C^\infty(M)$, one has
\[\langle \operatorname{Op}_h^{M}(a)u,v\rangle_{L^2(M)}=\langle u,\operatorname{Op}_h^{M}(\bar a)v\rangle_{L^2(M)}.\] Here $\bar a$ denotes the complex conjugate of $a$. In particular, if $a$ is real then $\operatorname{Op}_h^{M}(a)$ is formally self-adjoint on $L^2(M)$.
\item \emph{\textbf{Composition.}} (see Proposition \ref{prop: composition manifold}). There is an associative law $\diamond$ on $\mathcal A[\![\mathrm{i}\underline h]\!]$ such that for any $(a,b)\in S^{m_1}\times S^{m_2}$, for any $N\ge 0$, one has
\[\operatorname{Op}_h^{M}(a)\circ \operatorname{Op}_h^{M}(b)=\operatorname{Op}_h^{M}(T_N(a\diamond b))+R_N.\]
Moreover, letting $\mathcal C=\{h^{\delta (j+k)}\|a\|_{m_1,k}\|b\|_{m_2,k}\}$, one has that for any $s\le r-(m_1+m_2-N)$:
\[\|R_N\|_{H_h^r\to H_h^s}\lesssim_{\!\! N,r,s}^{\!\!\operatorname{Pol}_1(\mathcal C)} h^{N(1-2\delta)}.\]
At first order, we have $a\diamond b=ab-\frac{1}{2}\{a;b\}\mathrm{i}\underline h+\underline h^2\cdots$.
\item \emph{\textbf{Commutator.}} (see Proposition \ref{prop: commutator PDOs}). The difference $a\diamond b-b\diamond a$ does not feature even powers of $\underline h$.
\item \emph{\textbf{Inverse.}} (see Theorem \ref{thm: parametrix}) A formal power series $\mathbf a=a_0+\mathrm{i}\underline h\cdots$ is invertible in $\mathcal A[\![\mathrm{i}\underline h]\!]$ for $\diamond$ if and only if $a_0$ is nonvanishing. The inverse $\mathbf a^{\diamond_{-1}}$ satisfies $\mathbf{a}^{\diamond_{-1}}=a_0^{-1}+\mathrm{i}\underline h\cdots$.	
\item \emph{\textbf{Functional calculus.}} (see Theorem \ref{thm: functional calculus}) For any real symbol $a\in S^m$, such that $a+\mathrm{i}$ is elliptic in $S^m$, the operator $\mathrm{Op}_h^{M}(a)$ is self-adjoint on $L^2(M)$ with domain $H_h^m(M)$. For any smooth compactly supported function $f$, there is a formal power series $\bm{f}[a]$ such that for any $N\ge 0$, we have
\[f(\operatorname{Op}_h^{M}(a))=\operatorname{Op}_h^{M}(T_N(\bm{f}[a]))+R_N.\]
If $|a(x,\xi)+\mathrm{i}|\ge K\langle \xi\rangle^m$ with $K\in (0,1]$ and $\mathcal C=\{h^{\delta k}\|a\|_{m,k}\}$, we have
\[\|R_N\|_{H_h^r\to H_h^s}\lesssim_{\!\! f,N,r,s}^{\!\!\operatorname{Pol}(\mathcal C,K^{-1})} h^{N(1-2\delta)}.\]
One has $\bm{f}[a]=f(a)+\underline h^2\cdots$. Moreover, given another function $g$, we have $\bm{fg}[a]=\bm{f}[a]\diamond \bm{g}[a]$.
\end{enumerate}
\end{thm}

\subsubsection{Operators in charts}

The action of $\mathrm{Op}_h^M(a)$ in charts is as follows. It is a direct consequence of Proposition \ref{prop: change of variable pseudo euclidean}.
\begin{lem}[Operator in a chart]\label{lem: operator in a chart}Consider an isochore chart $(\varphi,\chi)$. For any $a\in S^m(T^*M)$ there is a formal symbol $L_{\varphi,\chi}a=\sum_{n=0}^{+\infty} (\mathrm{i}\underline{h})^n L_{\varphi,\chi,n}a$ such that for all $N\ge 0$:
\[\varphi_*\chi\mathrm{Op}_h^M(a)\varphi^*\chi=\mathrm{Op}_h^{\rm w}(\varphi_* T_N(L_{\varphi,\chi} a)+h^Nr_N).\]
Here $L_{\varphi,\chi,n}$ is a differential operator of order $2n$ with real coefficients independent of $\xi$ and always $\ge n$ derivatives hitting the $\xi$ variable. One has $L_{\varphi,\chi,0}a=\chi^2 a$, and $L_{\varphi,\chi,n}=0$ for all $n$ odd. Let $\mathcal C=\{h^{\delta k}\|a\|_{m,k}\}$. The remainder satisfies
\[\|r_N\|_{m-N,k}\lesssim_{\!\! N,k}^{\!\!\operatorname{Pol}_1(\mathcal C)} h^{-\delta(2N+k)}.\]
\end{lem}

\subsubsection{Continuity on Sobolev spaces} We now give the analogue of Proposition \ref{prop: continuity Sobolev} for manifolds.
\begin{lem}[Continuity on Sobolev spaces on manifolds] \label{lem: continuity sobolev mfld} Let $a\in S^m(T^*M)$, and set $\mathcal C=\{h^{\delta k}\|a\|_{m,k}\}$. Then, for all $r\in \mathbf{R}$,
\[\|\operatorname{Op}_h^{M}(a)\|_{H_h^r\to H_h^{r-m}}\lesssim_{\!\! r,m}^{\!\!\operatorname{Pol}_1(\mathcal C)} 1.\] \end{lem}
\begin{proof}For $u\in H_h^{r-m}(M)$, by the definition of the quantization and the Sobolev norm, and by the triangle inequality, we have:
\begin{equation*}
\begin{aligned}
\|\operatorname{Op}_h^{M}(a)u\|_{H_h^{r-m}(M)}^2& =\sum_k \|(\varphi_k)_*\chi_k \operatorname{Op}_h^{M}(a)u\|_{H_h^{r-m}(\mathbf{R}^d)}^2 \\
& \le \sum_{k,\ell} \|(\varphi_k)_* \chi_k\chi_\ell (\varphi_\ell)^* \mathrm{Op}_h^{\rm w}((\varphi_\ell)_* \chi_\ell' a_\ell) (\varphi_\ell)_* \chi_\ell u\|_{H_h^{r-m}(\mathbf{R}^d)}^2 \\
& \lesssim \sum_\ell \|\mathrm{Op}_h^{\rm w}((\varphi_\ell)_* \chi_\ell' a_\ell) (\varphi_\ell)_* \chi_\ell u\|_{H_h^{r-m}(\mathbf{R}^d)}^2 \\
& \lesssim \sum_\ell \|(\varphi_\ell)_*\chi_\ell u\|_{H_h^{r}(T^*\mathbf{R}^d)}^2=\|u\|_{H_h^r(M)}^2.
\end{aligned}
\end{equation*}
In the last two lines we have used the fact that multiplication by smooth functions and local changes of coordinates are continuous operators on $H_h^r(\mathbf{R}^d)$ with $h$-independent norms.
\end{proof}
We finally give a pseudolocality property for operators $\operatorname{Op}_h^{M}(a)$, which directly follows from the result on $\mathbf{R}^d$.
\begin{lem}[Pseudolocality of pseudodifferential operators]\label{prop: pseudolocality} Let $\psi_1,\psi_2\in C^\infty(M)$ be independent of $h$, with disjoint supports. Let $\mathcal C=\{h^{\delta k}\|a\|_{m,k}\}$. Then, for any $r,s\in \mathbf{R}$, one has:
\[\|\psi_1\operatorname{Op}_h^{M}(a)\psi_2\|_{H_h^r\to H_h^s}\lesssim_{\!\! r,s,\psi_1,\psi_2}^{\!\!\operatorname{Pol}_1(\mathcal C)} 1.\]

\end{lem}

\subsection{Reconstruction of the symbol} Lemmas \ref{lem: operator in a chart}--\ref{prop: pseudolocality} show that operators $\operatorname{Op}_h^{M}(a)$ behave as pseudodifferential operators in charts, and enjoy a pseudolocality property. It turns out that the converse is true: any reasonable pseudolocal operator that behaves as a pseudodifferential operator of order $m$ in charts can be expressed, up to a negligible remainder, as $\operatorname{Op}_h^{M}(a)$ for some symbol $a\in S^m$.

We give a general formulation. We let $\Xi=\{(U_\ell,\varphi_\ell,\chi_\ell',\chi_\ell)\}$ be as above. The next definition is not the usual one for $h$-pseudodifferential operators on manifolds. Yet, it is practical to attach to an operator the seminorms of its symbol in charts together with the constants in the pseudolocality estimates, see Lemma \ref{lem: lemma before construction of symbol} below.

\begin{defi}\label{defi: pseudo in chart} Let $m\in \mathbf{R}$. Let $\mathcal C=(C_{k})_{k\in \mathbf N}$ be a nondecreasing sequence of positive constants. Let $\vec{\mathbf a}=\{\mathbf a_{\ell}\}_{\ell\in \Xi}$ be an array of formal power series in $\mathrm{i}\underline h$, with $\mathbf a_{\ell}$ supported on $\{\chi_\ell'\neq 0\}$. We say that a sequence of operators $(A^N)_{N\ge 0}$ satisfies $\mathbf H \big(h,m,\vec{\mathbf a},\mathcal C)$ if
\begin{enumerate}
\item For all $N\ge 0$, for each $\ell\in \Xi$ we have:
\begin{equation} \label{eq: assume express charts AN}(\varphi_\ell)_*\chi_\ell' A^N \chi_\ell'\varphi_\ell^*=\mathrm{Op}_h^{\rm w}((\varphi_\ell)_*T_N(\mathbf a_{\ell}))+R_{N,\ell}.\end{equation}
with $\|R_{N,\ell}\|_{H_h^r\to H^{s}}\le C_{N+|r|+|s|} h^{N(1-2\delta)}$ provided $s\le r-(m-N)$. 
\item The coefficients of the expansion $\mathbf a_\ell:=\sum (\mathrm{i}\underline h)^k a_{\ell,k}$ satisfy $a_{\ell,k}\in S^{m-k}$, with
\[\|a_{\ell,k}\|_{m-k,j}\le C_{j+k} h^{-\delta(2k+j)}.\]
\item For any $\ell\in \Xi$ and any $s,r\in \mathbf R$:
\begin{equation} \label{eq: assum pseudo loc AN} \big\|(1-\chi_\ell') A^N\chi_\ell\big\|_{H_h^r\to H_h^s}\le C_{N+|r|+|s|} h^{N(1-2\delta)} , \quad \big\|\chi_\ell A^N(1-\chi_\ell')\big\|_{H_h^r\to H_h^{s}}\le C_{N+|r|+|s|} h^{N(1-2\delta)}.\end{equation}
\end{enumerate}
\end{defi} 
\begin{rem} Lemmas \ref{lem: operator in a chart}--\ref{prop: pseudolocality} precisely assert the following. For $a\in S^m$ and $h\in (0,1]$, let $C_k(a,h)=h^{\delta k}\|a\|_{m,k}$ and $\mathcal C=\{C_k(a,h)\}_{k\in \mathbf N}$. Then the constant family $(\operatorname{Op}_h^{M}(a))_{N\ge 0}$ satisfies $\mathbf H(h,m,\vec{\mathbf a},\mathcal C')$ with $\mathbf a_{\ell}=L_{\varphi_\ell,\chi_\ell'}a$, and where $\mathcal C'=\{C_k'\}_{k\in \mathbf N}$ satisfies $C'_k\lesssim^{\!\! \operatorname{Pol}_1(\mathcal C)}_{\!\! k,\Xi} 1$. In other words the new constants $C_k'$ are bounded by linear expressions in the variables $C_k(a,h)$.
\end{rem}

\begin{prop}\label{prop: constructing the symbol}Let $A$ be a linear operator on $C^\infty(M)$ such that the constant sequence $(A^N)_{N\ge 0}=(A)_{N\ge 0}$ satisfies $\mathbf H(h,m,\vec{\mathbf a},\mathcal C)$. Denote $\vec{\mathbf a}=\{\mathbf a_{\ell}\}_{\ell\in \Xi}$. There is a formal symbol 
\begin{equation} \label{eq: sigma tot}\sigma_{\rm tot}(\vec{\mathbf a})=\sum_{n\ge 0}(\mathrm{i}\underline h)^n \mathbf D^{2n}_{\Xi}\vec{\mathbf a}.\end{equation}
such that if $s\le r-(m-N)$:
\[A=\mathrm{Op}_h^M\big(T_N(\sigma_{\rm tot}(\vec{\mathbf a}))\big)+\mathcal O_{N,r,s,\Xi}^{\operatorname{Pol}_1(\mathcal C)}(h^{N(1-2\delta)})_{H_h^r\to H_h^{s}}.\]
Here $\mathbf D_{\Xi}^{2n}$ is a matrix differential operator of order $2n$ with real coefficients independent of $\xi$, and has $\ge n$ derivatives hitting the $\xi$ variable. The first terms are given by 
\[\mathbf D_{\Xi}^0 \vec{\mathbf a}=\pi(\vec{\mathbf a})=\sum \chi_\ell^2 \mathbf a_{\ell}, \qquad \mathbf D_{\Xi}^2 \vec{\mathbf a}=0.\]
\end{prop}
Note that the formal symbol $\sigma_{\rm tot}(\vec{\mathbf a})$ is canonically associated with $\vec{\mathbf a}$ but not with $A$. Proposition \ref{prop: constructing the symbol} is obtained from the next lemma, that will be used inductively in the construction of the symbol.

\begin{lem}\label{lem: lemma before construction of symbol} Let $(A^N)_{N\ge 0}$ satisfy $\mathbf{H}(h,m,\vec{\mathbf a},\mathcal C)$. Define
\[\pi(\vec{\mathbf a}):=\sum_\ell \chi_\ell^2\mathbf a_{\ell}.\]
Set $B^N:= A^N-\operatorname{Op}_h^{M}(T_N(\pi(\vec{\mathbf a}))$. Then $(B^N)_{N\ge 0}$ satisfies $\mathbf{H}(h,m,\vec{\mathbf b},\mathcal C')$ where:
\begin{itemize}
\item  There is a matrix operator $Q:=\sum_{n\ge 2} (\mathrm{i}\underline h)^n \mathbf D^{2n}_\Xi$ acting on arrays of formal power series $\vec{\mathbf a}\in \mathcal A[\![\mathrm{i}\underline{h}]\!]^\Xi$ such that 
\begin{equation} \label{eq: def mathbf b}\vec{\mathbf b}=Q\vec{\mathbf a}.\end{equation} 
Here $\mathbf D^{2n}_\Xi$ is a matrix of differential operators of order $2n$ with \textbf{real} coefficients independent of $\xi$ and always $\ge n$ derivatives hitting the $\xi$ variable. 
\item The new family $\mathcal C'=(C_k')$ satisfies $C_k'\lesssim_{\!\! k}^{\!\! \operatorname{Pol}_1(\mathcal C)} 1$.
\end{itemize}
\end{lem}

\begin{proof}

\textbf{1.} Let us first express the action of $A^N$ in charts. Fix $\ell_0\in \Xi$ and let $(\varphi,\chi):=(\varphi_{\ell_0},\chi_{\ell_0}')$. Using the partition of unity $\sum \chi_\ell^2=1$, we have
\[\chi \varphi_* A^N \varphi^*\chi=\sum_\ell \varphi_*\chi A^N \chi_\ell^2 \chi\varphi^*.\]
In turn, using $\chi_\ell\prec \chi_\ell'$,
\begin{equation} \label{eq: split AN charts}\chi \varphi_* A^N \varphi^*\chi=\sum_\ell \varphi_* \chi (\chi_\ell')^2A^N\chi_\ell' \chi_\ell^2 \chi\varphi^*+\sum_\ell \varphi_*\chi (1-(\chi_{\ell}')^2) A^N\chi_\ell^2 \chi \varphi^*.\end{equation}
We first control the second sum on the right-hand side of \eqref{eq: split AN charts}. By assumption, each of the operators $(1-\chi_\ell')A^N \chi_\ell$ is continuous $H_h^r\to H_h^{s}$ with
\[\big\|(1-\chi_\ell')A^N \chi_\ell\big\|_{H_h^r\to H_h^{s}}\lesssim_{\!\! N,r,s}^{\!\!\operatorname{Pol}_1(\mathcal C)} h^{N(1-2\delta)},\]
since multiplication operators by smooth functions are continuous on each Sobolev spaces we infer
\[\Big\|\sum \varphi_*\chi (1-(\chi_{\ell}')^2) A^N\chi_\ell^2 \chi \varphi^*\Big\|_{H_h^r\to H_h^{s}}\lesssim_{\!\! N,r,s,\Xi}^{\!\!\operatorname{Pol}_1(\mathcal C)} h^{N(1-2\delta)}.\]
We now control the first sum in \eqref{eq: split AN charts}. For each $\ell$, we write
\[\varphi_* \chi (\chi_\ell')^2A^N\chi_\ell' \chi_\ell^2 \chi\varphi^*=\varphi_*\chi \chi_{\ell}' (\varphi_{\ell})_* \big[(\varphi_{\ell})^*\chi_{\ell}' A^N\chi_{\ell}' (\varphi_\ell)^*\big] (\varphi_\ell)_*\varphi^* \chi_{\ell}^2\chi\]
By assumption, provided $s\le r-(m-N)$, the middle operator in brackets writes $\mathrm{Op}_h^{\rm w}(T_N((\varphi_\ell)_*\mathbf a_{\ell}))+\mathcal O^{\mathcal C}_{N,\Xi,r,s}(h^{N(1-2\delta)})$.
Letting $\kappa_\ell:=\varphi^{-1}\circ \varphi_\ell$, it follows that
\[\varphi_* \chi (\chi_\ell')^2A^N\chi_\ell' \chi_\ell^2 \chi\varphi^*=\chi \chi_{\ell}' \kappa_\ell^* \mathrm{Op}_h^{\rm w}\big((\varphi_\ell)_*T_N({\mathbf a}_{\ell})\big) (\kappa_\ell)_* \chi_\ell' \chi \chi_{\ell}^2+\mathcal O_{N,r,s,\Xi}^{\operatorname{Pol}_1(\mathcal C)}(h^{N(1-2\delta)})_{H_h^r\to H_h^{s}}.\]
We expand the truncated symbol
\[T_N(\mathbf a_{\ell}) =\sum_{k\le N} h^k a_{\ell,k}.\]
By Proposition \ref{prop: change of variable pseudo euclidean}, for each $k<N$, one has:
\[\chi \chi_{\ell}' \kappa_\ell^* \mathrm{Op}_h^{\rm w}\big((\varphi_\ell)_*a_{\ell,k}\big) (\kappa_\ell)_* \chi\chi_\ell'=\mathrm{Op}_h^{\rm w}\Big(\varphi_*\sum_{j<N-k} (\mathrm{i} h)^j \big(\mathbf D^{2j}_{\kappa_\ell} a_{\ell,k}\big) + h^{N-k}r_{N-k}^{(\ell)}\Big),\]
where the remainder $r_{N-k}^{(\ell)}$ is quantized into an operator satisfying
\[\|\mathrm{Op}_h^{\rm w}(r_{N-k}^{(\ell)})\|_{H_h^r\to H_h^s}\lesssim_{\!\! _{N,r,s}}^{\!\! \operatorname{Pol}_1(\mathcal C)} h^{-2\delta N}. \]
We compose with the multiplication by $\chi_\ell^2$:
\begin{equation} \label{eq: AN ell k in chart} \varphi_* \chi (\chi_\ell')^2\mathrm{Op}_h^{\rm w}((\varphi_\ell)_* a_{\ell,k})\chi_\ell' \chi_\ell^2 \chi\varphi^*=\mathrm{Op}_h^{\rm w}\Big(\varphi_*\sum_{j<N-k}(\mathrm{i} h)^j \big(\mathbf D^{2j}_\ell a_{\ell,k}\big) + h^{N-k}r_{N-k}^{(\ell)}\Big),\end{equation}
here the operators $\mathbf D^{2j}_\ell$ and the error may have changed but are still of the previous form. Explicitly, we have (using $\chi_\ell\prec \chi_\ell'$):
\begin{equation} \label{eq: form D  ell}\mathbf D^0_\ell=\chi^2\chi_\ell^2, \quad \mathbf D^{2}_\ell a=-\{(\chi_\ell')^2\chi^2a ,\chi_\ell^2\}=-\{\chi^2 a,\chi_\ell^2\}.\end{equation}
We form the arrays $\vec a_n=\{a_{\ell,n}\}_{\ell\in \Xi}$, that are such that
\[\vec{\mathbf a}=\sum_{n\ge 0} (\mathrm{i}\underline h)^n\vec a_n.\]
By summing the contributions \eqref{eq: AN ell k in chart} over the charts $\ell\in \Xi$ and the indices $k<N$, we find
\begin{equation} \label{eq: AN in chart phi chi}\varphi_* \chi A^N \chi\varphi^*= \mathrm{Op}_h^{\rm w}\Big(\varphi_*\sum_{n< N}(\mathrm{i}h)^n  \sum_{j+k=n}\mathbf D^{2j}_{\Xi,\chi,\varphi} \vec{a}_k\Big)+\mathcal O_{N,r,s}^{\operatorname{Pol}_1(\mathcal C)}(h^{N(1-2\delta)})_{H_h^r\to H_h^{s}},\end{equation}
provided $s\le r-(m-N)$. Here $\mathbf D^0_{\Xi,\chi,\varphi}{\vec{\mathbf a}}=\chi^2\pi(\vec{\mathbf a})$, and $\mathbf D^2_{\Xi,\chi,\varphi}=0$. To see this, recall \eqref{eq: form D  ell} and observe that 
\[\sum_\ell \chi_\ell^2 =1, \qquad \sum_\ell \{\chi^2 a,\chi_\ell^2\}=\{\chi^2 a,1\}=0.\]
Lastly, note that
\[\sum_{n< N}(\mathrm{i} h)^n  \sum_{j+k=n}\big(\mathbf D^{2j}_{\Xi,\chi,\varphi} \vec{a}_k\big)=T_N\Big(\sum (\mathrm{i} \underline h)^n\mathbf D^{2n}_{\Xi,\chi,\varphi} \vec{\mathbf a}\Big).\]

\textbf{2.} We now evaluate $\varphi_*\chi \operatorname{Op}_h^{M}(T_N(\pi(\vec{\mathbf a})))\varphi^*\chi$. Recall that 
\[T_N(\pi(\vec{\mathbf a}))=\sum_{k<N} (\mathrm{i} h)^k\sum_{\ell}\chi_\ell^2 a_{\ell,k}.\]
Working as above, we see that for $s\le r-(m-N)$, the operator is expressed in charts as
\begin{equation} \label{eq: Op pi a in charts} \begin{aligned}\varphi_*\chi \operatorname{Op}_h^{M}(T_N(\pi(\vec{\mathbf a})))\varphi^*\chi& =\mathrm{Op}_h^{\rm w}\Big( \varphi_*\sum_{n\le N} (\mathrm{i} h)^n\sum_{j+k=n}\sum_{\ell} \mathbf D^{2j}_{\varphi,\chi}\big( \chi_\ell^2 (\varphi_\ell)^* a_{\ell,k}\big)\Big)\\ & +\mathcal O^{\operatorname{Pol}_1(\mathcal C)}_{N,r,s}(h^{N(1-2\delta)})_{H_h^r\to H_h^s},\end{aligned}\end{equation}
with $\mathbf D^0_{\varphi,\chi}=\chi^2$ and $\mathbf D^{2}_{\varphi,\chi}=0$. Extracting the term $j=0$ from the sum gives:
\[\mathrm{Op}_h^{\rm w}\Big(\varphi_* \chi^2\sum_{n\le N} (\mathrm{i} h)^n\sum_\ell \big[(\chi_\ell)^2(\varphi_\ell)^* a_{\ell,k}\big]\Big).\]
which coincides with the subsum $j=0$ in \eqref{eq: AN in chart phi chi}. By taking the difference between \eqref{eq: AN in chart phi chi} and \eqref{eq: Op pi a in charts}, we obtain for $s\le r-(m-N)$:
\[\varphi_*\chi \big(A^N-\operatorname{Op}_h^{M}(T_N(\pi(\vec{\mathbf a})))\big) \chi\varphi^*= \mathrm{Op}_h^{\rm w}\Big(\sum_{n=0}^N (\mathrm{i} h)^nb_{\varphi,\chi,n}\Big)+\mathcal O^{\operatorname{Pol}_1(\mathcal C)}_{N,r,s}(h^{N(1-2\delta)}),\]
with
\[b_{\varphi,\chi,n}=\sum_{\substack{j+k= n \\ j\ge 2}} \mathbf D^{2j}_{\varphi,\chi}(\vec a_{k}).\]
Here $\vec a_k=\{a_{\ell,k}\}_{\ell\in \Xi}$ and $\mathbf D^{2j}_{\varphi,\chi}$ is a differential operator of order $\le 2j$, with real coefficients independent of $\xi$, with always $\ge j$ hitting the $\xi$ variable. 
The symbols $b_{\varphi,\chi,n}$ are controlled by
\[\|b_{\varphi,\chi,n}\|_{m-n,k}\lesssim_{\!\! n,k}^{\!\! \operatorname{Pol}_1(\mathcal C)} h^{-\delta(2n+k)}.\]

\textbf{3.} Let us now check the pseudolocality statement. For $(A^N)$, this is assumption \eqref{eq: assum pseudo loc AN}. As for $\operatorname{Op}_h^{M}(T_N(\pi(\vec{\mathbf a})))$, the result follows from Lemma \ref{prop: pseudolocality}. Hence $B^N=A^N-\operatorname{Op}_h^{M}(\pi(\vec{\mathbf a}))$ also satisfies \eqref{eq: assum pseudo loc AN}, with new constants $C_{N,r,s}'$ that are linear expressions in $\mathcal C$.  

\textbf{4.} By the above computations, letting $\mathbf b_{\ell}=\sum (\mathrm{i}\underline h)^n b_{\varphi_\ell,\chi_\ell',n}$ and $\vec{\mathbf b}=\{\mathbf b_\ell\}_{\ell\in \Xi}$, we have found an operator
\[Q:\mathcal A[\![\mathrm{i}\underline h]\!]^\Xi\longrightarrow \mathcal A[\![\mathrm{i}\underline h]\!]^\Xi\] such that $\vec{\mathbf b}=Q\vec{\mathbf a}=\sum_{n\ge 2} (\mathrm{i}\underline h)^n \mathbf D^{2n}_\Xi \vec{\mathbf a}$. Here $\mathbf D^{2n}_\Xi$ is a matrix differential operator of order $2n$ with real coefficients acting on arrays $\{a_\ell\}_{\ell\in \Xi}$. 

\end{proof}

We are now ready to prove Proposition \ref{prop: constructing the symbol}.
\begin{proof}[Proof of Proposition \ref{prop: constructing the symbol}] Using the operator $Q$ from \eqref{eq: def mathbf b}, we define inductively:
\begin{equation}\label{eq: def induction symbol} \left\{\begin{array}{l}
\vec {\mathbf b}_{0}=\vec {\mathbf a}, B_{0}^N= A^N \\
\vec {\mathbf b}_{n+1}=Q \vec{\mathbf b}_n \\
A_{n+1}^{N}=\mathrm{Op}_h^M(T_N(\pi(\vec{\mathbf b}_n))) \\
B_{n+1}^N=B_n^N-A_{n+1}^N
\end{array}\right.,\end{equation}

Since $(B_{0}^N)_{N\ge 0}$ satisfies $\mathbf{H}(m,\vec{\mathbf b}_{0},\mathcal C_0)$ with $\mathcal C_0=\mathcal C$, by applying Lemma \ref{lem: lemma before construction of symbol} inductively, we find that $(B_{n}^N)_{N\ge 0}$ satisfies $\mathbf{H}(h,m,\vec{\mathbf{b}}_{n},\mathcal C_n)$ for all $n$, where $\mathcal C_n$ consists of linear expressions in the elements of $\mathcal C_0$. In particular, provided $s\le r-(m-N)$:
\[(\varphi_\ell)*\chi_\ell' B_n^N (\varphi_\ell)^*\chi_\ell'=\mathrm{Op}_h^{\rm w}( \varphi_*T_N(\vec{\mathbf b}_n)_{\ell})+ \mathcal O^{\operatorname{Pol}_1(\mathcal C)}_{n,N,r,s}(h^{N(1-2\delta)})_{H_h^r\to H_h^s}.\]
Now by \eqref{eq: def mathbf b} we have
\[\vec{\mathbf b}_{n}=Q^n \vec{\mathbf b}_0.\]
Since $Q$ maps $\mathcal A[\![\mathrm{i}\underline h] \!]$ into $(\mathrm{i}\underline h)^2\mathcal A[\![\mathrm{i}\underline h] \!]$, the truncated symbol $T_N(\vec{\mathbf b}_n)=T_N(Q^n \vec{\mathbf b}_0)$ vanishes for $n\ge N/2$. Therefore, letting $n_0=\lfloor N/2\rfloor +1$, if $s\le r-(m-N)$ then
\[\|\varphi_*\chi B_{n_0}^N \varphi^*\chi\|_{H_h^r\to H_h^s}\lesssim_{\!\! N,r,s}^{\!\!\operatorname{Pol}_1(\mathcal C)} h^{N(1-2\delta)}.\]
Combined with the pseudolocality statement, this shows that 
\begin{equation} \label{eq: bound Bn0}\|B_{n_0}^N\|_{H_h^r\to H_h^{s}}\lesssim_{\!\! N,r,s}^{\!\!\operatorname{Pol}_1(\mathcal C)}h^{N(1-2\delta)}.\end{equation}
Recalling \eqref{eq: def induction symbol}, we have
\[A^N=B_0^N=\Big(\sum_{n=0}^{n_0-1} B_n^N-B_{n+1}^N\Big)+ B_{n_0}^N=\Big(\sum_{n=1}^{n_0} A_n^N\Big) +B_{n_0}^N.\]
Since $A_n^N=\operatorname{Op}_h^{M}(T_N(\pi(\vec{\mathbf b}_n)))$, by setting
\begin{equation} \label{eq: setting sigma tot}\sigma_{\rm tot}(\vec{\mathbf a}):=\sum_{n\ge 0} \pi(\vec{\mathbf b}_n)=\pi\Big(\sum_{n\ge 0} Q^n \vec{\mathbf a}\Big) \end{equation}
and using the estimate \eqref{eq: bound Bn0}, we obtain, for $s\le r-(m-N)$:
\[A^N=\operatorname{Op}_h^{M}\big(\varphi_*T_N(\sigma_{\rm tot}(\vec{\mathbf a}))\big)+\mathcal O_{N,r,s}^{\operatorname{Pol}_1(\mathcal C)}(h^{N(1-2\delta)})_{H_h^r\to H_h^{s}}.\]
Writing $Q=\sum_{n\ge 2}(\mathrm{i} \underline h)^n\mathbf D_{\Xi}^{2n}$, we see that
\[\sigma_{\rm tot}(\vec{\mathbf a})=\pi\Big(\sum_{n\ge 0}(\mathrm{i}\underline h)^n \widetilde{\mathbf D}_{\Xi}^{2n} \vec{\mathbf a}\Big),\]
for some new operators
\[\widetilde{\mathbf D}_{\Xi}^{0}=1, \quad \text{and}\quad \widetilde{\mathbf D}_{\Xi}^{2n}=\sum_{j_1+\ldots+j_k=n} \mathbf D_\Xi^{2j_1}\cdots \mathbf D_\Xi^{2j_k}\quad \text{for $n\ge 1$}.\]
The first terms in \eqref{eq: setting sigma tot} are $\mathbf D_\Xi^0\vec{\mathbf a}=\pi(\vec{\mathbf a})$ and $\mathbf D_\Xi^{2}\vec{\mathbf a}=0$.
\end{proof}

As an example, we give the symbol of the semiclassical Laplacian $h^2\Delta_g$.
\begin{prop}Let $g$ be a smooth Riemannian metric on $M$, and take $\operatorname{vol}=\operatorname{vol}_g$ in the definition of the quantization. There is a formal symbol $\mathbf r=\sum_{n\ \rm{even}} \underline h^nr_n$ with $r_n\in S^{-n}(T^*M)$ such that for $s\le r-(2-N)$ we have
\[h^2\Delta_g=\operatorname{Op}_h^{M}\big(|\xi|_g^2+h^2T_N(\mathbf r)\big)+\mathcal O(h^{N})_{H_h^r\to H_h^{s}}.\]\end{prop}
\begin{proof}In an isochore chart $(\varphi,\chi)$, the metric is represented by a matrix $g(x)$ satisfying $|\det g(x)|=1$. We let $g^{-1}=(g^{ij})_{ij}$ denote the inverse of $g$. Then, for $s\le r-(2-N)$, by using the composition and change of variable formulae:
\begin{align*}\varphi_*\chi h^2\Delta_g \varphi^* \chi & = \chi\Big(\sum_{i,j}(-\mathrm{i}h\partial_i)\big(g^{ij}(-\mathrm{i} h\partial_j)\Big)\chi \\ 
& = \sum_{i,j}\mathrm{Op}_h^{\rm w}(\chi)\mathrm{Op}_h^{\rm w}(\xi_i) \mathrm{Op}_h^{\rm w}(g^{ij})\mathrm{Op}_h^{\rm w}(\xi_j)\mathrm{Op}_h^{\rm w}(\chi)\\
& =\mathrm{Op}_h^{\rm w}(T_N(\mathbf a_{\varphi,\chi}))+\mathcal O(h^{N})_{H_h^r\to H_h^s}\end{align*}
with $\mathbf a_{\varphi,\chi}=\chi^2|\xi|_g^2+(\mathrm{i}\underline h)\cdots$. Here we have used $\sum g^{ij}\xi_i\xi_j=|\xi|_g^2$.
By Proposition \ref{prop: constructing the symbol}, we infer that there is a formal symbol $\mathbf r=\sum (\mathrm{i}\underline h)^nr_n$ with $r_n\in S^{-n}(T^*M)$ such that
\[h^2\Delta_g=\operatorname{Op}_h^{M}(|\xi|_g^2 +h^2T_N(\mathbf r))+\mathcal O(h^N)_{H_h^r\to H_h^s}.\]
Since $h^2\Delta_g$ is self-adjoint, it has real symbol, thus $r_n=0$ when $n$ is odd.
\end{proof}

\subsubsection{Changing quantization} Another consequence of Proposition \ref{prop: constructing the symbol} is the following result, that expresses how different quantizations are related to each other.
\begin{prop}Let $\widetilde{\operatorname{Op}}^M_h$ be another quantization procedure associated with a different atlas of isochore charts with cutoffs $\widetilde \Xi$. Let $a\in S^m$, and set $\mathcal C=\{h^{\delta k}\|a\|_{m,k}\}$. Then, for any $s\le r-(m-N)$, one has
\[\widetilde{\operatorname{Op}}^M_h(a)=\operatorname{Op}_h^{M}\Big(a +\sum_{\substack{2\le n< N \\ n\ \rm{even}}}(\mathrm{i}\underline h)^n \mathbf D^{2n}_{\Xi,\widetilde \Xi}a\Big)+\mathcal O_{N,r,s,\Xi,\widetilde \Xi}^{\operatorname{Pol}_1(\mathcal C)}(h^{N(1-2\delta)})_{H_h^r\to H_h^s}.\]
Here $\mathbf D^{2n}_{\Xi,\widetilde \Xi}$ is a differential operator of order $2n$ with real coefficients, and always $\ge n$ derivatives hitting the $\xi$ variable.	
\end{prop}
In particular, the symbol of a classical operator is well defined modulo $h^2$, independently of the choice of quantization.

\subsection{The algebra of pseudodifferential operators} The next result is similar to \cite[Lemma A.6]{DJN21}, which dealt with the $C^\infty_{\rm comp}$ calculus.
\begin{prop}[Composition]\label{prop: composition manifold} Let $a\in S^{m_1}(T^*M), b\in S^{m_2}(T^*M)$ and $\mathcal C=\{h^{\delta(j+k)}\|a\|_{m_1,j}\|b\|_{m_2,k}~:~j,k\ge 0\}$. Then, for $s\le r-(m_1+m_2-N)$ we have
\[\mathrm{Op}^M_h(a)\circ \mathrm{Op}_h^M(b)=\mathrm{Op}_h^{M}(T_N(a\diamond b))+\mathcal O^{\operatorname{Pol}_1(\mathcal C)}_{N,r,s}\big(h^{(1-2\delta )N}\big)_{H_h^r\to H_h^{s}}.\]
The formal symbol $a\diamond b$ is defined by
\[a\diamond b=ab+\frac{1}{2\mathrm{i}} \{a;b\} \underline{h}+\sum_{n=2}^{+\infty} (\mathrm{i}\underline{h})^n \mathbf D^{2n-2}(\mathrm{d}a\otimes \mathrm{d}b)_{|\rm{Diag}},\]
where $\mathbf D^{2n-2}$ is a differential operator of order $2n-2$ with real coefficients independent of $(\xi,\eta)$ acting on $\mathrm{d}a(x,\xi)\otimes \mathrm{d}b(y,\eta)$; the total number of derivatives hitting the $\xi,\eta$ variables is always $\ge n$. 
\end{prop}
\begin{rem}If $\mathbf a$ and $\mathbf b$ are formal power series in $\mathrm{i}\underline h$, we get a natural definition of $\mathbf a\diamond \mathbf b$ that extends the one for functions (i.e. one has $\iota(a)\diamond \iota(b)=a\diamond b$).  \end{rem}

\begin{proof}We essentially follow the proof of \cite[Lemma A.6]{DJN21}. We set 
\[C:=\mathrm{Op}_h^M(a)\mathrm{Op}_h^M(b)-\mathrm{Op}_h^M\Big(ab+\frac{h}{\mathrm{i}}\{a;b\}\Big).\]
Our goal is to express the action of $C$ in charts and then use Proposition \ref{prop: constructing the symbol} to write $C=\mathrm{Op}_h^M(\mathbf c)$, for some formal symbol $\mathbf c$.\bigskip

\textbf{1.} The first step is to make explicit the action of $C$ in charts. Let $A=\operatorname{Op}_h^{M}(a)$ and $B=\operatorname{Op}_h^{M}(b)$. We first compute $\varphi_*\chi AB\chi\varphi^*$. Let $\chi\prec \chi'$, then by pseudolocality
\begin{align*}
\varphi_*\chi AB\chi\varphi^*& = \chi\big[\varphi^*\chi' \operatorname{Op}_h^{M}(a)\chi'\varphi^*\big]\circ \big[\varphi_*\chi' \operatorname{Op}_h^{M}(b)\chi'\varphi^*\big]\chi\\ & + \mathcal O^{\operatorname{Pol}_1(\mathcal C)}_{N,r,s}(h^{N(1-2\delta)})_{H_h^r\to H_h^{s}}.
\end{align*}
We use Lemma \ref{lem: operator in a chart} to express (we write shortly $L_k:=L_{\varphi,\chi',k}$)
\[\varphi^*\chi' \operatorname{Op}_h^{M}(a)\chi'\varphi^*=\operatorname{Op}_h^{\rm w}\Big(\varphi_*\sum_{k<N} (\mathrm{i}h)^k (L_ka)\Big)+\mathcal O^{\operatorname{Pol}_1(\mathcal C)}_{N,r,s}(h^{N(1-2\delta)})_{H_h^{r}\to H_h^{s}},\]
provided $s\le r-(m_1-N)$. Recall that $L_0=(\chi')^2 a$ and $L_1=0$. We can do the same thing for $b$:
\[\varphi^*\chi' \operatorname{Op}_h^{M}(b)\chi'\varphi^*=\mathrm{Op}_h^{\rm w}\Big(\varphi_*\sum_{k'<N-k} (\mathrm{i} h)^{k'} L_{k'}b\Big)+\mathcal O_{N,s,t}^{\operatorname{Pol}_1(\mathcal C)}(h^{(N-k)(1-2\delta)})_{H_h^{s}\to H_h^{t}},\] 
for $t\le s-(m_2-N+k)$. By expanding the products, we find
\[\varphi_*\chi AB\chi\varphi^*=\chi\Big[\sum_{k+k'<N}(\mathrm{i}h)^{k+k'}\mathrm{Op}_h^{\rm w}(\varphi_*(L_ka)) \mathrm{Op}_h^{\rm w}(\varphi_*(L_{k'}b))\Big]\chi  +\mathcal O^{\operatorname{Pol}_1(\mathcal C)}_{N,r,t}(h^{N(1-2\delta)})_{H_h^r\to H_h^t},\]
for $t\le r-(m_1+m_2-N)$. For each pair $(k,k')$ with $k+k'=n$ we use the composition formula for the Weyl quantization on $\mathbf R^d$ to get
\begin{equation*}\begin{aligned}
\mathrm{Op}_h^{\rm w}(\varphi_*(L_ka)) \mathrm{Op}_h^{\rm w}(\varphi_*(L_{k'}b))& =\mathrm{Op}_h^{\rm w}\Big(\sum_{i<N-n} (\mathrm{i} h)^{i+k+k'}A_{2i}(D)(\varphi_*(L_k a)\otimes \varphi_*(L_{k'} b))\Big)\\ & +\mathcal O^{\operatorname{Pol}_1(\mathcal C)}_{N,r,t}(h^{(1-2\delta N)-n})_{H_h^r\to H_h^t}.
\end{aligned} \end{equation*}

Summing over $k+k'< N$ leads to
\begin{equation}\label{eq: AB in charts}
\begin{array}{ll}\varphi_*\chi AB\chi\varphi^*& =\chi\mathrm{Op}_h^{\rm w}\Big(\displaystyle \sum_{n<N}(\mathrm{i} h)^n\sum_{k+k'+i=n} A_{2i}(D)(\varphi_*(L_ka)\otimes \varphi_*(L_{k'}b))_{\rm Diag}\Big)\chi \\ & +\mathcal O^{\operatorname{Pol}_1(\mathcal C)}_{N,r,s}(h^{N(1-2\delta)})_{H_h^r\to H_h^{s}}.\end{array}
\end{equation}

For $n\ge 2$, the summands giving terms of the form $a\otimes \partial^j b$ or $\partial^j a\otimes b$ for some $j\ge 1$, i.e. with no differentiation on the symbol $a$ or $b$, correspond to indices $(k,k',i)=(n,0,0)$ or $(k,k',i)=(0,n,0)$ in the sum \eqref{eq: AB in charts}. The corresponding term in \eqref{eq: AB in charts} is
\begin{equation} \label{eq: no derivative AB chart} \chi\mathrm{Op}_h^{\rm w}\Big(\sum_{n< N} (\mathrm{i} h)^n\varphi_*\big((L_na)(L_0b)+(L_0)a(L_nb)\big)\Big)\chi.\end{equation}
On the other hand, by Lemma \ref{lem: operator in a chart}, we have
\begin{equation}\label{eq: op ab + bracket chart}\begin{aligned}
\varphi_*\chi'\operatorname{Op}_h^{M}(ab+\frac{h}{\mathrm{i}} \{a;b\})\chi'\varphi^*& =\mathrm{Op}_h^{\rm w}\Big(\varphi_* (\chi')^2\big(ab+\frac{h}{\mathrm{i}}\{a;b\}\big)+\varphi_*\sum_{n=2}^{N-1} (\mathrm{i} h)^n L_n\big(ab+\frac{h}{\mathrm{i}} \{a;b\}\big)\Big) \\ & +\mathcal O_{N,r,s}^{\operatorname{Pol}_1(\mathcal C)}(h^{N(1-2\delta)})_{H_h^r\to H_h^s}.
\end{aligned}\end{equation}
By Leibniz rule, since $L_n$ is a differential operator, we have for some coefficients $c_{ij}$:
\begin{equation} \label{eq: leibniz rule}L_n(ab)=aL_n b+bL_n a+\sum_{\substack{i+j\le 2n \\ i\ge 1,j\ge 1}} c_{ij}\partial^ia\partial^j b.\end{equation}
By taking the difference between \eqref{eq: AB in charts} and \eqref{eq: op ab + bracket chart}, we obtain 
\[\chi\varphi_*C\varphi^*\chi=\chi\mathrm{Op}_h^{\rm w}\Big(\sum_{n<N} (\mathrm{i} h)^n c_n\Big)\chi,\]
where
\begin{itemize}
\item $c_0=\varphi_*\big((\chi')^4ab-(\chi')^2ab\big)$ vanishes on the support of $\chi$.
\item $c_1=\varphi_*\big((\chi')^2\{a;b\}-\{(\chi')^2a;(\chi')^2b\}\big)$ vanishes on the support of $\chi$.
\item For $n\ge 2$, by \eqref{eq: no derivative AB chart} and \eqref{eq: op ab + bracket chart} and the Leibniz rule \eqref{eq: leibniz rule}, one has 
\[c_n=(L_na)(L_0b)+(L_0 a)L_n b-(L_na)(L_n b)+\sum_{\substack{i+j\le 2n, \\ i,j\ge 1}} c_{ij}\partial^ia \partial^jb.\]
Since $L_0=(\chi')^2$ and $\chi\prec \chi'$, the difference $(L_na)(L_0b)+(L_0 a)L_n b-(L_na)(L_n b)$ vanishes on the support of $\chi$.
\end{itemize}
By applying the composition formula one last time, we find, for $s\le r-(m_1+m_2-N)$:
\begin{equation} \label{eq: C in charts} \varphi_*\chi C \varphi^*\chi=\mathrm{Op}_h^{\rm w}\Big(\sum_{n=2}^{N-1} (\mathrm{i} h)^n\mathbf D^{2n-2}_{\varphi,\chi,\Xi}(\mathrm{d}a\otimes \mathrm{d}b)_{|\rm Diag}\Big)+\mathcal O^{\operatorname{Pol}_1(\mathcal C)}_{N,r,s}(h^{N(1-2\delta)})_{H_h^r\to H_h^s}.\end{equation}
Here $\mathbf D^{2n-2}_{\varphi,\chi,\Xi}$ is of order $2n-2$, and the number of derivatives hitting the $\xi$ variable in the expression $\mathbf D^{2n-2}_{\varphi,\chi,\Xi}(\mathrm{d}a\otimes \mathrm{d}b)_{|\rm Diag}$ is $\ge n-1$.

It is straightforward to show that $C$ satisfies the pseudolocality property.  We let
\begin{equation} \label{eq: c phi chi ell} \mathbf c_{\varphi_\ell,\chi_\ell'}:=\sum_{n=2}^{+\infty} (\mathrm{i}\underline{h})^n \mathbf D^{2n-2}_{\varphi_\ell,\chi_\ell',\Xi}(\mathrm{d}a\otimes \mathrm{d}b),\end{equation}
then set $\vec{\mathbf c}:=\{c_{\varphi_\ell,\chi_\ell'}\}_{\ell\in\Xi}$. We have precisely shown that the operator $C$ satisfies~$\mathbf H(h,m_1+m_2,\vec{\mathbf c},\mathcal C')$, where the family $\mathcal C'$ consists of linear expressions in the elements of $\mathcal C$.

\textbf{2.} Let $\sigma_{\rm tot}(\vec{\mathbf c})$ be the total symbol of $\vec{\mathbf c}$, which is defined by
\begin{equation} \label{eq: a diamond b in proof}\sigma_{\rm tot}(C):=\sum_{n=0}^{+\infty}\underline{h}^n\mathbf D_\Xi^{2n} \vec{\mathbf c},\end{equation}
where the operators $\mathbf D_{\Xi}^{2n}$ are as in \eqref{eq: sigma tot}. By combining \eqref{eq: a diamond b in proof} with \eqref{eq: c phi chi ell}, one can see that for some other family of operators $\{\mathbf D_\Xi^{2n-2}\}_{n\ge 0}$ we have:
\[\sigma_{\rm tot}(\vec{\mathbf c})=\sum_{n=2}^{+\infty}(\mathrm{i}\underline h)^n \mathbf D_\Xi^{2n-2} (\mathrm{d}a\otimes \mathrm{d}b)_{|\rm Diag}.\]
By Proposition \ref{prop: constructing the symbol}, if $s\le r-(m_1+m_2-N)$:
\[C=\mathrm{Op}_h^M(T_N(\sigma_{\rm tot}(\vec{\mathbf c})))+\mathcal O_{N,r,s}^{\operatorname{Pol}_1(\mathcal C)}(h^{N(1-2\delta)})_{H_h^r\to H^{s}}.\]
Setting
\[a\diamond b:=ab (\mathrm{i}\underline h)^0-\{a;b\} (\mathrm{i}\underline h)+\sigma_{\rm tot}(\vec{\mathbf c}),\]
we find
\[\operatorname{Op}_h^{M}(a)\operatorname{Op}_h^{M}(b)=\operatorname{Op}_h^{M}(T_N(a\diamond b))+\mathcal O^{\operatorname{Pol}_1(\mathcal C)}_{N,r,s}(h^{N(1-2\delta)})_{H_h^r\to H_h^{s}}.\]
This finishes the proof.\end{proof}
The next result is intuitive because composition of operators is associative.
\begin{lem}The operator $\diamond$ is associative, i.e. $\mathbf a\diamond (\mathbf b\diamond \mathbf c)=(\mathbf a \diamond \mathbf b)\diamond \mathbf c$.\end{lem}
\begin{proof}It is enough to prove the result for three smooth compactly supported functions $a,b,c$. Let $\mathcal C$ denote the collection of the $C^k$ norms of the symbols $a,b,c$. Since all functions are compactly supported we can work with $L^2\to L^2$ remainders in the calculus. 
By computing in two different ways the composition
\[\operatorname{Op}_h^{M}(a)\operatorname{Op}_h^{M}(b)\operatorname{Op}_h^{M}(c),\]
and taking the difference, we have
\[\operatorname{Op}_h^{M}\Big(T_N\big(T_N(a\diamond b)\diamond c-a\diamond T_N(b\diamond c)\big)\Big)=\mathcal O_N^{\mathcal C}(h^N)_{L^2\to L^2}.\]
Let us write
\[T_N\big(T_N(a\diamond b)\diamond c-a\diamond T_N(b\diamond c)\big)=\sum_{k<N}h^k \sum_{j<N} h^j\Big(\big((a\diamond b)_k\diamond c\big)_j- \big(a\diamond(b\diamond c)_k\big)_j\Big).\]
Since this produces a $\mathcal O_N^{\mathcal C}(h^N)_{L^2\to L^2}$ operator for all $N$, by letting $h\to 0$, we find that all the coefficients of the expansion have to vanish:
\[\forall n\ge 0, \ \sum_{k+j=n}\big((a\diamond b)_k\diamond c\big)_j - \big(a\diamond(b\diamond c)_k\big)_j=0.\]
The extension to formal power series is immediate by linearity.
\end{proof}

We give a result for the commutator of two pseudodifferential operators on $M$, which is analogous to  \cite[Lemma A.6]{DJN21}. I learned from the note of Guedes Bonthonneau \cite{bonthonneauEDP} that even on manifolds, for our choice of quantization, the formal symbol of the commutator does not feature any even powers of $\underline h$.

\begin{lem}One has
\begin{equation} \label{eq:a diam b - b diam a}a\diamond b-b\diamond a=\sum_{n\ge 0\ \rm{odd}}(\mathrm{i}\underline h)^n \mathbf D^{2n-2}(\mathrm{d}a\otimes \mathrm{d}b)_{|\rm Diag}.\end{equation}
\end{lem}
\begin{proof}It is enough to prove the result in the case where $a$ and $b$ are both smooth, real, and compactly supported. In this case, the operator $C:=[\operatorname{Op}_h^{M}(a),\operatorname{Op}_h^{M}(b)]$ is anti self-adjoint on $L^2(M)$. Let $\mathcal C=\{\|a\|_{C^j}\|b\|_{C^k}~:~j,k\ge 0\}$. By Proposition \ref{prop: composition manifold}, we have
\begin{equation} \begin{aligned}C& =\operatorname{Op}_h^{M}(T_N(a\diamond b-b\diamond a))+\mathcal O_N^{\operatorname{Pol}_1(\mathcal C)}(h^N)_{L^2\to L^2}\\ & =\sum_{n<N} (\mathrm{i} h)^n \mathbf D^{2n-2}(\mathrm{d}a\otimes \mathrm{d}b) +\mathcal O_N^{\operatorname{Pol}_1(\mathcal C)}(h^N)_{L^2\to L^2}.\end{aligned}\end{equation}
Since $C+C^*=0$, we infer
\[\sum_{n<N\ \rm{even}} (-1)^{n/2}h^n\mathbf D^{2n-2}(\mathrm{d}a\otimes \mathrm{d}b)=\mathcal O_N^{\operatorname{Pol}_1(\mathcal C)}(h^N)_{L^2\to L^2}.\]
By letting $h\to 0$, we deduce $\mathbf D^{2n-2}(\mathrm{d}a\otimes \mathrm{d}b)_{|\rm Diag}=0$ for all even $n\in \mathbf N$.
\end{proof}
By arguing as in \cite[Lemma A.6]{DJN21}, one can show that the operator $\mathbf D^{2n-2}_\Xi(\mathrm{d}a\otimes \mathrm{d}b)_{|\rm{Diag}}$ in \eqref{eq: a diamond b in proof} is actually of the form $\mathbf D^{2n-4}_\Xi(\mathrm{d}^2a\otimes \mathrm{d}^2b)_{|\rm{Diag}}$. This leads to the following proposition.

\begin{prop}[Commutator]\label{prop: commutator PDOs} For $a,b\in S^{m_1},S^{m_2}$, let $\mathcal C=\{h^{\delta(j+k)} \|a\|_{m_1,j}\|b\|_{m_2,k}\}$. If $s\le r-(m_1+m_2-N)$, then
\[\big[\mathrm{Op}_h^{M}(a),\mathrm{Op}_h^M(b)\big]=\mathrm{Op}_h^M\left(T_N(\mathbf c)\right)+\mathcal O^{\operatorname{Pol}_1(\mathcal C)}_{N,r,s}(h^{N(1-2\delta)})_{H_h^r\to H_h^s},\]
where $\mathbf c$ is the formal symbol given by
\begin{equation} \label{eq: def bf c}\mathbf c=a\diamond b-b\diamond a=-\mathrm{i} \underline h \{a;b\}+\sum_{n\ge 3\ \rm{odd}} (\mathrm{i}\underline{h})^n\mathbf D^{2n-4}_\Xi(\mathrm{d}^2a\otimes \mathrm{d}^2b)_{|\rm{Diag}},\end{equation}
here $\mathbf D^{2n-4}_\Xi$ is a differential operator of order $2n-4$ with real coefficients independent of $\xi$. The total number of derivatives hitting the $(\xi,\eta)$ variable in the expression $\mathbf D^{2n-4}_\Xi(\mathrm{d}^2a\otimes \mathrm{d}^2b)$ is always $\ge n$.\end{prop}
\begin{rem}At first order in $h$, the defect of commutativity of $\operatorname{Op}_h^{M}(a)$ with $\operatorname{Op}_h^{M}(b)$ is expressed by the nonvanishing of the Poisson bracket $\{a;b\}\neq 0$, which at the classical level means that the Hamiltonian flows generated by $a$ and $b$ do not commute.
\end{rem}

\subsection{Construction of a parametrix for elliptic operators} We review the parametrix construction for an elliptic operator. Our goal is to find an inverse to $\mathrm{Op}_h^M(a)$ modulo a remainder $\mathcal O(h^N)_{H_h^r\to H_h^{s}}$, when $a$ is an elliptic symbol.
We first state a lemma expressing derivatives of the inverse of a function, which can be proved by induction.
\begin{lem}\label{lem: derivatives inverse} Let $a$ be a smooth nonvanishing function. Then, for any nonzero multiindex $\alpha$ of length $n$, one has
\[\partial^\alpha a^{-1}=\sum_{2\le k\le n+1} \sum_{(\alpha_1,\ldots,\alpha_{k-1})} \frac{1}{a^k}\Big(\prod_{i=1}^{k-1} \partial^{\alpha_i}a\Big),\]
where the sum runs over tuples $(\alpha_1,\ldots,\alpha_{k-1})$ satisfying
\[\sum |\alpha_i|=n, \qquad \sum (|\alpha_i|-2)_+\le n+1-k.\]
\end{lem}
We can then estimate the seminorms of $a^{-1}$ in $S^m$ in terms of the seminorms of $a$ and the ellipticity constant, which is allowed to depend on $h$.
\begin{cor} \label{cor: bounds a inverse} Let $a\in S^m$ be a nonvanishing symbol, satisfying $|a|\ge \varepsilon h^\gamma\langle \xi\rangle^m$ for $0<\varepsilon\le 1$. Let
\[\mathcal C_1=\{h^{\delta k}\|a\|_{m,k}\}, \qquad \mathcal C_2=\{h^{\delta(k-2)_+}\|a\|_{m,k}\}.\]
Also, let $\sigma_1=\delta+\gamma$ and $\sigma_2=\max(\delta,\gamma)$. Then, for $i=1,2$:
\[\|a^{-1}\|_{-m,k}\lesssim_{\!\! k}^{\!\! \operatorname{Pol}(\mathcal C_i)} h^{-\gamma}h^{-\sigma_i k}\varepsilon^{-(k+1)}.\]
\end{cor}

\begin{thm}\label{thm: parametrix} Let $a\in S^m$ satisfy $|a(x,\xi)|\ge \varepsilon h^\gamma\langle \xi\rangle^m$ for some $\gamma\in [0,1/2)$ and $0<\varepsilon\le 1$. Let $\mathbf b=a^{\diamond_{-1}}$ denote the inverse of $a$ in $\mathcal A[\![\mathrm{i} \underline h]\!]$ for the law $\diamond$. Then $\mathbf b=\sum \underline h^n b_n$ with $b_n\in S^{-m-n}$, and in the notation of Corollary \ref{cor: bounds a inverse}, we have for $i\in \{1,2\}$:
\[\mathrm{Op}_h^M(a)\mathrm{Op}_h^M(T_N(\mathbf b))=\mathrm{Id} +\mathcal O_{N,r,s}^{\operatorname{Pol}(\mathcal C_i,\varepsilon^{-1})} (h^{-\gamma} h^{N(1-2\sigma_i)})_{H_h^r\to H_h^s}.\]
More precisely, the symbol $\mathbf b$ has the following form:
\[\mathbf b=\frac 1a+\sum_{n\ge 2\ \rm{even}}(\mathrm{i}\underline{h})^n\Big(\sum_{3\le k\le 2n}\frac{Q_{2n,k}(\mathbf Da,\ldots,\mathbf D^{2n}a)}{a^{k}}\Big).\]
Here $Q_{2n,k}(\bullet)$ is a polynomial expression involving derivatives of the symbol $a$. For any product $\prod (\partial^{\alpha_i} a)$ appearing in the numerator one has $\sum |\alpha_i|=2n$ and $\sum (|\alpha_i|-2)_+\le 2n-k$. 

\end{thm}

The basic idea is that any formal power series of the form $1-(\mathrm{i}\underline h) \mathbf r$ in $\mathcal A[\![\mathrm{i}\underline h]\!]$ admits an inverse $(1-(\mathrm{i}\underline h)\mathbf r)^{\diamond_{-1}}\in \mathcal A[\![\mathrm{i}\underline h]\!]$ for the law $\diamond$ in the form of a Neumann series. To obtain a quasi-inverse to the operator $\operatorname{Op}_h^{M}(a)$, it is enough to find an inverse to the symbol $a$ in $\mathcal A[\![\mathrm{i}\underline h]\!]$. By the previous observation, it is enough to find a function $b_0=b_0(\mathrm{i}\underline h)^0$ such that $a\diamond b_0=1-\mathrm{i}\underline h\mathbf r$, in which case the inverse of $a$ is given by $b_0\diamond (1-\mathrm{i}\underline h \mathbf r)^{\diamond_{-1}}$. We will simply choose $b_0=\frac 1a$.
\begin{proof}
\textbf{1.} Let $b_0=\frac 1a$. Let $(\mathcal C,\sigma)$ be either $(\mathcal C_1,\sigma_1)$ or $(\mathcal C_2,\sigma_2)$. By Corollary \ref{cor: bounds a inverse}, one has
\[\|b_0\|_{-m,k}\lesssim_{\!\! k}^{\!\! \operatorname{Pol}(\mathcal C)} \varepsilon^{-(k+1)}h^{-\gamma} h^{-\sigma k}.\]
Then, the composition formula shows that
\[\mathrm{Op}_h^M(a)\mathrm{Op}_h^M(b_0)=\mathrm{Id}-\mathrm{Op}_h^M(T_N(\mathbf r))+ \mathcal O^{\operatorname{Pol}(\mathcal C,\varepsilon^{-1})}_{N,r,s}(h^{N(1-2\sigma)-\gamma})_{H_h^r\to H_h^s},\]
where $\mathbf r=\sum_{n=2}^{+\infty} \underline{h}^n r_n$. Since $\{a;a^{-1}\}=0$, the symbol $\mathbf r$ has the following form:
\[\mathbf r=\sum_{n=2}^{\infty}(\mathrm{i}\underline h)^n \mathbf D^{2n-2}(\mathrm{d}a\otimes \mathrm{d}(a^{-1})).\]
By Lemma \ref{lem: derivatives inverse}, the coefficient of $(\mathrm{i}\underline h)^n$ is of the form
\begin{equation} \label{eq: formule r_n} r_n=\sum_{2\le k\le 2n} a^{-k}\Big(\sum_{(\alpha_1,\ldots,\alpha_k)} \prod \partial^{\alpha_i} a\Big).\end{equation}
Here the sum runs over tuples of multiindices with $|\alpha_i|\ge 1$, \ $\sum |\alpha_i|\le 2n$ and $\sum (|\alpha_i|-2)_+\le 2n-k$. As always, there are $\ge n$ derivatives hitting the $\xi$ variables. It follows that 
\[\|r_n\|_{-n,k}\lesssim_{\!\! n,k}^{\!\! \mathcal C} \varepsilon^{-(2n+k)}h^{-(2n+k)\sigma}.\]

\textbf{2.} We define a sequence of formal symbols $(r^{\diamond_n})$ by induction as follows:
\[\mathbf r^{\diamond_0}=1, \qquad \mathbf r^{\diamond_{n+1}}=\mathbf r\diamond \mathbf r^{\diamond_n} \quad \text{for $n\ge 0$}.\]
In particular $\mathbf r^{\diamond_1}=\mathbf r$. From \eqref{eq: formule r_n}, one can show by induction that for all $n\ge 1$, for all $k\ge 0$, we have
\begin{equation} \label{eq: formule (r diamond n)} (\mathbf r^{\diamond_n})_k=\sum_{2n\le j\le 2k} a^{-j}\sum_{(\alpha_1,\ldots,\alpha_{j})}\prod_{i=1}^j \partial^{\alpha_i}a,\end{equation}
where the sum runs over tuples satisfying $\sum |\alpha_i|=2k$ and $\sum (|\alpha_i|-2)_+\le 2k-j$.
Now, define 
\[\mathbf c:=1+\mathbf r+\ldots+\mathbf r^{\diamond_N}+\ldots=\sum_{n\ge 0} \mathbf r^{\diamond_n}.\]
Note that $\mathbf c\diamond (1-\mathbf r)=1$ as formal power series. Writing $\mathbf c=\sum (\mathrm{i}\underline h)^n c_n$, we deduce from \eqref{eq: formule (r diamond n)} that
\begin{equation} \label{eq: express cn}c_0=1, \quad c_1=0, \qquad c_n=\sum_{2\le j\le 2n} a^{-j}\sum_{(\alpha_1,\ldots,\alpha_{j})}\prod_{i=1}^j \partial^{\alpha_i}a\quad \text{for $n\ge 2$},\end{equation}
where the sum runs over tuples satisfying $\sum |\alpha_i|=2n$ and $\sum (|\alpha_i|-2)_+\le 2n-j$. Again one may check that
\[\|c_n\|_{-n,j}\lesssim_{\!\! n,j}^{\!\!\mathcal C} \varepsilon^{-(2n+j)} h^{-(2n+j)\sigma}.\]
Let us compute
\[\operatorname{Op}_h^{M}(a)\operatorname{Op}_h^{M}(T_N(b_0\diamond \mathbf c))=\sum_{n\le N} (\mathrm{i} h)^n\operatorname{Op}_h^{M}(a)\operatorname{Op}_h^{M}((b_0\diamond \mathbf c)_n).\]
For each $n$, we use the composition formula with the fact that 
\[\|(b_0\diamond \mathbf c)_n\|_{-m-n,k}\lesssim_{\!\! n,k}^{\!\!\operatorname{Pol}(\mathcal C,\varepsilon^{-1})} h^{-\gamma} h^{-\sigma(2n+k)}\] 
to obtain
\begin{align*}\operatorname{Op}_h^{M}(a)\operatorname{Op}_h^{M}((b_0\diamond \mathbf c)_n)& =\operatorname{Op}_h^{M}\big(T_{N-n}( a\diamond (b_0\diamond \mathbf c)_n)\big)\\ & +h^{-2n\sigma}h^{-\gamma}\mathcal O^{\operatorname{Pol}(\mathcal C,\varepsilon^{-1})}_{N,r,s}(h^{(N-n)(1-2\sigma)})_{H_h^r\to H_h^{s}}.\end{align*}
Summing over $k+n<N$, we obtain
\[\operatorname{Op}_h^{M}(a)\operatorname{Op}_h^{M}(T_N(b_0\diamond \mathbf c))=\mathrm{Op}_h^{M}\Big(\sum_{n+k< N}(\mathrm{i}h)^{n+k}\big(a\diamond (b_0\diamond \mathbf{c})_n\big)_k\Big)+\mathcal O^{\mathcal C,\varepsilon^{-1}}_{N,r,s}(h^{-\gamma} h^{N(1-2\sigma)})_{H_h^r\to H_h^s}.\]
By the associativity property of $\diamond$, we have
\[\sum_{n+k\le N}(\mathrm{i}h)^{n+k}\big(a\diamond (b_0\diamond \mathbf{c})_n\big)_k=\sum_{n+k\le N}(\mathrm{i}h)^{n+k}\big((a\diamond b_0)_n\diamond \mathbf{c}\big)_k=T_N((1-\mathbf r)\diamond \mathbf c)=T_N(1)=1.\]
Since $\operatorname{Op}_h^{M}(1)=\operatorname{Id}$, by letting $\mathbf b:=b_0\diamond \mathbf c$ we have obtained 
\[\operatorname{Op}_h^{M}(a)\operatorname{Op}_h^{M}(T_N(\mathbf b))=\operatorname{Id}+\mathcal O^{\operatorname{Pol}(\mathcal C,\varepsilon^{-1})}_{N,r,s}(h^{-\gamma} h^{N(1-2\sigma)})_{H_h^r\to H_h^s}.\]
To conclude the proof, note that by \eqref{eq: express cn}, the formal symbol $\mathbf b$ is indeed of the form $\mathbf b=\sum (\mathrm{i}\underline h)^nb_n$ with $b_0=\frac 1a$, $b_1=0$ and for $n\ge 2$:
\[b_n=\sum_{3\le 2k\le 2n} a^{-k}\sum_{(\alpha_1,\ldots,\alpha_k)}\prod \partial^{\alpha_i}a.\]
\textbf{3.} Let us show that $\mathbf b$ does not feature odd powers of $(\mathrm{i}\underline h)$. The idea is to express the adjoint of $\operatorname{Op}_h^{M}(a)^{-1}$ in two different ways. Fix $a\in S^m$, assume $|a|\ge \varepsilon\langle \xi\rangle^m$ and let $\mathcal C=\{\|a\|_{m,k}\}$. Then for $h$ small enough depending on $\mathcal C$ and $\varepsilon$ (see Corollary \ref{cor: exact inverse} below), the operator $\operatorname{Op}_h^{M}(a):C^\infty(M)\to C^\infty(M)$ is invertible. On the one hand
\[\big[\operatorname{Op}_h^{M}(a)^{-1}\big]^*=\mathrm{Op}_h^{M}(T_N(\mathbf b))^*+\mathcal O(h^N)=\operatorname{Op}_h^{M}(\overline{T_N(\mathbf b)})+\mathcal O_{N,\varepsilon}^{\operatorname{Pol}(\mathcal C)}(h^N).\]
On the other hand, 
\[\big[\operatorname{Op}_h^{M}(a)^{-1}\big]^*=[\operatorname{Op}_h^{M}(a)^*]^{-1}=\operatorname{Op}_h^{M}(\overline{a})^{-1}=\operatorname{Op}_h^{M}\Big(\sum_{n<N} (\mathrm{i} h)^n \overline{b_n}\Big)+\mathcal O_{N,r,m,\varepsilon}^{\operatorname{Pol}(\mathcal C)}(h^N)_{H_h^r\to H_h^{r+m}}.\]
By equating the above expressions, we deduce
\[\sum (-\mathrm{i} h)^n \overline{b_n}=\sum (\mathrm{i} h)^n \overline{b_n}+\mathcal O_{N,r,m,\varepsilon}^{\operatorname{Pol}(\mathcal C)}(h^N)_{H_h^r\to H_h^{r+m}}.\]
By letting $h\to 0$, we obtain $b_n=0$ for all $n$ odd.
\end{proof}
This implies the next corollary.
\begin{cor}[exact invertibility] \label{cor: exact inverse} Let $\mathcal C'=\{C_k'\}_{k\in \mathbf N}$ be a sequence of nonnegative numbers, and let $M,\varepsilon>0$ and $\gamma\in [0,1/2)$. Then there exists $h_0=h_0(\mathcal C',M,\varepsilon,\gamma,\delta)>0$ such that the following holds. Let $a\in S^m$, and take $i\in \{1,2\}$. We assume
\begin{itemize}
\item $|m|\le M$.
\item $a(x,\xi)\ge \varepsilon h^\gamma\langle \xi\rangle^m$.
\item For $i\in \{1,2\}$, we define $\mathcal C_i$ and $\sigma_i$ as in Corollary \ref{cor: bounds a inverse}. Writing $\mathcal C_i=\{C_{i,k}\}_{k\in \mathbf N}$, we assume $C_{i,k}\le C_i'$.
\end{itemize}
Then, for $0<h<h_0$, the operator $\operatorname{Op}_h^{M}(a):H_h^r\to H_h^{r-m}$ is invertible, with inverse satisfying
\[\|\operatorname{Op}_h^{M}(a)^{-1}\|_{H_h^{r-m}\to H_h^r}\lesssim_{\!\! M,r}^{\!\!\operatorname{Pol}(\mathcal C_i,\varepsilon^{-1})} 1.\]
In particular, for any $m\in \mathbf{R}$, provided $h$ is small enough depending on $m$, the operator $\operatorname{Op}_h^{M}(\langle \xi\rangle^m)$ is invertible, and there is a constant $C_m\ge 1$ which is independent of $h$ such that
\[C_m^{-1}\|\operatorname{Op}_h^{M}(\langle \xi\rangle^m)u\|_{L^2}\le \|u\|_{H_h^m(M)}\le C_m\|\operatorname{Op}_h^{M}(\langle \xi\rangle^m)u\|_{L^2}.\]
\end{cor}
\begin{proof}Let $\mathbf b$ be the inverse of $a$ in $\mathcal A[\![\mathrm{i} \underline h]\!]$. Take $N$ large enough and let $Q=\operatorname{Op}_h^{M}(T_N(\mathbf b))$. Then $Q$ is continuous $H_h^r\to H_h^{r+m}$ for all $r\in \mathbf{R}$. Moreover, by Theorem \ref{thm: parametrix}, for $s\le r-N$ we have
\[\operatorname{Op}_h^{M}(a)Q=\operatorname{Id}-R \qquad \text{with} \quad R=\mathcal O_{N,r,s}^{\mathcal C_i,\varepsilon^{-1}}(h^{-\gamma} h^{N(1-2\sigma_i)})_{H_h^r\to H_h^s}.\] 
We take $N=N(\gamma,\delta)$ large enough to ensure $N(1-2\sigma_i)-\gamma>0$. The series $S:=\sum_{k\ge 0} R^k$ converges in $\mathcal L(H_h^s,H_h^s)$ whenever $\|R\|_{H_h^s\to H_h^s}\le \frac 12$ in which case $\|S\|_{H_h^s\to H_h^s}\le 2$. For $s=0$ this is true whenever $0<h<h_0(\mathcal C_i,\varepsilon,N,\gamma,\delta)$. The identity $S=\operatorname{Id}+SR$ shows that $S:H_h^s\to H_h^s$ is continuous for all $s\in \mathbf{R}$. Then,
\[\operatorname{Op}_h^{M}(a) QS=\operatorname{Id},\]
thus $\operatorname{Op}_h^{M}(a):H_h^r\to H_h^{r-m}$ admits a right-inverse $QS$ satisfying 
\[\|QS\|_{H_h^r\to H_h^{r+m}}\lesssim_{\!\! m,r}^{\!\! \operatorname{Pol}(\mathcal C_i,\varepsilon^{-1})} 1.\] A similar argument shows that $\operatorname{Op}_h^{M}(a)$ admits a left-inverse satisfying the same estimates.
\end{proof}
When $a\in S^m$ is a real elliptic symbol, the operator $\operatorname{Op}_h^{M}(a+\mathrm{i})$ is automatically invertible by the spectral theorem. We give estimates on its inverse acting on Sobolev spaces.
\begin{prop}[Inverting $A+\mathrm{i}$] \label{prop: inverting A+i} Let $a\in S^m$ be a real symbol such that $a+\mathrm{i}$ is elliptic in $S^m$, with $|a+\mathrm{i}|(x,\xi)\ge K\langle \xi\rangle^m$. Let $\mathcal C=\{h^{\delta k}\|a\|_{m,k}\}$. Then, for all $s\in \mathbf{R}$, the operator $(\operatorname{Op}_h^{M}(a)+\mathrm{i})^{-1}$ is continuous $H_h^{s-m}\to H_h^{s}$, with
\[\|(\operatorname{Op}_h^{M}(a)+\mathrm{i})^{-1}\|_{H_h^{s-m}\to H_h^{s}}\lesssim_{\!\! m,s}^{\!\! \operatorname{Pol}(\mathcal C,K^{-1})} 1.\]
\end{prop}
\begin{proof}Let $A=\operatorname{Op}_h^{M}(a)$. Since $A$ is self-adjoint on $L^2(M)$ with domain $H_h^m(M)$, the inverse $Q:=(A+\mathrm{i})^{-1}$ is first defined as a continuous operator $L^2\to H_h^r(M)$, which satisfies $\|(A+\mathrm{i})^{-1}\|_{L^2\to L^2}\le 1$. Let $B=\operatorname{Op}_h^{M}((a+\mathrm{i})^{-1})$. Then, by the composition formula of Proposition \ref{prop: composition manifold}, we have
\begin{equation} \label{eq: B(A+i)}B(A+\mathrm{i})=\operatorname{Id}-R_1,\end{equation}
with $R_1:H_h^r\to H_h^{r+1}$. More precisely $\|R_1\|_{H_h^r\to H_h^{r+1}}\lesssim_{\!\! r}^{\!\!\operatorname{Pol}(\mathcal C,K^{-1})} 1$. Multiplying both sides of \eqref{eq: B(A+i)} by $Q$ on the right yields
\[Q=B+R_1Q.\]
Iterating down this formula, one finds for any $n\ge 1$:
\[Q=(\sum_{k<n}R_1^k)B+R_1^nQ.\]
In particular, for $0\le s\le n$,
\begin{equation} \begin{aligned}\|Qu\|_{H_h^s}& \le \Big(\sum_{k<n} \|R_1^k\|_{H_h^s\to H_h^s}\Big)\|Bu\|_{H_h^s}+\|R_1^n\|_{L^2\to H_h^s}\cdot \|Qu\|_{L^2} \\
	& \lesssim_{\!\! s}^{\!\!\operatorname{Pol}(\mathcal C,K^{-1})}\|u\|_{H_h^{s-m}}+\|u\|_{L^2}.\end{aligned}\end{equation}
	Thus, for all $s\ge 0$,
	\[\|Qu\|_{H_h^{s-m}\to H_h^{s}}\lesssim_{\!\! s}^{\!\!\operatorname{Pol}(\mathcal C,K^{-1})}1.\]
	By playing with the composition formula the other way around, one has
	\[Q=B+QR_2,\]
	with $\|R_2\|_{H_h^r\to H_h^{r+1}}\lesssim_{\!\! r}^{\!\!\operatorname{Pol}(\mathcal C,K^{-1})} 1$ for all $r\in \mathbf{R}$. Iterating this down, we infer
	\[Q=B\Big(\sum_{k<n} R_2^k\Big)+ QR_2^n.\]
	Thus, for $s\le 0$, by taking $n\ge -s$ and reasoning as above, we find
	\[\|Q\|_{H_h^{s-m}\to H^s}
	\lesssim_{\!\! m,s}^{\!\!\operatorname{Pol}(\mathcal C,K^{-1})}1.\]
	We have covered the cases $s\ge 0$ and $s\le 0$, which concludes the proof.
	
\end{proof}

\subsection{Functional calculus for pseudodifferential operators on manifolds} Let $a\in S^m(T^*M)$ be a real symbol, such that $a+\mathrm{i}$ is elliptic in $S^m$. Then $\operatorname{Op}_h^{M}(a)$ is a self-adjoint operator on $L^2(M)$ with domain $H_h^m(M)$, and if $f$ is a smooth compactly supported function on $\mathbf{R}$, we can define $f(\operatorname{Op}_h^{M}(a))$ by functional calculus as an operator acting on $L^2(M)$. It turns out that $f(\mathrm{Op}_h^M(a))$ is a pseudodifferential operator on $M$, that extends continuously to an operator $H_h^{-N}\to H_h^N$ for all $N\ge 0$. The proof relies on the Helffer--Sjöstrand formula (see \cite[Theorem 8.1]{dimassi}):
\begin{equation} \label{eq: helffer sjostrand}f(A)=\frac{1}{2\mathrm{i}\pi}\iint_{\mathbf C} \frac{\bar \partial \tilde f(z)}{(z-A)} \mathrm{d}z\wedge \mathrm{d}\bar z,\end{equation}
where $\mathrm{d}z=\mathrm{d}x+\mathrm{i} \mathrm{d}y$ and $\tilde f$ is a quasi-analytic extension of $f$. We recall that a quasi-analytic extension of $f\in C^\infty_{\rm comp}(\mathbf{R},\mathbf{C})$ is a function $\tilde f\in C^\infty_{\rm comp}(\mathbf C,\mathbf C)$ that coincides with $f$ on the real line and that satisfies
\[\forall N\ge 0, \qquad |\bar \partial\tilde f(z)|\lesssim_{\!\!} |\operatorname{Im}(z)|^N.\]
Note that such an extension is not unique.
\begin{thm}[Functional calculus in $S^m_\delta$]\label{thm: functional calculus} Let $a\in S^m(T^*M)$ be a real symbol, and set $\mathcal C=\{h^{\delta k}\|a\|_{m,k}\}$. We make the following ellipticity assumption on $a+\mathrm{i}$: there is some $K\in (0,1]$ such that
\begin{equation} \label{eq: assumption a+i}|a(x,\xi)+\mathrm{i}|\ge K \langle \xi\rangle^m.\end{equation}
Let $f:\mathbf R\to \mathbf C$ be a smooth compactly supported function. Then there is a formal symbol 
\[\bm{f}[a]=\sum_{n\ge 0\ \rm{even}} (\mathrm{i}\underline{h})^n\bm{f}[a]_n\] such that for any $r,s\in \mathbf{R}$,
\[f(\operatorname{Op}_h^{M}(a))=\mathrm{Op}_h^M(T_N(\bm{f}[a]))+\mathcal O_{f,N,r,s}^{\operatorname{Pol}(\mathcal C,K^{-1})}(h^{N(1-2\delta)})_{H_h^r\to H_h^s}.\]
More precisely, $\bm{f}[a]_0=f(a)$, and the higher order coefficients are given by
\begin{equation} \label{eq: cn functional calculus}\bm{f}[a]_n=\sum_{3\le k\le 2n} f^{(k-1)}(a)Q_{2n,k}(\mathbf Da,\ldots,\mathbf D^{2n}a)\ \ \text{for $n\ge 2$ even},  \end{equation}
where $Q_{2n,k}$ has the same form as in Theorem \ref{thm: parametrix} (but is not the same operator). Moreover, given two functions $f$ and $g$ we have 
\[\bm{(f+g)}[a]=\bm{f}[a]+\bm{g}[a]\quad \text{and}\quad \bm{fg}[a]=\bm{f}[a]\diamond \bm{g}[a].\] 
\end{thm}
We state two corollaries of this result.
\begin{cor}Assume that $f(a(\rho))\equiv 1$ on the support of $\chi$. Then,
\[f(\operatorname{Op}_h^{M}(a))\operatorname{Op}_h^{M}(\chi)=\operatorname{Op}_h^{M}(f(a))\operatorname{Op}_h^{M}(\chi)+\mathcal O(h^{\infty})_{L^2\to L^2}.\]\end{cor}
\begin{cor}Assume that $f(a(\rho))=a(\rho)$ on the support of $\chi$. Then,
\[f(\operatorname{Op}_h^{M}(a))\operatorname{Op}_h^{M}(\chi)=\operatorname{Op}_h^{M}(\chi)+\mathcal O(h^{\infty})_{L^2\to L^2}.\]
\end{cor}
\begin{rem}To obtain these corollaries, it is crucial that the sum \eqref{eq: cn functional calculus} does not involve the indices $k=1,2$. This follows from the precise expression of the coefficients in the expansion of $a\diamond b$ which do not feature terms of the form $h^n(\partial^{2n}a)b$ or $h^na(\partial^{2n}b)$, see Proposition \ref{prop: composition manifold}.
\end{rem}

\begin{proof}
\textbf{1.} Let $\tilde f$ be a quasi-analytic extension of $f$. With the Helffer--Sjöstrand formula \eqref{eq: helffer sjostrand} at hand, we are led to studying the resolvent $(\mathrm{Op}_h^M(a-z))^{-1}$. This has been done in the previous subsection, where we constructed an inverse to an elliptic operator, modulo a remainder in $\mathcal O(h^N)_{H_h^r\to H_h^{s}}$. Here, for any $z\in \mathbf C\backslash \mathbf R$, the symbol $a-z$ is elliptic, with elliptic constant $K(z)$ depending on $z$. This raises no issue since we have made explicit the dependencies in the ellipticity constant in Theorem \ref{thm: parametrix}. By assumption \ref{eq: assumption a+i} and since $a$ is real, there is a universal constant $c>0$ such that for all $z\in \mathbf C$, one has
\[|a(x,\xi)-z|\ge c K\frac{|\mathrm{Im}(z)|}{\langle z\rangle} \langle \xi\rangle^m.\]
In particular, since $\tilde f$ has compact support, the denominator is uniformly bounded for $z\in \operatorname{Supp}(\tilde f)$. For $z\in \mathbf C\backslash \mathbf R$, define $\mathbf b(z)=(a-z)^{\diamond_{-1}}$. By Theorem \ref{thm: parametrix}, one has for any $N\ge 0$ and $z\in \operatorname{Supp}(\tilde f)$:
\[\mathrm{Op}_h^M(a-z)\mathrm{Op}_h^M(T_N(\mathbf b(z)))=\mathrm{Id}+\mathcal O_{N,f}^{\operatorname{Pol}(\mathcal C,K^{-1}|\operatorname{Im}(z)|)}(h^{N(1-2\delta)})_{H_h^{-N}\to L^2}.\]
We multiply both sides by $(A-z)^{-1}=(\mathrm{Op}_h^M(a)-z)^{-1}$ on the left, and use the spectral theoretic bound $\|(A-z)^{-1}\|_{L^2\to L^2}\le \frac{1}{|\mathrm{Im}(z)|}$ to get
\begin{equation} \label{eq: (A-z) inverse} (A-z)^{-1}=\mathrm{Op}_h^M(T_N(\mathbf b(z)))+\mathcal O_N^{\operatorname{Pol}(\mathcal C,K^{-1}|\operatorname{Im}(z)|)}(h^{N(1-2\delta)})_{H_h^{-N}\to L^2}.\end{equation}
Since $\tilde f$ has compact support and $|\bar \partial \tilde f(z)|=\mathcal O(|\operatorname{Im}(z)|^\infty)$, by integrating \eqref{eq: (A-z) inverse} against $\bar \partial \tilde f$, one finds
\[f(A)=\frac{1}{2\pi \mathrm{i}}\int_{\mathbf C} \bar \partial \tilde f(z) \mathrm{Op}_h(T_N(\mathbf b(z)))\mathrm{d}z\wedge \mathrm{d}\bar z+\mathcal O_{N,f}^{\operatorname{Pol}(\mathcal C,K^{-1})}(h^{N(1-2\delta)})_{H_h^{-N}\to L^2}.\]
Now, recall from Theorem \ref{thm: parametrix} that $\mathbf b(z)$ is of the form
\[\mathbf b(z)=\frac{1}{z-a} +\sum_{n\ge 2\ \rm{even}}(\mathrm{i}\underline{h})^n\Big(\sum_{3\le k\le 2n}\frac{Q_{2n,k}(\mathbf Da,\ldots,\mathbf D^{2n}a)}{(z-a)^{k}}\Big).\]
Here $Q_{2n,k}(\bullet)$ is a polynomial expression in the derivatives of $a$ as in Theorem \ref{thm: parametrix}. Note that the terms in the numerators do not involve $z$. We use the fact that
\[\int_{\mathbf C} \frac{\bar \partial \tilde f(z)}{(z-a)^k} \mathrm{d}z\wedge \mathrm{d}\bar z=\frac{2\pi \mathrm{i}}{(k-1)!}f^{(k-1)}(a)\]
to obtain
\begin{equation} \label{eq: to obtain f(A)}f(A)=\mathrm{Op}_h^M(T_N(\bm{f}[a]))+\mathcal O_{N,f}^{\operatorname{Pol}(C,K^{-1})}(h^{N(1-2\delta)})_{H_h^{-N}\to L^2},\end{equation}
where $\bm{f}[a]=\sum_{n\ge 0\ \rm{even}} (\mathrm{i}\underline h)^n\bm{f}[a]_n$, with $\bm{f}[a]_0=f(a)$ and
\begin{equation} \label{eq: formula bm{f}[a]}\bm{f}[a]_n=\sum_{3\le k\le 2n} f^{(k-1)}(a)Q_{2n,k}(\mathbf Da,\ldots,\mathbf D^{2n}a)\ \ \text{for $n\ge 2$ even}.\end{equation}
Here the polynomials $Q_{2n,k}$ have changed in order to incorporate the factors $\frac{1}{(k-1)!}$. We will improve the error in \eqref{eq: to obtain f(A)} in step \textbf{3.} below.\medskip

\textbf{2.} Before that, we show that $\bm{(f+g)}[a]=\bm{f}[a]+\bm{g}[a]$ and $\bm{fg}[a]=\bm{f}[a]\diamond \bm{g}[a]$. The first identity is obvious from \eqref{eq: formula bm{f}[a]}. Let us prove the second identity. We fix $a\in S^m$ with $a+\mathrm{i}$ elliptic and will let $h\to 0$. Let us set $\mathcal C'=\{\|a\|_{m,k}\}_{k\in \mathbf N}$. On the one hand,
\begin{equation}\label{eq: f(a)g(a)} \begin{aligned}fg(\operatorname{Op}_h^{M}(a))&  = f(\operatorname{Op}_h^{M}(a))g(\operatorname{Op}_h^{M}(a))\\ & =\operatorname{Op}_h^{M}(T_N(\bm{f}[a]))\operatorname{Op}_h^{M}g(T_N(\bm{g}[a]))+\mathcal O(h^N)_{L^2\to L^2} \\
	& =\operatorname{Op}_h^{M}\big(T_{N}(T_N(\bm{f}[a])\diamond T_N(\bm{g}[a]))\big)+\mathcal O(h^{N})_{L^2\to L^2}\end{aligned}.\end{equation}
	On the other hand
	\begin{equation} \label{eq:fg(a)} \begin{aligned} fg(\operatorname{Op}_h^{M}(a))=\operatorname{Op}_h^{M}(T_{N}(\bm{fg}[a]))+\mathcal O(h^{N})_{L^2\to L^2}\end{aligned}.\end{equation}
	Since $a$ is independent of $h$, by equating \eqref{eq: f(a)g(a)} and \eqref{eq:fg(a)}, we find
	\[\forall N\ge 0, \quad T_{N}(T_N(\bm{f}[a])\diamond T_N(\bm{g}[a]))=T_{N}(\bm{fg}[a])+\mathcal O(h^N).\]
	In other words,
	\begin{equation*} \begin{aligned}
	\sum_{n<N}(\mathrm{i}h)^n\bm{fg}[a]_n& =(T_N(\bm{f}[a])\diamond T_N(\bm{g}[a]))_n+\mathcal O(h^N)\\ & =\sum_{i<N}\sum_{j,k<N} (\mathrm{i}h)^{i+j+k} (\bm{f}[a]_j\diamond \bm{g}[a]_k)_i+\mathcal O(h^N).
	\end{aligned}\end{equation*}
	By fixing $a$ and letting $h\to 0$, we can identify the coefficients in the $h$-expansion, and obtain that for all $n<N$:
	\[\bm{fg}[a]_n=\sum_{i+j+k=n} (\bm{f}[a]_j\diamond \bm{g}[b]_k)_i.\]
	This exactly means $\bm{fg}[a]=\bm{f}[a]\diamond \bm{g}[a]$.

	\textbf{3.} Finally, we improve the error in \eqref{eq: to obtain f(A)}. We go back to the case $\mathcal C=\{h^{\delta k}\|a\|_{m,k}\}$. To obtain a better error, the standard trick is to apply \eqref{eq: to obtain f(A)} with the function $f_k(x)=(x+\mathrm{i})^k f(x)$ instead of $f$, see \cite[Chapter 8]{dimassi}. Then $f_k(A)=(A+\mathrm{i})^k f(A)$ and \eqref{eq: to obtain f(A)} applied to $f_k$ gives
	\begin{equation} \label{eq: A+i k f(A)}(A+\mathrm{i})^k f(A)=\operatorname{Op}_h^{M}(T_N(\bm{f_k}[a]))+\mathcal O_{N,k,f}^{\mathcal C}(h^{N(1-2\delta)})_{H_h^{-N}\to L^2}.\end{equation}
	Letting $g_k:=(x+\mathrm{i})^k$, we have $\bm{f_k}[a]=\bm{g_k}[a]\diamond \bm{f}[a]$ by step \textbf{2}. By Corollary \ref{cor: exact inverse}, we know that $(A+\mathrm{i})^{-k}:L^2\to H_h^{mk}$ is continuous, thus \eqref{eq: A+i k f(A)} becomes
	\begin{equation} \label{eq: f(A)=(A+i)-k bla} f(A)=(A+\mathrm{i})^{-k}\operatorname{Op}_h^{M}(T_N(\bm{g_k}[a]\diamond \bm{f}[a]))+\mathcal O_{N,k}^{\operatorname{Pol}(\mathcal C,K^{-1})}(h^{N(1-2\delta)})_{H_h^{-N}\to H_h^{mk}}.\end{equation}
	Let $\mathbf b:=\bm{g_k}[a]^{\diamond_{-1}}$ be the inverse of $\bm{g_k}[a]$ for the law $\diamond$. Then, by Theorem \ref{thm: parametrix}, \[(A+\mathrm{i})^{-k}=\operatorname{Op}_h^{M}(T_N(\mathbf b))+\mathcal O_{N,k}^{\operatorname{Pol}(\mathcal C,K^{-1})}(h^{N(1-2\delta)}).\] Lastly, for $s\le r-N$, working modulo $\mathcal O_N^{\operatorname{Pol}(\mathcal C,K^{-1})}(h^{N(1-2\delta)})_{H_h^r\to H_h^s}$, we have
	\begin{equation}
\begin{aligned}(A+\mathrm{i})^{-k}\operatorname{Op}_h^{M}(T_N(\bm{g_k}[a]\diamond \bm{f}[a]))& =\operatorname{Op}_h^{M}(T_N(\mathbf b))\operatorname{Op}_h^{M}(T_N(\bm{g_k}[a]\diamond \bm{f}[a])) \\
	& =\operatorname{Op}_h^{M}\big(T_N\big(\mathbf b\diamond T_N(\bm{g_k}[a]\diamond \bm{f}[a])\big)\big)\\
	& =\operatorname{Op}_h^{M}(T_N(\mathbf b\diamond \bm{g_k}[a]\diamond \bm{f}[a])) \\
	&= \operatorname{Op}_h^{M}(T_N(\bm{f}[a])).\end{aligned}
	\end{equation}
	Thus by taking $k$ large enough depending on $N$, \eqref{eq: f(A)=(A+i)-k bla} becomes
	\[f(A)=\operatorname{Op}_h^{M}(T_N(\bm{f}[a]))+\mathcal O_{N,r,s}^{\operatorname{Pol}(\mathcal C,K^{-1})}(h^{N(1-2\delta)})_{H_h^r\to H_h^s},\]
	for $s\le r-N$, which improves \eqref{eq: to obtain f(A)}.
\end{proof}

\subsubsection{Taking more singular functions of $A$} 
One can study more singular functions of $A$, in the spirit of the work of K\"uster \cite{Kuster}. Actually, the tools developed in the present paper allow to generalize the results of K\"uster, who was concerned with Schr\"odinger operators $h^2\Delta_g+V$, to arbitrary elliptic $h$-pseudodifferential operators.

We recall the following result from \cite{Kuster}, which gives a quantitative control on a quasi-analytic extension of $f$.
\begin{lem}[{\cite[\S 3]{Kuster}}] \label{lem: kuster} Let $f\in C^\infty_{\rm comp}(\mathbf{R},\mathbf{R})$ satisfy $\operatorname{diam}(\operatorname{Supp}(f))\le 1$. Then, $f$ admits a quasi-analytic extension $\tilde f$ supported on $|\operatorname{Im}(z)|\le 1$, such that $\operatorname{diam}(\operatorname{Supp}(\tilde f))\le 2$, and such that for all $N\ge 0$, 
\[|\bar \partial_z \tilde f(z)|\le C_N |\operatorname{Im}(z)|^N \|f\|_{C^{N+3}}.\]
The constant $C_N$ is independent of $f$.
\end{lem}

To simplify the statement, in the next theorem, we only consider functions with support of size $\le 1$. The general case follows by using a partition of unity.
\begin{thm}[Functional calculus for singular functions]\label{thm: func calc singular func} Let $\delta,\gamma\in [0,1/2)$. Let $a\in S^m(T^*M)$ satisfy the ellipticity property
\[|a(x,\xi)+\mathrm{i}|\ge K \langle \xi\rangle^m.\]
Let $R\ge 1$. Let $f:\mathbf R\to \mathbf C$ be a smooth compactly supported function in $[-R,R]$, with $\operatorname{diam}(\operatorname{supp}(f))\le 1$. Set $\mathcal C_f=\{h^{\gamma k}\|f\|_{C^k}\}_{k\in \mathbf N}$. Set $\mathcal C_1=\{h^{\delta k}\|a\|_{m,k}\}$ and $\mathcal C_2=\{h^{\delta(k-2)_+}\|a\|_{m,k}\}$, also let $\sigma_1=\delta+\gamma$ and $\sigma_2=\max(\delta,\gamma)$. Then, for $i\in \{1,2\}$, assuming $\sigma_i<1/2$, we have
\[f(\operatorname{Op}_h^{M}(a))=\mathrm{Op}_h^M(T_N(\bm{f}[a]))+\mathcal O^{\mathcal C_i,\mathcal C_f,K^{-1}}_{N,R,r,s}(h^{N(1-2\sigma_i)})_{H_h^r\to H_h^s}.\]
\end{thm}

\begin{proof} The proof is analogous to that of Theorem \ref{thm: functional calculus}. The function $f_\gamma:=f(h^\gamma\cdot)$ is supported on $[-Rh^{-\gamma},Rh^{-\gamma}]$ and satisfies $\|f_\gamma\|_{C^k}\lesssim_{\!\! k,\mathcal C_f} 1$. Let $\omega\in C^{\infty}_{\rm comp}([-2,2],\mathbf{R})$ satisfy $\sum \omega(\cdot-n)=1$ on $\mathbf{R}$. We set $f_{\gamma,n}:=f_\gamma\omega_n$. We now express
\begin{equation} \label{eq: we now express}f(A)=f_\gamma(h^{-\gamma} A)=\sum_n f_{\gamma,n}(h^{-\gamma}A)=\sum_n \frac{1}{2\pi \mathrm{i}}\int \bar \partial \tilde f_{\gamma,n}(z) (z-h^{-\gamma}A)^{-1} \mathrm{d}z\wedge \mathrm{d}\bar z,\end{equation}
where $\tilde f_{\gamma,n}$ is a quasi-analytic extension of $f_{\gamma,n}$, chosen so as to satisfy the estimates of Lemma \ref{lem: kuster}:
\[|\bar \partial \tilde f_{\gamma,n}(z)|\lesssim_{\!\! N,\mathcal C_f}|\operatorname{Im}(z)|^N.\] 
We study separately the terms of the sum. Write
\[f_{\gamma,n}(A)=\frac{h^{\gamma}}{2\pi \mathrm{i}}\int \bar \partial \tilde f_{\gamma,n}(z)(h^{\gamma}z -A)^{-1}\mathrm{d}z\wedge \mathrm{d}\bar z.\]
For $z\in \mathbf C\backslash \mathbf{R}$, the symbol $a-h^\gamma z$ is elliptic with ellipticity constant 
\[K(z)=K \frac{h^\gamma|\operatorname{Im}(z)|}{\langle h^\gamma z\rangle}.\] 	
For $z\in \operatorname{Supp} \tilde f_{\gamma,n}$, one has $|z|\lesssim h^{-\gamma }R$ thus $K(z)\gtrsim K h^{\gamma}|\operatorname{Im}(z)|$. By running again the proof of Theorem \ref{thm: functional calculus} but using the precise version of Theorem \ref{thm: parametrix} where the ellipticity constant is allowed to degenerate with $h$, we obtain
\[f_{\gamma,n}(A)=\operatorname{Op}_h^{M}\big({\bm{f_{\gamma,n}}[h^{-\gamma}a]}\big)+\mathcal O(h^\gamma h^{N(1-2\sigma_i)})_{H_h^r\to H_h^s}.\]
It remains to sum over the $\mathcal O(Rh^{-\gamma})$ terms of the sum \eqref{eq: we now express} to obtain the result.
\end{proof}

\begin{rem}For $\gamma\ge \frac 12$, the operator $f(\operatorname{Op}_h^{M}(a))$ does not fit into the pseudodifferential calculus anymore, in the sense that the errors $\mathcal O(h^{N(1-2\sigma_i)})$ grow as $h\to 0$ and as $N$ increases. This operator should instead be viewed as a Fourier integral operator.
\end{rem}
\subsection{Egorov Theorem for the exotic propagator} We prove an exotic version of the Egorov theorem where the generator of the dynamics is allowed to depend on $h$.

\begin{prop}[Egorov theorem for exotic propagators] \label{prop: egorov exotic} Let $\delta,\nu\in [0,1/2)$ and set $\sigma=\max(\delta,\nu)$. Let $p=p_\delta+h^2r_\delta\in C^\infty_{\rm comp}(T^*M)$ be a real symbol, such that \[\|p_\delta\|_{C^\ell}\lesssim h^{-\delta(\ell-2)_+}, \qquad \|r_\delta\|_{C^\ell}\lesssim h^{-\delta(\ell+2)}.\] Let $\Phi^t$ denote the Hamiltonian flow generated by $p_\delta$. Let $P_h=\mathrm{Op}_h^M(p)$, and $U(t)=\mathrm{e}^{-\frac{\mathrm i}{h} tP_h}$. Let $a\in S^{\rm comp}_\nu(T^*M)$ and set $\mathcal C=\{h^{\nu k}\|a\|_{C^k}\}$. For any $t\in [0,1]$, there is a formal symbol $\mathbf a_t$ such that for any $r,s\in \mathbf{R}$, one has
\[U(-t)\mathrm{Op}_h^M(a)U(t)=\mathrm{Op}_h^M(T_N(\mathbf a_t)\circ \Phi^t)+\mathcal O_{N,r,s}^{\rm{Pol}_1(\mathcal C)}(h^{N(1-2\sigma)})_{H_h^r\to H_h^s},\]
uniformly in $t\in [0,1]$. More precisely, the symbol $\mathbf a_t$ is of the form
\begin{equation} \label{eq: formal at}\mathbf a_t=a+\sum_{n\ge 2}(\mathrm{i}\underline h)^n \Big(\sum_{1\le j\le 2n} B_{j,n}(t)\mathbf D^{j}_n a\Big).\end{equation}
Here $\|B_{j,n}\|_{C^\ell_{t,\rho}}\lesssim h^{-\delta(2n-j)}$, and $\mathbf D^j_n$ is a differential operator of order $j$ such that $\mathbf D^j_n 1=0$.
In particular, writing $T_N(\mathbf a_t)=a+r_N$, one has
\[\|a_t\|_{C^\ell}\lesssim_{\!\! \ell}^{\!\!\rm Pol_1(\mathcal C)} h^{-\sigma \ell}, \quad \|r_N\|_{C^\ell}\lesssim_{\!\! N,\ell}^{\!\! \operatorname{Pol}_1(\mathcal C)} h^{-\sigma(\ell+4)}.\]
\end{prop}
\begin{rem}The motivation for studying such symbols is that if $q$ satisfies $\|q\|_{C^\ell}\lesssim h^{-\delta (\ell-2)_+}$ and $f$ is a volume-preserving diffeomorphism, then, informally,
\[f_*\operatorname{Op}_h^{\rm w}(q)f^*\approx \operatorname{Op}_h^{\rm w}(p_\delta+r_\delta),\]
with $p_\delta=f_*q$ and $r_\delta$ of the above form.
\end{rem}
We first prove a short lemma about the mapping properties of $U(t)$ between Sobolev spaces.
\begin{lem}\label{lem: sobolev continuity propagator} Under the assumptions of Proposition \ref{prop: egorov exotic}, the operator $U(t)$ is continuous on $H_h^m$ for all $m\in \mathbf{R}$, with uniform estimates in $h$:
\[\|U(t)\|_{H_h^m\to H_h^m}\lesssim_{\!\! m} 1.\]\end{lem}
\begin{proof}Start from a smooth initial datum $u\in C^\infty(M)$. Let $\Lambda_m:=\operatorname{Op}_h^{M}(\langle \xi\rangle^m)$ and set $v(t):=\Lambda_m \mathrm{e}^{-\frac{\mathrm{i}}{h} tP_h}\Lambda_m^{-1}u$. Since $p$ has compact support, the operator $P_h=\operatorname{Op}_h^{M}(p)$ is continuous $H_h^r\to H_h^s$ for any $r,s\in \mathbf{R}$. It follows that $v(t)\in C^\infty(M)$. We compute
\begin{equation} \label{eq: d dt norm v}\frac{\mathrm{d}}{\mathrm{d}t}\big\|v(t)\big\|_{L^2(M)}^2= \frac{2}{h} \operatorname{Im} \langle \Lambda_m P_h\Lambda_m^{-1} v(t),v(t)\rangle_{L^2(M)}.\end{equation}
By the composition theorem of Proposition \ref{prop: composition manifold} and the parametrix construction from Theorem \ref{thm: parametrix}, we have
\[\Lambda_m P_h\Lambda_m^{-1}=\operatorname{Op}_h^{M}(p)+\mathcal O(h)_{L^2\to L^2}.\]
Here we have used the fact that $\|p\|_{C^\ell}\lesssim h^{-\delta(\ell-2)_+}$ to get a $\mathcal O(h)$ remainder. Since $p$ is real, the operator $\operatorname{Op}_h^{M}(p)$ is self-adjoint thus $\operatorname{Im}\langle \operatorname{Op}_h^{M}(p)v(t),v(t)\rangle_{L^2(M)}=0$. Going back to \eqref{eq: d dt norm v}, we obtain
\begin{equation} \label{eq: bound d d t v} \frac{\mathrm{d}}{\mathrm{d}t}\big\|v(t)\big\|_{L^2(M)}^2\le C_m\|v(t)\|_{L^2(M)}^2,\end{equation}
for some $h$-independent constant $C_m$. By Gr\"onwall's inequality, \eqref{eq: bound d d t v} implies $\|v(t)\|_{L^2(M)}^2\le \mathrm{e}^{C_mt}\|v(0)\|_{L^2(M)}^2$, i.e. $\|U(t)u\|_{H_h^m}\le \mathrm{e}^{C_mt} \|u\|_{H_h^m}$. This holds for any $u\in C^\infty(M)$, thus by density $U(t):H_h^m\to H_h^m$ is continuous and 
\[\|U(t)\|_{H_h^m\to H_h^m}\le \mathrm{e}^{C_mt}.\]
\end{proof}
We turn to the proof of Proposition \ref{prop: egorov exotic}.
\begin{proof}[Proof of Proposition \ref{prop: egorov exotic}] We look for $a_t$ of the form $a_t=\sum_{n\ge 0} h^j a_{t}^{(n)}$. Let $V(t):=U(t) \mathrm{Op}_h^M(a_t\circ \Phi^t)U(-t)$. Let us compute
\begin{align*}
\partial_t V(t)= & -\frac{\mathrm{i}}{h} P_hU(t)\mathrm{Op}_h^M(a_t\circ \Phi^t)U(-t) +U(-t)\mathrm{Op}_h^M(\dot a_t\circ \Phi^t + \mathrm da\cdot \partial_t \Phi^t) U(t)\\ & +\frac{\mathrm{i}}{h} U(t)\mathrm{Op}_h^M(a_t\circ \Phi^t) P_hU(-t).
\end{align*}
It implies, using the fact that $[U(t),P_h]=0$:
\[U(-t)\big[\partial_tV(t)\big]U(t)=-\frac{\mathrm{i}}{h} P_h\mathrm{Op}_h^M(a_t\circ \Phi^t) +\mathrm{Op}_h^M(\dot a_t\circ \Phi^t+ \mathrm da\cdot \partial_t \Phi^t)+\frac{\mathrm{i}}{h}\mathrm{Op}_h^M(a_t\circ \Phi^t) P_h. \]
Recall that $P_h=\mathrm{Op}_h^M(p)$ and $\mathrm da_t\cdot \partial_t \Phi^t=H_{p} a_t$, to get
\begin{equation} \label{eq: U(-t) dt V U(t)}U(-t)\big[\partial_tV(t)\big]U(t)=\frac{\mathrm{i}}{h} \big[\mathrm{Op}_h^M(a_t\circ \Phi^t),\mathrm{Op}_h^M(p)\big]+\mathrm{Op}_h^M(\dot a_t\circ \Phi^t+ H_{p_\delta}a_t).\end{equation}
We will prove in \eqref{eq: borne rang n proof egorov} below that for our choice of $a_t$ we have $\|a_t\|_{C^\ell}\lesssim h^{-\sigma \ell}$. Since $\|\Phi^t\|_{C^\ell}\lesssim h^{-\delta(\ell-1)_+}$ we obtain $\|a_t\circ \Phi^t\|_{C^\ell}\lesssim h^{-\sigma \ell}$. By the commutator formula of Proposition \ref{prop: commutator PDOs}, since all functions have compact support in $M$, one has for all $r,s\in \mathbf{R}$:
\begin{equation}\label{eq: by the commutator formula}\begin{array}{ll}\displaystyle \frac{\mathrm{i}}{h} \big[\mathrm{Op}_h^M(a_t\circ \Phi^t),\mathrm{Op}_h^M(p)\big]=& \displaystyle \mathrm{Op}_h^M\Big(\{ a_t\circ \Phi^t;p\} +\sum_{j=2}^{N-1}(\mathrm{i} h)^j \mathbf D^{2j-2}(\mathrm{d}^2(a_t\circ \Phi^t)\otimes  \mathrm{d}^2 p)_{|\rm diag}\Big)\\ & \displaystyle+\mathcal O_{N,r,s}(h^{N(1-2\sigma)-1})_{H_h^r\to H_h^s}.\end{array}\end{equation}
By the estimates on $\Phi^t$ and $p$, it follows from Fa\`a di Bruno's formula that
\begin{equation} \label{eq: it then follows from FdB}\sum_{j=2}^{N-1}(\mathrm{i} h)^j \mathbf D^{2j-2}(\mathrm d^2 (a_t\circ \Phi_t)\otimes \mathrm d^2p)_{|\rm diag}=\sum_{j=2}^{N-1}(\mathrm{i} h)^j\sum_{1\le |\alpha|\le 2j} (c_{j,\alpha} \partial^\alpha a_t)\circ \Phi^t,\end{equation}
for some coefficients $c_{j,\alpha}$ depending on $t,x,\xi$, satisfying
\begin{equation} \label{eq: cjalpha proof egorov} \|c_{j,\alpha}\|_{C^\ell(M\times [0,1))}\lesssim h^{-\delta(2j+\ell-|\alpha|)}.\end{equation}
We split the term $\{a_t\circ \Phi^t,p\}$ in \eqref{eq: by the commutator formula} as
\[\{a_t\circ \Phi^t,p\}=\{a_t\circ \Phi^t,p_\delta\}+h^2\{a_t\circ \Phi^t,r_\delta\}.\]
The first term $\{a_t\circ \Phi^t,p_\delta\}$ cancels out with $H_{p_\delta}a_t$ in \eqref{eq: U(-t) dt V U(t)}, while the second term $h^2\{a_t\circ \Phi^t,r_\delta\}$ can be incorporated in the sum \eqref{eq: it then follows from FdB} at the index $(j,|\alpha|)=(2,1)$. Indeed, it suffices to check that $\{a_t\circ \Phi^t,r_\delta\}=\partial r_\delta \partial \Phi^t\partial a_t$ and note that $\partial r_\delta\partial \Phi^t$ is bounded in $C^\ell$ norm by $h^{-\delta(\ell+3)}$.
Gathering everything, we find
\begin{equation} \label{eq: U(-t)dt V(t)U(t)} \begin{aligned}
	U(-t)\big[\partial_tV(t)\big]U(t)&= \displaystyle \mathrm{Op}_h^M\Big(\dot a_t\circ \Phi^t+\sum_{j=2}^{N-1} (\mathrm{i} h)^{j} \sum_{1\le |\alpha|\le 2j}  c_{j,\alpha}(\partial^\alpha a_t)\circ \Phi^t \Big)\\ & +\mathcal O_{N,r,s}(h^{N(1-2\sigma)-1})_{H_h^r\to H_h^s}.
	\end{aligned}\end{equation}
	We look for $a_t=T_N(\mathbf a_t)$ where $\mathbf a_t:=\sum_{n\ge 0} (\mathrm{i}h)^n a_t^{(n)}$ is a formal symbol. By injecting this in \eqref{eq: U(-t) dt V U(t)} and requiring that all the coefficients in the resulting formal $h$ expansion vanish, 
	we find that $(a_t^{(n)})_{n\ge 0}$ has to solve:
	\begin{equation} \label{eq: for a_t^n}\partial_t a_t^{(0)}=0, \qquad \partial_t a_t^{(n)}=-\sum_{\substack{i+j=n, \ j\ge 2 \\ 1\le |\alpha|\le 2j} }  c_{j,\alpha}\partial^\alpha a_{t}^{(i)}.\end{equation}
	The solution to \eqref{eq: for a_t^n} with initial condition $a_0=a$ is given by
	\begin{equation} \label{eq: def recursion egorov}a_t^{(0)}=a,\quad  a_t^{(1)}=0, \quad \text{and for $n\ge 2$: } a_t^{(n)}:=-\sum_{\substack{i+j=n, \ j\ge 2 \\ \  1\le |\alpha|\le 2j}}\int_0^t c_{j,\alpha}\partial^\alpha a_s^{(i)}\mathrm ds.\end{equation}
	An induction based on \eqref{eq: def recursion egorov} and \eqref{eq: cjalpha proof egorov} shows that $a_t^{(n)}$ is of the claimed form \eqref{eq: formal at}. It also follows that
	\begin{equation} \label{eq: borne rang n proof egorov} \|a_t^{(n)}\|_{C^\ell(T^*M)}\lesssim_{\!\! n,\ell} h^{-\sigma(2n+\ell)}, \quad \text{and}\quad \|a_t\|_{C^\ell}\lesssim_{\!\! N,\ell} h^{-\sigma \ell}.\end{equation}
	We set $a_t=T_N(\mathbf a_t)$. Then  \eqref{eq: U(-t) dt V U(t)} becomes
	\[U(-t)\big[\partial_tV(t)\big]U(t)=\mathcal O(h^{N(1-2\sigma)-1})_{H_h^r\to H_h^s}.\]
	By Lemma \ref{lem: sobolev continuity propagator}, we infer
	\begin{equation} \label{eq: partial V(t) is small}\partial_t V(t)=\mathcal O(h^{N(1-2\sigma)-1})_{H_h^r\to H_h^s}.\end{equation}
	Recall that $V(0)=\operatorname{Op}_h^{M}(a)$. For $t\in [0,1]$, by integrating \eqref{eq: partial V(t) is small} over $[0,t]$, and using Lemma \ref{lem: sobolev continuity propagator}, we deduce
	\[U(-t)\operatorname{Op}_h^{M}(a)U(t)=\operatorname{Op}_h^{M}(a_t\circ \Phi^t)+\mathcal O(h^{N(1-2\sigma)-1})_{H_h^r\to H_h^s}.\]
	This concludes the proof.
\end{proof}

\subsection{Egorov theorem for time-dependent symbols} \label{subsec: egorov time dep}

We give another version of the Egorov theorem for time-dependent symbols. For $M=\mathbf{R}^d$, it is essentially contained in \cite[Chapter 11]{zworski2012semiclassical}.

\begin{prop}[Egorov theorem for time-dependent symbols] \label{prop: egorov bis} Let $M=\mathbf{R}^d$ or be a closed manifold. Let $\{q_t\}$ be a piecewise smooth family of uniformly bounded (in $C^k$ norm for $k$ large enough) compactly supported real symbols, and let $U(t)$ solve
\[\left\{\begin{array}{ll} U(0)=\mathrm{Id}, \\ hD_t U(t)=\operatorname{Op}_h^{M}(q_t)U(t).\end{array}\right.\]
One defines a smooth family of symplectomorphisms $\{\kappa_t\}_{t\in [0,1]}$ of $T^*M$ by solving
\[\left\{\begin{array}{ll} \kappa_0=\mathrm{Id}, \\ \partial_t\kappa_t=H_{q_t},\end{array}\right.\]
where $H_{q_t}$ is the Hamiltonian vector field on $T^*M$ associated with $q_t$. Then, for any smooth symbol $a\in S_\nu^{\rm comp}(T^*M)$, there is a formal symbol $\mathbf b$ such that
\[U(1)^* \operatorname{Op}_h^{M}(a)U(1)=\operatorname{Op}_h^{M}(T_N(\mathbf b)\circ \kappa_1)+\mathcal O(h^{N(1-2\nu)})_{L^2}.\]
The symbol is of the form
\begin{equation} \label{eq: bf b for time dep}\mathbf b=a+\sum_{n\ge 2}(\mathrm{i} \underline h)^n \Big(\sum_{1\le j\le 2n} \mathbf D^j_n a\Big).\end{equation}
In particular, writing $T_N(\mathbf b)=a+r_N$, we have:
\begin{itemize}
\item if $\|a\|_{C^\ell}\lesssim h^{-\nu \ell}$ then $\|r\|_{C^\ell}\lesssim h^2h^{-\nu(\ell+4)}$.
\item if $\|a\|_{C^\ell}\lesssim h^{-\nu(\ell-2)_+}$ then $\|r\|_{C^\ell}\lesssim h^2 h^{-\nu(\ell+2)}$.
\end{itemize}
\end{prop}

One may reverse the perspective and consider the following situation. Let $\kappa:U_0\subset T^*\mathbf{R}^d\to U_1\subset T^*\mathbf{R}^d$ be a symplectomorphism sending $(0,0)$ to $(0,0)$. We assume that $U_0$ is relatively compact and star-shaped with respect to $(0,0)$. By \cite[Theorem 11.4]{zworski2012semiclassical}, $\kappa$ extends to a global symplectomorphism $\tilde \kappa$ of $T^*\mathbf{R}^d$, that coincides with the identity outside a compact set. Moreover, there is a piecewise smooth deformation of symplectomorphisms $\{\tilde \kappa_t\}_{t\in [0,1]}$ of $T^*\mathbf{R}^d$, and a piecewise smooth family of uniformly compactly supported symbols $q_t\in C^\infty_{\rm comp}(T^*\mathbf{R}^d,\mathbf{R})$, such that
\[\partial_t\tilde \kappa_t=H_{q_t}\tilde \kappa_t, \quad \tilde \kappa_0=\operatorname{Id}, \quad \tilde \kappa_1=\tilde \kappa.\]
Then, Proposition \ref{prop: egorov bis} produces a unitary operator $\mathcal T_\kappa:L^2(\mathbf{R}^d)\to L^2(\mathbf{R}^d)$ such that for any $a\in S_\nu^{\rm comp}(T^*\mathbf{R}^d)$, one has
\[\mathcal T_\kappa^*\mathrm{Op}_h^{\rm w}(a)\mathcal T_\kappa=\mathrm{Op}_h^{\rm w}\big(T_N(\mathbf b)\circ \tilde \kappa\big)+\mathcal O_{N,r,s}(h^{N(1-2\nu)})_{H_h^r\to H_h^s},\]
with $\mathbf b$ as in \eqref{eq: bf b for time dep}. In particular, if $a\in C^\infty_{\rm comp}(U_0)$ and $\chi\in C^\infty_{\rm comp}(U_1)$ then
\[\mathcal T_\kappa^*\mathrm{Op}_h^{\rm w}(a)\mathcal T_\kappa=\mathrm{Op}_h^{\rm w}(\chi(T_N(\mathbf b)\circ \kappa))+\mathcal O_{N,r,s}(h^{N(1-2\nu)})_{H_h^r\to H_h^s}.\]
The operator $\mathcal T_\kappa$ is called a \emph{local semiclassical quantization} of $\kappa$. Finally, if the graph of $\kappa$ is projectable, then we can find a generating function $\varphi(x,\eta)$ of $\tilde \kappa:(y,\eta)\to (x,\xi)$, and arrange $\mathcal T_\kappa$ to have a FIO representation
\[\mathcal T_\kappa u(x)=\frac{1}{(2\pi h)^d}\int_{\mathbf{R}^d}\mathrm{e}^{\frac{\mathrm{i}}{h} (\varphi(x,\eta)-\langle y,\eta\rangle)}\alpha_h(x,\eta)u(y)\mathrm{d}y\mathrm{d}\eta,\]
with $\alpha_h(x,\eta)=|\partial_{x,\eta}^2\varphi|^{1/2}+\mathcal O(h)$ --- of course one can always multiply $\alpha_h$ by a complex number of modulus $1$ and obtain another operator quantizing $\kappa$.

For more on this we refer to \cite[\S 4.1]{NZ09}. Note that when $\kappa$ is a linear symplectomorphism of $T^*\mathbf{R}^d$, it can be quantized into a \emph{metaplectic operator} that yields an exact Egorov Theorem.

\subsection{Counting eigenvalues} We give a Weyl type upper bound for the counting function of eigenvalues of elliptic pseudodifferential operators with exotic symbols.

\begin{prop}\label{prop: rough bounds counting func} Let $\mathcal C=(C_k)$ be a sequence of nonnegative numbers. Let $m,K,R>0$. Let $a\in S^m$ be a real symbol, satisfying the ellipticity property $a(x,\xi)\ge K\langle \xi\rangle^m-R$. We suppose $\|a\|_{m,k}\le C_k h^{-\delta k}$. Then, $\mathrm{Op}_h^M(a)$ has discrete spectrum $\{\lambda_j(h)\}$ on $L^2(M)$, and there is a constant $C>0$ such that
\begin{equation} \label{eq: bound eigenvalues}\#\{\lambda_j(h)\le \lambda\}\lesssim_{\!\! m,\mathcal C,K,R}h^{-d}(\lambda+C)^{d/m}.\end{equation}
\end{prop}
\begin{rem}Proposition \ref{prop: rough bounds counting func} does not simply follow from the functional calculus because Theorem \ref{thm: functional calculus} only allows counting eigenvalues in bounded energy windows. Instead, we use a min-max principle to compare eigenvalues of $\operatorname{Op}_h^{M}(a)$ to those of the Laplacian. 
\end{rem}
\begin{proof}By the ellipticity assumption on $a$, we can find constants $C,c>0$ depending on $K,R$ and $\mathcal C$ such that the operator $a+C- c(1+h^2\Delta_g)^{m/2}$ is nonnegative, in the sense that
\[\forall u\in C^\infty(M), \quad \langle \mathrm{Op}_h^{\rm w}(a)u,u\rangle_{L^2(M)}\ge \langle (c(1+h^2\Delta_g)^{m/2}-C)u,u\rangle_{L^2(M)}.\]
Letting $\{\lambda_j(h)\}$ and $\{\mu_j\}$ denote the ordered spectra of $\mathrm{Op}_h^{\rm w}(a)$ and $\Delta_g$, it follows from a min-max inequality that
\[\lambda_j(h)\ge c(1+h^2\mu_j)^{m/2}-C.\]
The upper bound \eqref{eq: bound eigenvalues} then follows from the Weyl's law
\[\#\{\mu_j\le x\}\lesssim 1+x^{d/2}.\]
\end{proof}

The next lemma allows controlling the trace class norm of a continuous operator $H_h^s(M)\to H_h^{s+r}(M)$ provided $r>d$ where $d$ is the dimension of $M$. 
\begin{lem}\label{lem: convert sobolev estimate into trace class estimate} Assume that $A:H_h^s\to H_h^{s+r}$ is a continuous operator for all $s$, with $r>d$. Then, $A$ is trace-class on $L^2(M)$ and
\[\|A\|_{\rm tr}\lesssim h^{-d}\|A\|_{L^2\to H_h^{r}} \qquad \big({\mathrm{or}~\lesssim h^{-d} \|A\|_{H_h^{-r}\to L^2}}\big).\]\end{lem}
\begin{proof}We take $s=0$ above. Then, one has
\[\|A\|_{\rm tr}=\|(1+h\sqrt{\Delta})^{-r}\cdot (1+h\sqrt{\Delta})^r A\|_{\rm tr}\lesssim \|(1+h\sqrt{\Delta})^{-r}\|_{\rm tr}\cdot \|A\|_{L^2\to H_h^r}.\]
Now, by Weyl's law, provided $r>d$, the operator $(1+h\sqrt{\Delta})^{-r}$ is trace-class and
\[\|(1+h\sqrt{\Delta})^{-r}\|_{\rm tr}=\sum_{\mu_j\in \mathrm{Spec}(\sqrt{\Delta})} \frac{1}{(1+h\mu_j)^{r}}\lesssim_{\!\! r} h^{-d}.\]
\end{proof}
	
	\printbibliography
\end{document}